%% file: MemBNPP.tex
\documentclass{memo-l}

\usepackage{amsfonts,amssymb,latexsym}

\usepackage{amsmath,amscd,amsthm}

\usepackage{pstricks,pst-node}

\usepackage{wrapfig}

\makeatletter

\renewcommand{\@seccntformat}[1]{\bf\@nameuse{the#1}.\quad}

\renewcommand\section{\@startsection{section}{1}%
                  \z@{.7\linespacing\@plus\linespacing}{.5\linespacing}%
                  {\normalfont\bfseries \boldmath}}

\renewcommand\subsection{\@startsection{subsection}{2}%
                  \z@{.5\linespacing\@plus.7\linespacing}{-.5em}%
                  {\normalfont\bfseries \boldmath}}

\renewcommand\subsubsection{\@startsection{subsubsection}{3}%
                  \z@{.3\linespacing\@plus.5\linespacing}{-.5em}%
                  {\normalfont\bfseries \boldmath}}

\makeatother

\theoremstyle{plain}

\newtheorem{theorem}{Theorem}[chapter]

 \newtheorem{assump}[theorem]{Assumption}

 \newtheorem{lem}[theorem]{Lemma}

 \newtheorem{cor}[theorem]{Corollary}

 \newtheorem{prop}[theorem]{Proposition}

\theoremstyle{definition}

 \newtheorem{rem}[theorem]{Remark}

\newcommand{\gl}{\mathfrak{g}}

\newcommand{\ul}{\mathfrak{u}}

\newcommand{\uj}{\mathfrak{u}_J}

\newcommand{\pj}{\mathfrak{p}_J}

\newcommand{\lj}{\mathfrak{l}_J}

\newcommand{\Uz}{\mathcal{U}_{\zeta}}

\newcommand{\hh}{{\mathfrak h}}

\newcommand{\Fr}{{\text{\rm Fr}}}

\newcommand{\ncf}{\mathcal{N}_1(\mathfrak{g}_{F})}

\newcommand{\ch}{\text{\rm ch}}

\newcommand{\ind}{\operatorname{ind}}

\newcommand{\Ext}{\operatorname{Ext}}

\newcommand{\opH}{\operatorname{H}}

\newcommand{\Hom}{\operatorname{Hom}}

\newcommand{\ga}{\gamma}

\newcommand{\la}{\lambda}

\newcommand{\al}{\alpha}

\newcommand{\be}{\beta}

\newcommand{\si}{\sigma}

\newcommand{\bZ}{{\mathbb Z}}

\newcommand{\BU}{{\mathbb U}}

\newcommand{\gr}{\text{\rm gr}\,}

\newcommand{\ad}{\operatorname{Ad}}

\newcommand{\adw}{\overline{\operatorname{Ad}}}

\newcommand{\Id}{\text{\rm Id}}

\makeindex

\begin{document}

\frontmatter

\title{Cohomology for quantum groups via the geometry of the nullcone}



\author{\sc Christopher P. Bendel}
\address
{Department of Mathematics, Statistics and Computer Science\\
University of
Wisconsin-Stout \\
Menomonie\\ WI~54751, USA}
\thanks{Research of the first author was supported in part by NSF
grant DMS-0400558} \email{bendelc@uwstout.edu}

\author{\sc Daniel K. Nakano}
\address
{Department of Mathematics\\ University of Georgia \\
Athens\\ GA~30602, USA}
\thanks{Research of the second author was supported in part by NSF
grant DMS-1002135} \email{nakano@math.uga.edu}

\author{\sc Brian J. Parshall}
\address
{Department of Mathematics\\ University of Virginia \\
Charlottesville\\ VA~22903, USA}
\thanks{Research of the third author was supported in part by NSF
grant DMS-1001900} \email{bjp8w@virginia.edu}

\author{\sc Cornelius Pillen}
\address{Department of Mathematics and Statistics \\ University of South
Alabama\\
Mobile\\ AL~36688, USA} \email{pillen@jaguar1.usouthal.edu}

\

\date{}



\begin{abstract}
Let $\zeta$ be a complex $\ell$th root of unity for an odd integer
$\ell>1$. For any complex simple Lie algebra $\mathfrak g$, let
$u_\zeta=u_\zeta({\mathfrak g})$ be the associated ``small" quantum
enveloping algebra. This algebra is a finite dimensional Hopf
algebra which can be realized as a subalgebra of the Lusztig
(divided power) quantum enveloping algebra $U_\zeta$ and as a
quotient algebra of the De Concini--Kac quantum enveloping algebra
${\mathcal U}_\zeta$. It plays an important role in the
representation theories of both $U_\zeta$ and ${\mathcal U}_\zeta$
in a way analogous to that played by the restricted enveloping
algebra $u$ of a reductive group $G$ in positive characteristic $p$
with respect to its distribution and enveloping algebras. In
general,
 little is known about the representation theory of quantum groups
(resp., algebraic groups) when $l$ (resp., $p$) is smaller than the
Coxeter number $h$ of the underlying root system. For example,
Lusztig's conjecture concerning the characters of the 
rational irreducible $G$-modules stipulates that 
$p \geq h$. The main result in this paper provides a
surprisingly uniform answer for the cohomology algebra
$\opH^\bullet(u_\zeta,{\mathbb C})$ of the small quantum group.
When $\ell>h$, this cohomology algebra has been calculated by
Ginzburg and Kumar \cite{GK}. Our result requires powerful tools
from complex geometry and a detailed knowledge of the geometry of
the nullcone of $\mathfrak g$. In this way, the methods point out
difficulties present in obtaining similar results for the restricted
enveloping algebra $u$ in small characteristics, though they do
provide some clarification of known results there also. Finally, we
establish that if $M$ is a finite dimensional $u_\zeta$-module, then
$\opH^\bullet(u_\zeta,M)$ is a finitely generated
$\opH^\bullet(u_\zeta,\mathbb C)$-module, and we obtain new results
on the theory of support varieties for $u_\zeta$.
\end{abstract}

\maketitle


\setcounter{page}{4}

\tableofcontents


\include{Introduction}

\mainmatter

\include{chapt1BNPP}

\include{chapt2BNPP}

\include{chapt3BNPP}

\include{chapt4BNPP}

\include{chapt5BNPP}

\include{chapt6BNPP}

\include{chapt7BNPP}

\include{chapt8BNPP}

\appendix
  \include{chapt9BNPP}

\backmatter

\bibliographystyle{amsalpha}

\printindex

\end{document}

%% file: Introduction.tex


\chapter*{Introduction }
\label{introsec}


Let ${\mathfrak g}_F$ be a finite dimensional, restricted Lie
algebra (as defined by Jacobson) over an algebraically closed field
$F$ of positive characteristic $p$, with restriction map $x\mapsto
x^{[p]}$, $x\in {\mathfrak g}_F$. The restricted enveloping algebra
$u:=u({\mathfrak g}_F)$ of ${\mathfrak g}_F$ is a finite dimensional
cocommutative Hopf algebra. In general, the cohomology algebra $\opH^\bullet(u,F)$ is difficult to compute. However,
Suslin-Friedlander-Bendel \cite{SFB1,SFB2} proved that, putting $A:=\opH^{2\bullet}(u,F)$ (the commutative subalgebra
of the cohomology algebra concentrated in even degrees),  the
(algebraic) scheme $\text{Spec}\,A$ is homeomorphic to the closed
subvariety ${\mathcal N}_1({\mathfrak g}_F):=\{x\in {\mathfrak g}_F:\ x^{[p]}=0\}$.
We call ${\mathcal N}_1({\mathfrak g}_F)$ the restricted nullcone of ${\mathfrak g}_F$;
it is a closed subvariety of the full nullcone ${\mathcal
N}({\mathfrak g}_F)$ which consists of all $[p]$-nilpotent elements
in ${\mathfrak g}_F$.

When ${\mathfrak g}_F$ is the Lie algebra of a reductive algebraic
group $G$ over $F$, the above results can be considerably sharpened.
For example, if $p>h$ (the Coxeter number of $G$), then
$\opH^{2\bullet}(u,F)\cong F[{\mathcal N}_1({\mathfrak g}_F)]$, the
coordinate algebra of ${\mathcal N}_1({\mathfrak g}_F)$ (cf.
Friedlander-Parshall \cite{FP2} and Andersen-Jantzen \cite{AJ}).
In addition, the condition $p\geq h$ implies that ${\mathcal
N}_1({\mathfrak g}_F)= {\mathcal N}({\mathfrak g}_F)$. However, when
$p\leq h$, there is no known calculation of $\opH^\bullet(u,F)$
(apart from some small rank cases). For all primes, ${\mathcal
N}_{1}({\mathfrak g}_F)$ is an irreducible variety \cite[(6.3.1)
Cor.]{NPV}, \cite[Thm. 4.2]{UGA2} and is the closure of a $G$-orbit.
These orbits have been determined in \cite{CLNP}, \cite[Thm.
4.2]{UGA2}.

Now let ${\mathfrak g}={\mathfrak g}_{\mathbb C}$ be a complex
simple Lie algebra, and let $U_\zeta=U_\zeta({\mathfrak g})$ be the
quantum enveloping algebra (Lusztig form) associated to $\mathfrak
g$ at a primitive $l$th root of unity $\zeta\in\mathbb C$. We regard
$U_\zeta$ as an algebra over ${\mathbb C}$ obtained by base change
from the quantum enveloping algebra over the cyclotomic field
${\mathbb Q}(\zeta)$. Here $l>1$ is an odd integer, not divisible
by 3 when $\mathfrak g$ has type $G_2$. The representation theory of
$U_\zeta$, when $l=p$, models or approximates the representation
theory of $G$. For example, part of the ``Lusztig program" to
determine the characters of the irreducible $G$-modules amounts to
showing, when $l=p\geq h$, that the characters of the irreducible
$G$-modules having high weights in the so-called Jantzen region
coincide with the characters of the analogous irreducible modules
for $U_\zeta$.  This result has been proved
by Andersen-Jantzen-Soergel \cite{AJS} for all large $p$ 
(no lower bound known except for small rank). Recently, Fiebig
\cite{Fie}, improving upon the methods of \cite{AJS}, 
has provided a specific (large) bound on $p$ for each root system 
sufficient for the validity of the Lusztig character formula.
The original conjecture for $p\geq h$ remains open.

For all $l$, Lusztig \cite{L1} has also introduced an analog,
denoted $u_\zeta:=u_\zeta({\mathfrak g})$, of the restricted
enveloping algebra $u$. Like $u$, $u_\zeta$ is a finite dimensional
Hopf algebra, though it is in general not cocommutative. It plays a
role in the representation theory of $U_\zeta$ much like that played
by $u$ in the representation theory of $G$. For
$l>h$, Ginzburg-Kumar \cite{GK} have calculated the cohomology
algebra $\opH^\bullet(u_\zeta,{\mathbb C})$. By an exact analogy with
the cohomology algebra of of $u$, they prove that
$\opH^{2\bullet}(u_\zeta,{\mathbb C}) \cong {\mathbb C}[{\mathcal
N}({\mathfrak g})]$, the coordinate algebra of the nullcone
${\mathcal N}({\mathfrak g})$ consisting of nilpotent elements in
$\mathfrak g$. Recently, Arkhipov-Bezrukavnikov-Ginzburg \cite[\S
1.4]{ABG}, taking \cite{GK} as a method to pass from the
representation theory of quantum groups to the geometry of the
nullcone, provide a proof of Lusztig's character formula for
$U_\zeta$ when $l>h$. These important connections have made the
small quantum group $u_{\zeta}$ an object of significant interest
(cf. \cite{ABBGM}, \cite{AG}, \cite{Be}, \cite{Lac1,Lac2}).

This paper presents new results on the cohomology of $u_\zeta$. In
particular, we compute the cohomology algebra $\opH^\bullet(u_\zeta,
{\mathbb C})$ in most of the remaining cases when $l\leq h$. Our results are
explicitly described in Section 1.2 below. We prove in Chapters 2--5
that $\opH^{2\bullet+1}(u_\zeta,{\mathbb C})=0$, while
$\opH^{2\bullet}(u_\zeta,{\mathbb C})$ is isomorphic to, in most
cases, the coordinate algebra of an explicitly described closed
subvariety ${\mathcal N}(\Phi_{0})$ of ${\mathcal N}({\mathfrak g})$
(constructed in a similar way to the variety ${\mathcal
N}_1({\mathfrak g}_F)$ discussed above). In Chapter 2, we rigorously
develop and present new results on the cohomology theory of
parabolic subalgebras for quantum groups. The application of
powerful tools from complex geometry represents at least one
advantage that the quantum enveloping algebra situation (in characteristic 
zero) has over
that in positive characteristic. In Chapters 3 and 5, we demonstrate
how the Grauert-Riemenschneider theorem and the normality 
of certain orbit closures in ${\mathcal N}({\mathfrak g})$,
namely, the varieties ${\mathcal N}(\Phi_0)$, play a vital
role in carrying out our cohomology calculations. Some of the major 
input in our calculations occurs in Chapter 4, where we analyze 
the combinatorics involving the multiplicity of the 
Steinberg module in certain cohomology modules related to unipotent 
radicals of parabolic subalgebras. 

In Chapter 6, we prove that $R:=\opH^\bullet(u_\zeta,{\mathbb C})$ is a
finitely generated ${\mathbb C}$-algebra. Also, we show that 
if $M$ is a finite dimensional
$u_{\zeta}$-module, then $\opH^\bullet(u_\zeta,M)$ is finitely generated
over $R$.  Chapter 7 adapts some of our methods to
the cohomology algebra $\opH^\bullet(u,F)$ in positive
characteristic, making some ad hoc computations in \cite{AJ} more
transparent. In particular, we can identify key vanishing results,
known over ${\mathbb C}$, and as yet unproved in positive
characteristic, which would be sufficient to extend the cohomology 
calculations to positive characteristic $p$.

Finally, by building on results in Chapter 6, we define support
varieties in the quantum setting in Chapter 8 and exhibit some new
calculations on support varieties. The theory of support varieties
for the restricted Lie algebra ${\mathfrak g}_F$ attached to a
reductive group $G$ in positive characteristic provides evidence for
beautiful connections between the representation theory of $G$ and
the structure and geometry of the restricted nullcone ${\mathcal
N}_1({\mathfrak g}_F)$ (\cite{NPV}, \cite{CLNP}, \cite{UGA1},
\cite{UGA2}). The results of this paper strongly reinforce this
expectation (as do those in \cite{ABG} mentioned above).

A number of applications and further developments which heavily depend on the 
foundational results of this paper have arisen since this manuscript first appeared as a preprint.
Drupieski \cite{Dr1,Dr2} has used the methods of this paper to investigate the cohomology and 
representation theory for Frobenius-Lusztig kernels of quantum groups. Moreover, Feldvoss and 
Witherspoon have shown how to apply our calculations to determine the representation type of the 
small quantum group \cite{FeW}. 

\vskip.3in
\begin{center} {\bf Acknowledgements}\end{center}

\medskip

The authors acknowledge Leonard Chastkofsky, Christopher Drupieski, William Graham, 
Shrawan Kumar and Eric Sommers for several useful discussions at various stages of
this project. In particular, Drupieski carefully went though the entire manuscript
helping us to correct and improve many results. We also thank Jesper Funch Thomsen for indicating how
recent work of Christophersen can be used to compute
$G_{1}$-cohomology for the group $E_{6}$, and two undergraduate
students Alloys J. Samz and Sara Wood at the University of
Wisconsin-Stout for their assistance with the MAGMA computations
used in this paper. Finally, the first, second, and last authors
acknowledge the hospitality and support of the Department of
Mathematics at the University of Virginia in Spring 2006 and Spring
2007 during work on this project.

%% file: chapt1BNPP.tex
\chapter{Preliminaries and Statement of Results}

 \renewcommand{\thesection}{\thechapter.\arabic{section}}

After introducing some notation, this chapter discusses the main results of this work. These results
are all concerned with the cohomology algebra $\opH^\bullet(u_\zeta({\mathfrak g}),{\mathbb C})$ of the small quantum group $u_\zeta({\mathfrak g})$ at a root $\zeta$ of
unity associated to a complex, simple Lie algebra $\mathfrak g$. Applications are presented to the calculation of support varieties over $u_{\zeta}$ for certain 
classes of modules for the quantum enveloping algebra $U_\zeta$.

\section{Some preliminary notation}\label{notation}
\renewcommand{\thetheorem}{\thesection.\arabic{theorem}}
\renewcommand{\theequation}{\thesection.\arabic{equation}}
\setcounter{equation}{0}
\setcounter{theorem}{0}

Let $\Phi$ be a finite and irreducible root system (in the classical sense). Fix a set
$\Pi=\{\alpha_1,\cdots,\alpha_n\}$ of simple roots (always labeled in the
standard way \cite[Appendix]{Bo}), and let $\Phi^+$ (respectively,
$\Phi^-$) be the corresponding set of positive (respectively, negative)
roots. The set $\Phi$ spans a real
vector space $\mathbb E$ with positive definite inner product
$\langle u,v\rangle$, $u,v\in \mathbb E$, adjusted so that
$\langle\alpha,\alpha\rangle=2$ if $\alpha\in\Phi$ is a short root. If $\Phi$ has
only one root length, all roots in $\Phi$ are both ``short" and ``long". 
Thus,  if $\Phi$ has two root lengths and if $\alpha\in\Phi$ is a long root, then
$\langle\alpha,\alpha\rangle=4$ (respectively, $6$) when $\Phi$ has type
$B_n,C_n,F_4$ (respectively, $G_2$). For $\alpha\in\Phi$, put
$d_\alpha:=\frac{\langle\alpha,\alpha\rangle}{2}\in\{1,2,3\}$.

Write $Q=Q(\Phi):={\mathbb Z}\Phi$ (the root lattice) and
$Q^+=Q^+(\Phi):={\mathbb N}\Phi^+$ (the positive root cone). If $J\subseteq\Pi$, let
$\Phi_J=\Phi\cap{\mathbb Z}J$. If $J=\varnothing$, put $\Phi_J=\varnothing$. If $J\not=\varnothing$, $\Phi_J$ is a closed subroot system
of $\Phi$. By definition, a nonempty subset $\Phi'\subseteq\Phi$ is a closed subroot system provided that, given $\alpha\in\Phi'$,
it holds that $-\alpha\in\Phi'$, and, given $\alpha,\beta\in\Phi'$ such that $\alpha+\beta\in\Phi$, it
holds that $\alpha+\beta\in\Phi'$.  As discussed below, for many of the closed subsystems $\Phi'$ important
in this work, it is possible to find
$w\in W$ and $J\subseteq \Pi$ so that $w(\Phi')=\Phi_J$; in other words, a simple set of roots for $\Phi'$ can be extended
to a simple set of roots for $\Phi$. If $\alpha=\sum_{i=1}^n m_i\alpha_i\in\Phi$, its
height is defined to be $\text{ht}\,\alpha:=m_1+\cdots+ m_n$.

For $\alpha\in\Phi$, write
$\alpha^\vee=\frac{2}{\langle\alpha,\alpha\rangle}\alpha$ for the corresponding coroot. Therefore,
$\Phi^\vee :=\{\alpha^\vee : \alpha\in\Phi\}$ is the dual root system in $\mathbb E$
defined by $\Phi$. We will always denote the short root of maximal height in $\Phi$ by $\alpha_0$; thus, $\alpha_0^\vee$ is the
unique long root of maximal length in $\Phi^\vee$.  For $1\leq i,j\leq n$, put
$c_{i,j}=\langle\alpha_j,\alpha_i^\vee\rangle\in\mathbb Z$. The
Cartan matrix $C:=[c_{i,j}]$ of $\Phi$ is symmetrizable in the sense that $DC$ is a
symmetric matrix, letting
$D$ be the diagonal matrix $\text{diag}[d_{\alpha_1},\cdots,d_{\alpha_n}]$.

Let $X_+\subset
\mathbb E$ be the positive cone of dominant weights, i.e., $X_+$ consists of all
$\varpi\in \mathbb E$ satisfying
$\langle\varpi,\alpha_i^\vee\rangle\in{\mathbb N}$, $i=1,\cdots, n$.
Define the fundamental dominant weights $\varpi_1,\cdots,\varpi_n\in X_+$
by the condition that $\langle\varpi_i,\alpha_j^\vee\rangle=\delta_{i,j}$, for $1\leq i,j\leq n$. Therefore,
$X_+={\mathbb N}\varpi_1\oplus\cdots\oplus {\mathbb N}\varpi_n$. For convenience, we will occasionally
let $\varpi_0$ denote the $0$-weight.
Put
$X={\mathbb Z}\varpi_1\oplus \cdots\oplus{\mathbb Z}\varpi_n$ (the
weight lattice). For  $\varpi\in X$, $\alpha\in\Phi$, $\langle\varpi,\alpha\rangle=d_\alpha\langle\varpi,\alpha^\vee\rangle\in\mathbb Z$.
 The weight lattice $X$ is partially ordered by putting $\lambda\geq
\mu$ if and only if $\lambda-\mu\in Q^+$.

The Weyl group $W$ of $\Phi$ is the finite
group of orthogonal transformations of $\mathbb E$ generated by
the reflections $s_\alpha:{\mathbb E}\to{\mathbb E}$, $\alpha\in\Phi$, defined by $s_\alpha(u)=
u-\langle u,\alpha^\vee\rangle \alpha$, $u\in\mathbb E$. If
$S:=\{s_{\alpha_1},\cdots, s_{\alpha_n}\}$, then $(W,S)$ is a
Coxeter system. Let $\ell:W\to{\mathbb N}$ be the length function on
$W$; thus, if $w\in W$, $\ell(w)$ is the smallest integer $m$ such that $w=s_{\beta_1}\cdots s_{\beta_m}$ for
 $\beta_i\in\Pi$.

Let $l$ be a fixed positive integer. For  $\alpha\in\Phi,m\in{\mathbb Z}$, let $s_{\alpha,m}:{\mathbb
E}\to{\mathbb E}$ be be the affine transformation defined by
$$s_{\alpha,m}(u)=u-(\langle
u,\alpha^\vee\rangle-ml)\alpha,\quad\forall u\in\mathbb E.$$
 If $m=0$, then $s_{\alpha,m}=s_{\alpha,0}=s_\alpha\in W$.
The affine Weyl group $W_l$
is the subgroup of the group  $ {\text{Aff}}({\mathbb E})$ of affine transformations of $\mathbb E$ generated by the reflections
$s_{\alpha,m}$, $\alpha\in\Phi, m\in\mathbb Z$.
Let $S_l :=\{s_{\alpha_1},\cdots,
s_{\alpha_n},s_{\alpha_{0},-1}\}$. Then $(W_l,S_l)$ is a Coxeter system. If $lQ$ is identified as the subgroup of
$\text{Aff}({\mathbb E})$ consisting of translations by $l$-multiples of elements in $Q$, then $W_l\cong W\ltimes lQ$. The extended affine Weyl group
$\widetilde W_l$ is obtained by putting $\widetilde W_l=W\ltimes
lX$. Although $\widetilde W_l$ need not be a Coxeter group, it contains $W_l$ as a
normal subgroup satisfying $\widetilde W_l/W_l\cong X/Q$. The
quotient map $\widetilde W_l\to W\cong\widetilde W_l/lX$ is denoted by
$w\mapsto \overline w$.

We will generally use the  ``dot" action of $W_l$ on $\mathbb E$: for $w\in W_l$, $u\in\mathbb E$,   put $w\cdot
u:=w(u+\rho)-\rho$,  where
$\rho:=\varpi_1+\cdots+\varpi_n=\frac{1}{2}\sum_{\alpha\in\Phi^+}\alpha\in X_+$
is the Weyl weight.

The Coxeter number of $\Phi$ is defined  to be $h=\langle
\rho,\alpha_{0}^{\vee} \rangle +1={\text{ht}}(\alpha_0^\vee) +1$. Thus, $h-1$ is the height
of the maximal root in $\Phi^\vee,$ or, equivalently, in $\Phi$. The integer $h$ plays an important role
in this paper, and often serves as a kind of ``event horizon" in representation theory.

Suppose that $A$ and $B$ are augmented algebras (over some common
field). In case $A$ is a subalgebra of $B$, it is called
{\it normal} in $B$ provided that $BA_+ = A_+B$, where $A_+$ denotes the augmentation
ideal of $A$. In this situation, we form the augmented algebra $B//A := B/I$, where $I :=
BA_+$, a two-sided ideal in $B$. It will sometimes be convenient to write $A\unlhd B$ to indicate that $A$
 is a normal subalgebra of $B$. If $A$ is a normal subalgebra in $B$, there is, of course, a spectral
 sequence which relates the cohomology of $B$ with that of $A$ and $A//B$. It will play an important role later
 in this paper. See Lemma~\ref{ABG} for more details.

\section{Main results}\label{mainresults}
\renewcommand{\thetheorem}{\thesection.\arabic{theorem}}
\renewcommand{\theequation}{\thesection.\arabic{equation}}
\setcounter{equation}{0}
\setcounter{theorem}{0}

 Let $G=G_{\mathbb C}$ be the connected, simple, simply connected
algebraic group over ${\mathbb C}$ with Lie algebra ${\mathfrak g}={\mathfrak g}_{\mathbb C}$
and root system $\Phi$ with respect to a fixed maximal torus $T\subset G$. Let ${\mathfrak t}={\text{Lie}}\,T$ be the corresponding maximal toral
subalgebra of $\mathfrak g$. Given $\alpha\in\Phi$, let ${\mathfrak
g}_\alpha\subset \mathfrak g$ be the $\alpha$-root space. Put ${\mathfrak
b}^+=\mathfrak t\oplus\bigoplus_{\alpha\in\Phi^+}{\mathfrak
g}_\alpha$ (the positive Borel subalgebra of $\mathfrak g$), and ${\mathfrak b}=
\mathfrak t\oplus\bigoplus_{\alpha\in\Phi^-}{\mathfrak g}_\alpha$
(the opposite Borel subalgebra).

For $J\subseteq \Pi$, let ${\mathfrak l}_J={\mathfrak
t}\oplus\bigoplus_{\alpha\in\Phi_J}{\mathfrak g}_\alpha$ be the Levi
subalgebra containing $\mathfrak t$ and having root system $\Phi_J$.
Then ${\mathfrak p}_J={\mathfrak l}_J\oplus{\mathfrak u}_J\supseteq
{\mathfrak b}$ is a parabolic subalgebra of $\mathfrak g$, where
${\mathfrak
u}_J=\bigoplus_{\alpha\in\Phi^-\backslash\Phi_J^-}{\mathfrak
g}_\alpha$ is the nilpotent radical of ${\mathfrak p}_J$. The parabolic subgroup
of $G$ with Lie algebra ${\mathfrak p}_J$ is denoted $P_J$. The Levi factor of $P_J$ having Lie algebra
${\mathfrak l}_J$ is denoted $L_J$. In particular, $B:=P_J$ where $J=\varnothing$ is the standard
Borel subgroup of $G$ containing $T$ (as its Levi factor) and corresponding to the negative roots.

The group $G$ acts on its
Lie algebra $\mathfrak g$ via the adjoint action. Under this action, $G$ stabilizes the subset ${\mathcal N}={\mathcal N}({\mathfrak g})$ of nilpotent elements in
${\mathfrak g}$---recall that $x\in{\mathfrak g}$ is nilpotent provided that $d\phi(x)$ acts as a nilpotent operator
 for any finite dimensional faithful rational representation $\phi:G\to \text{GL}(V)$. We will refer to $\mathcal N$ as the nullcone of ${\mathfrak g}$.  For $J\subseteq\Pi$, define
$${\mathcal N}(\Phi_{J}):=G\cdot {\mathfrak u}_{J}\subseteq{\mathcal N}\subset\mathfrak g.$$
Because the quotient variety $G/P_J$ is complete, an elementary geometric argument \cite[p. 68]{St} establishes that ${\mathcal N}(\Phi_J)$ is a closed subvariety of $\mathfrak g$, which is
clearly irreducible (since $G$ is irreducible). Also, ${\mathcal N}(\Phi_J)$ is $G$-stable under the restriction of the adjoint
action of $G$; it has codimension in $\mathfrak g$ equal to $\dim\,{\mathfrak l}_J$. In
particular, ${\mathcal N}(\Phi_\varnothing)=\mathcal N$ is a closed, irreducible (and normal) subvariety of $\mathfrak g$ of codimension $n={\text{rank}}({\mathfrak g})$.

The varieties
${\mathcal N}(\Phi_J)$ play a central role in this paper (for particular subsets $J\subseteq \Pi$ which depend on a choice of an integer $l$ introduced below).   They can also be
described in another way. It is known that the parabolic subgroup $P_J$ with Lie algebra 
${\mathfrak p}_J$ has an open, dense
orbit in ${\mathfrak u}_J$. Choose any element $x$ belonging to this orbit, and let ${\mathcal O}={\mathcal C}_J$ 
be the $G$-orbit of $x$. Then ${\mathcal N}(\Phi_J)$ is the Zariski closure
$\bar{\mathcal O}$ in ${\mathfrak g}$ of $\mathcal O$.  The orbits $\mathcal O$ in $\mathcal N$ which arise
as some ${\mathcal C}_J$, $J\subseteq\Pi$, are called Richardson orbits. If $J=\varnothing$, then 
${\mathcal N}(\Phi_J)={\mathcal N}$ is  the closure of the unique regular orbit in $\mathcal N$. 
If the Levi factors $L_J$ and $L_K$ are $G$-conjugate (i.e., if
the subsets $J,K$ of $\Pi$ are $W$-conjugate), then \cite[Thm. 2.8]{JR} establishes that 
${\mathcal C}_J={\mathcal C}_K$. Thus,
\begin{equation}\label{JRresult}
\text{\rm if}\,\,J,K\subseteq \Pi\,\,{\text{\rm are}}\,\,W{\text{\rm -conjugate, then}}\,\, {\mathcal N}(\Phi_J)={\mathcal N}(\Phi_K).\end{equation}

Let $l>1$ be a fixed positive integer. Let $\zeta\in\mathbb C$ be a primitive $l$th root of unity, and let
$U_\zeta:=U_{\zeta}({\mathfrak g})$ be the (Lusztig) quantum enveloping algebra attached to the complex
simple Lie algebra $\mathfrak g$ with root system $\Phi$. This algebra is obtained from a certain integral
form of the generic quantum group over ${\mathbb Q}(q)$ by specializing $q$ to $\zeta$.   
In addition, $U_\zeta$ has a Hopf algebra structure, which will
be more fully discussed in Section 2.2 below. Furthermore, $U_\zeta$ has a subalgebra $u_\zeta=
u_\zeta({\mathfrak g})$ which is often called the ``small" quantum group. It has finite
dimension equal to $l^{\dim\mathfrak g}$. As an augmented subalgebra of $U_\zeta$, the small 
quantum group $u_{\zeta}$ is normal. 
In fact, $U_\zeta//u_\zeta\cong {\mathbb U}({\mathfrak g})$, where ${\mathbb U}({\mathfrak g})$ is the
universal enveloping algebra of $\mathfrak g$.

It will usually be necessary to impose some mild restrictions on the integer $l$. These restrictions
are indicated in the two assumptions below. Throughout we will be careful to indicate  when the restrictions
are in force.

\begin{assump}\label{assumption} The integer $l$ is odd and greater than $1$.
If the root system
$\Phi$ has type $G_2$, then $3$ does not divide $l$.
In addition,  $l$ is not a bad prime for $\Phi$. (See
Section 3.1 for the definition of a bad prime.)
\end{assump}

Let
\begin{equation}\label{definePhi0}
\Phi_{0}=\Phi_{0,l}:=\{\alpha\in\Phi\,|\,\langle\rho,\alpha^\vee\rangle\equiv
0\,\text{\rm mod}\,l\}.
\end{equation}
If $l$ satisfies Assumption \ref{assumption}, then $\Phi_0$ is either the empty set $\varnothing$ or a closed subroot system of $\Phi$. When $\Phi_0$ is
 not empty,  there exists a non-empty subset $J\subseteq \Pi$ and an element
$w\in W$ such that
$$
w(\Phi_0)=\Phi_J, \text{ and, moreover, such that } w(\Phi_0^+) = \Phi_J^+.
$$
See Theorem \ref{identificationtheorem}. 
 If $w'\in W$
and $J'\subseteq \Pi$ also satisfy $w'(\Phi_0)=\Phi_{J'}$, then (\ref{JRresult}) guarantees
that ${\mathcal
N}(\Phi_0)={\mathcal N}(\Phi_{J'})$. Therefore, we can write ${\mathcal N}(\Phi_{0}):={\mathcal N}(\Phi_{J})$ as a well-defined closed subvariety
of the nullcone $\mathcal N$.

If $l\geq h$, then $\Phi_0=\varnothing$. If $\ell< h$, an explicit subset $J\subseteq\Pi$ such that $\Phi_0$ is a $W$-conjugate of $\Phi_J$ is
given in Chapter 3 for each of the classical types $A,B,C,D$. Similar information for the exceptional types 
$E_6,E_7,E_8, F_4,G_2$ (together with an explicit $w\in W$ such that $w(\Phi_0)=\Phi_J$)
is collected in Appendix A.1.

We will often require the following  additional
restriction on $l$.

\begin{assump}\label{assumption2} Let $l$ be a fixed positive integer
satisfying the Assumption~\ref{assumption}. Moreover, assume that when
the root system $\Phi$ is of type $B_{n}$ or type $C_{n}$, then $l>3$.
\end{assump}

 The theorem below presents a computation of the cohomology algebra
$\opH^\bullet(u_\zeta,{\mathbb C})=\Ext^\bullet_{u_\zeta}({\mathbb C},{\mathbb C})$ of the small quantum group $u_\zeta$.
It is a well-known result (see \cite[p. 232]{Mac} or \cite[Cor. 5.4]{GK}) that, given any Hopf algebra $H$ over a field $k$, the
cohomology algebra $\opH^\bullet(H,k)$ is a graded-commutative algebra (so that, in particular, the subalgebra concentrated in even
degrees is a commutative algebra over $k$). Additionally, because the algebras $u_\zeta$ are normal subalgebras of  $U_\zeta$, the
cohomology spaces $\opH^i(u_\zeta, {\mathbb C})$, for any non-negative integer $i$, acquire a natural action of $U_\zeta//u_\zeta\cong {\mathbb U}({\mathfrak g})$, the
universal enveloping algebra of ${\mathfrak g}$; see \cite[Lem. 5.2.1]{GK}.  Since $u_\zeta$ is finite dimensional, each space $\opH^i(u_\zeta,{\mathbb C})$ is therefore
a finite dimensional ${\mathfrak g}$-module, and thus it is a finite dimensional rational $G$-module. In the following result, we
identify $\opH^\bullet(u_{\zeta},{\mathbb C})$ as a rational $G$-module. In fact, in most cases, we can identify it as a rational $G$-algebra.

This theorem extends the main result due to Ginzburg and Kumar \cite[Main Thm.]{GK},  which calculated the cohomology algebra $\opH^\bullet(u_\zeta,{\mathbb C})$
in the special case in which $l>h$. Specifically, their work established that $\opH^\bullet(u_\zeta,{\mathbb C})=\opH^{2\bullet}(u_\zeta,{\mathbb C})$ is concentrated in even degrees, and it is isomorphic to the
coordinate algebra ${\mathbb C}[{\mathcal N}]$ of the nullcone $\mathcal N$ of $\mathfrak g$.

\begin{theorem}\label{MainThm} Let $l$ be as in Assumption \ref{assumption} and 
let $u_\zeta=u_\zeta({\mathfrak g})$ be the small 
quantum group associated to the complex simple Lie algebra 
$\mathfrak g$ with root system $\Phi$ at a primitive $l$th root of unity $\zeta$. Choose
$w \in W$ and $J \subseteq \Pi$ such that $w(\Phi_0^+) = \Phi_J^+$. Assume that $J$ is as listed
in Chapter 3 (for the classical cases) and in Appendix A.1 (for the exceptional cases). Then
the following identifications hold as $G$-modules under Assumption \ref{assumption}. 
Furthermore, in part(b)(i), the identification is as rational $G$-algebras under the additional Assumption \ref{assumption2}.
\begin{itemize}
\item[(a)] The odd degree cohomology vanishes (i.e., $\opH^{2\bullet+1}(u_{\zeta},{\mathbb C})=0$). 
\item[(b)] The even degree cohomology algebra $\opH^{2\bullet}(u_\zeta,
{\mathbb C})$ is described below.
\begin{itemize}
\item[(i)] Suppose that $l\nmid n+1$ when $\Phi$ is of type $A_{n}$ and $l \neq 9$ when $\Phi$ is of type $E_6$. Then
$$\opH^{2\bullet}(u_{\zeta},{\mathbb C})
\cong \operatorname{ind}_{P_{J}}^{G} S^{\bullet}({\mathfrak
u}_{J}^{*})\cong {\mathbb C}[G\times^{P_{J}}{\mathfrak u}_{J}]$$ 
where $G\times^{P_{J}}{\mathfrak u}_{J}$ is defined in Section 3.7. 
If we assume further that $l\neq 7,9$ when $\Phi$ is of type $E_{8}$,
then
$$\opH^{2\bullet}(u_{\zeta},{\mathbb C})\cong {\mathbb C}[{\mathcal N}(\Phi_{0})].$$
\item[(ii)] If $\Phi$ is of type $A_{n}$ and $l\mid n+1$ with $n+1=l(m+1)$, then
$$\opH^{2\bullet}(u_{\zeta},{\mathbb C})
\cong \operatorname{ind}_{P_{J}}^{G} \left( \bigoplus_{t=0}^{l-1}
S^{\frac{2\bullet-(m+1)t(l-t)}{2}} ({\mathfrak u}_{J}^{*})\otimes
\varpi_{t(m+1)}\right).$$
(Recall the convention that $\varpi_{0}=0.$)
\item[(iii)] If $\Phi$ is of type $E_{6}$ and $l = 9$, one can take $J = \{\al_4\}$. Then
$$\opH^{2\bullet}(u_{\zeta},{\mathbb C})
\cong \operatorname{ind}_{P_{J}}^{G} \left( S^{\bullet}({\mathfrak
u}_{J}^{*}) \oplus (S^{\frac{2\bullet - 12}{2}}({\mathfrak
u}_{J}^{*})\otimes \varpi_1) \oplus (S^{\frac{2\bullet -
12}{2}}({\mathfrak u}_{J}^{*})\otimes \varpi_6)\right).$$
\end{itemize}
\end{itemize}
\end{theorem}

We emphasize that the isomorphisms in (b(ii)) and (b(iii)) are only isomorphisms as rational $G$-modules. After considerable preparation in
Chapters 3 and 4, Theorem \ref{MainThm}
will be proved in Chapter 5 (see Sections 5.5 and 5.7).

 In most cases in Theorem \ref{MainThm}, the algebra $\opH^\bullet(u_\zeta, {\mathbb C})=\opH^{2\bullet}(u_\zeta,{\mathbb C})$ identifies with
 the coordinate algebra ${\mathbb C}[{\mathcal N}(\Phi_0)]$ of the affine variety ${\mathcal N}(\Phi_0)$, and it is, therefore, a finitely generated algebra over $\mathbb C$. However, in Chapter 6, we  show
 that, more generally, the algebra $\opH^{\bullet}(u_\zeta,{\mathbb C})$ is a finitely generated $\mathbb C$-algebra. These remaining
cases require individual arguments. More specifically, we have the following result.

\begin{theorem}\label{MainThm2} Let $l$ be as in Assumption \ref{assumption2}.
\begin{itemize}
\item[(a)] The algebra $\opH^{\bullet}(u_{\zeta},{\mathbb C})=\opH^{2\bullet}(u_\zeta,{\mathbb C})$
 is a finitely generated, commutative ${\mathbb C}$-algebra.
\item[(b)] For any finite dimensional $u_{\zeta}$-module $M$, $\opH^{\bullet}(u_{\zeta},M)$ is finitely generated
as a module for $\opH^{\bullet}(u_{\zeta}, {\mathbb C})$ (where the
action is described in \cite[Rem. 5.3, Appendix]{PW}).
\end{itemize}
\end{theorem}

Theorem \ref{MainThm2} has also recently been
proven by Mastnak, Pevtsova, Schauenberg, and Witherspoon \cite[Corollary 6.5]{MPSW}
only assuming the conditions of Assumption \ref{assumption}. They work in a more
general context of pointed Hopf algebras.

The above theorem points to an important application to the theory of support varieties for $u_\zeta$. Namely, Theorem \ref{MainThm2} implies that, given a finite dimensional
$u_\zeta$-module $M$, the support variety ${\mathcal V}_{\mathfrak
g}(M)$ can be defined as the maximal ideal spectrum
$\text{Maxspec}(R/J_{M})$, letting $J_{M}$ be the annihilator in
$R:=\opH^{\bullet}(u_{\zeta}(\mathfrak g),{\mathbb C})$ for its
natural action on $\operatorname{Ext}^{\bullet}_{u_{\zeta}({\mathfrak
g})}(M,M)$ (cf. \cite[\S 5]{PW}).

For $\lambda\in X_+$, let $\nabla_\zeta(\lambda)=\opH^0_\zeta(\lambda)$ denote the associated induced $U_\zeta$-module. (See
Section \ref{inductionfunctors} for more details on these modules---often called costandard modules.) Of course, each $\nabla_\zeta(\lambda)$ can be
regarded by restriction as a module for the small quantum group $u_\zeta$. Our third major result provides a computation of
 the support varieties of these  modules.

\begin{theorem}\label{MainThm3} Let $l$ be as in Assumption \ref{assumption2}. Let $\lambda\in X_{+}$ and
let $\nabla_\zeta(\lambda)=\opH^0_\zeta(\lambda)$ be the costandard module for
$U_\zeta$ of high weight $\lambda$. Suppose that $(l,p)=1$ for any
bad prime $p$ of $\Phi$. Then there exists a subset $J\subseteq \Pi$
such that
$${\mathcal V}_{\mathfrak g}(\nabla_{\zeta}(\lambda))=G\cdot {\mathfrak u}_{J}.$$
\end{theorem}

Besides the costandard $U_\zeta$-modules, another important class of $U_\zeta$-modules are the Weyl (or standard) modules $\Delta_\zeta(\lambda)$, $\lambda\in X_+$, which can be
defined as duals of the costandard modules $\nabla_\zeta(\lambda)$. Precisely, $\Delta_\zeta(\lambda)=\nabla_\zeta(\lambda^\star)^*$,
where $\lambda^\star$ is the image of $\lambda$ under the opposition involution on $X_+$.  The modules $\nabla_\zeta(\lambda)$
and $\Delta_\zeta(\lambda)$ are both finite dimensional with characters given by Weyl's classical character formula. The
statement of Theorem \ref{MainThm3} remains valid if each $\nabla_\zeta(\lambda)$ is replaced by $\Delta_\zeta(\lambda)$.

Theorem \ref{MainThm3} is proved in Section 8.3, where the subset $J$
is explicitly identified in terms of the closed subset
$\Phi_{\lambda,l}$ of $\Phi$ (defined in (3.1.1)). In particular, the theorem shows
that the support variety of $\nabla_\zeta(\lambda)$ is the closure
of a Richardson orbit in ${\mathcal N}({\mathfrak g})$. We also calculate support varieties 
of $\nabla_\zeta(\lambda)$ in the case of bad $l$. Within these computations we 
present examples, showing that even when $l>h$, if $p\mid l$
where $p$ is a bad prime for $\Phi$, then ${\mathcal V}_{\mathfrak
g}(\nabla_{\zeta}(\lambda))$ need not be the closure of a Richardson orbit. 
The calculation of support varieties given in Theorem \ref{MainThm3} perfectly mirrors similar 
results in \cite{NPV} for the restricted Lie algebra of a simple algebraic group over a field of positive characteristic.


%% file: chapt2BNPP.tex
\chapter{Quantum Groups, Actions, and Cohomology}
\renewcommand{\thesection}{\thechapter.\arabic{section}}

\label{secondsection}

This section lays out the framework for the results in the paper. 
Even though our main results involve the computation of the cohomology 
for the small quantum group $u_{\zeta}({\mathfrak g})$, it will be necessary to construct subalgebras ${\mathbb U}_q({\mathfrak u}_J)$ (and $U_\zeta({\mathfrak u}_J)$) corresponding to nilpotent radicals ${\mathfrak u}_J$
of parabolic subalgebras ${\mathfrak p}_{J}$ of $\mathfrak g$.  This construction
will necessitate the use of Lusztig's construction of a PBW basis for 
the full quantum enveloping algebras ${\mathbb U}_q({\mathfrak g})$. 

We also present results on the adjoint action of the ${\mathcal A}$-form 
(where ${\mathcal A}={\mathbb Q}[q,q^{-1}]$) of the quantum enveloping algebra ${\mathbb U}_q({\mathfrak p}_J)$ of the subalgebras ${\mathfrak p}_{J}$ 
on the ${\mathcal A}$-form on the algebras ${\mathbb U}_q({\mathfrak u}_J)$ associated to 
the unipotent radical ${\mathfrak u}_{J}$. These results will be 
essential for our cohomological calculations.

We will also make use of the DeConcini-Kac quantum enveloping algebra ${\mathcal U}_\zeta={\mathcal U}_\zeta({\mathfrak g})$, as well as its subalgebras ${\mathcal U}_\zeta({\mathfrak p}_J)$ and ${\mathcal U}_\zeta{(\mathfrak u}_J)$. While the (Lusztig) quantum enveloping algebra $U_\zeta(\mathfrak g)$ contains the
small quantum group $u_\zeta({\mathfrak g})$ as a subalgebra, $u_\zeta({\mathfrak g})$ is a 
quotient algebra of ${\mathcal U}_\zeta({\mathfrak g})$.   In the final part of this 
section, we prove several results on the formal characters of 
cohomology groups of the algebras ${\mathcal U}_\zeta({\mathfrak u}_J)$ in the context of an Euler characteristic calculation.


\section{Listings }\label{listingssec}
\setcounter{equation}{0}
\setcounter{theorem}{0}
\renewcommand{\thetheorem}{\thesection.\arabic{theorem}}
\renewcommand{\theequation}{\thesection.\arabic{equation}}
Let $\Phi$
be an arbitrary root system with associated Weyl group $W$. Let
$w_0\in W$ be the longest word, and set $N = |\Phi^+|$. A reduced
expression $w_0=s_{\beta_1}\cdots s_{\beta_N}$ (where the $\beta_i\in \Pi$ are not necessarily distinct)
determines a listing
\begin{equation*}\label{listing}
\gamma_1=\beta_1,\gamma_2=s_{\beta_1}(\beta_2),
\cdots,\gamma_N=s_{\beta_1}\cdots s_{\beta_{N-1}}(\beta_N)
\end{equation*}
of $\Phi^+$. Define a linear ordering on $\Phi^+$ by setting, for
$\beta,\delta\in \Phi^{+}$,
\begin{equation*}\label{ordering}
\beta\prec\delta\iff\beta=\gamma_i, \delta=\gamma_j\,\,{\text{\rm
with}} \,\,\,i<j.\end{equation*}
Of course, this ordering depends on the reduced
expression chosen for $w_0$.

Now let $J\subseteq \Pi$, and put $\Phi_J=\Phi\cap {\mathbb Z}J$,
the subroot system $\Phi_{J}$ generated by $J$. Set $W_J=\langle
s_{\al} : \al \in J\rangle$, the Weyl group of $\Phi_J$, let
$w_{0,J}\in W_J$ be the longest word in $W_J$, and let $w_J :=
w_{0,J}w_0$. Observe that $w_0=w_{0,J}w_J$ satisfies
$\ell(w_0)=\ell(w_{0,J})+\ell(w_J)$. Once the subset $J$ is fixed, it
will often be useful to fix
a reduced expression for $w_0$ by first choosing one for $w_{0,J}$
and then extending it by a reduced expression for $w_J$. If $|\Phi^+_J|=M$, then
$\gamma_1\prec\cdots\prec\gamma_M$ lists the positive roots
$\Phi^+_J$, while $\gamma_{M+1}\prec\cdots\prec\gamma_N$ lists the
remaining positive roots (those in $\Phi^+\backslash\Phi^+_J$). As
above, this ordering depends upon the choice of reduced expressions.

Since $w_{0,J}$ is the longest word in $W_J$, for $\be \in J$,
\begin{equation*}\label{longestword}
\ell(w_{0,J}s_{\be}) < \ell(w_{0,J}).
\end{equation*}
Furthermore, $w_J$ is a distinguished right coset representative for
$W_J$ in $W$ in the sense that $w_J$ is the unique element
of minimal length in its coset $W_Jw_J$. Thus, 
\begin{equation}\label{distinguished}
\ell(s_{\be}w_J) > \ell(w_J) \hskip.5in \text{ and hence }
\hskip.5in
    \ell(s_{\be}w_J) = \ell(w_J) + 1.
\end{equation}


\section{Quantum enveloping algebras }\label{quantumenvelopingalgebras}
\setcounter{equation}{0}
\setcounter{theorem}{0}
\renewcommand{\thetheorem}{\thesection.\arabic{theorem}}
\renewcommand{\theequation}{\thesection.\arabic{equation}}

We now assume that $\Phi$ is irreducible, and make use of the
notation introduced in Section \ref{notation}. Let
${\mathcal A}={\mathbb Q}[q,q^{-1}]$ be the $\mathbb Q$-algebra of Laurent
polynomials in an indeterminate $q$. Then ${\mathcal A}$ has
fraction field ${\mathbb Q}(q)$. We continue to restrict the integer
$l>1$ according to Assumption \ref{assumption}. Let
$\zeta=\sqrt[l]{1}\in\mathbb C$ be a primitive $l$th root of unity
and let $k={\mathbb Q}(\zeta)\subset\mathbb C$ be the cyclotomic field generated by
$\zeta$ over $\mathbb Q$. We will regard $k$ as an $\mathcal
A$-algebra by means of the homomorphism $\mathbb{Q}[q,q^{-1}]\to k$,
$q\mapsto \zeta$.

The quantum enveloping algebra $\BU_q=\BU_q({\mathfrak g})$ of
$\mathfrak g$ is the ${\mathbb Q}(q)$-algebra with generators
$E_{\alpha},K_\alpha^{\pm 1}, F_\alpha$, $\alpha\in \Pi$, which are subject to the 
relations (R1)---(R6) listed in \cite[(4.3)]{Jan3} (whose notation
we will generally follow, unless otherwise indicated). For example, 
the elements $K_\alpha=K_\alpha^{+1}$ and $K_\alpha^{-1}$, $\alpha\in\Pi$, 
generate a commutative subalgebra
of $\BU_q$ such that $K_\alpha K_\alpha^{-1}=1$. Also, for 
$\alpha,\beta\in\Pi$, we have
$$\begin{cases}
K_\alpha E_\beta K_\alpha^{-1}=q_\alpha^{\langle\beta,\alpha^\vee\rangle} E_\beta=q^{\langle\beta,\alpha\rangle}E_\beta;\\
K_\alpha F_\beta K_\alpha^{-1}= q_\alpha^{-\langle\beta,\alpha^\vee\rangle} F_\beta=q^{-\langle\beta,\alpha\rangle}F_\beta,\end{cases}
$$
where $q_\alpha=q^{d_\alpha}$.

In addition, there exists a Hopf
algebra structure on $\BU_q$, with comultiplication
$\Delta:{\mathbb U}_q\to{\mathbb U}_q\otimes{\mathbb U}_q$ and antipode $S:{\mathbb U}_q\to{\mathbb U}_q$ explicitly defined on
generators by
\begin{equation}\label{hopfalgebra}
\begin{cases} \Delta(E_\alpha)=E_\alpha\otimes 1 + K_\alpha\otimes E_\alpha\\
\Delta(F_\alpha)= F_\alpha\otimes K_\alpha^{-1}+ 1\otimes F_\alpha\\
\Delta(K_\alpha^{\pm 1})=K_\alpha^{\pm 1}\otimes K_\alpha^{\pm 1}\end{cases}\,\,{\text{\rm and}}\,\,
\begin{cases}S(E_\alpha)=-K_\alpha^{-1}E_\alpha\\
S(F_\alpha)=-F_\alpha K_\alpha\\
S(K_\alpha^{\pm 1})=K_\alpha^{\mp 1}.\end{cases}
\end{equation}
The comultiplication $\Delta$ is an algebra homomorphism. Moreover,  
the antipode $S$ is an algebra anti-isomorphism.
The counit $\epsilon:{\mathbb U}_q\to{\mathbb Q}(q)$ is the unique algebra homomorphism satisfying $\epsilon(E_\alpha)=\epsilon(F_\alpha)=0$ and
$\epsilon(K_\alpha^{\pm 1})=1$. See \cite[Ch. 4]{Jan3} for more details.

Let $\BU_q^+$ denote the subalgebra of $\BU_q$ generated by the
$E_{\al}$ ($\al \in \Pi$), $\BU_q^-$ denote the subalgebra of
$\BU_q$ generated by the $F_{\al}$ ($\al \in \Pi$), and $\BU_q^0$
denote the subalgebra of $\BU_q$ generated by the $K_{\al}^{\pm 1}$
$(\al \in \Pi)$. The generators of the algebra $\BU^+_q$ (resp., $\BU^{-}_q$) are 
subject only to the Serre relations (R5) (resp., (R6)) in \cite[p. 53]{Jan3}. 

Multiplication gives isomorphisms of vector spaces
\begin{equation}\label{triangle}\BU_q^+\otimes \BU_q^0\otimes \BU_q^-\overset\sim\to
\BU_q\overset\sim\leftarrow \BU_q^-\otimes \BU^0_q\otimes \BU_q^+.\end{equation}
There will be analogous subalgebras $U_\zeta^+$, $U_\zeta^-$ and $U_\zeta^0$ of the algebra $U_\zeta$ defined below upon specialization to
$\zeta$. In this case, the maps analogous to those
in (\ref{triangle}) defined by algebra multiplication
are also isomorphisms of vector spaces.

The quantum enveloping algebra $\BU_q$ has two $\mathcal A$-forms, $\BU_q^{\mathcal A}$
(due to Lusztig) and ${\mathcal U}_q^{\mathcal A}$ (due to De
Concini and Kac). In other words, 
${\mathbb U}_q^{\mathcal A}$ (resp., ${\mathcal U}_q^{\mathcal A}$)
is an $\mathcal A$-subalgebra of ${\mathbb U}_q$ which is free 
as an $\mathcal A$-module and which satisfies 
$${\mathbb U}_q^{\mathcal A}\otimes_{\mathcal A}{\mathbb Q}(q)
\cong {\mathbb U}_q\cong {\mathcal U}_q^{\mathcal A}
\otimes_{\mathcal A}{\mathbb Q}(q).$$
After passage to $\mathbb C$, these algebras play roles
analogous to the hyperalgebra of a reductive group (over a field of positive characteristic) and the universal
enveloping algebra of its Lie algebra, respectively. These $\mathcal A$-forms 
are defined below.

For an integer $i$, put 
$$[i]= \frac{q^i - q^{-i}}{q - q^{-1}},$$ 
and set, for $i>0$, $[i]^! =[i][i-1]\cdots [1]$. By
convention, $[0]^!=1$. For any integer $n$ and positive integer $m$, write
$$\left[\begin{matrix} n \\ m\end{matrix}\right]=\frac{[n][n-1]\cdots [n-m+1]}{[1][2]\cdots[m]}.$$
Set $\left[\begin{smallmatrix} n\\ 0\end{smallmatrix}\right]=1$, by definition. The expressions $[i]$ and $\left[\begin{matrix} n \\ m\end{matrix}\right]$
 all belong to
$\mathcal A$ (in fact, they belong to ${\mathbb Z}[q,q^{-1}]$). In case the 
root system has two root lengths, some
scaling of the variable $q$ is required. Thus, given any Laurent polynomial
$f\in\mathcal A$ and $\alpha\in\Pi$, let $f_\alpha\in\mathcal A$ be obtained by replacing $q$ throughout by
$q_\alpha=q^{d_\alpha}$.

For $\alpha\in\Pi$ and $m\geq 0$, let
$$\begin{cases} E^{(m)}_\alpha=
\frac{E_\alpha^m}{[m]^!_\alpha}\in{\mathbb U}_q \\
F^{(m)}_\alpha= \frac{F_\alpha^m}{[m]^!_\alpha}\in\BU_q.\end{cases}
$$
be the $m$th ``divided powers."    Let
$$\BU_{q}^{\mathcal A}=\BU_q^{\mathcal A}({\mathfrak g}):=\left\langle E^{(m)}_\alpha,\, F^{(m)}_\alpha,\, K_\alpha^{\pm 1}\,|\, \alpha\in\Pi, m\in{\mathbb N}\right\rangle\subset\BU_q,$$
where $\langle \cdots\rangle$ means ``$\mathcal A$-subalgebra generated by." The 
algebra $\BU_q^{\mathcal A}$ admits a triangular factorization
induced from (\ref{triangle}), in which $\BU^+_q$ (respectively, $\BU_q^-$, $\BU_q^0$) is 
replaced by the subalgebra $\BU^{{\mathcal A},+}_q$
(respectively, $\BU_q^{{\mathcal A}, -}$, $\BU_q^{{\mathcal A},0}$). Here $\BU^{{\mathcal A},+}_q$ 
(respectively, $\BU^{{\mathcal A},-}_q$) is defined to be the
$\mathcal A$-subalgebra of $\BU_q^{\mathcal A}$ generated by the elements $E_\alpha^{(m)}$ 
(respectively, $F_\alpha^{(m)}$) for $\alpha\in\Pi$
and $m\in\mathbb N$, and $\BU^{{\mathcal A},0}$ is defined to be 
the $\mathcal A$-subalgebra generated by the elements
$$K_\alpha^{\pm 1},\,\,  \left[\begin{matrix}K_\alpha; m\\ n\end{matrix}\right]:=\prod_{i=1}^n
\frac{K_\alpha q_\alpha^{m-i+1}-K_\alpha^{-1}q_\alpha^{-m-i+1}}{q_\alpha^i-q_\alpha^{-i}}
\in \BU^0_q,$$ for $\alpha\in\Pi$, $n\in\mathbb N$, and $m\in\mathbb Z$.

Put
\begin{equation*}\label{firstquantumgroup}
U_\zeta=U_\zeta({\mathfrak g}):=\BU_{q}^{\mathcal A}/ \langle K_\alpha^l-1, \alpha\in\Pi \rangle \otimes_{\mathcal A} {\mathbb C},
\end{equation*}
 where in this expression $\langle \cdots \rangle$
means ``ideal generated by $\cdots$''. Here $\mathbb C$ is regarded as a $\mathcal A$-algebra via the algebra homomorphism
${\mathcal A}\to{\mathbb C}$, $q\mapsto \zeta$. Let $U^+_\zeta$ (respectively, $U^-_\zeta$, $U^0_\zeta$) be the the image in
$U_\zeta$ of $\BU_q^{\mathcal A,+}\otimes_{\mathcal A}{\mathbb C}$ (respectively, $\BU_q^{\mathcal A,-}\otimes_{\mathcal A}{\mathbb C}$,
$\BU_q^{\mathcal A,0}\otimes_{\mathcal A}{\mathbb C}$).

The Hopf algebra structure on
$\BU_q$ induces a Hopf algebra structure on $\BU^{\mathcal A}_q$, and then passage to the field
$\mathbb C$, one obtains a Hopf algebra structure on the algebra $U_\zeta
$. By abuse of notation, let $E^{(m)}_\alpha$, etc., $\alpha\in\Pi$,
also denote the corresponding elements $1\otimes E^{(m)}_\alpha$, etc. in
$U_\zeta$ (note that $E_{\alpha}^{(1)} = E_{\alpha}$). Therefore,
$E_\alpha^l=F_\alpha^l=0$ and $K_\alpha^l=1$ (by construction) in $U_\zeta$. The
elements $E_\alpha,K_\alpha,F_\alpha$, $\alpha\in\Pi$, generate a
finite dimensional Hopf subalgebra, denoted $u_\zeta=u_\zeta({\mathfrak g})$,  of
$U_\zeta$. The algebra $u_\zeta$ is often called the small quantum group.

We now consider the DeConcini-Kac quantum enveloping algebra ${\mathcal U}_\zeta={\mathcal U}_\zeta({\mathfrak g})$. To begin, define 
${\mathcal U}^{\mathcal A}_q$ to be the $\mathcal
A$-subalgebra of $\BU_q$ generated by the $E_\alpha, F_\alpha,
K_\alpha^{\pm 1}$. There is an inclusion of ${\mathcal A}$-forms:
${\mathcal U}_q^{\mathcal A}\subseteq \BU_q^{\mathcal A}$. Then set
$${\mathcal U}_{\mathbb C}={\mathcal U}_{\mathbb C}({\mathfrak g}):={\mathcal U}^{\mathcal A}_q\otimes_{\mathcal A}{\mathbb C},$$
the algebra obtained by base-change (i.~e., $\mathbb C$ is regarded as an $\mathcal A$-algebra via the same algebra
homomorphism ${\mathcal A}\to{\mathbb C}$, $q\mapsto\zeta$, as above.)
Finally, put
\begin{equation*}\label{secondquantumgroup}
{\mathcal U}_\zeta={\mathcal U}_{\zeta}({\mathfrak g}):={\mathcal
U}_{\mathbb C}/\langle 1\otimes K_\alpha^l-1\otimes 1,\alpha\in\Pi\rangle.
\end{equation*}
As with the quantum enveloping algebra $U_\zeta$,  the algebra ${\mathcal U}_\zeta$ also inherits a Hopf algebra structure from that of ${\mathbb U}_q$.  
However, unlike the inclusion ${\mathcal U}_q^{\mathcal A}\subseteq
{\mathbb U}_q^{\mathcal A}$ of $\mathcal A$-forms above,  
${\mathcal U}_\zeta$
 is not generally a subalgebra of $U_\zeta$. The algebra ${\mathcal U}_\zeta$ has a central (and hence normal) subalgebra
${\mathcal Z}$ such that $u_{\zeta}\cong{\mathcal
U}_\zeta//{\mathcal Z}$; see \cite[Section 3]{DK} for more details. An explicit
description of ${\mathcal Z}$ will be given in Section 2.7.

All the $U_\zeta$-modules $M$ considered in this paper will be
assumed to be integrable and type 1. Usually, this will be assumed
without mention, but occasionally we repeat it for emphasis. In
other words, each $E_i,F_j$ acts locally nilpotently on $M$. In
addition, $M$ decomposes into a direct sum $\bigoplus_{\lambda\in X}M_\lambda$ of $M_\lambda$
weight spaces for $\lambda\in X$. Thus, if $v\in M_\lambda$,
$$\begin{cases} K_\alpha v=\zeta^{\langle\lambda,\alpha\rangle}v;\\
\left[\begin{smallmatrix}
K_\alpha; m\\
n\end{smallmatrix}\right]v=\left[\begin{smallmatrix}\langle
\lambda,\alpha\rangle +m\\
n\end{smallmatrix}\right]_{q=\zeta^{d_\alpha}}v\end{cases}
$$
for all $\alpha\in\Pi, m\in{\mathbb Z}, n\in{\mathbb N}$.

\section{Connections with algebraic
groups }\label{connectionswithalgebraicgroups}
\setcounter{equation}{0}
\setcounter{theorem}{0}
\renewcommand{\thetheorem}{\thesection.\arabic{theorem}}
\renewcommand{\theequation}{\thesection.\arabic{equation}}

Let $\BU({\mathfrak g})$ be the classical universal enveloping algebra of the complex simple Lie algebra $\mathfrak g$ with root system
$\Phi$. The Lie algebra $\mathfrak g$ has Chevalley generators $e_\alpha, f_\alpha, h_\alpha$, $\alpha\in\Pi$, which satisfy
certain well-known relations, providing a presentation of the algebra
$\BU({\mathfrak g})$ \cite[\S 18.3]{Hum}. The Frobenius
morphism is a surjective (Hopf) algebra homomorphism
\begin{equation*}\label{Frob}
\Fr:U_\zeta\twoheadrightarrow \BU({\mathfrak g}).
\end{equation*}
If $I$ is the augmentation ideal of $u_\zeta$, then
$\text{Ker Fr}=IU_\zeta=U_\zeta I$. In other words, $u_{\zeta}$ is a normal Hopf subalgebra
of $U_{\zeta}$ such that $U_{\zeta}//u_{\zeta} \cong \BU(\gl)$.
Given $\alpha\in\Pi$,  let $m$ be a positive integer. If $m= a\ell +b$, where $0\leq b<\ell$,
then
$$\begin{cases}\text{\rm Fr}(E_\alpha^{(m)})=\delta_{b,0} e_\alpha^{(a)}\\
\text{\rm Fr}(F_\alpha^{(m)})=\delta_{b,0}f_\alpha^{(a)}\\
\text{\rm Fr}(\left[\begin{smallmatrix} K_\alpha; l\\ 0\end{smallmatrix}\right])= h_\alpha.\end{cases}$$
In this expression, we have put $e_\alpha^{(a)}=\frac{1}{a!}e_\alpha^a$ and $f_\alpha^{(a)}=\frac{1}{a!}f_\alpha^a$ (cf. 
\cite{L1}).

By
means of the algebra homomorphism $\text{Fr}$, modules for the Lie algebra $\mathfrak g$ can
be ``pulled back" to give modules for $U_\zeta$: given a $\mathfrak
g$-module $M$, $M^{[1]}:=\Fr^*M$ denotes the $U_\zeta$ module
obtained from $M$ by making $x\in U_\zeta$ act as a linear
transformation on $M$ by $\Fr(x)$. If $M$ is a locally finite
${\mathbb U}({\mathfrak g})$-module (e.~g., finite dimensional),
then $M^{[1]}$ is an integrable, type 1 $U_\zeta$-module.
Conversely, a $U_\zeta$-module $N$ on which $u_\zeta$ acts trivially
can be made into a $\BU({\mathfrak g})$-module via the Frobenius
map, so that $N\cong M^{[1]}$ for a locally finite ${\mathbb
U}({\mathfrak g})$-module $M$.

Let $G$ be the complex simple, simply connected algebraic group
having Lie algebra $\mathfrak g$. Let $B\supset T$ (respectively, $B^+\supset T$) be the Borel subgroup
corresponding to the set $-\Phi^+=\Phi^-$ (respectively, $\Phi^+$) of negative (respectively, positive) roots. The category of locally finite
$\BU({\mathfrak g})$-modules is naturally isomorphic to the category
of rational $G$-modules.

The set $X_+$ of dominant weights for the root system $\Phi$ indexes the irreducible modules
for $U_\zeta$. Given $\lambda\in X_+$, let $L_\zeta(\lambda)$ be the
irreducible $U_\zeta$-module of high weight $\lambda$. On the other hand, let
$L(\lambda)$ be the irreducible rational $G$-module of high weight
$\lambda\in X_+$. We may also identify $L(\lambda)$ with the finite dimensional
irreducible ${\mathbb U}({\mathfrak g})$-module of high weight
$\lambda$. (Both $L_\zeta(\lambda)$ and $L(\lambda)$ are determined up to isomorphism.)
For $\lambda\in X_+$, we have
$$\text{\rm Fr}(L(\lambda))= L(\lambda)^{[1]}\cong L_\zeta(\ell\lambda).$$

The integral group algebra ${\mathbb Z}X$ has basis denoted
$e(\mu)$, $\mu\in X$. Given a finite dimensional $U_\zeta$-module
$M$,
$$\ch\, M
=\sum_{\mu\in X} \dim\,M_\mu\ e(\mu)\in{\mathbb Z}X$$ denotes its
(formal) character, where $M_\mu$ is the $\mu$-weight space for the action of $U_\zeta^0$ on $M$. Sometimes
for emphasis, we write $\ch_\zeta M$ for $\ch\, M$.

Similarly, if $M$ is a finite dimensional rational $G$-module
(respectively, ${\mathbb U}({\mathfrak g})$-module), let $\ch\,
M\in{\mathbb Z}X$ be its formal character with respect to the fixed
maximal torus $T$ (respectively, Cartan subalgebra ${\mathfrak h}=
{\text{Lie}}(T)$).


\section{Root vectors and
PBW-basis}\label{rootvectorsandPBW}
\setcounter{equation}{0}
\setcounter{theorem}{0}
\renewcommand{\thetheorem}{\thesection.\arabic{theorem}}
\renewcommand{\theequation}{\thesection.\arabic{equation}}

The ``root vector" generators $E_\alpha,F_\alpha$ for the quantum enveloping algebra $\BU_q$ are only defined for simple roots
$\alpha$. In what follows, it will be necessary to work with root vectors defined for
general (i.e., not simple) roots. In this direction, for each $\al \in \Pi$, Lusztig has defined an algebra automorphism
$T_{\al}:\BU_q\to\BU_q$. Here  we will follow \cite[Ch. 8]{Jan3}, where
the reader can find more details. If $s = s_{\al} \in W$ is the
simple reflection defined by $\al$,  we often write $T_{s} := T_{\al}$.
Given any $w \in W$, let $w = s_{\be_1}s_{\be_2}\cdots s_{\be_n}$ be
a reduced expression (so the $\beta_i\in\Pi$). Then define $T_w := T_{\be_1}\cdots T_{\be_n}
\in \text{Aut}(\BU_q)$. The automorphism $T_w$ is independent of the
reduced expression of $w$. In other words, the automorphisms
$T_{\al}$ extend to an action of the braid group of $W$ on $\BU_q$.

Now let $J \subseteq \Pi$ and fix a reduced expression $w_0 =
s_{\be_1}\cdots s_{\be_N}$ that begins with a reduced expression for the element
$w_{0,J}$ as in Section \ref{listingssec}. If $w_{0,J} =
s_{\be_1}\cdots s_{\be_M}$, then of course $s_{\be_{M+1}}\cdots s_{\be_N}$ is a
reduced expression for $w_J=w_{0,J}w_0$. There exists a linear ordering $\ga_1
\prec \ga_2 \prec \cdots \prec \ga_N$ of the positive roots, where
$\ga_i = s_{\be_1}\cdots s_{\be_{i-1}}(\be_i)$. For $\ga = \ga_i \in
\Phi^+$, the ``root vector'' $E_{\ga} \in \BU_q^+$ is defined by
$$
E_{\ga} = E_{\ga_i} := T_{s_{\be_1} \cdots s_{\be_{i-1}}}(E_{\be_i})
=
        T_{\be_1}\cdots T_{\be_{i-1}}(E_{\be_i}).
$$
If $\ga$ is simple, the ``new'' $E_{\ga}$ agrees with the original
generator $E_{\ga}$. More generally, $E_{\ga}$ has weight $\ga$.
Similarly,
$$
F_{\ga} = F_{\ga_i} := T_{s_{\be_1} \cdots s_{\be_{i-1}}}(F_{\be_i})
=
        T_{\be_1}\cdots T_{\be_{i-1}}(F_{\be_i}),
$$
a root vector of weight $-\ga$.

The subalgebra $\BU_q^+$ (respectively, $\BU_q^-$) has a PBW-like basis
consisting of all monomials $E_{\gamma_1}^{a_1}\cdots
E_{\gamma_N}^{a_N}$ (respectively, $F_{\gamma_1}^{a_1}\cdots
F_{\gamma_N}^{a_N}$), $a_1,\cdots, a_n\in{\mathbb N}$. Using divided
powers when necessary, one obtains PBW-bases for the specialized
quantum groups $U^+_\zeta$, $U^-_\zeta$, $u_{\zeta}^+$,
$u_{\zeta}^-$, $\mathcal{U}_{\zeta}^+$, and $\mathcal{U}_{\zeta}^-$.
The automorphisms $T_w$ induce automorphisms on $\BU^{\mathcal A}_q$
and hence on $U_\zeta$.

The monomial bases satisfy a key ``commutativity'' property
originally observed by Levendorski{\u\i} and Soibelman.
Here $E_{\ga_i}^0 = 1$ and $F_{\ga_i}^0 = 1$.

\begin{lem} (\cite[Thm. 9.3]{DP}, \cite{LS})\label{straighteninglemma} Let ${\mathcal A}'$ be a subring of
${\mathbb Q}(q)$ containing $\mathcal A$.
Suppose that $q^{2}-q^{-2}$ is invertible in ${\mathcal A}'$ when $\Phi$ is of type $B_{n}$, $C_{n}$, or $F_{4}$ and
additionally $q^{3}-q^{-3}$ is invertible in ${\mathcal A}'$ when $\Phi$ is of type $G_{2}$. In $\BU_q$, we have for $i < j$
\begin{itemize}
\item[(a)] $E_{\ga_i}E_{\ga_j} = q^{\langle\ga_i,\ga_j\rangle}E_{\ga_j}E_{\ga_i} + (*)$
where $(*)$ is an ${\mathcal A}'$-linear combination of monomials $
E_{\gamma_1}^{a_1}\cdots E_{\gamma_N}^{a_N}$, with $a_s=0$ unless
$i<s<j$;

\item[(b)]$F_{\ga_i}F_{\gamma_j}=q^{\langle\ga_i,\ga_j\rangle}F_{\ga_j}F_{\ga_i} + (*)$
where $(*)$ is an ${\mathcal A}'$-linear combination of monomials $
F_{\gamma_1}^{a_1}\cdots F_{\gamma_N}^{a_N}$, with $a_s=0$ unless
$i<s<j$.
\end{itemize}
\end{lem}

The requirements on ${\mathcal A}'$ (i.e.,  that $q^2-q^{-2}$ is invertible in types $B,C,F$, etc.) are not
explicitly stated in \cite{DP}, though they are implicit in the arguments in its appendix (see pp. 135--137 in
\cite{DP}; and also \cite[\S3.2]{Dr1}). We thank Chris Drupieski for pointing this out to us.


\section{Levi and parabolic subalgebras}\label{Leviandparabolicsubalgebras}
\setcounter{equation}{0}
\setcounter{theorem}{0}
\renewcommand{\thetheorem}{\thesection.\arabic{theorem}}
\renewcommand{\theequation}{\thesection.\arabic{equation}}

 Given a subset
$J \subseteq \Pi$, consider the Levi and parabolic Lie subalgebras
${\mathfrak l}_J$ and ${\mathfrak p}_J = {\mathfrak
l}_J\oplus{\mathfrak u}_J$ of $\mathfrak g$. We denote the
respective universal enveloping algebras by $\BU(\lj)$ and
$\BU(\pj)$. One can naturally define corresponding quantum
enveloping algebras $\BU_q(\lj)$ and $\BU_q(\pj)$. As (Hopf) subalgebras of
$\BU_q$, $\BU_q(\lj)$ is the subalgebra generated by the elements $\{E_{\al},
F_{\al} : \al \in J\} \cup \{K_{\al}^{\pm 1} : \al \in \Pi\}$, and
$\BU_q(\pj)$ is the subalgebra generated by $\{E_{\al}: \al \in J\}
\cup \{F_{\al}, K_{\al}^{\pm 1} : \al \in \Pi\}$. For example, if $J =
\varnothing$, then $\lj = \mathfrak{h}$, $\pj = \mathfrak{b}$,
$\BU_q(\lj) = \BU_q^0$, and $\BU_q(\pj) = \BU_q({\mathfrak b}) =
\BU_q^-\cdot \BU_q^0$. Specializing gives the subalgebras
$U_{\zeta}(\lj), U_{\zeta}(\pj), u_{\zeta}(\lj), u_{\zeta}(\pj)$ of
$U_\zeta$, and $\mathcal{U}_{\zeta}(\lj)$ and
$\mathcal{U}_{\zeta}(\pj)$ of ${\mathcal U}_\zeta$.

If we denote by $\pj^+ = \lj\oplus\uj^+$ the opposite parabolic
subalgebra (containing the positive Borel subalgebra ${\mathfrak
b}^+$), then one can analogously consider $\BU_q(\pj^+)$.


\section{The subalgebra $\BU_q(\uj)$}\label{subalgebrauJ}

\setcounter{equation}{0}
\setcounter{theorem}{0}
\renewcommand{\thetheorem}{\thesection.\arabic{theorem}}
\renewcommand{\theequation}{\thesection.\arabic{equation}}

The purpose of this section is to define a subalgebra $\BU_q(\uj)$ of $\BU_q$. It will play a role analogous to that played
by the subalgebra $\BU(\uj)$ of the universal enveloping algebra $\BU(\pj)$ corresponding to the nilpotent radical of $\pj$. As above, choose a reduced expression
for $w_0$ (beginning with one for $w_{0,J}$). Define $\BU_q(\uj)$ to
be the subspace of $\BU_q^-$ spanned by the
$F_{\ga_{M+1}}^{a_{M+1}}\cdots F_{\ga_N}^{a_N}$, $a_i \in
\mathbb{N}$. By Lemma \ref{straighteninglemma}, $\BU_q(\uj)$ is a
subalgebra of $\BU_q^-$.  In addition, the monomials above form a basis for
$\BU_q(\uj)$.  One can also verify directly that
$\BU_q(\uj)$ is independent of the choice of reduced expression for
$w_0$. However, this follows from a more general set-up which will
be useful for other purposes as well.

Given any $v \in W$, we define as in \cite[8.24]{Jan3} a subspace
$U^-[v] \subset \BU_q^-$ (respectively, $U^+[v] \subset \BU_q^+$) as
follows. Choose a reduced expression $v = s_{\eta_1}s_{\eta_2}\cdots
s_{\eta_t}$. For $1 \leq i \leq t$, set $f_i =
T_{s_{\eta_1}s_{\eta_2}\cdots s_{\eta_{i-1}}}(F_{\eta_i})$ (respectively,
$e_i = T_{s_{\eta_1}s_{\eta_2}\cdots s_{\eta_{i-1}}}(E_{\eta_i})$).
The $f_i$ (respectively, $e_i$) are in some sense ``root vectors'' like
those defined earlier. Then $U^-[v]$ (respectively, $U^+[v]$) is defined to
be the span of all monomials of the form $f_1^{a_1}f_2^{a_2}\cdots
f_t^{a_t}$ (respectively, $e_1^{a_1}e_2^{a_2}\cdots e_t^{a_t}$) for $a_i
\geq 0$. By \cite[2.2]{DKP}, $U^-[v]$ (respectively, $U^+[v]$) is a subalgebra
of ${\mathbb U}^-_q$ (resp., ${\mathbb U}^+_q$) and moreover is independent of the choice of reduced expression for
$v$.

Since $\BU_q(\uj) = T_{w_{0,J}}(U^-[w_J])$, $\BU_q(\uj)$ is a subalgebra of $\BU_q(\pj) \subset \BU_q$, independent of the reduced
expression for $w_J$. Furthermore, since the automorphism $T_{w_{0,J}}$
is also independent of the choice of reduced expression for
$w_{0,J}$, the algebra $\BU_q(\uj)$ depends only on $J$, not on our
above choices of reduced expressions. Following the procedure in
Section \ref{quantumenvelopingalgebras}, the algebra $\BU_q(\uj)$
can be specialized to give algebras $U_{\zeta}(\uj)$,
$u_{\zeta}(\uj)$, and $\Uz(\uj)$. For example, $U_{\zeta}(\uj)$ is
the subalgebra of $U_{\zeta}^-$ spanned by
$F_{\ga_{M+1}}^{(a_{M+1})}\cdots F_{\ga_N}^{(a_N)}$, $a_i \in
\mathbb{N}$.

Similarly one can define a subalgebra $\BU_q(\uj^+) \subset
\BU_q(\pj)$, and the corresponding specializations. Then
$\BU_q(\uj^+) = T_{w_{0,J}}(U^+[w_J])$. When $J = \varnothing$,
$\BU_q(\uj) = \BU_q^-$ and $\BU_q(\uj^+) = \BU_q^+$.


\section{Adjoint action}\label{adjointaction}
\setcounter{equation}{0}
\setcounter{theorem}{0}
\renewcommand{\thetheorem}{\thesection.\arabic{theorem}}
\renewcommand{\theequation}{\thesection.\arabic{equation}}

 Given a Hopf algebra $A$, the
adjoint action of $A$ is defined by setting, for $x,y \in A$,
$\text{Ad}(x)(y) = \sum x_{(1)}yS(x_{(2)})$ where $\Delta(x) = \sum
x_{(1)}\otimes x_{(2)}$ is the comultiplication and $S$ is the antipode. We
consider, in particular, the case $A=\BU_q$, where we record the
formulas here on generators: (for $m \in \BU_q$)
\begin{equation}\label{ad}
\begin{cases}
\ad(E_{\al})(m) &= E_{\al}m - K_{\al}mK_{\al}^{-1}E_{\al}\\
\ad(K_{\al}^{\pm 1})(m) &= K_{\al}^{\pm 1}mK_{\al}^{\mp 1}\\
\ad(F_{\al})(m) &= (F_{\al}m - mF_{\al})K_{\al}.
\end{cases}
\end{equation}
These expressions follow immediately from the formulas (\ref{hopfalgebra}).

By twisting the Hopf structure on $\BU_q$ by the algebra involution $\omega$ of
$\BU_q$ given by $\omega(E_{\al}) = F_{\al}$, $\omega(F_{\al}) =
E_{\al}$, and $\omega(K_{\al}) = K_{\al}^{-1}$, we obtain another
Hopf structure on $\BU_q$, with comultiplication $^\omega\Delta:=(\omega\otimes\omega)\circ\Delta\circ\omega$
and antipode $^\omega S=\omega\circ S\circ\omega$; see \cite[3.8]{Jan3}. In particular, this procedure
then gives an alternate adjoint action $\adw$, which satisfies:
\begin{equation}\label{adw}
\begin{cases}
\adw(E_{\al})(m) &= (E_{\al}m - mE_{\al})K_{\al}^{-1}\\
\adw(K_{\al}^\pm 1)(m) &= K_{\al}^{\pm 1}mK_{\al}^{\mp 1}\\
\adw(F_{\al})(m) &= F_{\al}m - K_{\al}^{-1}mK_{\al}F_{\al}.
\end{cases}
\end{equation}

\begin{prop}\label{adjointprop} The following stability results hold:
\begin{itemize}
\item[(a)] The subalgebra $\BU_q(\uj^+)$ is stable under the $\ad$ action of
$\BU_q(\pj^+)$ on itself.
\item[(b)] The subalgebra $\BU_q(\uj)$ is stable under the $\adw$ action of
$\BU_q(\pj)$ on itself.
\end{itemize}
\end{prop}

\begin{proof} Part (b) follows from part (a)
since, for $x \in \BU_q(\pj)$, $\adw(x) =
\omega\circ\ad(\omega(x))\circ \omega$.

We prove part (a). Since $\BU_q(\pj^+)$ is
generated as an algebra by $\{E_\al, K_{\al}, K_{\al}^{-1} : \al \in
\Pi\} \cup \{F_{\al} : \al \in J\}$, it suffices to show that
$\BU_q(\uj^+)$ is stable under the action of these elements. By
(\ref{ad}), $\text{Ad}(K_{\al})$ and $\text{Ad}(K_{\al}^{-1})$
simply act by scalar multiplication on the weight spaces which span
$\BU_q(\uj^+)$. Hence $\BU_q(\uj^+)$ is stable under
$\ad(K_{\al}^{\pm 1})$. For $\al \in \Phi^+ \backslash \Phi_J^+$, the
stability under $\ad(E_{\al})$ follows from the fact that
$\BU_q(\uj^+)$ is an algebra. It remains to prove stability under
$\ad(F_{\al})$ and $\ad(E_{\al})$ for $\al \in J$. Fix $\al \in J$
and set $\be = -w_{0,J}(\al) \in J$.

By Section \ref{subalgebrauJ}, $\BU_q(\uj^+) =
T_{w_{0,J}}(U^+[w_J])$. Given a reduced expression $w_J =
s_{\eta_1}s_{\eta_2}\cdots s_{\eta_t}$, $U^+[w_J]$ is spanned by
monomials of ordered root vectors:
$$
E_{\eta_1}, T_{s_{\eta_1}}(E_{\eta_2}), \dots, T_{s_{\eta_1}\cdots
s_{\eta_{t-1}}}(E_{\eta_t}).
$$
Since $w_J = w_{0,J}w_0$, $w_J^{-1}(\be)\in\Pi$. Consider also
$U^+[w_Js_{w^{-1}_J(\beta)}]$. Since $w_Js_{w^{-1}_J(\be)} =
s_{\be}w_J$, by (\ref{distinguished}), $\ell(w_Js_{w^{-1}_J(\be)}) =
\ell(w_J) + 1$. Therefore, the ordered root vectors defining
$U^+[w_Js_{w^{-1}_J(\be)}]$ are
$$
E_{\eta_1}, T_{s_{\eta_1}}(E_{\eta_2}), \dots, T_{s_{\eta_1}\cdots
s_{\eta_{t-1}}}(E_{\eta_t}), T_{w_J}(E_{w^{-1}_J(\be)}).
$$
By \cite[Prop. 8.20]{Jan3} $T_{w_J}(E_{w^{-1}_J(\be)}) =
E_{w_J(w_J^{-1}(\be))} = E_{\be}$. Now $U^+[w_J]$ is a subalgebra of
$U^+[w_Js_{w_J^{-1}(\be)}]$ spanned by monomials not involving the
last root vector $E_{\be}$. Moreover, by Lemma
\ref{straighteninglemma}, since $E_{\be}$ appears last in the
ordering of root vectors in $U^+[w_Js_{w_J^{-1}(\be)}]$,
\begin{equation}\label{soib1}
\text{if } u \in U^+[w_J] \subset U^+[w_Js_{w_J^{-1}(\be)}] \text{
is a monomial}, \text{ then } uE_{\be} -
q^{\langle\text{wt}(u),\be\rangle} E_{\be}u \in U^+[w_J],
\end{equation}
where $\text{wt}(u)$ denotes the weight of $u$.

Now consider the subalgebra $U^+[s_{\be}w_J]$. The ordered root
vectors defining $U^+[s_{\be}w_J]$ are
$$
E_{\be}, T_{s_{\be}}(E_{\eta_1}),
T_{s_{\be}}(T_{s_{\eta_1}}(E_{\eta_2})), \dots,
T_{s_{\be}}(T_{s_{\eta_1}\cdots s_{\eta_{t-1}}}(E_{\eta_t})).
$$
Then $T_{s_{\be}}(U^+[w_J])$ is evidently a subalgebra of
$U^+[s_{\be}w_J]$ spanned by monomials in all but the first root
vector $E_{\be}$. Moreover, by Lemma \ref{straighteninglemma}, since
$E_{\be}$ occurs first in the root ordering,
\begin{equation}\label{soib2}
\text{if } u \in U^+[w_J] \text{ is a monomial}, \text{ then }
E_{\be}T_{s_{\be}}(u) -
q^{\langle\text{wt}(T_{s_{\be}}(u)),\be\rangle}
T_{s_{\be}}(u)E_{\be} \in T_{s_{\be}}(U^+[w_J]).
\end{equation}

We now show that $\ad(F_{\al})$ preserves $\BU_q(\uj^+)$. Since
$\al,\be\in\Pi$, by \cite[Prop. 8.20]{Jan3}, 
$$F_{\al} =
T_{w_{0,J}s_{\be}}(F_{\be}),$$
while  \cite[8.14(3)]{Jan3} gives that
 $$F_{\be} =
-T_{s_{\be}}(E_{\be})K_{\be}^{-1} = -T_{s_{\be}}(E_{\be}K_{\be}).$$
Combining these equalities  gives $$F_{\al} =
T_{w_{0,J}s_{\be}}(-T_{s_{\be}}(E_{\be}K_{\be})) = -
T_{w_{0,J}}(E_{\be}K_{\be}).$$

Let $T_{w_{0,J}}(u)$ for $u \in U^+[w_J]$ be an arbitrary monomial
element of $\BU_q(\uj^+)$. Then we have
$$
\begin{array}{rll}
\ad(F_{\al})(T_{w_{0,J}}(u)) &= (F_{\al}T_{w_{0,J}}(u) -
T_{w_{0,J}}(u)F_{\al})K_{\al}
                & \text{ by } (\ref{ad})\\
    &= (-T_{w_{0,J}}(E_{\be}K_{\be})T_{w_{0,J}}(u) + T_{w_{0,J}}(u)T_{w_{0,J}}(E_{\be}K_{\be}))K_{\al}
            & \text { from above }\\
    &= T_{w_{0,J}}((uE_{\be}K_{\be}- E_{\be}K_{\be}u)K_{\be}^{-1})
            & \text{\cite[8.18(3)]{Jan3}}\\
    &= T_{w_{0,J}}(uE_{\be} - E_{\be}K_{\be}uK_{\be}^{-1}) & \\
    &= T_{w_{0,J}}(uE_{\be} - q^{\langle\text{wt}(u),\be\rangle}E_{\be}u) & \\
    &\in T_{w_{0,J}}(U^+[w_J]) = \BU_q(\uj^+) & \text{ by } (\ref{soib1})
\end{array}
$$
as claimed.

Lastly, we consider $E_{\al}$. Again $\al,\be\in\Pi$, so $E_{\al} =
T_{w_{0,J}s_{\be}}(E_{\be})$.
Let $T_{w_{0,J}}(u)$ for $u \in U^+[w_J]$ be an arbitrary monomial
element of $\BU_q(\uj^+)$. Then we have
$$
\begin{array}{rll}
\ad(E_{\al})(T_{w_{0,J}}(u)) &= E_{\al}T_{w_{0,J}}(u) -
            K_{\al}T_{w_{0,J}}(u)K_{\al}^{-1}E_{\al} &\text{ by } (\ref{ad})\\
    &= T_{w_{0,J}s_{\be}}(E_{\be}) T_{w_{0,J}}(u) -
            K_{\al}T_{w_{0,J}}(u)K_{\al}^{-1}T_{w_{0,J}s_{\be}}(E_{\be}) &\text{ from above }\\
    &= T_{w_{0,J}s_{\be}}(E_{\be}T_{s_{\be}}(u)- K_{\be}T_{s_{\be}}(u)K_{\be}^{-1}E_{\be}) &\text{\cite[8.18(3)]{Jan3}}\\
    &= T_{w_{0,J}s_{\be}}(E_{\be}T_{s_\be}(u) -
            q^{\langle\text{wt}(T_{s_{\be}}(u)),\be\rangle} T_{s_{\be}}(u)E_{\be}) &\\
    &\in T_{w_{0,J}s_{\be}}(T_{s_{\be}}(U^+[w_J])) = T_{w_{0,J}}(U^+[w_J]) = \BU_q(\uj^+)
        &\text{ by } (\ref{soib2})
\end{array}
$$
as claimed.
\end{proof}

The definitions of $\ad$ and $\adw$ now give the following.

\begin{cor}\label{corA} The algebra $\BU_q(\uj)$ is normal in $\BU_q(\pj)$ and the
algebra $\BU_q(\uj^+)$ is normal in $\BU_q(\pj^+)$. Furthermore,
normality also holds for the specializations $U_{\zeta}(\uj) \subset
U_{\zeta}(\pj)$, $u_{\zeta}(\uj) \subset u_{\zeta}(\pj)$, and
$\Uz(\uj) \subset \Uz(\pj)$ (as well as for the $+$-versions). Also,
$U_\zeta({\mathfrak l}_J)\cong U_\zeta({\mathfrak
p}_J)//U_{\zeta}({\mathfrak u}_J)$, ${u}_\zeta({\mathfrak
l}_J)\cong { u}_\zeta({\mathfrak p}_J)//u_{\zeta}({\mathfrak u}_J)$,
and ${\mathcal U}_\zeta({\mathfrak l}_J)\cong {\mathcal
U}_\zeta({\mathfrak p}_J)//{\mathcal U}_{\zeta}({\mathfrak u}_J)$.
\end{cor}

Recall the inclusion of $\mathcal A$-forms ${\mathcal
U}_q^{\mathcal A} \subseteq \BU_q^{\mathcal A}$ mentioned in \S2.2. While it is generally
false that ${\mathcal U}_q^{\mathcal A}$ is stable under the adjoint
action of $\BU_q^{\mathcal A}$ on itself, this property does hold
after base-change from ${\mathcal A}$ to a larger algebra ${\mathcal
B}$. More precisely, given $l$ satisfying Assumption \ref{assumption}, let $f_l(x) \in {\mathbb Q}[x]$
denote the minimal polynomial for a primitive $l$th root of unity.
Set $\mathcal{B} := {\mathbb Z}[q,q^{-1}]_{(\langle
f_l(q)\rangle)}$, i.e., the local ring determined by the maximal
ideal $\langle f_l(q)\rangle$. Defining $\BU_q^{\mathcal B}={\mathcal B}\otimes_{\mathcal A}\BU_q^{\mathcal A}$ and
${\mathcal U}_q^{\mathcal B}={\mathcal B}\otimes_{\mathcal A}{\mathcal U}_q^{\mathcal A}$,
there is an inclusion
${\mathcal U}_q^{\mathcal B} \subset U_q^{\mathcal B}$.

\begin{lem} \cite[Prop. 2.9.2(i)]{ABG}\label{ABGlemma} The adjoint action of $\BU_q^{\mathcal B}$ stabilizes
${\mathcal U}_q^{\mathcal B}$. That is,
$$\ad(\BU_q^{\mathcal B})({\mathcal U}_q^{\mathcal B})\subseteq {\mathcal U}_q^{\mathcal B} \hskip.5in \text{ and } \hskip.5in
\adw(\BU_q^{\mathcal B})({\mathcal U}_q^{\mathcal B}) \subseteq
{\mathcal U}_q^{\mathcal B}.$$ Hence, the adjoint action (either
$\ad$ or $\adw$) of $U_{\zeta}$ defines an action on ${\mathcal
U}_{\zeta}$.
\end{lem}

Let ${\mathcal Z} \subset \Uz$ be the central subalgebra such that
$\Uz//{\mathcal Z} \cong u_{\zeta}$ (cf. Section
\ref{quantumenvelopingalgebras}). In terms of root vectors,
${\mathcal Z}$ is the algebra generated by the elements $\{E_{\ga}^l,F_{\ga}^l :
\ga \in \Phi^+\}$. Given $J\subseteq \Pi$, define a central subalgebra $Z_J
\subset \Uz(\uj)$ as 
$$Z_J := {\mathcal Z} \cap \Uz(\uj).$$
 Clearly
$Z_J$ is the subalgebra generated by $\{F_{\ga}^l : \ga \in \Phi^{+}
\backslash \Phi_J^+\}$. Further,  $\Uz(\uj)//Z_J \cong
u_{\zeta}(\uj)$, leading to the following generalization of
\cite[Prop. 2.9.2]{ABG}.

\begin{cor}\label{corB} Under the induced $\adw$-action of $U_{\zeta}$ on $\Uz$ we
have the following:
\begin{itemize}
\item[(a)] $\adw(U_{\zeta}(\pj))$ stabilizes $\Uz(\uj)$, $Z_J$, and $u_{\zeta}(\uj)$.
\item[(b)] The action of $u_{\zeta}(\pj)$ on $Z_J$ is trivial.
\item[(c)] $\adw$ induces an action of $\BU(\pj)$ on $Z_J$.
\end{itemize}
\end{cor}

We conclude this section with some remarks on the action of a Hopf
algebra $H$ on an augmented algebra $A$. More precisely, $A$ is defined to be a left (respectively,
right) $H$-module algebra provided that $A$ is a left (respectively, right) $H$-module
such that:
\begin{itemize}
 \item[(i)] $h\cdot(ab)=\sum(h_{(1)}\cdot a)(h_{(2)}\cdot b)$ (respectively,
$(ab)\cdot h=\sum(a\cdot h_{(1)})(b\cdot h_{(2)}))$ for $a,b\in A, h\in H$,  putting $\Delta(h)=\sum h_{(1)}\otimes h_{(2)}$; 
\item[(ii)]
$h\cdot1_A=\epsilon(h)1_A$ (respectively, $1_A\cdot h=\epsilon(h)1_A$) for $h\in H$;  
\item[(iii)] $\epsilon_A(h\cdot a)=\epsilon(h)\epsilon_A(a)$ (respectively, $\epsilon_A(a\cdot h)=\epsilon_A(a)\epsilon(h)$),
for $a\in A, h\in H$, where $\epsilon_A$ is the augmentation map on $A$. 
\end{itemize}

The notion of a
left $H$-module algebra for a not necessarily augmented algebra $A$ is studied further in
\cite[4.1.1,
4.1.9]{Mon}.

For example, $H$ is a right $H$-module algebra defined by 
$$ a\cdot h=\ad_r(h)(a)=\sum S(h_{(1)})ah_{(2)}, \quad a,h\in H.$$
Thus, $\ad_r(hh')=\ad_r(h')\ad_r(h)$, $h,h'\in H$. In the context of ${\mathbb U}_q$, we have
\begin{equation}\label{rightmodulealgebra}
\begin{cases} \ad_r(E_\alpha)=\adw(K_\alpha^{-1}E_\alpha)\\
\ad_r(K_\alpha^{\pm 1})=\adw(K^{\mp 1})\\
\ad_r(F_\alpha)=-\adw(F_\alpha K_\alpha).
\end{cases}
\end{equation}
With these relations, there are evident analogues of the above lemma and its corollary. In particular, the augmented
algebra ${\mathcal U}_q$ is a left (respectively, right) module algebra for $U_\zeta$ using $\ad$ or $\adw$ (respectively,
$\ad_r$). Under the right action of $U_\zeta$ on ${\mathcal U}_\zeta$, Corollary \ref{corB} holds replacing $\adw$ throughout
by $\ad_r$.

Now let $A$ be a right $H$-module algebra, and let $V$ be
a left $A$-module. A left action $v\mapsto h\cdot v$ of $H$ on $V$ is called compatible with that
of $A$ provided that $V$ is a left $H$-module with this action and  $a(h\cdot v)= \sum h_{(1)}\cdot (a\cdot h_{(2)})v$
for $a\in A,v\in V, h\in H$.  If $A$ is a left $H$-module algebra and $V$ is a left $H$-module, the action
is compatible provided the last condition is replaced by $h\cdot (av)=\sum (h_{(1)}\cdot a)(h_{(2)}\cdot v)$.


\section{Finite dimensionality of cohomology
groups}\label{finitedimofcohogroups}

\setcounter{equation}{0}
\setcounter{theorem}{0}
\renewcommand{\thetheorem}{\thesection.\arabic{theorem}}
\renewcommand{\theequation}{\thesection.\arabic{equation}}

 For any field $k$, given
augmented $k$-algebras $A \subset B$ with $A$ normal in $B$ (cf.
Sections \ref{notation} and \ref{adjointaction}), there exists a Lyndon-Hochschild-Serre
(LHS) spectral sequence (cf. \cite[\S 5.2, 5.3]{GK}).
More precisely, we have the following result.\footnote{Recall that $\opH^\bullet(A,M):=\Ext^\bullet_A(k,M)$.}

\begin{lem}\label{ABG} Assume
that $B$ is flat as a right $A$-module.
 For any left $B$-module $M$,
there is a natural action of the quotient algebra $B//A$ on $\opH^{\bullet}(A,M)$ which
gives rise to a (first quadrant) spectral sequence
$$E_2^{i,j}= \opH^i(B//A, \opH^j(A,M)) \Rightarrow \opH^{i+j}(B, M).$$
\end{lem}

In this result, the action of $B//A$ arises from the identification
$H^\bullet(A,M)\cong\Ext^\bullet_B(B//A,M)$. Then the right multiplication
action of $B//A$ on itself induces a left $B//A$-module structure on
$\Ext^\bullet_B(B//A,M)$. In some cases, the action of $B//A$ also
comes about through an action of $B//A$ on the reduced bar resolution $D_\bullet=D_\bullet(A)$ of
$k$. More precisely, suppose that $H$ is a Hopf algebra acting on the
augmented algebra $A$ on the  {\it right} as explained above. Then
$H$ acts on the right on $D_\bullet$. Here
$D_n=A\otimes_k A_{+}^{\otimes n}$; explicitly,
$$a\otimes[a_1|\cdots|a_n]\cdot h=\sum a\cdot h_{(1)}[a_1\cdot h_{(2)}|\cdots|a_n\cdot h_{(n+1)}].$$
This action commutes with the differentials $d_\bullet:D_n\to D_{n-1}$. Now suppose that $M$ is
a left $A$-module with a compatible left action of $H$.   The left action of $H$ on $M$ and the right action
of $H$ on $B_\bullet$ define a left $H$-action on the complex $\Hom_A(D_\bullet, M)$ computing $H^\bullet(A,M)$.
The $H$-action on $H^\bullet(A,M)$ can be interpreted in the context of Lemma \ref{ABG}, by putting
$B=H\# A$, the smash product of $H$ and $A$. (See \cite{Mon}.) Then $B$ is a flat right $A$-module, $M$ inherits
a natural $B$-module structure and $B//A\cong H$. It is easy to verify that the two actions of $H$ on
$H^\bullet(A,M)$ identify.\footnote{We are grateful for discussions on these issues with Chris Drupieski, who
pointed out that our original set-up with $H$ acting on the left of $A$ was not correct.}

We return to the situation of the previous section. Under the
$\adw$-action of $U_{\zeta}(\pj)$ on itself, $U_{\zeta}(\pj)$ admits
the structure of a Hopf module algebra. The action of $U_{\zeta}(\pj)$ on $\Uz(\pj)$ can be
extended to an action on the bar resolution which computes the
cohomology $\opH^{\bullet}(\Uz(\pj),{\mathbb C})$. Thus, there is
 a natural action of $U_{\zeta}(\pj)$ on $\opH^{\bullet}(\Uz(\pj),{\mathbb C})$
and further on $\opH^{\bullet}(\Uz(\uj),{\mathbb C})$. In
particular, $U^0_\zeta$ acts on these cohomology spaces.

\begin{lem}\label{cohomologyofA}
 For each nonnegative integer $n$, the cohomology space $\opH^n({\mathcal U}_\zeta({\mathfrak u}_J),{\mathbb C})$
  is a finite dimensional vector space.\end{lem}

\begin{proof}If $Z_J$ is the central subalgebra of ${\mathcal
U}_\zeta({\mathfrak u}_J)$ defined after Lemma \ref{ABGlemma},
${\mathcal U}_\zeta({\mathfrak u}_J)//Z_J\cong u_\zeta({\mathfrak
u}_J)$. Certainly, ${\mathcal U}_\zeta({\mathfrak u}_J)$ is flat (even free) over $Z_J$. By Lemma \ref{ABG}, there is a spectral sequence
$$E_2^{i,j}=\opH^i(u_\zeta({\mathfrak u}_J),\opH^j(Z_J,{\mathbb
C}))\Rightarrow \opH^{i+j}({\mathcal U}_\zeta({\mathfrak
u}_J),{\mathbb C}).$$
 The subalgebra $Z_J$ is central so it follows from \cite[Lemma 5.2.2]{GK} that
 the action of $u_\zeta({\mathfrak u}_J)$ on $\opH^\bullet (Z_J,{\mathbb
 C})$ is trivial, and
 $$E^{\bullet,\bullet}_2\cong \opH^\bullet(u_\zeta({\mathfrak
 u}_J),{\mathbb C})\otimes \opH^\bullet(Z_J,{\mathbb C}).$$
 Since $Z_J$ is the symmetric algebra based on the vector space
 ${\mathfrak u}_J^{[1]}$, we get that $\opH^\bullet(Z_J,{\mathbb
 C})\cong \Lambda^\bullet({\mathfrak u}_J^*)^{[1]}$. Moreover,
 $u_\zeta({\mathfrak u}_J)$ is finite dimensional, so
 $\opH^i(u_\zeta({\mathfrak u}_J),{\mathbb C})$ is finite dimensional
 for any integer $i$. The result follows because the cohomology
 $\opH^{n}({\mathcal U}_\zeta({\mathfrak u}_J),{\mathbb C})$ is
a subquotient of $\oplus_{i+j=n}E_2^{i,j}$ which is finite
dimensional.\end{proof}

 Also, the cohomology of ${\mathcal U}_\zeta({\mathfrak u}_J)$ can be computed by means
 of the reduced bar resolution (cf. \cite[Ch. X, \S 2]{Mac}). Although it is not clear from this
 point of view that the cohomology is finite dimensional in any
 homological degree, it does have a natural $U_\zeta^0$-action induced from the $\overline{\text{\rm Ad}}$-action
 on $U_\zeta({\mathfrak p}_J)$. But
 then the above lemma establishes it is finite dimensional, and this
 fact implies each cohomology space is a weight module for
 $U^0_\zeta$. We summarize this result as follows.

 \begin{cor} The cohomology $\opH^\bullet({\mathcal U}_\zeta({\mathfrak
 u}_J),{\mathbb C})$ is a weight module for $U^0_\zeta$. For any
 $\lambda\in X$, $\opH^\bullet({\mathcal U}_\zeta({\mathfrak
 u}_J),{\mathbb C})_\lambda$ is finite dimensional.\end{cor}

 \begin{proof}It remains only to establish the last statement. But
 if $\bar D_\bullet=D_\bullet({\mathcal U}_\zeta({\mathfrak u}_J))$ is the reduced bar-resolution, then, given any weight
 $\mu$, for large
 $n$, $(\bar D_n)_\mu=0$. \end{proof}


\section{Spectral sequences
 and the Euler characteristic}\label{spectralseqandEuler}

\setcounter{equation}{0}
\setcounter{theorem}{0}
\renewcommand{\thetheorem}{\thesection.\arabic{theorem}}
\renewcommand{\theequation}{\thesection.\arabic{equation}}

We study the cohomology of $\Uz(\uj)$ and $u_{\zeta}(\uj)$, using a filtration on $\Uz(\uj)$ (respectively, $u_{\zeta}(\uj)$)
introduced for $\Uz$ in \cite{DK}.

Let $A = {\mathcal U}_{\zeta}({\mathfrak u}_{J})$. By Section
\ref{subalgebrauJ}, $A$ has a basis of monomial elements
$$
\{F_{\ga_{M+1}}^{a_{M+1}}F_{\ga_{M+2}}^{a_{M+2}}\cdots
F_{\ga_N}^{a_N} : a_i \in \mathbb{N}\},
$$
where $N = |\Phi^+|$ and $M = |\Phi_J^+|$. For $\overline{a} :=
(a_{M+1},\cdots, a_N) \in \mathbb{N}^{N - M}$, set
$$
F_{\overline{a}} :=
F_{\ga_{M+1}}^{a_{M+1}}F_{\ga_{M+2}}^{a_{M+2}}\cdots
F_{\ga_N}^{a_N}.
$$
Place a total (lexicographical) ordering $\prec$ on $\mathbb{N}^{N -
M}$ as follows. Put $\overline{a} \prec \overline{b}$ if and only if
there exists $M + 1 \leq i \leq N$ such that $a_i < b_i$ and $a_j =
b_j$ for all $j > i$. With this ordering, one can define an
$\mathbb{N}^{N- M }$-filtration on $A$. Given $\overline{a} \in
\mathbb{N}^{N- M }$, let $A_{\overline{a}}$ be the subalgebra of $A$
generated by
$$
\{F_{\overline{b}} : \overline{b} \preceq \overline{a}\}.
$$
By Lemma \ref{straighteninglemma}, this is a multiplicative
filtration on $A$. That is, $A_{\overline{a}}\cdot A_{\overline{b}}
\subseteq A_{\overline{a}+\overline{b}}.$

Moreover, by Lemma \ref{straighteninglemma} again, the associated
graded algebra $\gr A$ is generated by $\{X_{\alpha} : \alpha\in
\Phi^{+}\backslash\Phi^{+}_{J}\}$ subject to the relations:
\begin{equation*}\label{relations}
X_{\alpha}\cdot X_{\beta}=\zeta^{\langle \alpha,\beta
\rangle}X_{\beta}\cdot X_{\alpha}\ \ \text{if $\alpha\prec \beta$}.
\end{equation*}
This filtration induces a filtration on $u_{\zeta}(\uj)$ such that
the algebra $\text{gr\,}u_{\zeta}({\mathfrak u}_{J})$ is also
generated by $\{X_{\alpha} : \alpha\in
\Phi^{+}\backslash\Phi^{+}_{J}\}$ subject to the above relations, in
addition to the condition:
\begin{equation*}\label{nilpotent}
X_{\alpha}^{l}=0\ \ \text{for $\alpha\in
\Phi^{+}\backslash\Phi_{J}^{+}$}.
\end{equation*}

Let $\Lambda^{\bullet}_{\zeta,J}$ be the graded algebra with
generators $\{x_{\alpha} : \alpha\in
\Phi^{+}\backslash\Phi^{+}_{J}\}$ where $\text{deg}(x_{\alpha})=1$
subject to the following relations:
\begin{equation*}\label{relations2}
x_{\alpha}\cdot x_{\beta}+\zeta^{-\langle \alpha,\beta
\rangle}x_{\beta}\cdot x_{\alpha} = 0 \ \ \text{if $\alpha \prec
\beta$};
\end{equation*}
\begin{equation*}\label{relations3}
x_{\alpha}^{2}=0\ \ \text{for $\alpha\in
\Phi^{+}\backslash\Phi_{J}^{+}$}.
\end{equation*}
There exists a graded (by degree) algebra isomorphism
$\opH^{\bullet}(\gr{\mathcal U}_{\zeta}({\mathfrak u}_{J}),{\mathbb
C}) \cong \Lambda^{\bullet}_{\zeta,J}$ (cf. \cite[Prop. 2.1]{GK}).
This is also an isomorphism of $U_{\zeta}^0$-modules, where
$\Lambda^\bullet_{\zeta,J}$ is regarded as a $U^0_\zeta$-module by
assigning $x_\alpha$ weight $\alpha$.

\begin{prop}\label{euler} (a) In the character group ${\mathbb Z}X$,
we have
 $$ \sum_{n =
0}^{\infty}(-1)^n\operatorname{ch}\opH^n(\Uz(\uj),{\mathbb C}) =
\sum_{n = 0}^{\infty}(-1)^n\operatorname{ch}\Lambda^n_{\zeta,J}.$$

(b) If $\lambda\in X$ is a weight of $U^0_\zeta$ in
$\opH^n({\mathcal U}_\zeta({\mathfrak u}_J),{\mathbb C})$, then
$\lambda$ is a weight of $U^0_\zeta$ in $\Lambda^n _{\zeta,J}$.
\end{prop}

\begin{proof} Let $A = {\mathcal U}_{\zeta}({\mathfrak u}_{J})$. By Corollary 2.8.3 and
 the discussion above, both $\opH^\bullet(A,{\mathbb C})$ and $\opH^\bullet( \text{gr } A,{\mathbb C})$
  have weight space decompositions with finite dimensional weight spaces. Let $A_{+}$ and $\text{gr } A_+$ denote the augmentation
ideals of $A$ and $\text{gr } A$, respectively. Let $C^{\bullet}(A)$
and $C^{\bullet}(\text{gr } A)$ be the complexes obtained by taking
duals of the respective reduced bar resolutions. More precisely,
$C^{n}(A)=\text{Hom}_{\mathbb C}((A_{+})^{\otimes n}, {\mathbb C})$
and $C^{n}(\text{gr }A)=\text{Hom}_{\mathbb C}((\text{gr
}A_{+})^{\otimes n}, {\mathbb C})$. Note that $A_+$ and $\text{gr }
A_+$ are isomorphic as $U^0_\zeta$-modules. The same holds for
$C^{n}(A)$ and $C^{n}(\text{gr } A)$ and the differentials of both
complexes are $U^0_\zeta$-module maps. Thus, for a weight $\lambda$,
$\opH^{\bullet}(A,{\mathbb C})_\lambda$ and $\opH^{\bullet}(\text{gr
} A,{\mathbb C})_\lambda$ identify with the cohomologies of the
complexes $C^{\bullet}(A)_{\la}$ and $C^{\bullet}(\text{gr }
A)_{\la}$, respectively.

By the Euler-Poincar\'e principle (cf. \cite[Lemma 7.3.11]{GW}),
$$\begin{aligned}
\chi(\opH^\bullet(A,{\mathbb C})_\lambda) &:=\sum_{n =
0}^{\infty}(-1)^n\dim\,\opH^n(A,{\mathbb C})_\lambda \\
&= \sum_{n =
0}^{\infty}(-1)^n \dim\,C^{n}(A)_{\la}
\\  &=\sum_{n=0}^{\infty}(-1)^n\dim\,C^n(\text{gr } A)_\lambda\\
 &=\sum_{n=0}^\infty
(-1)^n\dim\,\opH^n(\text{gr }A,{\mathbb C})_\lambda\\
 &=:
\chi(\opH^\bullet(\text{gr }A,{\mathbb C})_\lambda).\end{aligned}
$$
Part (a) follows from the cohomology calculation for $\gr\,A$
noted above.

Let $A_\bullet$ be the increasing filtration on $A$ indexed by
$\Lambda:={\mathbb N}^{N-M}$, viewed as a poset using the
lexicographic ordering $\prec$ above. It induces a (downward)
filtration on the complex $C^{\bullet}(A)_{\la}$ as follows. For
$\gamma, \eta \in \Lambda$, set ${A_{+}}_{\gamma}=A_{\gamma}\cap
A_{+}$, and define
$$\begin{cases} B^{n}_{[\prec\eta]}=\sum_{\sum \gamma_{i}\prec\eta} {A_{+}}_{\gamma_1}
\otimes {A_{+}}_{\gamma_2} \otimes \cdots \otimes
{A_{+}}_{\gamma_n},\\
B^{n}_{[\preceq \eta]}=\sum_{\sum
\gamma_{i}\preceq \eta} {A_{+}}_{\gamma_1} \otimes
{A_{+}}_{\gamma_2} \otimes\cdots \otimes {A_{+}}_{\gamma_n}.\end{cases}$$
 Then $B^{n}_{[\prec\eta]}\subseteq
B^{n}_{[\preceq \eta]}$, so setting, for $\la \in X(T)$,
$$\begin{cases} C^{n}(A)_{\la,[\prec\eta]}=\text{Hom}_{\mathbb C}(A_{+}^{\otimes n}/B^{n}_{[\prec\eta]},{\mathbb C})_{\la}, \\
C^{n}(A)_{\la, [\preceq \eta]}=\text{Hom}_{\mathbb C}(A_{+}^{\otimes
n}/B^{n}_{[\preceq \eta]},{\mathbb C})_{\la},\end{cases}$$
it follows that
$$C^{n}(A)_{\la,[\preceq \eta]}\subseteq C^{n}(A)_{\la, [\prec
\eta]}.$$
Moreover, if $\eta,\zeta\in \Lambda$ with $\zeta \prec
\eta$, then ${C^{n}(A)_{\la,[\prec\eta]}}\subseteq {C^{n}(A)_{\la,
[\prec\zeta]}}.$

The grading on $\text{gr A}$ leads in a natural way to a grading of
the complex $C^{\bullet}(\text{gr } A)_{\la}.$ For $\eta \in
\Lambda$, $C^{\bullet}(\gr A)_{\lambda,[\eta]}$ denotes the graded
component corresponding to $\eta$, and we can identify
$${C^\bullet(A)_{\la,[\prec\eta]}}/{C^\bullet(A)_{\la,[\preceq
\eta]}}$$ with ${C^{\bullet}(\text{gr } A)_{\la,[\eta]}}.$ Also,
$$C^\bullet(\gr\, A)_\lambda\cong \bigoplus_{\eta\in\Lambda}
C^\bullet(A)_{\la,[\prec\eta]}/C^\bullet(A)_{\la,[\preceq \eta]}.$$

For a fixed weight $\lambda$, $C^n(\gr\, A)_\lambda\neq 0$ for
finitely many $n$, and
$C^\bullet(A)_{\la,[\prec\eta]}/C^\bullet(A)_{\la,[\preceq
\eta]}\neq 0$ for finitely many $\eta$. Let
$\overline{\Lambda}=\{\eta\in \Lambda:\
C^n_{\la,[\prec\eta]}/C^n_{\la,[\preceq \eta]}\neq 0,\ \text{for
some $n$}\}$. Then $\overline{\Lambda}$ is a finite totally ordered
set in ${\mathbb N}^{N-M}$ which induces a filtration (which can be
indexed by ${\mathbb N}$) on $C^{\bullet}(\gr\, A)_\lambda$.
Therefore, we have a spectral sequence
$$E_1^{i,j}= (\opH^{i+j}(\gr A,{\mathbb C})_{\la})_{(i)}\Rightarrow \opH^{i+j}(A,{\mathbb C})_{\lambda}.$$
This shows that $\opH^n({\mathcal U}_\zeta({\mathfrak u}_J),{\mathbb
C})_{\lambda}$ is a subquotient of
$(\Lambda^n_{\zeta,J})_{\lambda}$, and part (b) follows.

\end{proof}


\section{Induction functors}\label{inductionfunctors}

\setcounter{equation}{0}
\setcounter{theorem}{0}
\renewcommand{\thetheorem}{\thesection.\arabic{theorem}}
\renewcommand{\theequation}{\thesection.\arabic{equation}}

Let ${\mathcal C}$ (respectively, ${\mathcal C}^{\leq}$) be the category of
type 1, integrable representations of $U_\zeta$ (respectively,
$U_{\zeta}({\mathfrak b})$). The restriction functor ${\text{\rm
res}}:{\mathcal C}\to{\mathcal C}^\leq$ has a right adjoint
induction functor $\opH^0_{\zeta}(-) =
\opH^0(U_\zeta/U_\zeta({\mathfrak b}),-):{\mathcal C}^\leq\to{\mathcal
C}$ defined by
$$\opH^0_{\zeta}(M)= (M\otimes k[U_\zeta])^{U_\zeta({\mathfrak b})}\cong
{\mathcal F}(\Hom_{U_\zeta({\mathfrak b})}(U_\zeta,M)).$$ In this
expression $k[U_\zeta]$ denotes the coordinate algebra of $U_\zeta$.
Also, the functor ${\mathcal F}(-)$ assigns to any $U_\zeta$-module
the largest type 1, integrable submodule. We refer to \cite[(2.8),
(2.10)]{APW} and \cite[(2.9)]{RH} for further discussion and
explanation of notation.

Any dominant weight $\lambda\in X_+$ can be viewed as a one-dimensional $U_\zeta({\mathfrak
b})$-module, and so provides an induced module
$$\nabla_\zeta(\lambda):=\opH^0(U_\zeta/U_\zeta({\mathfrak
b}),\lambda).$$
This module has an irreducible socle isomorphic to $L_\zeta(\lambda)$. In
 addition, there is an equality
$$\ch\,\nabla_\zeta(\lambda)=\ch\,L(\lambda)=\sum_{x\in W}(-1)^{\ell(x)}
e(w\cdot\lambda)/\sum_{x\in W}(-1)^{\ell(x)}e(x\cdot 0)$$ of formal
characters, in which the expression on the right is just the Weyl character formula.  (Recall that $L(\lambda)$ denotes the
irreducible representation for the complex group $G$ (or its Lie algebra $\mathfrak g$) of high
weight $\lambda$.)

We will also use the induction functors
$\opH^0(U_\zeta/U_\zeta({\mathfrak p}_J),-)$ (respectively,
$\opH^0(U_\zeta({\mathfrak p}_J)/U_\zeta({\mathfrak b}),-)$) from
the category of type 1, integrable $U_\zeta({\mathfrak
p}_J)$-modules (respectively, $U_\zeta({\mathfrak b})$-modules) to type 1,
integrable $U_\zeta$-modules (respectively, $U_\zeta({\mathfrak
p}_J)$-modules). Note that if $\lambda$ is a one-dimensional
$U_{\zeta}({\mathfrak b})$-module then $\opH^0(U_\zeta({\mathfrak
p}_J)/U_\zeta({\mathfrak b}),\lambda)$ is trivial as a
$U_{\zeta}({\mathfrak u}_{J})$-module.

Let $(X_J)_+\subseteq X$ be the set of $J$-dominant weights, i.e.,
$\lambda\in X$ belongs to $(X_J)_+$ if and only if
$\langle\lambda,\alpha^\vee\rangle \in{\mathbb N}$ for all
$\alpha\in J$. The set $(X_J)_+$ indexes the irreducible (type 1,
integrable) $U_\zeta({\mathfrak l}_J)$-modules. For
$\lambda\in(X_J)_+$,
$$\nabla_{J,\zeta}(\lambda):=\opH^0(U_\zeta({\mathfrak p}_J)/U_\zeta({\mathfrak
b}),\lambda)$$
 has irreducible socle isomorphic to $L_{J,\zeta}(\lambda)$, the
 irreducible $U_\zeta({\mathfrak l}_J)$-module of high weight
 $\lambda$.

If $\lambda\in (X_J)_+$ satisfies
$\langle\lambda+\rho,\alpha^\vee)=\ell-1$ for all $\alpha\in J$, we
call $\lambda$ a $J$-Steinberg weight. Then
$\nabla_{J,\zeta}(\lambda)$ is a projective (and injective)
irreducible $U_\zeta({\mathfrak l}_J)$-module (in the category of
type 1, integrable modules). It remains irreducible, projective, and
injective upon restriction to ${\mathfrak u}_\zeta({\mathfrak
l}_J)$.

Finally, it will usually be more convenient to write
$\ind_{U_\zeta({\mathfrak b})}^{U_\zeta({\mathfrak p}_J)}(-)$ in place
of $\opH^0(U_\zeta({\mathfrak p}_J)/U_\zeta({\mathfrak b}),-)$.

%% file: chapt3BNPP.tex
 \chapter{Computation of $\Phi_{0}$ and ${\mathcal N}(\Phi_{0})$}\renewcommand{\thesection}{\thechapter.\arabic{section}}

Throughout this chapter, $\Phi$ is an irreducible root system in an Euclidean space $\mathbb E$ with
weight lattice $X$. We will be concerned with certain special
closed subroot systems $\Phi_{\lambda,l}$ of $\Phi$ which are defined by an integer $l>1$ and a weight
$\lambda\in X$. These are introduced in Section 3.1. After classifying these subroot systems
in Sections 3.2 and 3.3 (see also Appendix \ref{tables1} for the exceptional cases), Section 3.5 takes up some related issues involving the normality of orbit closures. Finally, Section 3.6 uses these results to 
identify  the coordinate algebras of these orbit closures as certain induced $G$-modules.  All these results will play an important role later in this paper; see Chapter 4, for example.

\section{Subroot systems defined by weights}\label{bases}
\renewcommand{\thetheorem}{\thesection.\arabic{theorem}}
\renewcommand{\theequation}{\thesection.\arabic{equation}}
\setcounter{equation}{0}
\setcounter{theorem}{0}

 A prime $p$ is called {\it bad} for the root
system $\Phi$ provided that there exists a closed subsystem $\Phi'$
of $\Phi$ such that $Q/Q'$ has $p$-torsion, where $Q=Q(\Phi)$ and $Q'=Q(\Phi')$ are the root
lattices of $\Phi$ and $\Phi'$, respectively. If $p$ is not bad, then $p$
is called a {\it good} prime for $\Phi$. Equivalently, $p$ is good
if and only if $p$ does not appear as the coefficient of a simple
root in the decomposition of the maximal root in $\Phi$ as an integral linear
combination of simple roots; see \cite[I,\S4]{SS}. The good primes for the various types
of irreducible root systems are thus easily determined.  Therefore, making use of the explicit
expressions for the maximal root in $\Phi$ displayed in \cite[Plates I--IX]{Bo}, the good primes are given
as follows:
\begin{itemize}
\item
$\Phi$ of type $A_{n}$, all $p$;
\item
$\Phi$ of type $B_{n}$, $C_{n}$, $D_{n}$, $p\geq 3$;
\item
$\Phi$ of type $E_{6}$, $E_{7}$, $F_{4}$, $G_{2}$, $p\geq 5$;
\item
$\Phi$ of type $E_{8}$, $p\geq 7$.
\end{itemize}
Now let $l>1$ be an integer (not necessarily prime). We will say that $l$ is good for $\Phi$
provided that $l$ is not divisible by a bad prime for $\Phi$.
Otherwise, $l$ is bad for $\Phi$.

Additionally, a good integer $l$
is said to be {\it very good} for $\Phi$ provided that if $\Phi$ has
type $A_n$, then $l$ and $n+1$ are relatively prime. In cases of non-irreducible root
systems (which may arise in the case of the root system of a Levi factor of a parabolic subgroup),
the integer $l$ is good
(respectively, very good) provided that it is good (respectively, very good) for
every irreducible component of $\Phi$.

The significance of good integers $l$ comes about
because of certain closed subsystems $\Phi_{\lambda,l}$ constructed
from weights $\lambda\in X$. Precisely, put
\begin{equation*}\label{defnofsubsystems}
\Phi_{\lambda,l}:=\{\alpha\in\Phi\,|\,\langle\lambda+\rho,\alpha\rangle\equiv
0\,\,{\text{\rm mod}}\,l\}.
\end{equation*}
Recall that, by our conventions, for $\mu\in X$ and $\alpha\in\Phi$, $\langle \mu,\alpha\rangle\in\mathbb Z$. 
In practice, when $l$ is clear from context, we will just denote
the subset $\Phi_{\lambda,l}$ simply by  $\Phi_{\lambda}$. Obviously,
$\Phi_\lambda$ (when it is not the empty set) is a closed subroot system of $\Phi$.

When working with a quantum
enveloping algebra $U_\zeta$, where $\zeta=\sqrt[l]{1}$, our assumptions on $l$
(i.e., $l$ is odd in types $B_n,C_n$ and $F_4$, and, in type $G_2$,
$l$ is not divisible by $3$) mean that each integer $d_\alpha=\frac{\langle\alpha,\alpha\rangle}{2}$ is relatively prime to $l$.
 Since $d_\alpha\alpha^\vee=\alpha$, it therefore follows that
\begin{equation}\label{defnofsystemstwo}
\Phi_\lambda=\{\alpha\in\Phi\,|\,\langle\lambda+\rho,\alpha^\vee\rangle\equiv
0\,\,{\text{\rm mod}}\,l\}.
\end{equation}
Thus, in the sequel,  we can always take (\ref{defnofsystemstwo}) as the
definition of $\Phi_\lambda$. We denote by $\Phi_\lambda^+$ the intersection   of $\Phi_\lambda $ with $ \Phi^+$. It is useful to observe that, for any
$w\in \widetilde W_l$,
\begin{equation*}\label{actionchange}
\overline{w}(\Phi_\lambda)= \Phi_{w\cdot\lambda}.
\end{equation*}

The following result, while quite elementary, is essential for our
work in this paper.

\begin{lem}\label{subsystemlemma} Let $l>1$ be an odd integer.
  Assume that $l$ is good for $\Phi$. For $\lambda\in X$,
there exists a set of simple roots $\Pi'$ for $\Phi$ such that
$\Pi'\cap\Phi_\lambda$ is a set of simple roots for $\Phi_\lambda$. In particular,
there exists a $w\in W$ and a subset $J\subseteq \Pi$ such that
$w(\Phi_\lambda)=\Phi_J$.  Furthermore, $w$ may be chosen so that $w(\Phi_{\lambda}^+) = \Phi_J^+$.
\end{lem}

\begin{proof} First, observe that
$${\mathbb Q}\Phi_{\lambda}\cap \Phi=\Phi_\lambda.$$
In fact, if a root $\alpha$ belongs to the left-hand side, then
$m\alpha\in{\mathbb Z}\Phi_\lambda$, for some integer $m$ which is
divisible only by bad primes. Then $\langle\lambda+\rho,m\alpha\rangle\equiv 0$ mod$\,l$.
But since $(m,l)=1$, it follows that $\langle\lambda+\rho,\alpha\rangle\equiv 0$ mod$\,l$,
i.e., $\alpha\in\Phi_\lambda$. This means that the hypotheses of \cite[IV.1.7, Prop. 24]{Bo}
are satisfied, and this quoted result implies the conclusion of the lemma. 
The final claim follows from the fact that all simple systems in $\Phi_J$ are conjugate
under the action of the Weyl group. \end{proof}

Consider the case of $\Phi_0$.  For {\it all} $l\geq h$, it is
immediately true that $\Phi_0=\varnothing$, and so the conclusion of the lemma 
trivially holds.
However, this is not the case in general. For example, suppose that $\Phi$ has type
$F_4$ with $l = 3$. Letting $\epsilon_i$, $1\leq i\leq 4$, be an orthonormal basis for ${\mathbb R}^4$, $\Phi$ can be
realized explicitly as the following set of 48 vectors
$$\Phi=\{\pm\epsilon_i\}_{1\leq i\leq 4}\cup\{\pm\epsilon_i\pm\epsilon_j\}_{1\leq i<j\leq 4}\cup
\{\frac{1}{2}(\pm\epsilon_1\pm\epsilon_2\pm\epsilon_3\pm\epsilon_4)\}.$$
We can think of $\Phi$ as the set of all $\alpha\in \bigoplus{\mathbb Q}\epsilon_i$ such that $2\alpha\in\oplus{\mathbb Z}\epsilon_i$ and $\alpha$ has integral square length $||\alpha||\leq 2.$ Also, $\Pi:=\{\epsilon_4,\frac{1}{2}(\epsilon_1-\epsilon_2-\epsilon_3-\epsilon_4),\epsilon_2-\epsilon_3,\epsilon_3-\epsilon_4,\}$
forms a set of simple roots (cf. \cite[Appendix, Plate VIII]{Bo}).  Now take $l=3$ and let $\lambda=0$. Since $\rho=\frac{1}{2}(11\epsilon_1+5\epsilon_2+3\epsilon_2+\epsilon_1)$, it is directly checked that  $\Phi_0$ has
$$\{\epsilon_1-\epsilon_2,\epsilon_2+\epsilon_4\}\cup\{-\epsilon_3,\frac{1}{2}(\epsilon_1+\epsilon_2+\epsilon_3
-\epsilon_4)\}$$
as a set of simple roots. Thus, $\Phi_0$ has type $A_2\times A_2$, and there is
clearly no $J \subset \Pi$ for which $\Phi_J$ has type $A_2\times
A_2$. More generally, still for type $F_4$, if $3$ divides
$l$, say $l=3l'$, then the conclusion of the lemma fails for any $\lambda=(l'-1)\rho$.

On the other hand, $l=9$ satisfies Assumption 1.2.1, but $9$ is not good for $F_4$. Here
$\Phi_0=\{\pm\frac{1}{2}(\epsilon_1+\epsilon_2+\epsilon_3-\epsilon_4)\}$ has type $A_1$,  so  $\Phi_0$ does  satisfy
the conclusion of the lemma. We will see in Theorem  \ref{identificationtheorem} below that this
fact holds generally as long as $l$ satisfies Assumption 1.2.1.

\subsection{Richardson orbits}\label{richardsonorbits} Let $G$ be a complex, simple and simply connected algebraic group over $\mathbb C$
with root system $\Phi$.  For $J\subset\Pi$, the (standard) parabolic subgroup $P_J=L_J\ltimes U_J\supseteq B$ of $G$ has a
dense (open) orbit ${\mathcal C}'_J$ in the Lie algebra ${\mathfrak u}_J$ of $U_J$ under the adjoint action of $P_J$. 
In particular, if $J=\varnothing$, then $L_J=T$, and $P_J=B$, the Borel subgroup corresponding to $\Phi^-$.  The corresponding Richardson orbit ${\mathcal C}_J$ is the $G$-orbit $G\cdot x$ for any $x\in {\mathcal C}_J'$. Therefore, when $J=\varnothing$, ${\mathcal C}_J$ is the regular or principal nilpotent orbit.  Also,
it is straightforward to show that the (Zariski) orbit closure $\overline{\mathcal C}_J$ equals $G\cdot {\mathfrak u}_J$.
In addition, $\overline{\mathcal C}_J$ has dimension $2\dim {\mathfrak u}_J$.

 If $J,K$ are
$W$-conjugate subsets of $\Pi$, the Johnston-Richardson theorem
states that ${\mathcal C}_J={\mathcal C}_K$. Hence, given $\la \in
X$, if there exists $w \in W$ and $J\subseteq \Pi$ with
$w(\Phi_{\la}) = \Phi_J$, then $\la$ defines in a unique way a Richardson class
${\mathcal C}_\lambda$ in $\mathcal N$ by setting ${\mathcal
C}_\lambda={\mathcal C}_J$.   For
$\la \in X$ with $w(\Phi_{\la}) = \Phi_J$ as above,  set
$\mathcal{N}(\Phi_{\la}) := G\cdot \ul_J =\overline{\mathcal C}_J\subseteq \mathcal{N}$.

For more details on the above results, see \cite[Chap. 5]{Hum1}. An important generalization of Richardson classes
(namely, induced classes) will be discussed when it is needed in \S3.5.


\section{The case of the classical Lie algebras}
\renewcommand{\thetheorem}{\thesection.\arabic{theorem}}
\renewcommand{\theequation}{\thesection.\arabic{equation}}
\setcounter{equation}{0}
\setcounter{theorem}{0}

In \cite[\S3.1-3.7]{CLNP},
 explicit
determinations were given for all irreducible root systems $\Phi$ of a $J \subset \Pi$ so that $w(\Phi_0) =
\Phi_J$ when $l = p$ is a good prime. The proofs there work equally well
when $l$ satisfies Assumption \ref{assumption}. As noted before, we use the root
notation and ordering of Bourbaki \cite{Bo}. The results below will be useful
in Chapter 4.

The first theorem below treats the
cases when $\Phi$ has type $A$ or $B$, and the second theorem below summarizes
the situation in types $C$ and $D$. Since $\dim\,{\mathcal N}(\Phi_0)=2\dim\,{\mathfrak u}_J$,
we have $\dim\,{\mathcal N}(\Phi_0)=|\Phi|-|\Phi_0|$. Also, in types $A$---$D$, Assumption \ref{assumption}
means that $l$ is good, so that Lemma \ref{subsystemlemma} is applicable.\footnote{In the statement
of the theorem, root system terms $A_0, B_0, \cdots$ should be ignored.}

\begin{theorem}\label{subsystemtypesAB} Let $l$ be as in Assumption \ref{assumption},
${\mathfrak g}$ be a
classical simple Lie algebra with $\Phi$ of type $A_{n}$ (respectively,
$B_{n}$), and $h=n+1$ (respectively, $2n$) be the Coxeter number of $\Phi$.
\begin{itemize}
\item[(a)] If $l\geq h$ then ${\mathcal N}(\Phi_{0}) ={\mathcal
N}({\mathfrak g})$ and $\dim {\mathcal N}(\Phi_{0})=|\Phi|$.
\item[(b)] Suppose that $l< h$ where $h-1=lm+s$ with $m> 0$ with
and $0\leq s \leq l-1$. Then ${\mathcal N}(\Phi_{0})=G\cdot
{\mathfrak u}_{J}$ where $J\subseteq \Pi$ such that when
\begin{itemize}
\item[(i)] $\Phi$ is of type $A_{n}$,
$$\Phi_{0}\cong \Phi_{J}\cong \underbrace{A_{m}\times\dots \times A_{m}}_{\text{$s+1$ times}}
\times \underbrace{A_{m-1} \times \dots \times A_{m-1}}_{\text{$l-s-1$
times}};$$ where $\dim {\mathcal N}(\Phi_{0})=n(n+1)-m(lm+2s-l+2)$.
\item[(ii)] $\Phi$ is of type $B_{n}$,
$$
\Phi_{0}\cong \Phi_{J}\cong
\begin{cases} \underbrace{A_{m}\times\dots \times A_{m}}_{\text{$\frac{s}{2}$ times}}
\times \underbrace{A_{m-1} \times \dots \times
A_{m-1}}_{\text{$\frac{l-s-1}{2}$ times}}\times B_{\frac{m+1}{2}}
& \text{if $s$ is even ($m$ odd)},\\
\underbrace{A_{m}\times\dots \times A_{m}}_{\text{$\frac{s+1}{2}$
times}} \times \underbrace{A_{m-1} \times \dots\times
A_{m-1}}_{\text{$\frac{l-s-2}{2}$ times}}\times B_{\frac{m}{2}} &
\text{if $s$ is odd ($m$ even).}
\end{cases}
$$
Also,
$$
\dim {\mathcal N}(\Phi_{0})=
\begin{cases} 2n^{2}-\frac{m(lm+2s-l+3)+1}{2} & \text{if $s$ is even ($m$ odd)},\\
2n^{2}-\frac{m(lm+2s-l+3)}{2} & \text{if $s$ is odd ($m$ even). }
\end{cases}
$$
\end{itemize}
\end{itemize}
\end{theorem}

\begin{theorem}\label{subsystemtypesCD} Let $l$ be as in Assumption \ref{assumption}, ${\mathfrak g}$ be a classical
simple Lie algebra with $\Phi$ of type $C_{n}$ (respectively, $D_{n}$), and
$h=2n$ (respectively, $2n-2$) be the Coxeter number of $\Phi$.
\begin{itemize}
\item[(a)] If $l\geq h$ then ${\mathcal N}(\Phi_{0}) ={\mathcal
N}({\mathfrak g})$ and $\dim {\mathcal N}(\Phi_{0})=|\Phi|$.
\item[(b)] Suppose that $l< h$ where $h+1=lm+s$ with $m> 0$
and $0\leq s \leq l-1$. Then ${\mathcal N}(\Phi_{0})=G\cdot
{\mathfrak u}_{J}$ where $J\subseteq \Pi$ such that when
\begin{itemize}
\item[(i)] $\Phi$ is of type $C_{n}$,
$$
\Phi_{0}\cong \Phi_{J}\cong
\begin{cases} \underbrace{A_{m}\times\dots \times A_{m}}_{\text{$\frac{s}{2}$ times}}
\times \underbrace{A_{m-1} \times \dots \times
A_{m-1}}_{\text{$\frac{l-s-1}{2}$ times}}\times C_{\frac{m-1}{2}}
& \text{if $s$ is even},\\
\underbrace{A_{m}\times\dots \times A_{m}}_{\text{$\frac{s-1}{2}$
times}} \times \underbrace{A_{m-1} \times \dots \times
A_{m-1}}_{\text{$\frac{l-s}{2}$ times}}\times C_{\frac{m}{2}} &
\text{if $s$ is odd.}
\end{cases}
$$
Also,
$$
\dim {\mathcal N}(\Phi_{0})=
\begin{cases} 2n^{2}-\frac{m(lm+2s-l-1)+1}{2} & \text{if $s$ is even ($m$ odd)},\\
2n^{2}-\frac{m(lm+2s-l-1)}{2} & \text{if $s$ is odd ($m$ even).}
\end{cases}
$$
\item[(iv)] $\Phi$ is of type $D_{n}$,
$$
\Phi_{0}\cong \Phi_{J}\cong
\begin{cases} \underbrace{A_{m}\times\dots \times A_{m}}_{\text{$\frac{s}{2}$ times}}
\times \underbrace{A_{m-1} \times \dots \times
A_{m-1}}_{\text{$\frac{l-s-1}{2}$ times}}\times D_{\frac{m+1}{2}}
& \text{if $s$ is even and $m\geq 3$},\\
\underbrace{A_{m}\times\dots \times A_{m}}_{\text{$\frac{s-1}{2}$
times}} \times \underbrace{A_{m-1} \times \dots \times
A_{m-1}}_{\text{$\frac{l-s}{2}$ times}}\times D_{\frac{m+2}{2}}
& \text{if $s$ is odd}, \\
\underbrace{A_{m}\times\dots \times A_{m}}_{\text{$\frac{s}{2}$
times}} \times \underbrace{A_{m-1} \times \dots \times
A_{m-1}}_{\text{$\frac{l-s+1}{2}$ times}} & \text{if $s$ is even and
$m=1$.}
\end{cases}
$$
Also,
$$\dim {\mathcal N}(\Phi_{0})=
\begin{cases} 2n^{2}-2n-\frac{m(lm+2s-l+1)-1}{2} & \text{if $s$ is even ($m$ odd)},\\
2n^{2}-2n-\frac{m(lm+2s-l+1)}{2} & \text{if $s$ is odd ($m$ even).}
\end{cases}
$$
\end{itemize}
\end{itemize}
\end{theorem}


\section{The case of the exceptional Lie algebras}\label{exceptionalsec}
\renewcommand{\thetheorem}{\thesection.\arabic{theorem}}
\renewcommand{\theequation}{\thesection.\arabic{equation}}

\setcounter{equation}{0}
\setcounter{theorem}{0}

 For the exceptional types $G_2, F_4, E_6, E_7,$ and $E_8$, the precise
 determination of the subroot system $\Phi_{0}$ and the variety ${\mathcal N}(\Phi_{0})$ can be
carried out by hand. In most cases, the determination of this variety can
 be deduced from its dimension ($\dim {\mathcal
N}(\Phi_{0})=|\Phi|-|\Phi_{0}|$) and the fact that ${\mathcal N}(\Phi_{0})$ is the
closure of a Richardson orbit; see  \cite[\S4.2]{CLNP}. When this
information is not sufficient, the correct Richardson orbit can be pinned down by
using the Weyl group as discussed in \cite[\S 4.3]{CLNP}. In fact, for
computational purposes in Chapter \ref{combinSteinbergsec}, for each
value of $l$ satisfying Assumption \ref{assumption}, we identify an
explicit element $w \in W$ and a subset $J \subset \Pi$ such that
$w(\Phi_0^+) = \Phi_J^+$.  As observed before, the choices of $w$ and $J$ are not
unique in general.  The computer package MAGMA \cite{BC,BCP} was used
to verify these facts. The tables providing the description of ${\mathcal
N}(\Phi_{0})$, $w$, and $J$ for various possible values of $l$ are presented in the Appendix \ref{tables1}.


\section{Standardizing $\Phi_0$}\label{subsecidentification}
\renewcommand{\thetheorem}{\thesection.\arabic{theorem}}
\renewcommand{\theequation}{\thesection.\arabic{equation}}
\setcounter{equation}{0}
\setcounter{theorem}{0}

If $l$ satisfies Assumption \ref{assumption} and if $\Phi$ has classical type, then
$l$ is automatically good for $\Phi$. Thus, in this case, the theorem below follows
 immediately from Lemma \ref{subsystemlemma}. However, in the exceptional types, we need
 to quote the computer results tabulated in Appendix \ref{tables1} in case $l$ is not
 good for $\Phi$ (but still satisfies Assumption \ref{assumption}).  We essentially worked out
 the case of $F_4$ after the statement of Lemma \ref{subsystemlemma}, where $l=9$ is the only
 value that needs to be considered.

For future reference, we summarize this result as follows.

\begin{theorem}\label{identificationtheorem}
Let $l$ be as in Assumption \ref{assumption}. Then there exists $w
\in W$ and a subset $J\subseteq \Pi$ such that $w(\Phi_{0,l}^+) =
\Phi_J^+$.
\end{theorem}


\section{Normality of orbit closures}\label{normality}
\renewcommand{\thetheorem}{\thesection.\arabic{theorem}}
\renewcommand{\theequation}{\thesection.\arabic{equation}}

\setcounter{equation}{0}
\setcounter{theorem}{0}
We consider certain nilpotent orbit closures for
a complex simple Lie algebra $\gl = \gl_{\mathbb{C}}$ having
root system $\Phi$, etc. Let $G$ be a complex (connected) algebraic group of the same root type as $\gl$.
Since we are only interested in the adjoint action of $G$ on its Lie algebra $\mathfrak g$, we do not
require $G$ to be simply connected. In case $\Phi$ has type $B_n$ (resp., $C_n$, $D_n$), we 
will take $G=SO_{2n+1}({\mathbb C})$ (resp., $Sp_{2n}({\mathbb C})$, $SO_{2n}({\mathbb C})$).

 In the discussion
below we will make use of the considerable work available in determining normal
$G$-orbit closures in the nilpotent variety ${\mathcal N} = {\mathcal N}(\gl)$.

Fix
an integer $l$ satisfying Assumption \ref{assumption}, and let $\Phi_0 =
\Phi_{0,l}$. We are especially interested when the variety ${\mathcal N}(\Phi_0)$ is normal. In fact, we
will verify that, in almost all cases, it is normal.

\subsection{The classical case.} When $\Phi$ has type $A_n$, a famous result of Kraft-Procesi 
\cite[\S0, Theorem]{KP1}
states all orbit closures in $\mathcal N$ are normal varieties. In this case, nilpotent orbits in $\mathfrak g= {\mathfrak sl}_{n+1}({\mathbb C})$
are naturally in one-to-one correspondence with the set ${\mathcal P}(n+1)$ of partitions $\eta=(\eta_1,\eta_2,\cdots,
\eta_{n+1})$ of $n+1$.  Here $\eta_1\geq\eta_2\geq\cdots\geq \eta_n\geq 0$ and $\eta_1+\cdots+\eta_{n+1}=n+1$.
We also write $\eta\vdash n+1$ to mean that $\eta\in{\mathcal P}(n+1)$.  Also, if a part $a$ is repeated $b$ times,
we often write $a^b$ in place of $\underbrace{a,\cdots,a}_{b}$ in $\eta$. Thus, if ${\mathcal O}_\eta$ denotes the corresponding orbit, the elements $x\in {\mathcal O}_\eta$ are
 just those nilpotent $(n+1)\times (n+1)$-matrices which have Jordan blocks of sizes $\eta_i\times\eta_i$, $1\leq i\leq n$.  A parabolic subgroup $P_J$
also determines a {\it composition} $\eta=\eta_J=(\eta_1,\cdots,\eta_{n+1})$ of $n+1$ meaning that each $\eta_i\geq 0$
and $\eta_1+\cdots+\eta_{n+1}=n+1$. Thus, when displayed in the usual way as matrices, the Levi factor $L_J$ of $P_J$
corresponds to blocks of sizes $\eta_1\times\eta_1, \dots,\eta_{n+1}\times\eta_{n+1}$ down the diagonal. Two
Levi factors $L_J$ and $L_K$ are conjugate in the group $G=SL_{n+1}({\mathbb C})$ if and only if the partitions defined by rearrangement of the two compositions
$\eta_J$ and $\eta_K$ are equal.   We let $\widetilde\eta_J$ denote this partition.
A well-known result of Kraft states that ${\mathcal N}(\Phi_J)$ is the closure in $\mathcal N$ of the nilpotent class
defined by the partition $\widetilde\eta_J'$ dual to $\widetilde\eta_J$.  In this way, the Richardson orbits in
type $A_n$ 
described in Theorem 3.2.1 can be explicitly identified as certain ${\mathcal O}_\eta$, $\eta\vdash n+1$. 

 The set ${\mathcal P}(n)$ (as well as subsets considered below) is partially ordered by putting $\eta\trianglelefteq\sigma$
provided that, for each $i$, $\eta_1+\cdots+\eta_i\leq\sigma_1+\cdots+\sigma_i$. Then, given $\eta,\sigma\in {\mathcal P}(n+1)$,
${\mathcal O}_\eta\subseteq\overline{{\mathcal O}_\sigma}$ if and only if $\eta\trianglelefteq\sigma$. 
For a nonempty subset $\Gamma\subseteq{\mathcal P}(n)$, a pair in $\Gamma$ is
a pair $(\eta,\sigma)$ of distinct elements
in $\Gamma$ such that $\eta\trianglelefteq\sigma$. If there is no other element $\tau\in\Gamma$ which is strictly
between $\eta$ and $\sigma$ in the ordering $\trianglelefteq$, the pair is called minimal.

For the  classical types $B_n$ and $C_n$, the nilpotent classes   are
also in natural one-to-one correspondence with certain sets of partitions. Thus, for a positive integer $N$, let ${\mathcal P}_1(N)$
be the set of partitions $\eta\vdash N$ in which each even part $\eta_i$ is repeated an even number of times. Similarly,
if $N$ is even, let ${\mathcal P}_{-1}(N)$ consist of those $\eta\vdash N$ in which each odd part $\eta_i$ is repeated
an even number of times. When $\Phi$ has type $B_n$ (resp., $C_n$), 
the $G$-orbits on $\mathcal N$ correspond naturally to the
$\eta\in{\mathcal P}_1(2n+1)$ (resp., $\eta\in{\mathcal P}_{-1}(2n)$). In fact, given $\eta\in{\mathcal P}_\epsilon(N)$ $(\epsilon=\pm 1)$, the
corresponding orbit is ${\mathcal O}_{\epsilon,\eta}:={\mathcal O}_{\eta}\cap \mathfrak g$, i.e., it is the intersection
of the nullcone $\mathcal N$ of $G$ with the corresponding $SL_{2n+1}({\mathbb C})$-orbit (resp., $SL_{2n}({\mathbb C})$-orbit)
indexed by $\eta$. 

If $\Phi$ has type $D_{2n}$, then the $G$-orbits in $\mathcal N$ correspond to the elements in ${\mathcal P}_{1}(2n)$, except in the case in which all
the parts of $\eta$ are even (i.e., $\eta$ is very even); in this case, $\eta$ defines two orbits ${\mathcal O}_\eta^I$
and ${\mathcal O}_\eta^{II}$ (which combine under the action of the full orthogonal group $O(2n)$ on $\mathfrak g$).

If $\eta\in{\mathcal P}_\epsilon(m)$ for some positive integer $m$, call $\eta$ an $\epsilon$-partition.  Given
a pair $(\eta,\sigma)$ in ${\mathcal P}_\epsilon(n)$, suppose that $\eta_i=\sigma_i$, $1\leq i\leq r$, and
$\eta'_j=\sigma'_j$, $1\leq j\leq s$. In other words, the first $r$ rows and first $s$ columns in the Young
diagrams associated to $\eta$ and $\sigma$ are the same. Suppose that $(\eta_1,\cdots,\eta_r)$ is
an $\epsilon$-partition, and consider the pair $(\overline\eta,\overline\sigma)$ in ${\mathcal P}_{(-1)^s\epsilon}(m)$
obtained by lettting $\overline\eta=(\eta_{r+1},\cdots)$, $\overline\sigma=(\sigma_{s+1},\cdots)$, and
$m=\eta_{r+1}+\cdots \eta_{n}$. In this case, write $(\eta,\sigma)\to(\overline\eta,\overline\sigma)$, allowing
the possibility that no rows or columns were removed, i.e., $\overline\sigma=\sigma$ and $\overline\eta=\eta$. 
Then a pair $(\eta,\sigma)$ is called irreducible if $(\eta,\sigma)\to(\overline\eta,\overline\sigma)$ implies
that $(\eta,\sigma)=(\overline\eta,\overline\sigma)$.  By means of this process, every irreducible pair
$(\eta,\sigma)$ in ${\mathcal P}_\epsilon(n)$ can be reduced to an irreducible minimal pair $(\overline\eta,
\overline\sigma)$ in ${\mathcal P}_{\epsilon'}(m)$. The irreducible minimal pairs $(\eta,\sigma)$ are
classified in \cite[3.4]{KP2}; there are precisely  8 distinct minimal irreducible pairs.  Of course, a
pair $(\eta,\sigma)$ may be minimal irreducible in say ${\mathcal P}_1(2n)$, but not minimal irreducible
when regarded as a pair in ${\mathcal P}_{-1}(2n)$. 

 A main result, stated in \cite[Thm. 16.2]{KP2}, establishes
that, given $\sigma\in{\mathcal P}_\epsilon(n)$, the orbit closure $\overline{\mathcal O}_{\epsilon,\sigma}$ is normal in 
  a codimension 2 class ${\mathcal O}_{\epsilon,\eta}\subset\overline{\mathcal O}_{\epsilon,\sigma}$ if
  and only if the irreducible minimal pair $(\overline\eta,\overline\sigma)\in{\mathcal P}_{\epsilon'}(n')$
  obtained from $(\eta,\sigma)$ (by the row and column removal process described above) is not the pair
  \begin{equation}\label{badpair}  (2m,2m), (2m-1,2m-1,1,1)\in {\mathcal P}_1(4m)\times{\mathcal P}_1(4m).
  \end{equation}
  So $n'=4m$ in this case.  We will use this result to study the normality of the ${\mathcal N}(\Phi_0)$.

Given $\eta\in{\mathcal P}(2n+1)$, there exists a unique $\eta\in{\mathcal P}_1(2n+1)$ which is largest partition (with respect to
$\trianglelefteq$) among all $\sigma\in{\mathcal P}_1(2n+1)$ satisfying $\sigma\trianglelefteq\eta$. It is denoted $\eta_B$ and
it is called the $B$-collapse of $\eta$. See \cite[Lemma 6.3.3]{CM} where its (simple) construction is indicated. Similarly,
if $\eta\in {\mathcal P}_1(2n)$ (resp., $\eta\in{\mathcal P}_{-1}(2n)$), there is a unique largest $\eta_D$ (resp.,
$\eta_C$) in ${\mathcal P}_1(2n)$ (resp., ${\mathcal P}_{-1}(2n)$) among those elements $\sigma\in{\mathcal P}_1(2n)$ (resp., $\sigma\in{\mathcal P}_{-1}(2n)$) satisfying $\sigma\trianglelefteq\eta$.

The extension of Kraft's result described above to the other classical types requires the notion of an
induced nilpotent class. Thus, let ${\mathfrak p}={\mathfrak l}\oplus\mathfrak n$ be a parabolic subalgebra of
$\mathfrak g$. For any nilpotent class ${\mathcal O}_{\mathfrak l}$ in $\mathfrak l$, the parabolic subgroup
$P$ with Lie algebra $\mathfrak p$ has a unique open (nilpotent) orbit ${\mathcal O}_{\mathfrak l}'$ in ${\mathcal O}_{\mathfrak l}+{\mathfrak n}$, which in turn defines a nilpotent orbit, denoted $\ind_{\mathfrak l}^{\mathfrak g}{\mathcal O}_{\mathfrak l}$
since it can be proved that it does not depend on the choice of parabolic subalgebra $\mathfrak p$ having Levi factor $\mathfrak l$.
Furthermore, induction of nilpotent classes is transitive, i.e., for ${\mathfrak l}_1\subseteq{\mathfrak l}_2\subseteq{\mathfrak l}_3$, $\ind_{{\mathfrak l}_1}^{{\mathfrak l}_3}=\ind_{{\mathfrak l}_2}^{{\mathfrak l}_3}\circ\ind_{{\mathfrak l}_1}^{{\mathfrak l}_2}$ on nilpotent orbits in ${\mathfrak l}_1$. In addition, if ${\mathcal O}_J$ is the Richardson class defined by
$J\subset\Pi$, then ${\mathcal O}_J=\ind_{{\mathfrak l}_J}^{\mathfrak g}{\mathcal O}_0$, where ${\mathcal O}_0$ denotes the
trivial class in ${\mathfrak l}_J$. Since $[{\mathfrak l}_J,{\mathfrak l}_J]$ is a direct sum of classical simple Lie algebras,
${\mathcal O}_0$ is defined by a single column partition $(1^s)$ on each simple component. Therefore, using
\cite[Thm. 7.3.3]{CM}, ${\mathcal O}_J$ can be explicitly described as a nilpotent class ${\mathcal O}_{\epsilon,\eta}$. 

Write $N=m'l+s'$ where $0\leq s' \leq l-1$.  Then ${\mathcal N}(\Phi_{0})= \overline{{\mathcal
O}}_{\sigma_{X}}$ where $\sigma=(l^{m'},s')$ and $\sigma_{X}$ is
the $X$-collapse of $\sigma$ ($X\in\{B,C,D\}$). See \cite{UGA1} for more details. Now we analyze each type.

\medskip
\noindent
\underline{Case 1:  $\Phi$ has type $B_n$.} First, suppose that $s'$ is odd. Then $\sigma_B=(l^{m'},s')=\sigma$.
In this case, because $l$ is odd, it is impossible to reduce $\sigma$ to a partition $(2m,2m)\in{\mathcal P}_1(4m)$
(which occurs in (\ref{badpair})) by removing rows $\sigma_1,\cdots,\sigma_r$ (which automatically form
an $\epsilon$-partition) and an {\it even} number of columns $(\sigma_1',\cdots,\sigma_s')$. Therefore,
${\mathcal N}(\Phi_0)$ is normal in this case.

\medskip\noindent
	\underline{Case 2: $\Phi$ has type $B_n$ and $s'$ is even.} Then $\sigma_B=\sigma$ if $s'=0$ and
$\sigma_B=(l^{m'},s'-1,1)$ if $s'>0$. Again, it is clearly impossible to reduce such a partition to
one of the form $(2m,2m)$ by removing rows and an even number of columns. Therefore, ${\mathcal N}(\Phi_0)$
is normal in this case.

\medskip\noindent
	\underline{Case 3: $\Phi$ has type $D_n$, thus  $2n=m'l+s'$.}
	 If $s'$ is odd or 0, then $\sigma_D=\sigma$, and the situation is very similar to
that in Case 1, and normality follows.  If $s'$ is a positive even integer, then $\sigma_D=(l^{m'}, s'-1,1)$,
placing in the same situation as in Case 2.  Thus, ${\mathcal N}(\Phi)$ is normal in type $D$.

\medskip\noindent
	\underline{Case 4: $\Phi$ has type $C_n$, $2n=m'l+s'$, and $s'$ is even (thus $m'$ is even). }
Then $\sigma_C=\sigma=(l^{m'},s')\in{\mathcal P}_{-1}(2n)$. Now we must remove an {\it odd}
number $t$ of columns to get to some ${\mathcal P}_1(4m)$. Clearly, in order to obtain $(2m,2m)$,
we must have $s'<t<l$. Thus, we can assume that $s'+1<l$. The possible $\eta$ for which
$(\eta,\sigma)$ is minimal in ${\mathcal P}_{-1}(2n)$ are 
$$
\begin{cases}\eta_1=(l^{m'-2},l-1,l-1,s'+2);\\
\eta_2=(l^{m'},1,1),\quad s'=2;\\
\eta_3=(l^{m'},s'-2,2),\quad s>2. \\  \end{cases}$$
But because $s'<t$, it is impossible to obtain $(2m-1,2m-1,1,1)$ from any of these partitions while
obtaining $(2m,2m)$ from $\sigma$ by removing rows and (an even number $t>s'$) of columns.

\medskip\noindent
	\underline{Case 5:  $\Phi$ has type $C_n$, $2n=m'l+s'$, and $s'$ is odd (thus $m'$ is odd). }
In this case, 
$$\sigma_C= (l^{m'-1},l,s'+1).$$  
If $l-1>s'+1$, it is impossible to obtain a partition
$(2a,2a)$, $a>0$, from $\sigma_C$ by removing rows (from top to bottom) and columns (from left to right). 
If $l-1=s'+1$, then all the $m'-1$ rows of length $l$ must be removed and then an {\it odd} number of
columns must be removed. Since $l-1=s'+1$ is even, it is again not possible to arrive at a
partition $(2a,2a)$, $a>0$.

\medskip
We conclude in all classical cases that ${\mathcal N}(\Phi_0)$ is normal.

\subsection{Exceptional cases}  In case $J\subseteq\Pi$ consists of mutually   orthogonal short roots, 
an important result of Broer \cite[Thm. 4.1]{Br2} can be applied,  establishing that ${\mathcal N}(\Phi_J)$ is a normal variety. In particular,
this applies to the full nullcone ${\mathcal N}={\mathcal N}(\varnothing)$ (a well-known result of Kostant)
and the closure ${\mathcal N}_{\text{subreg}}={\mathcal N}(\{{\pm\alpha}\})$,
$\alpha\in\Pi$ (short), of the subregular class. Using the tables in Appendix \ref{tables1} (where $J$ is explicitly
described), we can treat
the various exceptional types below. Notice that ${\mathcal N}_{\text subreg}$ is the closure of the
unique nilpotent class of codimension $2$ in $\mathcal N$. 

\medskip
\noindent{\underline{Type $G_2$}.} The relevant ${\mathcal N}(\Phi_0)$ are either the full nullcone
$\mathcal N$ (for $l\geq 6$) or ${\mathcal N}_{\text subreg}$. As remarked above, these are normal.

\medskip\noindent{\underline{Type $F_4$.}} There are four possible orbits, having Carter-Bala labels $F_4(a_2)$, $F_4(a_2)$, $F_4(a_1)$,
and $F_4$ and corresponding distinguished Dynkin diagrams labeled $(0200)$, $(0202)$, $(2202)$, and
$(2222)$, respectively (using \cite[p. 128]{CM}). By \cite[Thm. 1]{Br1}, the corresponding orbit closures are normal.

\medskip\noindent{\underline{Type $E_6$.}} There are four relevant orbits having Bala-Carter labels $A_4+A_1$,
$E_6(a_3)$, $E_6(a_1)$, and $E_6$. The last three are Richardson classes, with Levi factor root systems $A_1\times
A_1\times A_1$, $A_1$, and  $\varnothing$. Thus, Broer's theorem \cite[Thm. 4.1]{Br2} applies to guarantee
these have normal orbit closures. The final case $A_4 + A_1$ has normal orbit closure, by \cite[p. 296]{So1}.

\medskip\noindent{\underline{Type $E_7$}.} There are seven relevant orbits having Bala-Carter
labels $A_4+A_2$, $A_6$, $E_6(a_1$), $E_7(a_1)$, $E_7(a_2)$, $E_7(a_3)$, and $E_7$. Again, \cite[Thm. 4.1]{Br2}
implies the last five have normal orbit closures. Using techniques from \cite{So1} and \cite{So2}, Sommers
\cite{So3}
has informed us that he has verified the normality for the remaining two cases (unpublished). 

\medskip\noindent{\underline{Type $E_8$}.} There are nine relevant orbits having Bala-Carter labels
$A_6+A_1$, $E_8(b_6)$, $E_8(a_6)$, $E_8(a_5)$, $E_8(a_4)$, $E_8(a_3)$, $E_8(a_2)$, $E_8(a_1)$, 
and $E_8$. The last five have normal orbit closures, again by \cite[Thm. 4.1]{Br2}. Again, Sommers has
informed us that he has verified normality for $E_8(a_6)$ and $E_8(a_5)$. The remaining two cases, 
$A_6+A_1$ (when $l=7$) and $E_8(b_6)$ (when $l=9$) remain open at present.  

\subsection{Summary}
We summarize the analysis in the following
theorem.

\begin{theorem}\label{normalitytheorem}
 Let $l$ be as in Assumption \ref{assumption} and $J\subseteq \Pi$ so
that ${\mathcal N}(\Phi_{0})=G\cdot {\mathfrak u}_{J}$. If $\Phi$ is
of type $E_8$, assume that $l\neq 7,9$. Then ${\mathcal
N}(\Phi_{0})$ is a normal variety.
\end{theorem}


\section{Resolution of singularities}\label{resolutionofsingularitiessec}
\renewcommand{\thetheorem}{\thesection.\arabic{theorem}}
\renewcommand{\theequation}{\thesection.\arabic{equation}}

\setcounter{equation}{0}
\setcounter{theorem}{0}
 We maintain the notation of the
previous section. Let $J\subseteq \Pi$ and $P_{J}$ be the
associated parabolic subgroup.  From the Bruhat decomposition, it
follows that the quotient map $G\overset\pi\to G/P_J$ has local sections
in the sense that $G/P_J$ has an open covering by affine spaces
$X\cong {\mathbb A}^{\dim {\mathfrak u}_J}$ such that $\pi^{-1}X\cong X\times P_J$. Thus, the  orbit space
$G\times^{P_J}{\mathfrak u}_J:= (G\times {\mathfrak u}_J)/P_J$ for the natural right action of $P_J$
on $G\times {\mathfrak u}_J$ satisfies 
$${\mathbb C}[G\times^{P_J}{\mathfrak u}_J]
\cong ({\mathbb C}[G]\otimes S^\bullet({\mathfrak u}_J^*))^{P_J}=:\ind_{P_J}^GS^\bullet({\mathfrak u}_J^*).$$
Also, the natural projection $G\times^{P_J}{\mathfrak u}_J\to G/P_J$ identifies $G\times^{P_J}{\mathfrak
u}_J$ with the cotangent bundle of the smooth variety $G/P_J$.  For background and more discussion
of the topics in this section, see \cite[pp. 90--97]{Jan4}. 

In addition, there is the moment (or collapsing) map $\mu:G\times^{P_J}{\mathfrak u}_J\to G\cdot{\mathfrak u_J}$
defined by mapping the $P_J$-orbit $[x,u]$ of $(x,u)\in G\times {\mathfrak u}_J$ to $x\cdot u\in G\cdot{\mathfrak
u}_J$. Then $\mu$  is a desingularization of $G\cdot{\mathfrak u}_J$, in the sense that it is 
a birational, proper morphism of varieties, if and only if the following condition holds:
\begin{equation}\label{richardson} 
\text{\rm for all $x\in {\mathcal C}'_J$, $C_G(x)=C_{P_J}(x)$.}
\end{equation}
Recall that the condition $x\in{\mathcal C}'_J$ means just that $P_J\cdot x$ is dense in ${\mathfrak u}_J$. It is
well known that $C_G(x)^o=C_{P_J}(x)^o$, i.~e., that the two centralizers $C_G(x)$ and $C_{P_J}(x)$
have the same connected components of the identity; see \cite[Cor. 5.2.2]{Car}.

\begin{lem}\label{even}  (a) If $\mu$ is a desingularization, then 
$$ {\mathbb C}[G\times^{P_J}{\mathfrak u}_J]\cong {\mathbb C}[\mu^{-1}{\mathcal C}_J]\cong {\mathbb C}[{\mathcal C}_J],$$
where, for a complex variety $X$, ${\mathbb C}[X]$ denotes the algebra of regular functions on $X$. (Recall
that ${\mathcal C}_J:=G\cdot x$ for any $x\in{\mathcal C}'_J$.)

(b) If $x$ is an
even nilpotent element (in the sense of Bala-Carter, see \cite[Ch. 8]{CM}),   condition (\ref{richardson}) on centralizers holds.\end{lem}

\begin{proof} (a) holds by \cite[Remark, p. 95]{Jan4}, and (b) follows from  \cite[Remark, p. 93]{Jan4}.
\end{proof}

 In the theorem below, we show that (\ref{richardson})  holds in all the situations of interest in this paper.
 
If $\mu$ is a desingularization, and if, in addition,  ${\overline{\mathcal O}_J}=G\cdot u_J$ is a normal variety, then
\begin{equation}\label{inductionidentification}
\text{ind}_{P_{J}}^{G} S^{\bullet}({\mathfrak u}_{J}^{*})\cong
{\mathbb C}[G\times^{P_{J}} {\mathfrak u}_{J}] \cong {\mathbb
C}[G\cdot {\mathfrak u}_{J}].
\end{equation}
This follows because  ${\overline{\mathcal O}_J}\backslash{\mathcal O}_J$ has codimension
at least 2.

Now we can state the following result.

\begin{theorem}\label{isomorphismofinducedthm}
Let $l$ be as in Assumption \ref{assumption} and choose
$J\subseteq \Pi$ so
that ${\mathcal N}(\Phi_{0})=G\cdot {\mathfrak u}_{J}$. 

(a) The moment map $\mu:G\times^{P_J}{\mathfrak u}_J\to G\cdot{\mathfrak u}_J$ is
a $G$-equivariant desingularization of $G\cdot{\mathfrak u}_J$.

(b) If $\Phi$ is
of type $E_8$, assume that $l\neq 7,9$. Then the identifications
(\ref{inductionidentification} ) hold. \end{theorem}
 
\begin{proof} We first prove that the condition (\ref{richardson}) holds, namely, that
$C_G(x)\subseteq P_J$, where $x\in {\mathfrak u}_J$. Then (a) holds by the discussion above Lemma \ref{even}.
Consequently, Theorem \ref{normalitytheorem} and the discussion right above establishes (b).  
 
The verification of (\ref{richardson}) will be case-by-case.
 
 \medskip
\noindent\underline{Case 1: $\Phi$ has type $A_n$.} Without loss of generality we can 
assume that $G=GL_{n}(k)$. The centralizer is connected so $C_G(x)=C_{G}(x)^{0}$. 
Since $x$ is Richardson we have by \cite[Corollary 5.2.4]{Car} $C_{G}(x)^{0}\subseteq 
P_{J}$.

\medskip\noindent
\underline{Case 2: $\Phi$ has type $B_n$.}  Let $N=2n+1$ and, as before, write
$N=lm'+s'$ where $0\leq s' \leq l-1$
 and $m'>0$. Set $\eta=(l^{m'},s')$ and
recall that ${\mathcal N}(\Phi_{0})= \overline{{\mathcal
O}}_{\eta_{B}}$where $\eta_{B}$ is the $B$-collapse of
$\eta$. 

For type $B_{n}$ we have
$$
\eta_{B}=
\begin{cases} (l^{m'},s') & \text{if $s'$ is odd or $s'=0$},\\
(l^{m'},s'-1,1) & \text{if $s'$ is even and $s'\neq 0$}.
\end{cases}
$$
In either case each of the nonzero parts are odd so the associated
weighted Dynkin diagram has even entries \cite[\S 5.3]{CM}.
Therefore, the orbit ${\mathcal O}_{\eta_{B}}$ is even, and the
centralizer $C_G(x)$ is contained in $P_{J}$ by Lemma \ref{even}.

\medskip\noindent    
\underline{Case 3: $\Phi$ has type $D_n$.} This case is similar to that in
Case 2, and it is left to the reader.

\medskip\noindent
\underline{Case 4: $\Phi$ has type $C_n$.}  Now let $N=2n$ and write
$N=lm'+s'$. Then 
$$
\eta_{C}=
\begin{cases} (l^{m'},s') & \text{if $m'$ is even ($s'$ even)},\\
(l^{m'-1},l-1,s'+1) & \text{if $m'$ is odd ($s'$ odd)}.
\end{cases}
$$
In the first case, when $m'$ and $s'$ are both even, let $b$ be the number
of distinct even nonzero parts. So $b=1$ and by \cite[p. 92]{CM} the
component group $A({\mathcal C}_{\eta_{C}})\cong ({\mathbb
Z}/2{\mathbb Z})^{b-1}$ is trivial. Here the component group is defined as $C_{G'}(x)/C_{G'}(x)^o$
for any $x\in {\mathcal C}'_J$, where $G'$ is the adjoint group $PSp_{2n}$. Thus, $C_G(x)$ is generated
by $C_G(x)^o$ and the center of $G$, which is contained in $P_J$. Thus, $C_G(x)\subseteq C_{P_J}(x)$
here also.

Now suppose that $m'$ and $s'$ are odd.   In this case, the component
group $A({\mathcal C}_J)$ is isomorphic to ${\mathbb Z}/2{\mathbb Z}$, so
 a different line of reasoning is required. In \cite[Cor. 7.7]{H},
Hesselink provides a necessary and sufficient criterion for having
$C_{G}(x)\subseteq P_{J}$ for $x\in {\mathcal C}_J'$. Here $\eta_{C}$ satisfies condition
(i) of \cite[Cor. 7.7]{H} where $\epsilon=1$ and $\eta_{C}\in
\text{Pai}(2n,m'-1)$ (see \cite[\S 6.1]{H} for the definitions/notation).

\medskip\noindent
\underline{Case 5: $\Phi$ has exceptional type $G_2$, $F_4$, $E_6$, $E_7$, or $E_8$.}
 In these cases, we refer to Appendix \ref{tables1}, where the classes ${\mathcal C}_J$ are all identified
 in the Bala-Carter notation. Using the tables given in \cite[pp. 128-134]{CM}, we see that all but two of the orbits 
 are even orbits (so that Lemma \ref{even} is applicable). In the two other cases, namely,
 type $A_4+ A_1$ in type $E_6$ or type $A_6+A_1$ in type $E_8$, the full component group
 $C_G(x)/C_G(x)^o$ is determined to be trivial there, where the component group 
 is identified with the fundamental
 group $\pi_1({\mathcal C}_J)$ of the orbit. More directly, we can use the determination of this
 component group given in \cite{CM} in all cases. 
 
 The theorem is completely proved. \end{proof}

%% file: chapt4BNPP.tex
\chapter{Combinatorics and the Steinberg
Module}\label{combinSteinbergsec}\renewcommand{\thesection}{\thechapter.\arabic{section}}

In the computation of the cohomology algebra
$\opH^{\bullet}(u_{\zeta}({\mathfrak g}),{\mathbb C})$ for $l>h$ in \cite{GK},
a key step is showing that the space of $u_{\zeta}({\mathfrak h})$-invariants on
$\opH^{\bullet}(\Uz(\mathfrak{u}),{\mathbb C})$ is one dimensional (where
${\mathfrak h} \subset {\mathfrak g}$ is the Cartan subalgebra).
This fact follows because the space of $u_{\zeta}({\mathfrak h})$-invariants 
on $\Lambda^{\bullet}_{\zeta,\varnothing}$ is one dimensional.
This fact is far from being true for $l \leq h$. For small $l$, a more
intricate analysis is needed, namely, we must consider the multiplicity of a
certain ``Steinberg module'' in $\opH^{\bullet}(\Uz(\uj),{\mathbb
C})$. This computation will then be used in Chapters 5 and 6 in order to
complete the desired cohomology computations.

\section{Steinberg weights}\label{Steinbergweightssec}\renewcommand{\thetheorem}{\thesection.\arabic{theorem}}
\renewcommand{\theequation}{\thesection.\arabic{equation}}\setcounter{equation}{0}
\setcounter{theorem}{0}
 If $l$ satisfies Assumption \ref{assumption}, by Theorem
 \ref{identificationtheorem},
we can choose $w \in W$ and $J\subseteq \Pi$ such that $w(\Phi_{0}) = \Phi_J$.  Clearly,  it can
be additionally assumed that $w(\Phi_0^+)= \Phi^+_J$. 
In the classical cases, the particular choice of $w$ and $J\subseteq \Pi$ will not generally matter for the 
arguments that follow. However, when $\Phi$ has type $A_n$ with $l$ dividing $n + 1$, a 
special $w$ and $J$
are identified in (\ref{wepi}).   Also, in the exceptional cases, each pair  $w \in W$ and $J\subseteq \Pi$
identified in Appendix A.1 satisfies the property $w(\Phi_0^+)=\Phi^+_J$

\begin{lem}\label{Steinbergwtslemma} Let $w \in W$ be such that 
$w(\Phi_{0}^+) = \Phi_J^{+}$ for some $J \subset \Pi$. For all $\al \in J$,
$\langle w\cdot 0, \alpha^{\vee} \rangle = l-1$.
\end{lem}

\begin{proof} Since $w\cdot 0 = w(\rho) - \rho$, the claim is equivalent to showing that
$l = \langle w(\rho),\al^{\vee}\rangle = \langle
\rho,w^{-1}(\al)^{\vee}\rangle$ for all $\al \in J$. But, by our assumption on $w$, 
$w^{-1}(J)$ is the unique set $\Pi_0$ of simple roots for $\Phi_0$ contained in
$\Phi_0^+ = \Phi^+ \cap \Phi_0$. So, the lemma asserts that $\langle
\rho,\be^{\vee}\rangle = l$ for all $\beta\in\Pi_0$. Therefore,
although $w \in W$ is not uniquely determined, if the lemma holds
for one choice of $w$, it holds for all choices. Now, in the
exceptional cases, for each $l$, an element $w \in W$ and a subset
$J \subseteq \Pi$ satisfying $w(\Phi_0^+) = \Phi_J^+$ are identified in
Appendix \ref{tables1}. In these cases, the lemma can be checked
directly by using the tables in Appendix \ref{tables2} which explicitly give $w\cdot 0$.

Now assume that $\Phi$ has classical type $A_n$, $B_n$, $C_n$, or
$D_n$. We can also assume that $\Pi_0\not=\emptyset$. Then $\Phi_0$
consists of all roots $\alpha$ such that the coroot $\alpha^\vee$
has height $\text{ht}(\alpha^\vee)=\langle\rho,\alpha^\vee\rangle$
divisible by $l$. In particular, $l < h$. If $\beta^\vee$ is any
coroot of height $m>0$, then, for any positive integer $i$, $3\leq
i<m$, it is easy to see (in each possible case), that
$\beta^\vee=\delta^\vee +\gamma^\vee$ for $\delta,\gamma\in\Phi$
satisfying $\text{ht}(\delta^\vee)=i$. Then $\beta= a\delta +
b\gamma$, where $a,b$ are positive rational numbers. If
$\beta\in\Phi_0^+$ with $\text{ht}(\beta^\vee)=tl$, $t>1$, then
\ $\beta^\vee=\delta^\vee+\gamma^\vee$ with
$\delta,\gamma\in\Phi_0^+$. If $\beta\in\Pi_0$, it follows that
$\text{ht}(\beta^\vee)= l$, i.~e.,
$\langle\rho,\beta^\vee\rangle=l$, as required.\end{proof}

For chosen $w \in W$  and $J\subseteq \Pi$ such that $w(\Phi_0^+) = \Phi_J^+$, set 
\begin{equation}\label{steinbergmodule}
M:=(\text{ind}_{U_{\zeta}({\mathfrak b})}^{U_{\zeta}({\mathfrak
p}_{J})} w\cdot 0)^{*}.
\end{equation} 
By the lemma,  $w \cdot 0$ is a
$J$-Steinberg weight (see Section 2.10). The module $M$ is therefore
isomorphic to a ``Steinberg'' type module on $U_{\zeta}({\mathfrak
l}_{J})$ that remains irreducible if viewed as a
$u_{\zeta}({\mathfrak l}_{J})$-module. The highest weight of $M$ is
$-w_{0,J}(w\cdot 0)$ and the lowest weight of $M$ is $-w\cdot 0$.
Note that the module $M$ does depend on the choice of $w$.


\section{Weights of
$\Lambda^{\bullet}_{\zeta,J}$}\label{weightsinlambda}\renewcommand{\thetheorem}{\thesection.\arabic{theorem}}
\renewcommand{\theequation}{\thesection.\arabic{equation}}\setcounter{equation}{0}
\setcounter{theorem}{0}
 By Section
\ref{adjointaction}, the $\adw$-action induces an action of
$U_{\zeta}(\pj)$ (and hence also of $u_{\zeta}(\pj)$) on $\Uz(\uj)$.
This defines an action of $U_{\zeta}(\pj)$ on the cohomology
$\opH^{\bullet}(\Uz(\uj),{\mathbb C})$. See also Section
\ref{finitedimofcohogroups}. In Theorem \ref{multSteinberg} below,
we determine
$$\text{Hom}_{u_{\zeta}({\mathfrak l}_{J})}
((\text{ind}_{U_{\zeta}({\mathfrak b})}^{U_{\zeta}({\mathfrak
p}_{J})} w\cdot 0)^{*}, \opH^{\bullet}(\Uz(\uj),{\mathbb C})).$$ By
Proposition~\ref{euler}(b), it follows that, as a $U_{\zeta}^0$-module,
$\opH^{\bullet}(\Uz(\uj),{\mathbb C})$ is a subquotient of
$\Lambda^{\bullet}_{\zeta,J}$. The key ingredients to proving Theorem
\ref{multSteinberg} below are the following computational results concerning the weights in
$\Lambda^{\bullet}_{\zeta,J}$.

\begin{prop}\label{keypropositiononweights} Let $l$ be
 as in Assumption \ref{assumption}.
Choose $w\in W$ and $J\subseteq\Pi$ such that $w(\Phi_{0}^+)=\Phi_{J}^+$. Let $\ga$ be a
$J$-dominant weight of $\Lambda^i_{\zeta,J}$ such that $\ga =
-w_{0,J}(w\cdot 0) + l\nu$ for some $\nu \in X$.
\begin{itemize}
\item[(a)] Suppose that $l\nmid n+1$ when $\Phi$ is of
type $A_{n}$ and $l \neq 9$ when $\Phi$ is of type $E_6$. Then $\ga
= -w_{0,J}(w\cdot 0)$ (i.e., $\nu = 0$) and $i = \ell(w)$.

\item[(b)] If $\Phi$ is of type $A_{n}$ with $n+1= l(m+1)$
and $w$ is as defined in (\ref{wepi}), then $\ga$ is one of the
following, for $0 \leq t \leq l - 1$:
$$\ga = -w_{0,J}(w\cdot 0) +
 l\varpi_{t(m+1)} \hskip.3in \text{ with } \hskip.3in i =
  \ell(w) + (m+1)t(l-t).$$
We set $\varpi_{0}=0.$

\item[(c)] If $\Phi$ is of type $E_6$ and $l = 9$
(assuming that $w$ and $J$ are as in Appendix \ref{tables1}), then
$\ga$ is one of the following:
$$\ga = -w_{0,J}(w\cdot 0) \text{ with } i = \ell(w) = 8,$$
$$\ga = -w_{0,J}(w\cdot 0) + \ell\varpi_1 \text{ with } i = 20,$$
$$\ga = -w_{0,J}(w\cdot 0) + \ell\varpi_6 \text{ with } i = 20.$$
\end{itemize}
\end{prop}

\noindent One should observe that the weight $\nu$ in the statement
of the proposition must necessarily be $J$-dominant by Lemma
\ref{Steinbergwtslemma}. The proposition will be proved below. See
Section \ref{proofofkeyprop}.

\section{Multiplicity of the Steinberg
module}\label{multiplicityofSteinbergmodule} \renewcommand{\thetheorem}{\thesection.\arabic{theorem}}
\renewcommand{\theequation}{\thesection.\arabic{equation}}\setcounter{equation}{0}
\setcounter{theorem}{0}

Assuming that
Proposition \ref{keypropositiononweights} holds, we can now
determine how often the ``Steinberg module'' $M$ (introduced in
\ref{steinbergmodule}) appears in
$\opH^{\bullet}(u_{\zeta}(\uj),{\mathbb C})$. When $l \geq h$, $w =
1$, $J = \varnothing$, and $M = {\mathbb C}$.
In parts (b) and (c) of the theorem, the notation $l\varpi_j$ is used
to denote the one-dimensional $U_{\zeta}({\mathfrak l}_J)$-module with 
weight $l\varpi_j$.

\begin{theorem}\label{multSteinberg} Let $l$ be as in Assumption
 \ref{assumption}.
Choose $w\in W$  and $J\subseteq\Pi$ such that $w(\Phi_{0}^+)=\Phi_{J}^+$.
\begin{itemize}
\item[(a)] Suppose that $l\nmid n+1$ when $\Phi$ is
of type $A_{n}$ and $l \neq 9$ when $\Phi$ is of type $E_6$.
Then as $U_{\zeta}(\lj)$-modules
$$
\Hom_{u_{\zeta}({\mathfrak l}_{J})}((\ind_{U_{\zeta}({\mathfrak b})}
^{U_{\zeta}({\mathfrak p}_{J})}w\cdot
0)^*,\opH^{i}(\Uz(\uj),{\mathbb C})) =
\begin{cases}
{\mathbb C} &\text{ if } i = \ell(w)\\
0 &\text{ otherwise}.
\end{cases}
$$
\item[(b)] If $\Phi$ is of type $A_{n}$ with $n+1= l(m+1)$ and $w$ is
as defined in (\ref{wepi}),
then as $U_{\zeta}(\lj)$-modules
\begin{align*}
\Hom_{u_{\zeta}({\mathfrak l}_{J})}((\ind_{U_{\zeta}({\mathfrak b})}
^{U_{\zeta}({\mathfrak p}_{J})}w\cdot 0)^*,&\opH^i(\Uz(\uj),{\mathbb
C}))\\ &=
\begin{cases}
{\mathbb C} &\text{ if } i = \ell(w)\\
l\varpi_{t(m+1)}\oplus l\varpi_{(l-t)(m+1)} &\text{ if } i = \ell(w) +(m+1)t(l-t) \\
    &\quad \text{ for } 1 \leq t \leq (l-1)/2\\
0 & \text{ otherwise}.
\end{cases}
\end{align*}
\item[(c)] If $\Phi$ is
 of type $E_6$ and $l = 9$ (assuming that $w$ and $J$
are as in Appendix \ref{tables1}), then as $U_{\zeta}(\lj)$-modules
$$
\Hom_{u_{\zeta}({\mathfrak l}_{J})}((\ind_{U_{\zeta}({\mathfrak b})}
^{U_{\zeta}({\mathfrak p}_{J})}w\cdot
0)^*,\opH^{i}(\Uz(\uj),{\mathbb C})) =
\begin{cases}
l\varpi_1 \oplus l\varpi_6 & \text{ if } i = 20\\
{\mathbb C} & \text{ if } i = \ell(w) = 8\\
0 & \text{ otherwise}.
\end{cases}
$$
\end{itemize}
\end{theorem}

\begin{proof} Since
$\text{Hom}_{u_{\zeta}({\mathfrak l}_{J})}
((\text{ind}_{U_{\zeta}({\mathfrak b})}^{U_{\zeta}({\mathfrak
p}_{J})} w\cdot 0)^{*}, \opH^{\bullet}(\Uz(\uj),{\mathbb C}))$ is a
module for $U_{\zeta}(\lj)$ on which $u_{\zeta}(\lj)$ acts
trivially, it is also a (finite dimensional---hence completely
reducible) module for the universal enveloping algebra $\BU(\lj)$
(see Section \ref{connectionswithalgebraicgroups}). Thus, if
$\text{Hom}_{u_{\zeta}({\mathfrak l}_{J})}
((\text{ind}_{U_{\zeta}({\mathfrak b})}^{U_{\zeta}({\mathfrak
p}_{J})} w\cdot 0)^{*}, \opH^{\bullet}(\Uz(\uj),{\mathbb C})) \neq
0$, any $U_{\zeta}(\lj)$-composition factor will be of the form
$L_J(\nu)^{[1]}$ for a $J$-dominant weight $\nu$. In other words,
there must be a nonzero $U_\zeta({\mathfrak l}_J)$-homomorphism
\begin{equation}\label{nonzero}
(\text{ind}_{U_{\zeta}({\mathfrak b})}^{U_{\zeta}({\mathfrak
p}_{J})} w\cdot 0)^{*}\otimes L_J(\nu)^{[1]} \to
\opH^{\bullet}(\Uz(\uj),{\mathbb C}).
\end{equation}
Hence, the weight $-w_{0,J}(w\cdot 0) + l\nu$ must appear in
$\opH^{\bullet}(\Uz(\uj),{\mathbb C})$, and so also in
$\Lambda^{\bullet}_{\zeta,j}$ by Proposition~\ref{euler}(b). The theorem
now follows from Proposition \ref{keypropositiononweights} if each
weight listed therein does indeed give rise to a non-trivial
homomorphism as in (\ref{nonzero}).

By Proposition~\ref{euler}(a),
\begin{equation}\label{one}
\sum_{n = 0}^{\infty}(-1)^n\operatorname{ch}\opH^n(\Uz(\uj),{\mathbb
C}) = \sum_{n =
0}^{\dim(\uj)}(-1)^n\operatorname{ch}\Lambda^n_{\zeta,J}.
\end{equation}
Since $\opH^{\bullet}(\BU(\uj),{\mathbb C})$ can be computed from
$\Lambda^{\bullet}(\uj^*)$ considered as a complex (with
appropriately defined differential), we similarly have
\begin{equation}\label{two}
\sum_{n =
0}^{\dim(\uj)}(-1)^n\operatorname{ch}\opH^n(\BU(\uj),{\mathbb C}) =
\sum_{n = 0}^{\dim(\uj)}(-1)^n\operatorname{ch}\Lambda^n(\uj^*).
\end{equation}
Clearly, $\operatorname{ch}_{\zeta}\Lambda^n_{\zeta,J} =
\operatorname{ch}\Lambda^n(\uj^*).$ Hence (\ref{one}) and
(\ref{two}) give
\begin{equation}\label{three}
\sum_{n = 0}^{\infty}(-1)^n\operatorname{ch}\opH^n(\Uz(\uj),{\mathbb
C}) = \sum_{n =
0}^{\dim(\uj)}(-1)^n\operatorname{ch}\opH^n(\BU(\uj),{\mathbb C}).
\end{equation}
Let $^JW = \{x \in W ~|~ x(\Phi^-)\cap \Phi^+ \subset \Phi^+
\backslash \Phi_J^+\}$ be the set of distinguished right $W_J$-coset
representatives in $W$. By \cite[Thm. 2.5.1.3]{W},
\begin{equation}\label{four}
\sum_{n =
0}^{\dim(\uj)}(-1)^n\operatorname{ch}\opH^n(\BU(\uj),{\mathbb C}) =
\sum_{x \in {^JW}}(-1)^{\ell(x)}\ch\,L_J(-w_{0,J}(x\cdot 0)).
\end{equation}
 Also, (cf. \cite[Prop. 3.4.5]{HK})
$$
\operatorname{ch}_{\zeta}\ind_{U_{\zeta}
(\mathfrak{b_l}_J)}^{U_{\zeta}(\lj)}(-w_{0,J}(x\cdot 0)) =
\operatorname{ch} L_J(-w_{0,J}(x\cdot 0)).
$$
Combining this with (\ref{three}) and (\ref{four}) gives
\begin{equation}\label{character}
\sum_{n = 0}^{\infty}(-1)^n\operatorname{ch}\opH^n(\Uz(\uj),{\mathbb
C}) = \sum_{x \in {^JW}}(-1)^{\ell(x)}\operatorname{ch}
\ind_{U_{\zeta}(\mathfrak{b_l}_J)}^{U_{\zeta}(\lj)}(-w_{0,J}(x\cdot
0)).
\end{equation}

The weights $\gamma$ in Proposition \ref{keypropositiononweights}
are all ``Steinberg weights'' whose induced module is injective over
$U_{\zeta}(\lj)$. Hence, they do not appear as a composition factor
in any other induced module. For part (a), since only one such
weight occurs, it gives rise to a composition factor on the
right-hand side of (\ref{character}) which cannot be canceled out by
any other factor. Hence it appears as well on the left-hand side
which completes the proof. For parts (b) and (c), while multiple
``Steinberg weights'' appear, these weights are all distinct and
hence give rise to distinct composition factors on the right-hand
side of (\ref{character}) which cannot cancel each other out. Hence, 
they all appear on the left-hand side as well.
\end{proof}

\begin{rem}\label{Steinbergrem} In those cases where the cohomology is two-dimensional,
the differences of the weights are neither sums of positive roots nor sums of negative roots. Hence
the isomorphisms in the theorem also hold as
$U_{\zeta}(\pj)$-modules.
\end{rem}


\section{Proof of Proposition \ref{keypropositiononweights}}
\label{proofofkeyprop}\renewcommand{\thetheorem}{\thesection.\arabic{theorem}}
\renewcommand{\theequation}{\thesection.\arabic{equation}}\setcounter{equation}{0}
\setcounter{theorem}{0}

The remainder of Section \ref{combinSteinbergsec} is devoted to
proving Proposition \ref{keypropositiononweights}. Note first of all
that the weight $-w_{0,J}(w\cdot 0)$ does appear in
$\Lambda^{\bullet}_{\zeta,J}$ in degree $\ell(w)$ (cf.
\cite[7.3]{GW}, \cite[Prop. 2.2]{FP1}). So the goal is to show that
(in most cases) a weight $\nu$ satisfying the hypothesis must in
fact be zero. In Sections \ref{classicalone}--4.8,
the classical root systems will be considered. For these, we will
mainly work with the $\epsilon$-basis that represents $\Phi$
\cite[p. 250]{Bo} and $\langle -,- \rangle$ will always denote the
ordinary Euclidean inner product. In Section \ref{classicalone}, we
first show that $\langle\nu,\al^{\vee}\rangle = 0$ for all $\al \in
J$. To do this, for each of the classical root systems $X \in \{A_n,
B_n, C_n, D_n\}$ we show the existence of an element $\delta_X \in
X$ with the following properties:
\begin{equation}\label{a}
\max_{\la}\langle\la,\delta_X \rangle := \max\{\langle \lambda,
\delta_X \rangle\; |\; \lambda \text{ a weight of
}\Lambda^{\bullet}_{\zeta,J}\} = \langle -w_{0,J}(w\cdot 0),\delta_X
\rangle \text{ and }
\end{equation}
\begin{equation}\label{b}
\langle \varpi_{j}, \delta_X \rangle>0 \text{ for all fundamental
weights } \varpi_{j} \text{ corresponding to } \al_j \in J.
\end{equation}
When $\Phi$ is of type $A_n$ or $C_{n}$, one can choose
$\delta_X=\sum_{\alpha\in J} \alpha^{\vee}$ and the maximum value in
(\ref{a}) turns out to be simply $(l-1)|J|$.

Now assume that a $\delta_X$ satisfying properties (\ref{a}) and
(\ref{b}) exists and that $-w_{0,J}(w\cdot 0)+l\nu$ is a
$J$-dominant weight of $\Lambda^{i}_{\zeta,J}$ for some $J$-dominant
weight $\nu$. Then
\begin{equation*}
\langle -w_{0,J}(w\cdot 0),\delta_X \rangle +l \langle \nu, \delta_X
\rangle = \langle -w_{0,J}(w\cdot 0)+l\nu, \delta_X \rangle \leq
\max_{\la}\langle\la, \delta_X \rangle =
 \langle -w_{0,J}(w\cdot 0), \delta_X\rangle.
\end{equation*}
This forces $\nu$ to vanish on $J$. To show that $\nu = 0$, it
remains to show that $\langle\nu,\al^{\vee}\rangle = 0$ for $\al \in
\Pi\backslash J$. In Section \ref{classicalBCD}, this is shown for
types $B_n$, $C_n$, $D_n$. Sections \ref{classicalA} and
\ref{classicalAII} are devoted to dealing with type $A_n$ and
showing how the ``extra'' weights arise in Proposition
\ref{keypropositiononweights}.

In Section \ref{exceptionalI}, similar ideas along with direct
computations will be used to deal with the exceptional root systems.


\section{The weight $\delta_X$ }\label{classicalone}\renewcommand{\thetheorem}{\thesection.\arabic{theorem}}
\renewcommand{\theequation}{\thesection.\arabic{equation}}\setcounter{equation}{0}
\setcounter{theorem}{0}

In this section, we construct a weight $\delta_X$ which satisfies
properties (\ref{a}) and (\ref{b}). We will first consider the case
$m \geq 2$ in Theorems \ref{subsystemtypesAB} and
\ref{subsystemtypesCD}.

For $\Phi$ of type $A_n$ and
$$\Phi_{0}\cong \Phi_{J}\cong
\underbrace{A_{m}\times\dots \times A_{m}}_{\text{$r_{1}$
times}}\times \underbrace{A_{m-1} \times \dots \times
A_{m-1}}_{\text{$r_{2}$ times}};$$ 
we choose $J$ in the ``natural" way, namely such that 
$$\Pi\backslash J= \{\alpha_{t(m+1)}\;|\; 1 \leq t \leq r_1\} \cup \{\alpha_{r_1(m+1)+sm}\; | \; 1 \leq s \leq r_2-1\}$$
and $w$ such that $w(\Phi_0^+) = \Phi_J^+.$ The Dynkin diagram below illustrates our choice for  $J$.
\begin{center}
\begin{picture}(400,70)(-5,-20)
\multiput(0,35)(15,0){1}{\circle*{5}} 
\multiput(15,35)(5,0){3}{\circle*{1}} 
\multiput(40,35)(25,0){1}{\circle*{5}} 
\put(20,35){\oval(50,10)}
\put(0,35){\line(1,0){10}}
\put(30,35){\line(1,0){25}}
\put(-5,25){$\underbrace{\mbox{\hspace{.7in}}}$}
\put(10,5){$A_m$}
\put(55,35){\line(1,0){25}}
\multiput(55,35)(15,0){2}{\circle*{5}} 
\multiput(85,35)(5,0){3}{\circle*{1}} 
\multiput(110,35)(15,0){1}{\circle*{5}} 
\put(90,35){\oval(50,10)}
\put(100,35){\line(1,0){20}}
\put(65,25){$\underbrace{\mbox{\hspace{.7in}}}$}
\put(80,5){$A_m$}
\multiput(125,35)(5,0){3}{\circle*{1}} 
\put(140,35){\line(1,0){20}}
\multiput(150,35)(15,0){1}{\circle*{5}} 
\multiput(165,35)(5,0){3}{\circle*{1}} 
\multiput(190,35)(15,0){1}{\circle*{5}} 
\put(170,35){\oval(50,10)}
\put(180,35){\line(1,0){40}}
\put(145,25){$\underbrace{\mbox{\hspace{.7in}}}$}
\put(160,5){$A_m$}
\put(-5,0){$\underbrace{\mbox{\hspace{2.8in}}}$}
\put(80,-20){$r_1 \mbox{ times }$}
\multiput(205,35)(5,0){1}{\circle*{5}} 
\multiput(220,35)(15,0){1}{\circle*{5}} 
\multiput(235,35)(5,0){2}{\circle*{1}} 
\multiput(255,35)(25,0){1}{\circle*{5}} 
\put(237,35){\oval(46,10)}
\put(215,35){\line(1,0){15}}
\put(245,35){\line(1,0){25}}
\put(215,25){$\underbrace{\mbox{\hspace{.6in}}}$}
\put(230,5){$A_{m-1}$}
\multiput(270,35)(5,0){1}{\circle*{5}} 
\multiput(285,35)(15,0){1}{\circle*{5}} 
\multiput(300,35)(5,0){2}{\circle*{1}} 
\multiput(320,35)(25,0){1}{\circle*{5}} 
\put(302,35){\oval(46,10)}
\put(270,35){\line(1,0){25}}
\put(310,35){\line(1,0){20}}
\put(280,25){$\underbrace{\mbox{\hspace{.6in}}}$}
\put(295,5){$A_{m-1}$}
\multiput(335,35)(5,0){3}{\circle*{1}} 
\multiput(360,35)(15,0){1}{\circle*{5}} 
\multiput(375,35)(5,0){2}{\circle*{1}} 
\multiput(395,35)(25,0){1}{\circle*{5}} 
\put(377,35){\oval(46,10)}
\put(350,35){\line(1,0){20}}
\put(385,35){\line(1,0){10}}
\put(355,25){$\underbrace{\mbox{\hspace{.6in}}}$}
\put(370,5){$A_{m-1}$}
\put(215,0){$\underbrace{\mbox{\hspace{2.6in}}}$}
\put(285,-20){$r_2 \mbox{ times }$}
\end{picture}
\end{center}
Let
$$\delta_{A}=(\underbrace{1,0,\dots,0-1,\dots,1,0,\dots,0,-1}_{\text{$r_{1}$ times}},
\underbrace{1,0,\dots,0,-1,\dots,1,0,\dots,0,-1}_{\text{$r_{2}$
times}})$$ in the orthonormal basis describing $\Phi$ in
$n+1$-dimensional Euclidean space ${\mathbb E}$ \cite[p. 250]{Bo}.
Note that the first $r_{1}$-groupings of $(1,0,\dots,0,-1)$ have
$m+1$-components, while the last $r_{2}$-groupings of
$(1,0,\dots,0,-1)$ have $m$-components. In this case,
$\delta_A=\sum_{\alpha\in J} \alpha^{\vee}$, and evidently
$\langle\varpi_{j},\delta_{A}\rangle > 0$ for all $\varpi_{j}$
corresponding to simple roots in $J$ (property (\ref{b})).

To show property (\ref{a}), let $\lambda$ be a weight of
$\Lambda^{i}_{\zeta,J}$. Then $\lambda$ is a sum of distinct
positive roots not in $\Phi_{J}^{+}$. If $\beta$ is a positive root
then $\langle \beta,\delta_A \rangle =0,\pm 1, \pm 2$ (i.~e.
$\beta=\epsilon_{i}-\epsilon_{j}$ with $i<j$). Set
$$A[t]=\{\beta\in \Phi^{+}\backslash\Phi_{J}^{+} \;|\; \langle
\beta,\delta_A \rangle =t\}.$$ A quick count shows that
\begin{eqnarray*}
|A[2]|&=& \frac{(r_{1}+r_{2}-1)(r_{1}+r_{2})}{2}\\
|A[1]|&=& (m-1)(r_{1}+r_{2}-1)r_{1}+(m-2)(r_{1}+r_{2}-1)r_{2}.
\end{eqnarray*}
Therefore,
\begin{eqnarray*}
\text{max}_{\lambda}\langle \lambda,\delta_A \rangle &=&2|A[2]|+|A[1]| \\
&=& (r_{1}+r_{2}-1)(mr_{1}+(m-1)r_{2})\\
&=& (r_{1}+r_{2}-1)|J|.
\end{eqnarray*}
According to Theorem \ref{subsystemtypesAB}, we can set $r_{1}=s+1$
and $r_{2}=l-s-1$. Consequently,
$$\text{max}_{\lambda}\langle \lambda,\delta_A \rangle = (l-1)|J|=
\langle -w_{0,J}(w\cdot 0),\delta_A \rangle $$ because $A_{n}$ is
simply laced.

For the other classical Lie algebras, let $\Phi$ be of type $X_{n}$
where $X=B,C,D$ with
$$\Phi_{0}\cong \Phi_{J}\cong
\underbrace{A_{m}\times\dots \times A_{m}}_{\text{$r_{1}$
times}}\times \underbrace{A_{m-1} \times \dots
A_{m-1}}_{\text{$r_{2}$ times}}\times X_{q}. $$ 
Again the ``natural" choice for $J$, is such that 
$$\Pi\backslash J= \{\alpha_{t(m+1)}\;|\; 1 \leq t \leq r_1\} \cup \{\alpha_{r_1(m+1)+sm}\; | \; 1 \leq s \leq r_2\}$$
with
$$\delta_{X}=
(\underbrace{1,0,\dots,0-1,\dots,1,0,\dots,0,-1}_{\text{$r_{1}$
times}},
\underbrace{1,0,\dots,0,-1,\dots,1,0,\dots,0,-1}_{\text{$r_{2}$
times}},1,0,\dots,0).
$$
Notice that in all cases $-w_{0,J}\delta_X = \delta_X$. By using
\cite[p. 102]{GW}, one can verify that $\langle \varpi_{j},
\delta_{X} \rangle> 0$ for all $\varpi_j$ that correspond to simple
roots in $J$. If $\beta$ is a positive root in $X_{n}$ then $\langle
\beta,\delta_{X} \rangle =0,\pm 1, \pm 2$, so set $X[t]=\{\beta\in
\Phi^{+}\backslash\Phi_{J}^{+} \;|\;\langle \beta,\delta_{X} \rangle
=t\}$. By using our computations for type $A_{n}$ with consideration
of other positive roots in $X_{n}$, it follows that
$$|X[2]|=\begin{cases}
(r_{1}+r_{2})^{2}
& \text{$X_{n}=B_{n}$};\\
(r_{1}+r_{2})(r_{1}+r_{2}+1)
& \text{$X_{n}=C_{n}$};\\
(r_{1}+r_{2})^{2} & \text{$X_{n}=D_{n}$}
\end{cases}
$$
and
$$|X[1]|=\begin{cases}
2(r_{1}+r_{2})(r_{1}(m-1)+r_{2}(m-2))+(2q-1)(r_{1}+r_{2})

& \text{$X_{n}=B_{n}$};\\
2(r_{1}+r_{2})(r_{1}(m-1)+r_{2}(m-2))+2(q-1)(r_{1}+r_{2})

& \text{$X_{n}=C_{n}$};\\
2(r_{1}+r_{2})(r_{1}(m-1)+r_{2}(m-2))+2(q-1)(r_{1}+r_{2})

& \text{$X_{n}=D_{n}$}.
\end{cases}
$$
Hence,
\begin{eqnarray*}
\text{max}_{\lambda}\langle \lambda, \delta_{X}\rangle&=&
2|X[2]|+|X[1]|\\
&=& \begin{cases}
2(r_{1}+r_{2})(r_{1}m+r_{2}(m-1))+(2q-1)(r_{1}+r_{2})
& \text{$X_{n}=B_{n}$};\\
2(r_{1}+r_{2})(r_{1}m+r_{2}(m-1))+2q(r_{1}+r_{2})

& \text{$X_{n}=C_{n}$};\\
2(r_{1}+r_{2})(r_{1}m+r_{2}(m-1))+2(q-1)(r_{1}+r_{2})

& \text{$X_{n}=D_{n}$}.
\end{cases}
\end{eqnarray*}
On the other hand, by using the expression of $\rho$ for the
classical Lie algebras in terms of the $\epsilon$-basis (see
\cite[p. 107]{GW}), one sees that
$$\langle -w_{0,J}(w\cdot 0),\delta_X \rangle =\begin{cases}
(l-1)(r_{1}m+r_{2}(m-1))+(l-1)(q-\frac{1}{2})
& \text{$X_{n}=B_{n}$};\\
(l-1)(r_{1}m+r_{2}(m-1))+(l-1)(q)
& \text{$X_{n}=C_{n}$};\\
(l-1)(r_{1}m+r_{2}(m-1))+(l-1)(q-1) & \text{$X_{n}=D_{n}$}.
\end{cases}
$$

In Theorems \ref{subsystemtypesAB} and \ref{subsystemtypesCD},
$r_{1}+r_{2}=\frac{l-1}{2}$ for all root systems $\Phi$ of types
$B_{n}$, $C_{n}$, or $D_{n}$, and so
$\operatorname{max}_{\la}\langle\la,\delta_{X}\rangle = \langle
-w_{0,J}(w\cdot 0),\delta_{X}\rangle$ as desired.

Next we consider the case $m=1$. Again let $\Phi$ be of type $X_{n}$
where $X=A,B,C,D$ with
$$\Phi_{0}\cong \Phi_{J}\cong \begin{cases}
\underbrace{A_{1}\times\dots \times A_{1}}_{\text{$r$ times}} & \text{ for $X_{n}=A_n$, $C_n$, or $D_{n}$.}\\
\underbrace{A_{1}\times\dots \times A_{1}}_{\text{$r$ times}} \times
A_1& \text{ for $X_{n}=B_n$.}
\end{cases} $$
Here we choose $J= \{\alpha_{2t-1}\;|\; 1 \leq t \leq r\}$ and  
$J= \{\alpha_{2t-1}\;|\; 1 \leq t \leq r\}\cup \{\alpha_{n}\}$, respectively.
Let
$$\delta_{X}=\begin{cases}
(\underbrace{1,-1,\dots,1,-1}_{\text{$r$ times}},
\underbrace{0,\dots,0}_{\text{$z$ times}})
& \text{ for $X_{n}=A_{n}$, $C_n$, or $D_{n}$};\\
(\underbrace{1,-1,\dots,1,-1}_{\text{$r$ times}},
\underbrace{0,\dots,0}_{\text{$z$ times}},1) & \text{ for
$X_{n}=B_{n}$}.
\end{cases}
$$
As above, one can verify that $\langle \varpi_{j},\delta_{X}
\rangle> 0$ for all $\varpi_j$ that correspond to simple roots in
$J$. Also, we conclude that
$$|X[2]|=\begin{cases}
\frac{r(r-1)}{2}
& \text{$X_{n}=A_{n}$};\\
r^{2}
& \text{$X_{n}=B_{n}$, $C_n$};\\

r(r-1) & \text{$X_{n}=D_{n}.$}
\end{cases}
$$
and
$$|X[1]|=\begin{cases}
rz
& \text{$X_{n}=A_{n}$};\\
2rz
& \text{$X_{n}=C_n$, $D_n$};\\
2rz+r+z & \text{$X_{n}=B_{n}$}.

\end{cases}
$$
Hence,
\begin{eqnarray*}
\text{max}_{\lambda}\langle \lambda, \delta_{X}\rangle&=&
2|X[2]|+|X[1]| =
\begin{cases}
r(r+z-1)
& \text{$X_{n}=A_{n}$};\\
(r+\frac{1}{2})(2r+2z)
& \text{$X_{n}=B_{n}$};\\
r(2r+2z)

& \text{$X_{n}=C_{n}$};\\
r(2(r+z-1))

& \text{$X_{n}=D_{n}$}.
\end{cases}
\end{eqnarray*}
On the other hand, by using the expression of $\rho$ for the
classical Lie algebras in terms of the $\epsilon$-basis (see
\cite[p. 107]{GW}) one sees that
$$\langle -w_{0,J}(w\cdot 0),\delta_X \rangle =\begin{cases}
r(l-1)
& \text{$X_{n}=A_{n}$, $C_n$, $D_n$};\\
(r+ \frac{1}{2})(l-1)
& \text{$X_{n}=B_{n}$}.\\
\end{cases}
$$

By Theorems \ref{subsystemtypesAB} and \ref{subsystemtypesCD}, $r$
is the number of copies of $A_1$ and $z = n + 1 - 2r$ (respectively, $n - 2r
- 1$, $n - 2r$) in type $A_n$ (respectively, $B_n$, $C_n$ or $D_n$). One
obtains that
\begin{equation*}
r+z =\begin{cases} l
& \text{$X_{n}=A_{n}$};\\
\frac{l-1}{2}
& \text{$X_{n}=B_{n}$, $C_n$};\\
\frac{l+1}{2} &\text{$X_{n}= D_n$}.
\end{cases}
\end{equation*}
Again,
$\operatorname{max}_{\la}\langle\la,\delta_{X}\rangle = \langle
-w_{0,J}(w\cdot 0),\delta_{X}\rangle$.


\section{Types $B_n, C_n, D_n$}\label{classicalBCD}\renewcommand{\thetheorem}{\thesection.\arabic{theorem}}
\renewcommand{\theequation}{\thesection.\arabic{equation}}\setcounter{equation}{0}
\setcounter{theorem}{0}

 In this section $\Phi$ is always of type
$B_{n}$, $C_{n}$ or $D_{n}$. Under this assumption we show that the
only weight $\nu$ satisfying the hypothesis of Proposition
\ref{keypropositiononweights} is the zero weight. Note that our
restriction on the root systems and $l$ being odd implies that
$\text{gcd}(l,(X : Q))=1$. For any such weight $\nu$ we observe that
$l\nu\in Q$ because $-w_{0,J}(w\cdot 0)+l\nu$ is a weight of
$\Lambda^{i}_{\zeta,J}$. It follows that $\nu \in Q$.

Our results in Section \ref{classicalone} show that both
$-w_{0,J}(w\cdot 0)$ and $-w_{0,J}(w\cdot 0)+l\nu$ consist of a sum
of all the roots in $X[1]\cup X[2]$ with some additional terms
involving roots in $X[0]$. Therefore, $l\nu$ is a sum of distinct
roots in $X[0] \cup - X[0]$.

We express $\delta_X$ and $\nu$ in the $\epsilon$-basis as $\delta_X
= \sum_{i=1}^{n}\delta_{X,i} \epsilon_i$ and $\nu =
\sum_{i=1}^{n}\nu_{i} \epsilon_i$, respectively. Notice that $\nu
\in Q$
implies that $\nu \in {\mathbb Z}\epsilon_1 \oplus \cdots
\oplus { \mathbb Z}\epsilon_n.$
Define the following sets
\begin{equation}\label{Sab}
S_{(a,b)}=\{(i,j) : \delta_{X,i}=a,\; \delta_{X,j}=b, \text{ and } i
<j\} \text{ and } S_{(a)}=\{i:\; \delta_{X,i}=a\}.
\end{equation}
 The case when $\Phi$ is of type $B_{n}$ will be discussed in detail.
 The verification that $\nu=0$ for the cases when $\Phi$ is of type $C_n$ and $D_n$ are left to the
reader.

Any positive root of the form $\epsilon_i - \epsilon_j$ in $B[0]$
satisfies $(i,j) \in S_{(0,0)}\cup S_{(1,1)}\cup S_{(-1,-1)}$, any
positive root of the form $\epsilon_i + \epsilon_j$ in $B[0]$
satisfies $(i,j) \in S_{(0,0)}\cup S_{(1,-1)}\cup S_{(-1,1)}$ and
any positive root of the form $\epsilon_i$ in $B[0]$ satisfies $i
\in S_{(0)}$. Using (\ref{Sab}), we can express

\begin{eqnarray*}
l\nu&=& \sum_{(i,j)\in S_{(1,1)}}m_{i,j}(\epsilon_{i}-\epsilon_{j})
+ \sum_{(i,j)\in
S_{(-1,-1)}}\widetilde{m}_{i,j}(\epsilon_{i}-\epsilon_{j})\\
&+&\sum_{(i,j)\in S_{(1,-1)}}n_{i,j}(\epsilon_{i}+\epsilon_{j})
+\sum_{(i,j)\in
S_{(-1,1)}}\widetilde{n}_{i,j}(\epsilon_{i}+\epsilon_{j})\\
&+&\sum_{(i,j)\in S_{(0,0)}}p_{i,j}(\epsilon_{i}-\epsilon_{j})+
\sum_{(i,j)\in
S_{(0,0)}}\widetilde{p}_{i,j}(\epsilon_{i}+\epsilon_{j})+\sum_{i\in
S_{(0)}} {q}_{i} \epsilon_{i}
\end{eqnarray*}
with $m_{i,j},\widetilde{m}_{i,j},n_{i,j},\widetilde{n}_{i,j}, q_i,
p_{i,j}, \widetilde{p}_{i,j}=0,\pm 1$.

The above expression shows that if $\delta_{B,i}= 1$ then $l\nu_{i}
= \sum_{j} (m_{i,j}+n_{i,j}+m_{j,i}+\widetilde{n}_{j,i}).$ For fixed
$i$, it follows from Theorem \ref{subsystemtypesAB} that the number
of pairs of the form $(i, j)$ together with those of form $( j,i)$
in $S_{(1,1)}$ are less than $(l+1)/2$. A similar counting argument
shows that the number of pairs of the form $(i,j)$ in $S_{(1,-1)}$
together with the number of pairs of the form $( j,i)$ in
$S_{(-1,1)}$ are  less than $(l-1)/2$. Therefore, $|\sum_{j}(
m_{i,j}+n_{i,j}+m_{j,i}+\widetilde{n}_{j,i})| \leq (l-1)$ and
$\delta_{B,i}= 1$ implies $\nu_{i} =0$. Similarly, one can argue
that $\delta_{B,i}= -1$ implies $\nu_{i} =0$.

Any simple root $\epsilon_i - \epsilon_{i+1} \in \Pi\backslash J$
satisfies either $\delta_{B,i}= -1$ and $\delta_{B,i+1}= 1$ or
$\delta_{B,i}= 0$ and $\delta_{B,i+1}= 0$. It follows from above
that, for any $\alpha\in \Pi\backslash J$,
\begin{equation*}
\langle l\nu, \alpha^{\vee} \rangle =\langle \sum_{i \in S_{(0)}}
{q}_{i} \epsilon_{i}+ \sum_{(i,j)\in
S_{(0,0)}}p_{i,j}(\epsilon_{i}-\epsilon_{j})+ \sum_{(i,j)\in
S_{(0,0)}}\widetilde{p}_{i,j}(\epsilon_{i}+\epsilon_{j}),
\alpha^{\vee} \rangle.
\end{equation*}
For $m>1$, there are no roots in $\Pi\backslash J$ with
$\delta_{B,i}= 0$ and $\delta_{B,i+1}= 0$ and the inner product on
the right-hand side vanishes. Since $\nu$ vanishes on $J$, one
concludes $\nu=0$, as desired.

Assume $m=1$ and let $i$ be such that $\delta_{B,i}= 0$, then $l
\nu_i = q_{i}+\sum_{j} (p_{i,j}+ \tilde{p}_{i,j} + p_{j,i}+
\tilde{p}_{j,i})$. Theorem \ref{subsystemtypesAB} implies that for
fixed $i$ there are less than $(l-1)/2$ pairs of the form $(i, j)$
or $( j,i)$ in $S_{(0,0)}$. It follows that $l|\nu_i| = |
q_{i}+\sum_{j} (p_{i,j}+ \tilde{p}_{i,j} + p_{j,i}+
\tilde{p}_{j,i})| < l.$ This forces $\nu=0$.


\section{Type $A_n$}\label{classicalA}\renewcommand{\thetheorem}{\thesection.\arabic{theorem}}
\renewcommand{\theequation}{\thesection.\arabic{equation}}\setcounter{equation}{0}
\setcounter{theorem}{0}

 In this (and the next)
section $\Phi$ is always of type $A_{n}$ with simple roots
$\alpha_1, \dots , \alpha_n$. We will show that a weight $\nu$
satisfying the hypothesis of Proposition
\ref{keypropositiononweights} equals the zero weight unless $l$
divides $n+1$. For $\mu\in Q$, $\mu_{i}$ always denotes the
coefficient of $\mu$ in its expansion in terms of the
$\epsilon$-basis (i.~e., $\mu=\sum_{i=1}^{n+1} \mu_{i}
\epsilon_{i}$).

First consider the case when $m > 1$ in Theorem
\ref{subsystemtypesAB}. Let
$$\Phi_{0}\cong \Phi_{J}\cong
\underbrace{A_{m}\times\dots \times A_{m}}_{\text{$s+1$
times}}\times \underbrace{A_{m-1} \times \dots \times
A_{m-1}}_{\text{$l-s-1$ times}}$$ where $n=lm+s$ and $0\leq s \leq
l-1$.

We choose $J$ as in Section 4.5 and fix a particular $w \in W$ with $w(\Phi_0^+)
=\Phi_J^+$ as follows. Partition the set $\{\frac{n}{2}, \frac{n}{2}-1, ... ,
-\frac{n}{2}+1, -\frac{n}{2}\}$, i.~e., the set of coordinates of
$\rho$ in the $\epsilon$-basis, into its congruence classes modulo
$l$ and order each congruence class in decreasing order. Then we
order the congruence classes according to the size of their largest
element from highest to lowest. The resulting array is the
coordinate vector of a $W$-conjugate of $\rho$. We denote this
conjugate by $w\rho$ and $w \in W$ denotes the unique permutation
that sends $\rho$ to $w\rho$. If we identify the Weyl group with the
symmetric group in $n+1$ letters, then $w$ can be described as
follows.
\begin{equation}\label{siti}
\text{For } 1 \leq i \leq n+1 \text { define } s_i, t_i \text{ via }
i-1 = s_il + t_i \text{ with } 0 \leq t_i \leq l-1.
\end{equation}
\begin{equation}\label{wi} \text{Then } w(\epsilon_i) = \epsilon_{w(i)} \text{ where } w(i) =
\begin{cases}
t_i(m+1)+s_i+1 &\text{ if } 0 \leq t_i \leq s\\
t_im+s_i +s+2&\text{ if } s+1 \leq t_i \leq l-1.
\end{cases}
\end{equation}

\begin{equation}\label{w0ji} \text{Moreover, } w_{0,J}w(i) =
\begin{cases}
(t_i+1)(m+1)-s_i &\text{ if } 0 \leq t_i \leq s\\
(t_i+1)m-s_i +s+1&\text{ if } s+1 \leq t_i \leq l-1.
\end{cases}
\end{equation}
For any $u \in W$ define $Q(u):= \{\alpha \in \Phi^+\;|\; u\alpha
\in \Phi^-\}$. Using (\ref{siti}) and (\ref{wi}) we find that
\cite[7.3]{GW}
\begin{equation}\label{Qw} Q(w) = \{ \epsilon_i - \epsilon_j\;|\; s_i < s_j, t_i > t_j\}
\text{ and } w\cdot 0 = \sum_{\{\ga \in Q(w)\}}-w(\ga).
\end{equation}
Define
\begin{equation*}\label{fi}
f(i)=
\begin{cases}
w_{0,J}w(lm+i)=1+(i-1) (m+1) &\text{ if } 1 \leq i \leq s+1\\
w_{0,J}w(l(m-1)+i)=1+ (i-1)m+(s+1) &\text{ if } s+2\leq i \leq l,
\end{cases}
\end{equation*}
and
\begin{equation}\label{gi}
g(i)=
\begin{cases}
w_{0,J}w(i)=i(m+1) &\text{ if } 1 \leq i \leq s+1\\
w_{0,J}w(i)=im+(s+1) &\text{ if } s+2\leq i \leq l.
\end{cases}
\end{equation}
Next set $\delta^i = \epsilon_{f(i)} - \epsilon_{g(i)}$ for $1\leq i
\leq l$. Then $\displaystyle{\delta_A = \sum_{i=1}^{l} \delta^i =
\sum_{i=1}^{l}
    \epsilon_{f(i)} - \epsilon_{g(i)}.}$\\
The ${\alpha}_{g(i)} = \epsilon_{g(i)} - \epsilon_{f(i+1)}$ are
precisely the simple roots contained in $\Pi$ but not in $J$.

Using the notation introduced in (\ref{Sab}), we partition $A[0]$
into the following subsets:
\begin{eqnarray*}
R^+_{(0,0)}&=& \{ \epsilon_i - \epsilon_j \in A[0] \; |\; (i,j) \in S_{(0,0)}\},\\
R^+_{(1,1)}&=&\{ \epsilon_i - \epsilon_j \in A[0] \; |\; (i,j) \in
S_{(1,1)}\} =
    \{ \epsilon_{f(i)} - \epsilon_{f(j)} \; | \; 1\leq i < j \leq l\},\\
R^+_{(-1,-1)}&=&\{ \epsilon_i - \epsilon_j \in A[0] \; |\; (i,j) \in
S_{(-1,-1)}\}= \{ \epsilon_{g(i)} - \epsilon_{g(j)} \; | \; 1\leq i
< j \leq l\}.
\end{eqnarray*}
In addition we set $ R_{(a,a)} = R^+_{(a,a)} \cup -R^+_{(a,a)}, a
\in \{-1,0,1\}$. Notice that both sets $R_{(1,1)}$ and $R_{(-1,-1)}$
form root systems
 of type $A_{l-1}$ with simple roots
\begin{equation}\label{betai}
\beta_i := \epsilon_{f(i)} - \epsilon_{f(i+1)} \text{ and
 } \tau_i := \epsilon_{g(i)} - \epsilon_{g(i+1)},
\end{equation}
respectively.

Next define
\begin{equation*}
S^+ := \{\epsilon_{f(i)} - \epsilon_{f(j)}\; | \; 1 \leq i \leq s+1,
s+2 \leq j \leq l\} \text{ and } S:=S^+ \cup - S^+.
\end{equation*}
Then it follows from (\ref{w0ji}) through (\ref{gi}) that
\begin{equation}\label{R+}
 R^+_{(-1,-1)} \cap w_{0,J}w (Q(w))= \varnothing \text{ and } R^+_{(1,1)} \cap w_{0,J}w (Q(w))= S^+.
\end{equation}
One concludes that the weight $-w_{0,J}(w\cdot 0)$ is the sum of all
roots in $A[1]\cup A[2]$ together with certain roots in
$R^+_{(0,0)}$ and the roots in $ S^+$. The elements of $S^+$ can
also be characterized as those roots in $R^+_{(1,1)}$ that contain
$\beta_{s+1}$. It is important to note that no roots of
$R^+_{(-1,-1)}$ contribute to $-w_{0,J}(w\cdot 0)$.

Next assume that $\la$ is a weight of $\Lambda^{\bullet}_{\zeta, J}$
such that $\langle \lambda +w_{0,J}(w\cdot 0), \alpha \rangle = 0$
for all $\alpha \in J.$ Set $\nu = \lambda +w_{0,J}(w\cdot 0)$.
Using the $\beta$-basis of $R_{(1,1)}$ and the $\gamma$-basis of
$R_{(-1,-1)}$, we express $\nu$ in the form
\begin{equation}\label{nu}
\nu = \sum_{i=1}^{l-1} k_i {\beta}_i + \sum_{i=1}^{l-1} l_i{\tau}_i
+ \sum_{\eta \in R^+_{(0,0)}} m_{\eta} \eta.
\end{equation}
Since $\nu$ is the zero weight when restricted to $J$, one observes
that $\langle \nu, \delta^i \rangle= 0$ for $1\leq i \leq l$. Since
$\langle \nu, \delta^1 \rangle = 0$, $k_1- l_1=0$ and, inductively,
it follows from $\langle \nu, \delta^i \rangle = 0$ that $k_i = l_i$
for $1\leq i \leq l-1$. Moreover, it follows from (\ref{R+}) that
all $k_i$ and $l_i$ are nonnegative. One concludes that $\nu$ is a
sum of distinct roots in $R_{(0,0)}$ together with distinct roots in
$R^+_{(-1,-1)}$ and in $R^+_{(1,1)}\backslash S^+$.

Finally, assume that $\lambda_1$ and $\lambda_2$ are two weights of
$\Lambda^{\bullet}_{\zeta, J}$ such that $\langle \lambda_i
+w_{0,J}(w\cdot 0), \alpha \rangle = 0$ for all $\alpha \in J$. For
example $\la_1$ could be of the form $-w_{0,J}(u\cdot 0)$ for some
$u \neq w$ with $u(\Phi_0^+) = \Phi_J^+$ and $\la_2$ could be equal to
$-w_{0,J}(u\cdot 0)+l \nu$, where $\nu$ is the zero weight when
restricted to $J$. It follows from our above arguments that
$\lambda_2 - \lambda_1$ is a sum of distinct roots in $ R_{(-1,-1)}
\cup R_{(1,1)}\backslash S \cup R_{(0,0)}$. The elements in
$R_{(1,1)}\backslash S$ form a root system of type $A_{s} \times
A_{l-s-2} $, spanned by the roots $\{\beta_1, \dots, \beta_s\}\cup
\{\beta_{s+2}, \dots, \beta_{l-1}\}$, as defined in (\ref{betai}).
We can decompose $\la_2 - \la_1 = \ga_1 +\ga_2+\ga_3$ where the
support of $\ga_1$ lies entirely in the type $A_s$ component of
$R_{(1,1)}\backslash S$, the support of $\ga_2$ lies entirely in the
type $A_{l-s-2}$ component of $R_{(1,1)}\backslash S$, and the
support of $\ga_3$ lies entirely in $R_{(-1,-1)} \cup R_{(0,0)}$.

Next observe that $\alpha_{g(i)} = \beta_i - \delta^i$. It follows
that $\langle \la_2 - \la_1, \alpha_{g(i)} \rangle =\langle \la_2 -
\la_1, \beta_i \rangle = \langle \ga_1+ \ga_2 , \beta_{i}\rangle $.
The inner product of $\la_2 - \la_1$ with the roots in
$\Pi\backslash J$ is then given by the following:
\begin{equation}\label{ladiff}
\langle \la_2 - \la_1, \alpha_{g(i)} \rangle =
\begin{cases}
\langle \ga_1, \beta_i \rangle &\text{ if }1 \leq i \leq s,\\
\langle \ga_1 + \ga_2, \beta_i \rangle &\text{ if } i=s+1,\\
\langle \ga_2, \beta_i \rangle &\text{ if }s+2 \leq i \leq l-1.\\
\end{cases}
\end{equation}
Since $\ga_1$ is a sum of distinct roots of a root system of type
$A_s$ a direct computation shows that $|\langle \ga_1,\beta_i
\rangle| \leq s+1$. Similarly, $|\langle \ga_2,\beta_i \rangle| \leq
l-s-1$. Now $\la_2 - \la_1 \in lX$ implies that either $\la_1 =
\la_2$ or $l = s+1$. Hence, Proposition
\ref{keypropositiononweights} holds for type $A_n$ as long as $l$
does not divide $n+1$.

Suppose now that $m = 1$. Then
$$\Phi_{0}\cong \Phi_{J}\cong
\underbrace{A_{1}\times\dots \times A_{1}}_{\text{$s+1$ times}}.$$
Essentially the same argument as above works here as well. We
highlight some of the differences and leave the details to the
interested reader. The definition of $w$ holds as above with $m =
1$. Define $f(i)$, $g(i)$, $\delta^i$, $\delta_A$, and $\al_{g(i)}$
just as above with $m = 1$. Note that for $i \geq s + 2$, $f(i) =
g(i)$ and $\delta^i = 0$. In the definitions of $R^{+}_{(1,1)}$ and
$R^{+}_{(-1,-1)}$, we have $1 \leq i < j \leq s + 1$, and the root
systems $R_{(1,1)}$ and $R_{(-1,-1)}$ are of type $A_s$, while
$S^{+} = \varnothing$. For $s+2\leq i \leq l-1$, the $\beta_i =
\tau_i$ form a basis for the root system $R_{(0,0)}$ of type
$A_{l-s-2}$. The equivalent of (\ref{R+}) is now
\begin{equation*}
 R^+_{(a,a)} \cap w_{0,J}w (Q(w))= \varnothing \text{ where } a \in \{1,0,-1\}.
\end{equation*}
In (\ref{nu}), the index $i$ should run from $i = 1$ to $i = s$.
Next we consider $\la_2 - \la_1$, as defined above. We decompose
$\la_2 - \la_1 = \ga_1 + \ga_2+ \ga_3$ with the support of $\ga_1$
in $R_{(1,1)}$, the support of $\ga_2$ in $R_{(0,0)}$ and the
support of $\ga_3$ in $R_{(-1,-1)}$. Then equation (\ref{ladiff})
remains valid. Hence one obtains the same conclusion.


\section{Type $A_n$ with $l$ dividing
$n+1$}\label{classicalAII}\renewcommand{\thetheorem}{\thesection.\arabic{theorem}}
\renewcommand{\theequation}{\thesection.\arabic{equation}}\setcounter{equation}{0}
\setcounter{theorem}{0}
Let $\Phi$ be a root system of type $A_n$. We begin with the special case $l=n+1=h$. Here $J= \varnothing $. We will make use of  the following Lemma.

\begin{lem}\label{lequalh} Let $\Phi$ be of type $A_n$, $l=n+1$ and $\nu \in X$. The weight $l \nu$ appears in  $\Lambda^{\bullet}_{\zeta,\varnothing}$ if and only if $\nu = \varpi_i$ for some $0 \leq i \leq n$, where $\varpi_{0}=0$.
\end{lem}
\begin{proof}
Assume that $l\nu = (n+1)\nu$ is a weight of $\Lambda^{\bullet}_{\zeta,\varnothing}$. It follows from the argument in \cite[2.2, 6.1]{AJ} that $\nu = u \varpi_i$, for some $u \in W$ and $0 \leq i \leq n$.  Next assume that  $\nu = u \varpi_i\neq \varpi_i$. Note that  $(n+1) \varpi_i = \sum_{j=1}^i j(n+1-i)\alpha_j + \sum_{j=i+1}^{n} i(n+1-j) \alpha_{j}$ and that $(n+1) u \varpi_i= \sum_{j=1}^i[ j(n+1-i)- q_j(n+1)]\alpha_j + \sum_{j=i+1}^{n} [i(n+1-j)- p_j(n+1)] \alpha_{j}$ for some $q_j, p_j \geq 0$. Since $(n+1)u \varpi_i$ is a sum of positive roots it follows that $q_1=0$, while $u \varpi_i \neq \varpi_i$ implies $q_i \geq 1$. Therefore, there exists a $j$ with $1 \leq j < i$ such that $q_j = 0$ and $q_{j+1} \geq 1$. Since $(n+1)u \varpi_i \in \Lambda^{\bullet}_{\zeta,\varnothing}$, 
it is a sum of {\it distinct} positive roots. From the preceding
decomposition into simple roots and the assumption on $j$, this sum includes precisely $j(n+1-i)$ distinct roots that contain the simple root $\alpha_j$. However, this sum contains at  most $(j+1)(n+1-i) - (n+1)$ distinct roots that contain $\al_{j+1}$ and hence at most $(j+1)(n+1-i) - (n+1)$ distinct roots that contain $\alpha_j+\alpha_{j+1}$.
On the other hand, there are only $j$ distinct roots that contain $\alpha_j$ but not $\alpha_{j+1}$.   But $(j+1)(n+1-i) - (n+1)+j = j(n+1-i)-i+j <  j(n+1-i)$, which is a contradiction. 

The weight $l\varpi_i$ is precisely the sum of all positive roots containing $\alpha_i$. Hence, it is a weight of $\Lambda^{\bullet}_{\zeta,\varnothing}$.
\end{proof}

Assume throughout the remainder of this
section that $l$ divides $n + 1$. 
We continue to identify the Weyl group
with the symmetric group in $n+1$ letters. Then $l =s+1$ and the
definition of $w$ given in (\ref{siti}) and (\ref{wi}) can be
simplified to
\begin{equation}\label{wepi}
w(\epsilon_i) = \epsilon_{w(i)} \text{ where } w(i) = t_i(m+1)+s_i
+1.
\end{equation}
We now follow the arguments used in \cite[6.2]{AJ}. Recall that $n+1
= (m+1)l$. We define the element $\sigma \in W$ as follows:
\begin{equation*}\label{sigma}
\sigma = (1,2, ... , l)(l+1, l+2, ... , 2l) \cdots (ml+1, ml+2, ...
, (m+1)l).
\end{equation*}
Direct computation shows that $w(\si^t\cdot 0) = -l\varpi_{g(t)}$
for $1\leq t \leq l-1$. Setting $\varpi_{g(0)} = 0$ yields
\begin{equation}\label{wdot0}
w \cdot 0 = w\si^t \cdot 0 + l \varpi_{g(t)} \text{ for } 0\leq t
\leq l-1.
\end{equation}
We find that
\begin{align*}
Q(w\si^t) = \{ \epsilon_i - \epsilon_j\;&|\; s_i < s_j,\si^t( t_i + 1) > \si^t( t_j + 1)\} \\
        & \cup
\{ \epsilon_i - \epsilon_j\;|\; s_i = s_j, t_i < t_j \text{ and }
    \si^t( t_i + 1) > \si^t( t_j + 1)\}.\\
\end{align*}
The cardinality of the first set in the above union is equal to the
cardinality of $Q(w)$ (see (\ref{Qw})) while the second set can be
identified with $Q(\sigma^t)$. Using this decomposition of
$Q(w\si^t)$, \cite[7.3]{GW}, and \cite[6.2(3)]{AJ}, we conclude that
\begin{equation}\label{length}
\ell(w\sigma^t) = \ell(w) + \ell(\sigma^t) = \ell(w) +(m+1)t(l-t).
\end{equation}

Next, assume that $l\nu -w_{0,J}(w\cdot 0)$ is a weight of
$\Lambda^{\bullet}_{\zeta, J}$ such that $\langle l\nu, \alpha
\rangle = 0$ for all $\alpha \in J$. The discussion in Section
\ref{classicalA} shows that $l\nu$ is a sum of distinct roots in
$R^+_{(-1,-1)}\cup R_{(0,0)}\cup R^+_{(1,1)}$. We can decompose $\nu
= \ga_1 +\ga_2$ where the support of $\ga_1$ lies entirely in
$R^+_{(1,1)}$ and the support of $\ga_2$ lies entirely in
$R^+_{(-1,-1)} \cup R_{(0,0)}$. As before, the inner product
$\langle l \nu, \alpha_{g(i)} \rangle = \langle l \nu, \beta_i
\rangle =\langle l \ga_1, \beta_i \rangle$ is completely determined
by the contribution coming from $R^+_{(1,1)}$, the positive roots of
a type $A_{l-1}$ root system. Let $\kappa_i$
denote the fundamental weight corresponding to the simple root
$\beta_i$ of the root system of type $A_{l-1}$. It follows from the above lemma that
$l\nu = l\kappa_i$. Moreover,  it follows from the construction that $\kappa_i= \varpi_{g(i)}$. Finally, by (\ref{wdot0}), for each $1 \leq
i \leq l - 1$, we have
$$
-w_{0,J}(w\cdot 0) + l\varpi_{g(i)} = -w_{0,J}(w\si^i\cdot 0),
$$
and the latter weight is a weight of $\Lambda^{\bullet}_{\zeta,J}$
(cf. \cite[7.3]{GW},\cite[Prop. 2.2]{FP1}). Further, by
(\ref{length}), this lies in degree $\ell(w\si^i)= \ell(w) +
(m+1)t(l-i)$. Since $g(i) = i(m+1)$, the result follows.



\section{Exceptional Lie algebras}\label{exceptionalI}\renewcommand{\thetheorem}{\thesection.\arabic{theorem}}
\renewcommand{\theequation}{\thesection.\arabic{equation}}\setcounter{equation}{0}
\setcounter{theorem}{0}
 In this section, we assume that the root system $\Phi$
is of exceptional type. We show that if $\nu$ satisfies the
hypothesis of Proposition \ref{keypropositiononweights}, then $\nu =
0$ except in the case when $\Phi$ is of type $E_6$ and $l = 9$. Note
that in the excluded case $l$ is divisible by $(X : \bZ\Phi) = 3$.
Our goal is to show that if $-w_{0,J}(w\cdot 0) + l\nu$ with $\nu$
being $J$-dominant is a weight of $\Lambda_{\zeta,J}^{\bullet}$,
then $\nu = 0$. For the exceptional Lie algebras, an explicit choice
of $w$ and $J$ is listed in Appendix  \ref{tables1}. One can then
explicitly compute the value of $-w_{0,J}(w\cdot 0)$. This was again
done with the aid of MAGMA \cite{BC,BCP} and the results are given
in tables in Appendix  \ref{tables2}. The weights in the tables are
listed with respect to the basis $\{\varpi_{1},\dots, \varpi_{n}\}$
of fundamental dominant weights.

Since the dimension of $\Lambda_{\zeta,J}^{\bullet}$ is finite, with
the aid of a computer, one could in principle compute all possible
weights of $\Lambda_{\zeta,J}^{\bullet}$ and compare them to
$-w_{0,J}(w\cdot 0)$ modulo $l$. For types $F_4$ and $G_2$, this can
readily be done and one finds that $\nu = 0$. For type $E_n$, the
size of $\Lambda_{\zeta,J}^{\bullet}$ is sufficiently large as to
make the computations somewhat impractical on a typical desktop
computer. As such, we present an alternative approach which makes
use of some of the ideas from the preceding sections on classical
root systems to show directly that $\nu = 0$ or reduce the
computations to a more manageable number.

In what follows, let $\delta^{\vee} = \sum_{\al \in J}\al^{\vee}$.
This is a slight abuse of notation since $\delta^{\vee}$ may not
equal $(\sum_{\al \in J}\al)^{\vee}$ but should not lead to any
confusion. Recall that $\langle -w_{0,J}(w\cdot
0),\delta^{\vee}\rangle = (l-1)|J|$. As in Section
\ref{proofofkeyprop}, set
$\max_{\la}\langle\la,\delta^{\vee} \rangle := \max\{\langle
\lambda, \delta^{\vee}\rangle\; |\;\lambda \text{ a weight of
}\Lambda^{i}_{\zeta,J}\}$. As with the classical cases, the key to
showing that $\nu = 0$ is that
$\max_{\la}\langle\la,\delta^{\vee}\rangle$ is in general ``close''
to $(l-1)|J|$.

To make this more precise, set
$$
E[t] := \{\beta \in \Phi^{+}\backslash \Phi_J^{+} ~|~ \langle
\beta,\delta^{\vee}\rangle = t\}.
$$
Then we can decompose $\Phi^{+}\backslash \Phi_J^{+}$ as a disjoint
union:
$$
\Phi^{+}\backslash \Phi_J^{+} = E[<0] \cup E[0] \cup E[>0]
$$
where
$$
E[>0] = \cup_{t>0}E[t] \text{ and } E[<0] = \cup_{t<0}E[t].
$$
That is, we separate the positive roots into those which give a
positive, zero, or negative inner product with $\delta^{\vee}$.
Since weights in $\Lambda^{\bullet}_{\zeta,J}$ are composed of sums
of distinct positive roots from $\Phi^+ \backslash \Phi_J^+$,
clearly,
$$
\max_{\la}\langle\la,\delta^{\vee}\rangle = \langle\sum_{\be \in
E[>0]}\be,\delta^{\vee}\rangle.
$$
For convenience, for the remainder of this section, set $\la :=
\sum_{\be \in E[>0]}\be$. One might think of $\la$ as a conical
representative of those weights having maximum inner product with
$\delta^{\vee}$. If $\si$ is a weight of $\Lambda_{\zeta,J}^i$ with
$\langle\si,\delta^{\vee}\rangle =  \langle\la,\delta^{\vee}\rangle$,
then we would have $\si = \la + z$ where $z$ is a sum of distinct roots
which lie in $E[0]$. With
the aid of MAGMA the weight $\la$ can be readily computed for a
given $l$, $w$, and $J$. For each relevant case, the weight $\lambda$
is given in the tables in Appendix \ref{tables2}.

Let $x = -w_{0,J}(w\cdot 0) + l\nu$ be a $J$-dominant weight of
$\Lambda_{\zeta,J}^{\bullet}$.  Our goal is to show that (in all but one case) the
only such weight that occurs is when $\nu = 0$.  Recall that, as mentioned in
Section \ref{proofofkeyprop}, the weight $-w_{0,J}(w\cdot 0)$ appears in
$\Lambda_{\zeta,J}^{\ell(w)}$.  To show that $\nu = 0$, we show that
$\langle \nu,\al^{\vee}\rangle = 0$ for all $\al \in \Pi$. We
separate this into two cases: $\al \in J$ and $\al \in \Pi\backslash
J$. The first case follows if it can be shown that
$$|\langle -w_{0,J}(w\cdot 0),\delta^{\vee}\rangle - \langle \la,\delta^{\vee}\rangle| < l.
$$
The numbers $\langle -w_{0,J}(w\cdot 0),\delta^{\vee}\rangle$ and
$\langle \la,\delta^{\vee}\rangle$ can be found by direct calculation.
These values are given in the tables in  Appendix \ref{tables2}.
We find that the desired inequality holds in all but one case (type $E_8$ when $l =
7$).  In that one remaining case, a different argument will be used to show that
$\langle\nu,\al^{\vee}\rangle = 0$ for all $\al \in J$.

We now outline the basic process for showing that $\langle\nu,\al^{\vee}\rangle = 0$
for all $\al \in \Pi\backslash J$. Note that there will be precisely one case (type $E_6$
when $l = 9$) when this fails to hold.   Once it is known that
$\langle\nu,\al^{\vee}\rangle = 0$ for all $\al \in J$, it follows that
$\langle\nu,\delta^{\vee}\rangle = 0$.  Hence $\langle x,\delta^{\vee}\rangle =
\langle -w_{0,J}(w\cdot 0) + l\nu,\delta^{\vee}\rangle = \langle
-w_{0,J}(w\cdot 0),\delta^{\vee}\rangle = (l-1)|J|$.  By the definition of $\la$,
we have $\langle\la,\delta^{\vee}\rangle \geq \langle x,\delta^{\vee}\rangle = (l - 1)|J|$.
The underlying idea of the proof is that $\langle\la,\delta^{\vee}\rangle$ and
$\langle x,\delta^{\vee}\rangle$ are in many cases equal or in general differ by only a
small amount.  We express $x$ in the form $x = \la - a + b + z$ where $a$, $b$, and
$z$ consist of (possibly empty) sums of distinct roots lying in $E[>0]$, $E[<0]$, and
$E[0]$, respectively. While $z$ can consist of arbitrarily many
elements from $E[0]$, $a$ and $b$ are constrained to consist of a
small number of elements from $E[>0]$ or $E[<0]$ depending upon how
close $\langle\la,\delta^{\vee}\rangle$ is to $(l-1)|J|$. In each
case, this can be explicitly described.

Given $\al \in \Pi\backslash J$, by direct computation, one can find bounds $A$ and $B$
(integers) such that for an arbitrary linear combination $z$ as
above, one has $A \leq \langle z,\al^{\vee}\rangle \leq B$.  In some
cases, these bounds will be sufficient to conclude that $\nu = 0$.
For such cases, the bounds will be given below.  When that is not
sufficient, using the expression of $x$ as $\la + a - b + z$, and considering the
possibilities for $a$ and $b$, one can then obtain bounds $A' \leq
\langle x,\al^{\vee} \rangle \leq B'$ for each $\al \in \Pi\backslash J$.
In those cases needed, the
bounds are listed in Appendix  \ref{tables2}. On the other hand,
$$
\langle x,\al^{\vee}\rangle = \langle - w_{0,J}(w\cdot 0) +
l\nu,\al^{\vee}\rangle = \langle -w_{0,J}(w\cdot
0),\al^{\vee}\rangle + l\langle\nu,\al^{\vee}\rangle.
$$
The value of $\langle -w_{0,J}(w\cdot 0),\al^{\vee}\rangle$ can also be computed, and
(as needed) is listed in Appendix \ref{tables2}.
Comparing this to the bounds on $\langle x,\al^{\vee}\rangle$, one often finds
that $\langle\nu,\al^{\vee}\rangle$ must be zero. When the bounds
allow for a nonzero $\nu$, MAGMA is used to verify that no solutions
exist (except in type $E_6$ when $l = 9$) by checking all
possibilities for $a$, $b$, and $z$. When needed for efficiency, the
known bounds can be used to limit the possible choices for $z$. The
basic details for each case are given below.


\vskip.25cm \noindent {  Type $E_6$:} 

\vskip.35cm \noindent $l = 11$: Here $\la = -w_{0,J}(w\cdot 0)$.
For $x = \la - a + b + z$ as above, we have
\begin{align*}
\langle -w_{0,J}(w\cdot 0),\delta^{\vee}\rangle &= \langle x,\delta^{\vee}\rangle\\
&= \langle\la,\delta^{\vee}\rangle - \langle a,\delta^{\vee}\rangle
	+ \langle b,\delta^{\vee}\rangle + \langle z,\delta^{\vee}\rangle\\
&= \langle -w_{0,J}(w\cdot 0),\delta^{\vee}\rangle - \langle a,\delta^{\vee}\rangle
	+ \langle b,\delta^{\vee}\rangle.
\end{align*}
By definition $\langle a,\delta^{\vee}\rangle \geq 0$ and $\langle b,\delta^{\vee}\rangle \leq 0$,
and so $a$ and $b$ must be empty.
Hence $x = \la + z$.  Since we are also assuming that
$x = -w_{0,J}(w\cdot 0) + l\nu$, it follows that $z$ would need to equal $l\nu$.
However, for $\al \in \Pi\backslash J$,
one finds that $-4 \leq \langle z,\al^{\vee}\rangle \leq 6$. Hence
$z$ cannot equal $11\nu$ unless $\nu = 0$.

\vskip.25cm \noindent $l = 9$: In this case, $\langle
-w_{0,J}(w\cdot 0),\delta^{\vee}\rangle$ and $\langle
\la,\delta^{\vee}\rangle$ differ by two. Also $E[>0] = E[1]$ and
$E[<0] = E[-1]$. So there are three ways we can express $x$ in the
form $\la - a + b + z$:
\begin{itemize}
\item[(1)] $x = \la - x_1 - x_2 + z$ where $x_1, x_2 \in E[1]$ ($x_1 \neq x_2$),
\item[(2)] $x = \la - x_1 + y_1 + z$ where $x_1 \in E[1]$ and $y_1 \in E[-1]$,
\item[(3)] $x = \la + y_1 + y_2 + z$ where $y_1, y_2 \in E[-1]$ ($y_1 \neq y_2$).
\end{itemize}
Using MAGMA, we compute all such weights $\la - a + b + z$ and check whether or not they can be
equal to $-w_{0,J}(w\cdot 0) + l\nu$ for a $J$-dominant weight $\nu$. In case (1),
we find precisely one pair of elements in $E[1]$ that works with $z$ being
an empty sum and $\nu$ being zero. That is, $-w_{0,J}(w\cdot 0)$ is a sum of eight
distinct roots in $E[1]$. In case (2), no sums over $E[0]$ work. In
case (3), we find however two cases where $\la$ plus two elements of
$E[-1]$ and eight elements of $E[0]$ equals $-w_{0,J}(w\cdot 0) +
9\nu$ for a $J$-dominant weight $\nu$. In one case $\nu = \varpi_1$
and in the other $\nu = \varpi_6$. The reader should be aware that
these weights give rise to the exceptions stated in Theorem
\ref{MainThm}(b)(iii).

\vskip.25cm \noindent $l = 7$: Here $\langle -w_{0,J}(w\cdot
0),\delta^{\vee}\rangle = \langle\la,\delta^{\vee}\rangle$. Arguing as in the
$l = 11$ case, it follows that
$x$ must be of the form $x = \la + z$. For $\al \in \Pi\backslash
J$, one finds the bounds on $\langle x,\al^{\vee}\rangle$ listed in Appendix  \ref{tables2}.
For each $\al$, in order to have $x = -w_{0,J}(w\cdot 0) + 7\nu$, we must
have $\langle\nu,\al^{\vee}\rangle = 0$.

\vskip.25cm \noindent $l = 5$: In this case, $\langle
-w_{0,J}(w\cdot 0),\delta^{\vee}\rangle$ and $\langle
\la,\delta^{\vee}\rangle$ differ by one. So there are two ways we
can express $x$ in the form $\la - a + b + z$:
\begin{itemize}
\item[(1)] $x = \la - x_1 + z$ where $x_1 \in E[1]$,
\item[(2)] $x = \la + y_1 + z$ where $y_1 \in E[-1]$.
\end{itemize}
Using MAGMA, we compute all possibilities and find that only one
such expression gives $x = -w_{0,J}(w\cdot 0) + l\nu$.  Specifically,
this occurs in case (1) with $z$ being the sum of a pair of
elements from $E[0]$. Again, $x = -w_{0,J}(w\cdot 0)$, i.e., $\nu =
0$.


\vskip.25cm \noindent {  Type $E_7$:} 

\vskip.35cm \noindent $l = 17$: Here $\la = -w_{0,J}(w\cdot 0)$.
Hence, $x = \la + z$, and $z$ would need to equal $l\nu$. However,
for $\al \in \Pi\backslash J$, one finds that $-9 \leq \langle
z,\al^{\vee}\rangle \leq 10$. Hence, $z$ cannot equal $17\nu$ unless
$\nu = 0$.

\vskip.25cm \noindent $l = 15$: In this case, $\langle
-w_{0,J}(w\cdot 0),\delta^{\vee}\rangle$ and $\langle
\la,\delta^{\vee}\rangle$ differ by two. Also $E[>0] = E[1]$ and
$E[<0] = E[-1]$.  As in the type $E_6$, $l = 9$ case, there are three ways we
can express $x$ in the form $\la - a + b + z$. For $\al \in \Pi\backslash J$, one finds the
bounds on $\langle x,\al^{\vee}\rangle$ listed in Appendix  \ref{tables2}.
For each $\al$, in order to
have $x = -w_{0,J}(w\cdot 0) + 15\nu$, we must have
$\langle\nu,\al^{\vee}\rangle = 0$.

\vskip.25cm \noindent $l = 13$: Here $\la = -w_{0,J}(w\cdot 0)$.
Hence $x = \la + z$, and $z$ would need to equal $l\nu$. However,
for $\al \in \Pi\backslash J$, one finds that $-6 \leq \langle
z,\al^{\vee}\rangle \leq 8$. Hence $z$ cannot equal $13\nu$ unless
$\nu = 0$.

\vskip.25cm \noindent $l = 11$: Here $\la = -w_{0,J}(w\cdot 0)$.
Hence $x = \la + z$, and $z$ would need to equal $l\nu$. However,
for $\al \in \Pi\backslash J$, one finds that $-4 \leq \langle
z,\al^{\vee}\rangle \leq 6$. Hence $z$ cannot equal $11\nu$ unless
$\nu = 0$.

\vskip.25cm \noindent $l = 9$: In this case, $\langle
-w_{0,J}(w\cdot 0),\delta^{\vee}\rangle$ and $\langle
\la,\delta^{\vee}\rangle$ differ by two. Also
$E[>0] = E[1] \cup E[2] \cup E[3]$ and $E[<0] = E[-1]\cup E[-2] \cup E[-3]$.
If we express $x$ in the form $\la - a + b + z$, note that $a$ cannot involve
any terms from $E[3]$ and $b$ cannot involve any terms from $E[-3]$.
So there are five ways we can express $x$ in the form
$\la - a + b + z$:
\begin{itemize}
\item[(1)] $x = \la - x_1 + z$ where $x_1 \in E[2]$,
\item[(2)] $x = \la - x_1 - x_2 + z$ where $x_1, x_2 \in E[1]$ ($x_1 \neq x_2$),
\item[(3)] $x = \la - x_1 + y_1 + z$ where $x_1 \in E[1]$ and $y_1 \in E[-1]$,
\item[(4)] $x = \la + y_1 + y_2 + z$ where $y_1, y_2 \in E[-1]$ ($y_1 \neq y_2$),
\item[(5)] $x = \la + y_1 + z$ where $y_2 \in E[-2]$.
\end{itemize}
For $\al \in \Pi\backslash J$, one finds the bounds on $\langle x,\al^{\vee}\rangle$ listed in
Appendix  \ref{tables2}. For each $\al$, in order to have $x =
-w_{0,J}(w\cdot 0) + 9\nu$, we must have
$\langle\nu,\al^{\vee}\rangle = 0$.

\vskip.25cm \noindent $l = 7$: In this case, $\langle
-w_{0,J}(w\cdot 0),\delta^{\vee}\rangle$ and $\langle
\la,\delta^{\vee}\rangle$ differ by three. Also
$E[>0] = E[1] \cup E[2] \cup E[3]$ and $E[<0] = E[-1]\cup E[-2] \cup E[-3]$.
So there are ten ways we can express $x$ in the form
$\la - a + b + z$. We leave the details to the interested reader.
In this case, for $\al \in \Pi\backslash J$, the
bounds on $\langle x,\al^{\vee}\rangle$ allow for the possibility
that $\langle\nu,\al^{\vee}\rangle \neq 0$. Using MAGMA, we compute
all possibilities and find that only one
such expression gives $x = -w_{0,J}(w\cdot 0) + l\nu$.  Specifically,
this occurs for an $x$ of the form $x = \la - a$ where $a$ is a sum of three
distinct roots in $E[1]$ and $b$ and $z$ are empty.
Again, $x = -w_{0,J}(w\cdot 0)$, i.e., $\nu = 0$.

\vskip.25cm \noindent $l = 5$: In this case, $\langle
-w_{0,J}(w\cdot 0),\delta^{\vee}\rangle$ and $\langle
\la,\delta^{\vee}\rangle$ differ by three. Also
$E[>0] = E[1] \cup E[2] \cup E[3]$ and $E[<0] = E[-1]\cup E[-2] \cup E[-3]$.
As in the $l = 7$ case,
there are ten ways we can express $x$ in the form $\la - a + b + z$.
Here $\Pi\backslash J = \{\al_4\}$ and one finds that $-10 \leq
\langle x,\al_4^{\vee}\rangle \leq -6$. Since $\langle
-w_{0,J}(w\cdot 0),\al_4^{\vee}\rangle = -9$, in order to have $x =
-w_{0,J}(w\cdot 0) + 5\nu$, we must have
$\langle\nu,\al_4^{\vee}\rangle = 0$.


\vskip.25cm \noindent {  Type $E_8$:} 

\vskip.35cm \noindent $l = 29$: Here $\la = -w_{0,J}(w\cdot 0)$.
Hence $x = \la + z$, and $z$ would need to equal $l\nu$. However,
for $\al \in \Pi\backslash J$, one finds that $-16 \leq \langle
z,\al^{\vee}\rangle \leq 18$. Hence $z$ cannot equal $29\nu$ unless
$\nu = 0$.

\vskip.25cm \noindent $l = 27$: In this case, $\langle
-w_{0,J}(w\cdot 0),\delta^{\vee}\rangle$ and $\langle
\la,\delta^{\vee}\rangle$ differ by two. Also $E[>0] = E[1]$ and $E[<0] = E[-1]$.
As in the type $E_6$, $l = 9$ case, there are three ways we can
express $x$ in the form $\la - a + b + z$. For $\al \in
\Pi\backslash J$ one finds the bounds on $\langle x,\al^{\vee}\rangle$ listed in Appendix
\ref{tables2}. In order to have $x = -w_{0,J}(w\cdot 0) + 27\nu$, we
must have $\langle\nu,\al^{\vee}\rangle = 0$.

\vskip.25cm \noindent $l = 25$: In this case, $\langle
-w_{0,J}(w\cdot 0),\delta^{\vee}\rangle$ and $\langle
\la,\delta^{\vee}\rangle$ differ by four. Also $E[>0] = E[1]$ and $E[<0] = E[-1]$.
So there are five ways we can express $x$ in the form $\la - a + b + z$:
\begin{itemize}
\item[(1)] $x = \la - x_1 - x_2 - x_3 - x_4 + z$ where $x_i \in E[1]$ (distinct),
\item[(2)] $x = \la - x_1 - x_2 - x_3 + y_1 + z$ where $x_i \in E(1)$ (distinct) and
$y_1 \in E[-1]$,
\item[(3)] $x = \la - x_1 - x_2 + y_1 + y_2 + z$ where $x_i \in E[1]$ (distinct) and
$y_i \in E[-1]$ (distinct),
\item[(4)] $x = \la - x_1 + y_1 + y_2 + y_3 + z$ where $x_1 \in E[1]$ and $y_i \in E[-1]$
(distinct),
\item[(5)] $x = \la + y_1 + y_2 + y_3 + y_4 + z$ where $y_i \in E[-1]$ (distinct).
\end{itemize}
For $\al \in \Pi\backslash J$ one finds the bounds on $\langle x,\al^{\vee}\rangle$
listed in Appendix \ref{tables2}. In order to have $x = -w_{0,J}(w\cdot 0) + 25\nu$, we
must have $\langle\nu,\al^{\vee}\rangle = 0$.

\vskip.25cm \noindent $l = 23$: In this case, $\langle
-w_{0,J}(w\cdot 0),\delta^{\vee}\rangle = \langle
\la,\delta^{\vee}\rangle$. Hence $x = \la + z$. For $\al \in
\Pi\backslash J$ one finds the bounds on $\langle x,\al^{\vee}\rangle$ listed in Appendix
\ref{tables2}. In order to have $x = -w_{0,J}(w\cdot 0) + 23\nu$, we
must have $\langle\nu,\al^{\vee}\rangle = 0$.

\vskip.25cm \noindent $l = 21$: In this case, $\langle
-w_{0,J}(w\cdot 0),\delta^{\vee}\rangle$ and $\langle
\la,\delta^{\vee}\rangle$ differ by four. Also $E[>0] = E[1]\cup E[2]$ and
$E[<0] = E[-1]\cup E[-2]$. So there are fourteen ways we can express $x$ in the form
$\la - a + b + z$. We leave the details to the interested reader.
For $\al \in \Pi\backslash J$ one finds the bounds on $\langle x,\al^{\vee}\rangle$ listed in
Appendix \ref{tables2}. In order to have $x = -w_{0,J}(w\cdot 0) +
21\nu$, we must have $\langle\nu,\al^{\vee}\rangle = 0$.

\vskip.25cm \noindent $l = 19$: In this case, $\langle
-w_{0,J}(w\cdot 0),\delta^{\vee}\rangle = \langle
\la,\delta^{\vee}\rangle$. Hence $x = \la + z$. For $\al \in
\Pi\backslash J$ one finds the bounds on $\langle x,\al^{\vee}\rangle$ listed in Appendix
\ref{tables2}. In order to have $x = -w_{0,J}(w\cdot 0) + 19\nu$, we
must have $\langle\nu,\al^{\vee}\rangle = 0$.

\vskip.25cm \noindent $l = 17$: In this case, $\langle
-w_{0,J}(w\cdot 0),\delta^{\vee}\rangle = \langle
\la,\delta^{\vee}\rangle$. Hence $x = \la + z$. For $\al \in
\Pi\backslash J$ one finds the bounds on $\langle x,\al^{\vee}\rangle$ listed in Appendix
\ref{tables2}. In order to have $x = -w_{0,J}(w\cdot 0) + 17\nu$, we
must have $\langle\nu,\al^{\vee}\rangle = 0$.

\vskip.25cm \noindent $l = 15$: In this case, $\langle
-w_{0,J}(w\cdot 0),\delta^{\vee}\rangle$ and $\langle
\la,\delta^{\vee}\rangle$ differ by eight. Also,
$E[>0] = E[1] \cup E[2] \cup E[3]$ and $E[<0] = E[-1]\cup E[-2] \cup E[-3]$.
As a result, there are numerous ways we can express
$x$ in the form $\la - a + b + z$. We leave the details to the interested reader.
Further, for $\al \in
\Pi\backslash J$, the bounds on $\langle x,\al^{\vee}\rangle$ allow
for the possibility that $\langle\nu,\al^{\vee}\rangle \neq 0$. By
analyzing the constraints placed on $a$, $b$, and $z$ (to afford a
nonzero $\nu$), the resulting possibilities are all computed with
MAGMA, and one finds that the only $x$ that works is precisely $-w_{0,J}(w\cdot 0)$,
i.e., $\nu = 0$.

\vskip.25cm \noindent $l = 13$: In this case, $\langle
-w_{0,J}(w\cdot 0),\delta^{\vee}\rangle$ and $\langle
\la,\delta^{\vee}\rangle$ differ by three. Also
$E[>0] = E[1] \cup E[2] \cup E[3]$ and $E[<0] = E[-1]\cup E[-2] \cup E[-3]$.
As in the type $E_7$, $l = 7$ case, there are ten ways we can
express $x$ in the form $\la - a + b + z$. For $\al \in
\Pi\backslash J$ one finds the bounds on $\langle x,\al^{\vee}\rangle$ listed in Appendix
\ref{tables2}. In order to have $x = -w_{0,J}(w\cdot 0) + 13\nu$, we
must have $\langle\nu,\al^{\vee}\rangle = 0$.

\vskip.25cm \noindent $l = 11$: In this case, $\langle
-w_{0,J}(w\cdot 0),\delta^{\vee}\rangle$ and $\langle
\la,\delta^{\vee}\rangle$ differ by two. Also
$E[>0] = E[1] \cup E[2] \cup E[3]$ and $E[<0] = E[-1]\cup E[-2] \cup E[-3]$.
As in the type $E_7$, $l = 9$ case, there are five ways we can
express $x$ in the form $\la - a + b + z$.
For $\al \in \Pi\backslash J$ one finds the bounds on $\langle x,\al^{\vee}\rangle$
listed in Appendix \ref{tables2}. In order to have $x = -w_{0,J}(w\cdot 0) + 11\nu$, we
must have $\langle\nu,\al^{\vee}\rangle = 0$.

\vskip.25cm \noindent $l = 9$: In this case, $\langle
-w_{0,J}(w\cdot 0),\delta^{\vee}\rangle$ and $\langle
\la,\delta^{\vee}\rangle$ differ by three. Also,
$E[>0] = E[1] \cup E[2] \cup E[3]$ and $E[<0] = E[-1]\cup E[-2] \cup E[-3]$.
As in the type $E_7$, $l = 7$ case, there are ten ways we can
express $x$ in the form $\la - a + b + z$. For $\al \in
\Pi\backslash J$, the bounds on $\langle x,\al^{\vee}\rangle$ allow
for the possibility that $\nu \neq 0$. To reduce the number of
possibilities that need to be checked, observe that there is a
sizable difference between $\langle \lambda,\al_1^{\vee}\rangle =
18$ and $\langle -w_{0,J}(w\cdot 0),\al_1^{\vee}\rangle = 8$. Since
$\langle\nu,\al_1^{\vee}\rangle = 0$, we must have $\langle -a + b +
z,\al_1^{\vee}\rangle = -10$. We find that $\langle -a +
b,\al_1^{\vee}\rangle \geq -3$ and $\langle z,\al_1^{\vee}\rangle
\geq -8$. This reduces the possibilities to a number manageable for
MAGMA to compute all the possible cases. The only $x$ that works is
precisely $-w_{0,J}(w\cdot 0)$, i.e., $\nu = 0$.

\vskip.25cm \noindent $l = 7$: In this case, $\langle
-w_{0,J}(w\cdot 0),\delta^{\vee}\rangle$ and $\langle
\la,\delta^{\vee}\rangle$ differ by $8$ which is larger than $7$.
Here we may not immediately conclude that
$\langle\nu,\al^{\vee}\rangle = 0$ for all $\al \in J$. However,
since $\langle\nu,\al^{\vee}\rangle \geq 0$ for all $\al \in J$,
this must be true for all but possibly one $\al$ for which one could
have $\langle\nu,\al^{\vee}\rangle = 1$.

Suppose the latter case holds.  Then we would have $\langle
x,\delta^{\vee}\rangle = 49$ whereas $\langle
\la,\delta^{\vee}\rangle = 50$. So there would be only two ways in
which $x$ could occur ($x = \la - x_1 + z$ where $x_1 \in E[1]$ or
$x = \la + y_1 + z$ where $y_1 \in E[-1]$). Here $\Pi\backslash J =
\{\al_4\}$ and one finds that $-21 \leq \langle
x,\al_4^{\vee}\rangle \leq -19$. On the other hand, $\langle
-w_{0,J}(w\cdot 0),\al_4^{\vee}\rangle = -17$. Since
$\langle\nu,\al_4^{\vee}\rangle$ is an integer, these bounds show
that we cannot have $x = -w_{0,J}(w\cdot 0) + 7\nu$ for any $\nu$,
contradicting our assumption. Therefore, 
$\langle\nu,\al^{\vee}\rangle = 0$ for all $\al \in J$.

Now, our standard argument can be used to show that
$\langle\nu,\al_4^{\vee}\rangle = 0$. Here
$E[>0] = E[1] \cup E[2] \cup E[3]$ and $E[<0] = E[-1]\cup E[-2] \cup E[-3]$.
Since $\langle -w_{0,J}(w\cdot
0),\delta^{\vee}\rangle$ and $\langle \la,\delta^{\vee}\rangle$
differ by $8$ there are numerous ways we can express $x$ in the form
$\la - a + b + z$. One finds that $-21 \leq \langle
x,\al_4^{\vee}\rangle \leq -12$. Since $\langle -w_{0,J}(w\cdot
0),\al_4^{\vee}\rangle = -17$, In order to have $x = -w_{0,J}(w\cdot
0) + 7\nu$, we must have $\langle\nu,\al_4^{\vee}\rangle = 0$.


\vskip.25cm \noindent {  Type $F_4$:} 

\vskip.35cm \noindent $l = 11$: Here $\la = -w_{0,J}(w\cdot 0)$.
Hence $x = \la + z$, and $z$ would need to equal $l\nu$. However,
for $\al \in \Pi\backslash J$, one finds that $-4 \leq \langle
z,\al^{\vee}\rangle \leq 5$. Hence, $z$ cannot equal $11\nu$ unless
$\nu = 0$.

\vskip.25cm \noindent $l = 9$: In this case, $\langle
-w_{0,J}(w\cdot 0),\delta^{\vee}\rangle$ and $\langle
\la,\delta^{\vee}\rangle$ differ by two. Also
$E[>0] = E[1] \cup E[2]$ and $E[<0] = E[-1]\cup E[-2]$.
As in the type $E_7$, $l = 9$ case, there are five ways we can
express $x$ in the form $\la - a + b + z$. For $\al \in
\Pi\backslash J$, the bounds on $\langle x,\al^{\vee}\rangle$ allow
for the possibility that $\langle\nu,\al^{\vee}\rangle \neq 0$.
Using MAGMA, we compute all possibilities and see that the only $x$ that works
is precisely $-w_{0,J}(w\cdot 0)$, i.e., $\nu = 0$.

\vskip.25cm \noindent $l = 7$: In this case, $\langle
-w_{0,J}(w\cdot 0),\delta^{\vee}\rangle$ and $\langle
\la,\delta^{\vee}\rangle$ differ by one. Also,
$E[>0] = E[1] \cup E[2] \cup E[3]$ and $E[<0] = E[-1]\cup E[-2] \cup E[-3]$.
If we express $x$ in the form $\la - a + b + z$, then $a$ can contain
terms only from $E[1]$ and $b$ can contain terms only from $E[-1]$.
As in the type $E_6$, $l = 5$ case, there are two ways we can
express $x$ in the form $\la - a + b + z$. For $\al \in
\Pi\backslash J$, the bounds on $\langle x,\al^{\vee}\rangle$ allow
for the possibility that $\langle\nu,\al^{\vee}\rangle \neq 0$.
Using MAGMA, we compute all possibilities and see that the only $x$ that works
is precisely $-w_{0,J}(w\cdot 0)$, i.e., $\nu = 0$.

\vskip.25cm \noindent $l = 5$: Here $\la = -w_{0,J}(w\cdot 0)$.
Hence, $x = \la + z$, and $z$ would need to equal $l\nu$. Here
$\Pi\backslash J = \{\al_2\}$, and one finds that $0 \leq \langle z,\al_2^{\vee}\rangle \leq 1$.
Hence $z$ cannot equal $5\nu$ unless $\nu = 0$.


\vskip.25cm \noindent {  Type $G_2$:} 

\vskip.35cm \noindent $l = 5$: Here $\la = -w_{0,J}(w\cdot 0)$.
Hence $x = \la + z$, and $z$ would need to equal $l\nu$.
Here
$E[0] = \{3\al_1 + 2\al_2 = \varpi_2\}$.  Hence, the only non-empty
option for $z$ is $z = \varpi_2$ which is not equal to $5\nu$ for
any $\nu$.  Hence, $x = -w_{0,J}(w\cdot 0)$, i.e., $\nu = 0$.

%% file: chapt5BNPP.tex
\chapter{The Cohomology Algebra $\opH^{\bullet}(u_{\zeta}({\mathfrak g}),{\mathbb C})$}
\label{cohomologysection}
\renewcommand{\thesection}{\thechapter.\arabic{section}}

The identification of $\opH^{\bullet}(u_{\zeta}({\mathfrak
g}),{\mathbb C})$ for small $l$ will proceed in several steps. The
computation is motivated by the analogous problem for the restricted Lie algebra
${\mathfrak g}_F$ over the algebraically closed field $F$ of positive characteristic $p$.
In that case, the support
variety $\mathcal{V}_{\gl_F}(F)$ of  the trivial module $F$ is homeomorphic (as a topological
space) to the the
restricted nullcone $\ncf=\{x\in{\gl}_F\,|\,x^{[p]}=0\}$. By \cite{CLNP}, the variety
$\ncf$ identifies with the closed subset $G\cdot \uj\subset \gl_F$, for an appropriate
subset $J \subset
\Pi$. When $p \geq h$,  $J = \varnothing$ and $\uj =
\ul$.

To attack the computation of the cohomology of $u_{\zeta}({\mathfrak g})$, we
consider the parabolic subgroup $P_J$ associated to this subset $J
\subset \Pi$ with $w(\Phi_{0}^+)=\Phi_{w\cdot 0}^+=\Phi_{J}^+$. Then we proceed as follows:

\begin{itemize}

\item In Section 5.1, the cohomology of $u_{\zeta}({\mathfrak g})$ is shown to be related to
that of $u_{\zeta}({\mathfrak p}_J)$.

\item In Section 5.2, the cohomology  of $u_{\zeta}({\mathfrak p}_J)$ is  shown to be related 
to the cohomology of $u_{\zeta}({\mathfrak u}_J)$.

\item Sections 5.3--5.4 present the key computation
for $\opH^{\bullet}(u_{\zeta}({\mathfrak u}_J),{\mathbb C})$.

\item Sections 5.5--5.7 complete the computation of
$\opH^{\bullet}(u_{\zeta}({\mathfrak g}),{\mathbb C}).$

\end{itemize}

Throughout this chapter, fix $l$ satisfying Assumption \ref{assumption}. Fix $w\in W$ and $J\subseteq \Pi$
so that $w(\Phi^+_0)=\Phi_J^+$. 

While the goal of
this paper is to make cohomological computations in the case that $l
< h$, the arguments are also valid for $l \geq h$. In
that case, we would have $\Phi_0 = \varnothing$, $J = \varnothing$,
and $w = \Id$. Then $P_J = B$, $U_J = U$, $L_J = T$, ${\mathfrak
p}_J = {\mathfrak b}$, ${\mathfrak u}_J = {\mathfrak u}$,
${\mathfrak l}_J = {\mathfrak t}$, and the module $M =
(\ind_{U_\zeta({\mathfrak b})}^{U_\zeta({\mathfrak p}_J)} w\cdot
0)^* = {\mathbb C}$. As such, our results recapture the  calculation for
 $l> h$ given by Ginzburg and Kumar  in \cite{GK}.


\section{Spectral sequences, I}\label{firstsectioncohocal}\renewcommand{\thetheorem}{\thesection.\arabic{theorem}}
\renewcommand{\theequation}{\thesection.\arabic{equation}}\setcounter{equation}{0}
\setcounter{theorem}{0}
 By Lemma \ref{Steinbergwtslemma},
$w\cdot 0$ is $J$-dominant, i.e., $\langle w\cdot 0,\alpha^\vee\rangle$ is a non-negative integer for all
$\alpha\in J$. Using the factorization
$$\ind_{U_\zeta({\mathfrak b})}^{U_\zeta({\mathfrak g})}=\ind_{U_\zeta({\mathfrak p}_J)}^{U_\zeta({\mathfrak g})}\circ
\ind_{U_\zeta({\mathfrak b})}^{U_\zeta({\mathfrak p}_J)}$$
of functors, 
there is a Grothendieck spectral sequence
  \begin{equation}\label{firsteqncoho}
E_{2}^{i,j}=R^{i}\text{ind}_{U_{\zeta}({\mathfrak
p}_{J})}^{U_{\zeta}({\mathfrak g})}
R^{j}\text{ind}_{U_{\zeta}({\mathfrak b})}^{U_{\zeta}({\mathfrak
p}_{J})} w\cdot0 \Rightarrow R^{i+j}\text{ind}_{U_{\zeta}({\mathfrak
b})}^{U_{\zeta}({\mathfrak g})}w\cdot 0.
\end{equation}
However, since  $w\cdot 0$ is $J$-dominant,
$R^{j}\text{ind}_{U_{\zeta}({\mathfrak b})}^{U_{\zeta}({\mathfrak
p}_{J})} w\cdot 0=0$ for $j>0$. Consequently, this spectral sequence collapses and so
yields, by \cite[Cor. 3.8]{A},
\begin{equation}\label{firstcohoeqn2}
R^{i}\text{ind}_{U_{\zeta}({\mathfrak p}_{J})}^{U_{\zeta}({\mathfrak
g})} \left(\text{ind}_{U_{\zeta}({\mathfrak b})}^{U_{\zeta}({\mathfrak
p}_{J})} w\cdot 0\right)= R^{i}\text{ind}_{U_{\zeta}({\mathfrak
b})}^{U_{\zeta}({\mathfrak g})} w\cdot 0=
\begin{cases}
{\mathbb C} & \text{ if } i = \ell(w)\\
0 & \text{ if } i\neq \ell(w).
\end{cases}
\end{equation}

The following spectral sequence provides a connection between the
cohomology of $u_{\zeta}({\mathfrak p}_J)$ and that of
$u_{\zeta}({\mathfrak g})$.

\begin{theorem}\label{firstsectioncohocalthm}
 Let $w\in W$ such that $w(\Phi_{0}^+)=\Phi_{J}^+$ where
$J\subseteq \Pi$. There exists a first quadrant spectral sequence of rational $G$-modules
$$E_{2}^{i,j}=R^{i}\ind_{{P}_{J}}^{G}
\operatorname{H}^j\left(u_{\zeta}({\mathfrak p}_{J}),
\operatorname{ind}_{U_{\zeta}({\mathfrak b})}^{U_{\zeta}({\mathfrak
p}_{J})} w\cdot 0\right) \Rightarrow
\operatorname{H}^{i+j-\ell(w)}(u_{\zeta}({\mathfrak g}),{\mathbb
C}).$$
\end{theorem}

\begin{proof} We follow the construction in \cite[I 6.12]{Jan1}.
Form the functors ${\mathcal F}_1, {\mathcal F}_2:U_\zeta({\mathfrak p}_J){\text{--mod}}
\to G$--mod defined by setting
$$
{\mathcal F}_{1}(-)=\text{Hom}_{u_{\zeta}({\mathfrak g})}\left({\mathbb
C},\text{ind}_{U_{\zeta} ({\mathfrak p}_{J})}^{U_{\zeta}({\mathfrak
g})}(-)\right) \text{ and } {\mathcal F}_{2}(-)=\text{ind}_{P_{J}}^{G}
\text{Hom}_{u_{\zeta}({\mathfrak p}_{J})}\left({\mathbb C},-\right).
$$
  The
reader should observe that we are implicitly using the Frobenius map
(cf. Section \ref{connectionswithalgebraicgroups}) and the following
identification of functors:
$$\text{ind}_{P_{J}}^{G}(-)\cong
\text{ind}_{U_{\zeta}({\mathfrak p}_{J})//u_{\zeta}({\mathfrak
p}_{J})}^{U_{\zeta}({\mathfrak g})//u_{\zeta}({\mathfrak g})}(-).
$$
The functors ${\mathcal F}_1$ and ${\mathcal F}_2$ are naturally
isomorphic. Therefore, there exist two Grothendieck spectral sequences:
\begin{equation}\label{firstcohoeqn3}\begin{cases}
{^\prime E}_{2}^{i,j}=\opH^{i}\left(u_{\zeta}({\mathfrak
g}),R^{j}\text{ind}_{U_{\zeta} ({\mathfrak
p}_{J})}^{U_{\zeta}({\mathfrak g})}(\text{ind}_{U_{\zeta}({\mathfrak
b})} ^{U_{\zeta}({\mathfrak p}_{J})}w\cdot 0)\right)\Rightarrow
(R^{i+j}{\mathcal F}_{1})(\text{ind}_{U_{\zeta}({\mathfrak b})}
^{U_{\zeta}({\mathfrak p}_{J})}w\cdot 0);\\
 E_{2}^{i,j}=R^{i}\text{ind}_{P_{J}}^{G}
\opH^j\left(u_{\zeta}({\mathfrak p}_{J}),
\text{ind}_{U_{\zeta}({\mathfrak b})} ^{U_{\zeta}({\mathfrak
p}_{J})}w\cdot 0)\right) \Rightarrow (R^{i+j}{\mathcal
F}_{2})(\text{ind}_{U_{\zeta}({\mathfrak b})} ^{U_{\zeta}({\mathfrak
p}_{J})}w\cdot 0)\end{cases}
\end{equation}
necessarily converging to the same abutment.

By (\ref{firstcohoeqn2}), the first spectral sequence ${^\prime E}^{i,j}_2$ collapses, leading to
an identification
$$(R^{\bullet+\ell(w)}{\mathcal F}_{1})(\text{ind}_{U_{\zeta}({\mathfrak b})}
^{U_{\zeta}({\mathfrak p}_{J})}w\cdot 0)\cong
\opH^{\bullet}(u_{\zeta}({\mathfrak
g}),R^{\ell(w)}\text{ind}_{U_{\zeta}({\mathfrak
b})}^{U_{\zeta}({\mathfrak g})} w\cdot 0)\cong
\opH^{\bullet}(u_{\zeta}({\mathfrak g}),{\mathbb C}).$$ Combining
this with the second spectral sequence $E_2^{i,j}$ proves
the theorem.
\end{proof}


\section{Spectral sequences, II}\label{secondsectioncohocal}\renewcommand{\thetheorem}{\thesection.\arabic{theorem}}
\renewcommand{\theequation}{\thesection.\arabic{equation}}\setcounter{equation}{0}
\setcounter{theorem}{0}

In this section, we reidentify the term $\operatorname{H}^j\left(u_{\zeta}({\mathfrak p}_{J}),
\operatorname{ind}_{U_{\zeta}({\mathfrak b})}^{U_{\zeta}({\mathfrak
p}_{J})} w\cdot 0\right)$ occurring in the spectral sequence 
 in
  Theorem \ref{firstsectioncohocalthm}. Using 
the Lyndon-Hochshild-Serre spectral sequence in Lemma \ref{ABG} for
$u_{\zeta}(\uj) \unlhd u_{\zeta}(\pj)$. Note that
$u_{\zeta}(\pj)//u_{\zeta}(\uj) \cong u_{\zeta}(\lj)$. Thus, there is a spectral 
sequence
\begin{equation*}\label{secondcohocaleqn1}
E_2^{i,j} = \opH^i(u_{\zeta}({\mathfrak
l}_{J}),\opH^j(u_{\zeta}({\mathfrak u}_{J}),
\ind_{U_{\zeta}({\mathfrak b})}^{U_{\zeta}({\mathfrak p}_{J})}w\cdot
0)) \Rightarrow \opH^{i + j}(u_{\zeta}({\mathfrak p}_{J}),
\ind_{U_{\zeta}({\mathfrak b})}^{U_{\zeta}({\mathfrak p}_{J})}w\cdot
0).
\end{equation*}
All of the cohomology groups involved admit an action of
$U_{\zeta}(\pj)$ induced from the $\ad_r$-action. Furthermore, the action
of the subalgebra $u_{\zeta}(\pj)$ is trivial. Hence, there is an
action of $\BU(\pj)$ (or equivalently $P_J$). Moreover, the spectral
sequence preserves this action.

Since $\ind_{U_{\zeta}({\mathfrak b})}^{U_{\zeta}({\mathfrak
p}_{J})}w\cdot 0$ is trivial as a $U_{\zeta}({\mathfrak
u}_{J})$-module, the left-hand side of the spectral sequence may be
reidentified as follows:

\begin{equation}\label{secondcohoeqn2}
E_2^{i,j} = \opH^i\left(u_{\zeta}({\mathfrak
l}_{J}),\opH^j(u_{\zeta}({\mathfrak u}_{J}),{\mathbb C})\otimes
\ind_{U_{\zeta}({\mathfrak b})}^{U_{\zeta}({\mathfrak p}_{J})}w\cdot
0\right) \Rightarrow \opH^{i + j}(u_{\zeta}({\mathfrak
p}_{J}),\ind_{U_{\zeta}({\mathfrak b})}^{U_{\zeta}({\mathfrak
p}_{J})}w\cdot 0).
\end{equation}

\begin{prop}\label{secondcohoprop}
 Let $w \in W$ and $J \subset \Pi$ be such that
$w(\Phi_0^+) = \Phi_J^+$. Then for all $j \geq 0$ there is an
isomorphism of $\BU(\pj)$-modules
$$
\opH^j\left(u_{\zeta}({\mathfrak p}_{J}),\ind_{U_{\zeta}({\mathfrak
b})}^{U_{\zeta}({\mathfrak p}_{J})}w\cdot 0\right) \cong
\Hom_{u_{\zeta}({\mathfrak l}_{J})}\left((\ind_{U_{\zeta}({\mathfrak
b})}^{U_{\zeta}({\mathfrak p}_{J})} w\cdot
0)^*,\opH^j(u_{\zeta}({\mathfrak u}_{J}),{\mathbb C})\right).
$$
\end{prop}

\begin{proof} Since $\ind_{U_{\zeta}({\mathfrak b})}^{U_{\zeta}({\mathfrak p}_{J})}w\cdot 0$
is projective as a $u_{\zeta}({\mathfrak l}_{J})$-module, in the
spectral sequence (\ref{secondcohoeqn2}), $E_2^{i,j} = 0$ for all $i
> 0$. Thus the spectral sequence collapses giving for all $j \geq 0$
\begin{align*}
\opH^j\left(u_{\zeta}({\mathfrak p}_{J}),\ind_{U_{\zeta}({\mathfrak
b})}^{U_{\zeta}({\mathfrak p}_{J})} w\cdot 0\right) &\cong
\Hom_{u_{\zeta}({\mathfrak l}_{J})}\left({\mathbb C},\opH^j(u_{\zeta}({\mathfrak u}_{J}),{\mathbb C})\otimes\ind_{U_{\zeta}({\mathfrak b})}^{U_{\zeta}({\mathfrak p}_{J})}w\cdot 0\right)\\
&\cong \Hom_{u_{\zeta}({\mathfrak
l}_{J})}\left((\ind_{U_{\zeta}({\mathfrak b})}^{U_{\zeta}({\mathfrak
p}_{J})} w\cdot 0)^*,\opH^j(u_{\zeta}({\mathfrak u}_{J}),{\mathbb
C})\right).
\end{align*}
\end{proof}


\section{An identification theorem}\label{thirdsectioncohocal}\renewcommand{\thetheorem}{\thesection.\arabic{theorem}}
\renewcommand{\theequation}{\thesection.\arabic{equation}}\setcounter{equation}{0}
\setcounter{theorem}{0}

The following theorem now gives an identification as
$U_{\zeta}^0$-modules (or equivalently $\BU({\mathfrak h})$-modules
where ${\mathfrak h} \subset \gl$ is the Cartan subalgebra) of the
$\Hom$-group appearing in Proposition \ref{secondcohoprop}.

\begin{theorem}\label{thirdsectioncohocalthm}
 Let $l$ be as in Assumption \ref{assumption} and $w\in W$ such that $w(\Phi_{0}^+)=\Phi_{J}^+$.
\begin{itemize}
\item[(a)] Suppose that $l\nmid n+1$ when $\Phi$ is of type $A_{n}$ and $l \neq 9$ when $\Phi$
is of type $E_6$. Then as $U_{\zeta}^0$-modules
$$
\Hom_{u_{\zeta}({\mathfrak l}_{J})}\left((\ind_{U_{\zeta}({\mathfrak b})}
^{U_{\zeta}({\mathfrak p}_{J})}w\cdot
0)^*,\opH^{s}(u_{\zeta}({\mathfrak u}_{J}),{\mathbb C})\right) \cong
S^{\frac{s-\ell(w)}{2}}\left({\mathfrak u}_{J}^{*}\right)^{[1]}.
$$
\item[(b)] If $\Phi$ is of type $A_{n}$ with $n+1= l(m+1)$ and $w\in W$ is as defined in
(\ref{wepi}), then as $U_{\zeta}^0$-modules
$$
\Hom_{u_{\zeta}({\mathfrak l}_{J})}\left((\ind_{U_{\zeta}({\mathfrak b})}
^{U_{\zeta}({\mathfrak p}_{J})}w\cdot
0)^*,\opH^{s}(u_{\zeta}({\mathfrak u}_{J}),{\mathbb C})\right) \cong
\bigoplus_{t=0}^{l-1} S^{\frac{s-\ell(w)-(m+1)t(l-t)}{2}}\left({\mathfrak
u}_{J}^{*}\right)^{[1]}\otimes l\varpi_{t(m+1)},
$$
where $\varpi_{0}=0$.
\item[(c)] If $\Phi$ is of type $E_6$ and $l = 9$ (assuming that $w$ and $J$
 are as in Appendix \ref{tables1}), then
as $U_{\zeta}^0$-modules
\begin {eqnarray*}
\Hom_{u_{\zeta}({\mathfrak l}_{J})}\left((\ind_{U_{\zeta}({\mathfrak b})}
^{U_{\zeta}({\mathfrak p}_{J})}w\cdot
0)^*,\opH^{s}(u_{\zeta}({\mathfrak u}_{J}),{\mathbb C})\right) &\cong&
S^{\frac{s-\ell(w)}{2}}\left({\mathfrak u}_{J}^{*}\right)^{[1]} \oplus
\left(S^{\frac{s-20}{2}}\left({\mathfrak u}_{J}^{*}\right)^{[1]}\otimes l\varpi_1\right) \\
&\ \ \ \ \ \oplus& \left(S^{\frac{s-20}{2}}\left({\mathfrak
u}_{J}^{*}\right)^{[1]}\otimes l\varpi_6\right).
\end{eqnarray*}
\end{itemize}
\end{theorem}

\begin{proof} For convenience
set $M = (\ind_{U_{\zeta}({\mathfrak b})}^{U_{\zeta}({\mathfrak
p}_J)}w\cdot 0)^*$ and ${\mathcal G}(-) =
\Hom_{u_{\zeta}(\lj)}(M,-)$. Since the module $M$ is injective
(equivalently projective) as a $u_{\zeta}(\lj)$-module, the functor
${\mathcal G}(-) = \Hom_{u_{\zeta}(\lj)}(M,-)$ is an exact functor.

The argument will proceed by induction on successive quotients of
$\Uz(\uj)$ (cf. \cite[2.4]{GK}). Let $N = |\Phi^+ \backslash
\Phi_J^+|$. Note that previously $N$ was used to denote $|\Phi^+|$,
but this should not cause any confusion here. As in Chapter 2,
choose any fixed ordering of root vectors $f_1$, $f_2$, \dots, $f_N$
in $\Uz(\uj)$ corresponding to the positive roots in
$\Phi^+\backslash\Phi_J^+$. For purposes of this argument, the
precise ordering is irrelevant. Each $f_i^l$ is central
in $\Uz(\uj)$. For $0 \leq i \leq N$, set $A_i = \Uz(\uj)//\langle
f_1^l,f_2^l,\dots,f_i^l\rangle$ where $\langle \dots \rangle$
denotes ``the subalgebra generated by \dots'' (with $A_0 =
\Uz(\uj)$). Note that $A_N = u_{\zeta}(\uj)$. For $1 \leq i \leq N$,
set $B_i = \langle f_i^l\rangle \subset A_{i-1}$ be the subalgebra
generated by $f_i^l$. Note that each $B_i$ is a polynomial algebra
in one variable. Note also that $B_i$ is normal (in fact central) in
$A_{i-1}$, and $A_{i-1}//B_i \cong A_i$. For $0 \leq i \leq N$, let
$V_i$ be an $i$-dimensional vector space with basis
$\{x_1,x_2,\dots,x_i\}$. Further consider $V_i$ as a
$U_{\zeta}^0$-module by letting $x_i$ have weight $\ga_i$. That is,
each $x_i$ is ``dual'' to the element $f_i$. In particular, $V_N
\cong \uj^*$. Here $V_0 = \{0\}$.

Consider first part (a). We prove inductively for $0 \leq i \leq N$
that as $U_{\zeta}^0$-modules we have
$$
{\mathcal G}\big{(}\opH^{s}(A_i,{\mathbb C})\big{)} \cong
\begin{cases}
S^{r}(V_i)^{[1]} &\text{ if }s = 2r + \ell(w)\\
0 &\text{ else}.
\end{cases}
$$
The case $i = N$ is precisely the statement of the theorem. For $i =
0$, this is precisely Theorem \ref{multSteinberg} where by
convention we take $S^0(V_0) = {\mathbb C}$.

Assume now that the claim is true for $i - 1$, and we will show that
it is true for $i$. We will make use of the LHS spectral sequence of
Lemma \ref{ABG} for $B_i \unlhd A_{i-1}$. Since $B_i$ is a
polynomial algebra in one variable, its cohomology is an exterior
algebra in one variable. Precisely, for each $i$, we have as a
$U_{\zeta}^0$-module:
$$
\opH^b(B_i,{\mathbb C}) =
\begin{cases}
{\mathbb C} &\text{ if } b = 0\\
{\mathbb C}_i^{[1]} &\text { if } b = 1\\
0 &\text{ else},
\end{cases}
$$
where ${\mathbb C}_i$ is a one-dimensional vector space with basis
element $x_i$ (i.~e., of weight $\ga_i$).

The spectral sequence is
$$
E_{2}^{a,b} = \opH^a(A_i,\opH^b(B_i,{\mathbb C})) \Rightarrow
\opH^{a + b}(A_{i-1},{\mathbb C}).
$$
The algebra $A_i$ acts on $B_i$ via the $\adw$-action. By Corollary
\ref{corB}(b), this action is trivial. Therefore this spectral
sequence can be rewritten as
$$
E_{2}^{a,b} = \opH^a(A_i,{\mathbb C})\otimes\opH^b(B_i,{\mathbb C})
\Rightarrow \opH^{a + b}(A_{i-1},{\mathbb C}).
$$
Alternately, this easily follows from the above description of
$\opH^b(B_i,{\mathbb C})$.

Again using the $\adw$-action, $u_{\zeta}(\lj)$ acts on all terms of
the spectral sequence and this action is preserved by the
differentials. Since the functor ${\mathcal G}(-)$ is exact, we can
apply it to the above spectral sequence and obtain a new spectral
sequence. For convenience, we abusively use the same name:
$$
E_{2}^{a,b} = {\mathcal G}\big{(}\opH^a(A_i,{\mathbb
C})\otimes\opH^b(B_i,{\mathbb C})\big{)} \Rightarrow
    {\mathcal G}\big{(}\opH^{a + b}(A_{i-1},{\mathbb C})\big{)}.
$$
The above description of $\opH^b(B_i,{\mathbb C})$ shows that
$u_{\zeta}(\lj)$ acts trivially on it. And hence, this spectral
sequence may be rewritten as
$$
E_{2}^{a,b} = {\mathcal G}\big{(}\opH^a(A_i,{\mathbb
C})\big{)}\otimes\opH^b(B_i,{\mathbb C}) \Rightarrow
    {\mathcal G}\big{(}\opH^{a + b}(A_{i-1},{\mathbb C})\big{)}.
$$

Observe that $E_2^{a,b} = 0$ for $b \geq 2$. That is, the spectral
sequence consists of at most two nonzero rows. So only the first
differential $d_2 : E_2^{a,1} \to E_2^{a + 2,0}$ could potentially
be nonzero. The first row $E_2^{a,0} = \opH^a(A_i,{\mathbb C})$ is
precisely what we are trying to identify inductively. Note also that
for all $a$, the second row $E_2^{a,1} \cong E_{2}^{a,0}\otimes
{\mathbb C}_i^{[1]}$. In particular, $E_2^{a,1} \neq 0$ if and only
if $E_2^{a,0} \neq 0$.

By the inductive hypothesis, we know that the abutment
$$
{\mathcal G}\big{(}\opH^{a + b}(A_{i-1},{\mathbb C})\big{)} \cong
\begin{cases}
S^{r}(V_{i-1})^{[1]} &\text{ if } a + b = 2r + \ell(w)\\
0 &\text{ else}.
\end{cases}
$$
In particular, ${\mathcal G}\big{(}\opH^{a + b}(A_{i-1},{\mathbb
C})\big{)} = 0$ for $a + b < \ell(w)$.

Let $A \geq 0$ be least value of $a$ such that $E_2^{a,0} \neq 0$.
Hence, $A$ is necessarily the least value of $a$ such that
$E_2^{a,1} \neq 0$. In particular, $E_2^{A-2,1} = 0$ and hence
$E_{\infty}^{A,0} \cong E_2^{A,0}/d_2(E_2^{A-2,1}) = E_2^{A,0}$. By
the inductive hypothesis, we conclude that $A = \ell(w)$. So for all
$a < \ell(w)$ we have
$$
{\mathcal G}\big{(}\opH^{a}(A_i,{\mathbb C})\big{)} = 0.
$$

Next we claim that $E_2^{\ell(w) + a,0} = 0 = E_2^{\ell(w) + a,1}$
for all odd $a > 0$. This can be seen inductively on $a$. For
example, since $E_2^{\ell(w) - 1,1} = 0$,
$$
E_2^{\ell(w) + 1,0} = E_2^{\ell(w) + 1,0}/d_2(E_2^{\ell(w) - 1,1})
\cong E_{\infty}^{\ell(w) + 1,0} \subset {\mathcal
G}\big{(}\opH^{\ell(w) + 1}(A_{i-1},{\mathbb C})\big{)} = 0.
$$
Inductively, for odd $a > 0$, we similarly have $E_2^{\ell(w) + a -
2,1} = 0$ and so
$$
E_2^{\ell(w) + a,0} = E_2^{\ell(w) + a,0}/d_2(E_2^{\ell(w) + a -
2,1}) \cong E_{\infty}^{\ell(w) + a,0} \subset {\mathcal
G}\big{(}\opH^{\ell(w) + a}(A_{i-1},{\mathbb C})\big{)} = 0.
$$
In other words, for all odd $a > 0$, we have (as claimed)
$$
{\mathcal G}\big{(}\opH^{\ell(w) + a}(A_i,{\mathbb C})\big{)} = 0.
$$
Summarizing: our spectral sequence has $E_2^{a,b} = 0$ for $a <
\ell(w)$ and $E_2^{\ell(w) + a,b} = 0$ for all odd $a > 0$. That is,
the columns are initially zero and then begin to alternate nonzero
(potentially) and zero thereafter.

Furthermore, for all even $a \geq 0$, we then have
$$
\ker\{d_2 : E_2^{\ell(w) + a,1} \to E_2^{\ell(w) + a + 2,0}\}
\subset E_{\infty}^{\ell(w) + a,1} \subset {\mathcal
G}\big{(}\opH^{\ell(w) + a + 1}(A_{i-1},{\mathbb C})\big{)} = 0.
$$
Therefore $d_2: E_2^{\ell(w) + a,1} \to E_2^{\ell(w) + a + 2,0}$ is
injective. Hence, we have for all even $a \geq 0$,
\begin{align*}
E_2^{\ell(w) + a,0}/E_2^{\ell(w) + a - 2,1} &\cong
        E_2^{\ell(w) + a,0}/d_2(E_2^{\ell(w) + a - 2,1})
        \cong E_{\infty}^{\ell(w) + a,0}\\
    &\cong {\mathcal G}\big{(}\opH^{\ell(w) + a}(A_{i-1},{\mathbb C})\big{)} =
        S^{a/2}(V_{i-1})^{[1]}.
\end{align*}
So we have a short exact sequence of $U_{\zeta}^0$-modules:
$$
0 \to E_2^{\ell(w) + a - 2,1} \to E_2^{\ell(w) + a,0} \to
S^{a/2}(V_{i-1})^{[1]} \to 0.
$$
But identifying these $E_2$-terms gives
$$
0 \to {\mathcal G}\big{(}\opH^{\ell(w) + a - 2}(A_i,{\mathbb
C})\big{)}\otimes {\mathbb C}_i^{[1]} \to {\mathcal
G}\big{(}\opH^{\ell(w) + a}(A_i,{\mathbb C})\big{)} \to
S^{a/2}(V_{i-1})^{[1]} \to 0.
$$
Inducting now on even $a \geq 0$, we may assume that
$$
{\mathcal G}\big{(}\opH^{\ell(w) + a - 2}(A_i,{\mathbb C})\big{)} =
S^{(a - 2)/2}(V_i)^{[1]}.
$$
The short exact sequence becomes
$$
0 \to S^{(a - 2)/2}(V_i)^{[1]}\otimes {\mathbb C}_i^{[1]} \to
{\mathcal G}\big{(}\opH^{\ell(w) + a}(A_i,{\mathbb C})\big{)} \to
S^{a/2}(V_{i-1})^{[1]} \to 0
$$
or setting $a = 2r$,
$$
0 \to S^{r-1}(V_i)^{[1]}\otimes {\mathbb C}_i^{[1]} \to {\mathcal
G}\big{(}\opH^{\ell(w) + 2r}(A_i,{\mathbb C})\big{)} \to
S^{r}(V_{i-1})^{[1]} \to 0.
$$
Hence as a $U_{\zeta}^0$-module,
$$
{\mathcal G}\big{(}\opH^{\ell(w) + 2r}(A_i,{\mathbb C})\big{)} \cong
(S^{r-1}(V_i)^{[1]}\otimes {\mathbb C}_i^{[1]})\oplus
S^{r}(V_{i-1})^{[1]}.
$$
The left hand factor consists of $r$-fold symmetric powers in the
$x_j$ which contain at least one $x_i$, while the righthand factor
consists of $r$-fold symmetric powers in $x_j$ with $1 \leq j \leq i
- 1$. Hence
$$
{\mathcal G}\big{(}\opH^{\ell(w) + 2r}(A_i,{\mathbb C})\big{)} \cong
S^{r}(V_i)^{[1]}
$$
which along with the above conclusions verifies the inductive claim
and hence part (a) of the theorem.

For parts (b) and (c) a similar argument can be used whose details
are left to the interested reader. Of crucial importance here is the
degrees in which the ``extra'' classes arise in parts (b) and (c) of
Theorem \ref{multSteinberg}. For example, consider part (b). We have
${\mathcal G}\big{(}\opH^i(\Uz(\uj),{\mathbb C})\big{)} = 0$ unless
$i = \ell(w) + (m + 1)t(l-t)$ for $0 \leq t \leq l - 1$. Observe
that $t(l - t)$ (and, hence, $(m + 1)t(l-t)$) is necessarily even.
Thus the extra cohomology classes appear in degrees having the same
parity as $\ell(w)$. As such, in the above argument, we will have a
similar phenomenon happening in the spectral sequence: $E_2^{a,b} =
0$ for $a < \ell(w)$ or $a = \ell(w) + a'$ with $a' > 0$ being odd,
and the analogous argument will give the claim. Similarly in part
(c), the extra cohomology classes lie in a degree with
the same parity as $\ell(w)$.

\end{proof}


\section{Spectral sequences, III}\label{4thsectioncohocal} \renewcommand{\thetheorem}{\thesection.\arabic{theorem}}
\renewcommand{\theequation}{\thesection.\arabic{equation}}\setcounter{equation}{0}
\setcounter{theorem}{0}

As mentioned previously, the $\Hom$-groups in
Theorem \ref{thirdsectioncohocalthm} admit an action of $\BU(\pj)$
induced from the $\ad_r$-action. On the other hand $\BU(\pj)$ acts
naturally on $\uj$ by the adjoint action or on $\uj^*$ by the
coadjoint action. This can be further extended to an action on
$S^{\bullet}(\uj^*)$. With this action, the isomorphisms in Theorem
\ref{thirdsectioncohocalthm} also hold as $\BU(\pj)$-modules (or
equivalently $P_J$-modules).

\begin{lem}\label{4thsectionlemma}
 The isomorphisms of Theorem \ref{thirdsectioncohocalthm}
 also hold as $\BU(\pj)$-modules where
the actions are as described above.
\end{lem}

\begin{proof}
Consider first the generic case - part (a). We use notation as in
the proof of Theorem \ref{thirdsectioncohocalthm}. Let $Z_J =
\langle f_1^l,\dots,f_N^l\rangle \subset \Uz(\uj)$ be the central
subalgebra such that $\Uz(\uj)//Z_J \cong u_{\zeta}(\uj)$ (cf. also
Section \ref{adjointaction}). Consider the spectral sequence
$$
E_2^{a,b} = \opH^a(u_{\zeta}(\uj),\opH^b(Z_J,{\mathbb C}))
\Rightarrow \opH^{a + b}(\Uz(\uj),{\mathbb C})
$$
given by Lemma \ref{ABG}. This is a spectral sequence of
$U_{\zeta}(\pj)$-modules. Since $u_{\zeta}(\pj)$ acts trivially on
$\opH^b(Z_J,{\mathbb C})$ (cf. Corollary \ref{corB}(b)), this may be
rewritten as
$$
E_2^{a,b} = \opH^a(u_{\zeta}(\uj),{\mathbb
C})\otimes\opH^b(Z_J,{\mathbb C})) \Rightarrow \opH^{a +
b}(\Uz(\uj),{\mathbb C}).
$$
Furthermore, applying the functor ${\mathcal G}(-)$ we get a new
spectral sequence (using the same name)
$$
E_2^{a,b} = {\mathcal G}\big{(}\opH^a(u_{\zeta}(\uj),{\mathbb
C})\big{)}\otimes\opH^b(Z_J,{\mathbb C}) \Rightarrow
\mathcal{G}\big{(}\opH^{a + b}(\Uz(\uj),{\mathbb C})\big{)},
$$
whose differentials still preserve the action of $U_{\zeta}(\pj)$.
Since $u_{\zeta}(\uj)$ acts trivially on both
$\opH^{\bullet}(u_{\zeta}(\uj),{\mathbb C})$ and
$\opH^{\bullet}(\Uz(\uj),{\mathbb C})$, and $u_{\zeta}(\lj) \cong
u_{\zeta}(\pj)//u_{\zeta}(\uj)$, we have
$$
{\mathcal G}\big{(}\opH^{\bullet}(u_{\zeta}(\uj),{\mathbb C})\big{)}
= \Hom_{u_{\zeta}(\lj)}(M,\opH^{\bullet}(u_{\zeta}(\uj),{\mathbb
C})) \cong
\Hom_{u_{\zeta}(\pj)}(M,\opH^{\bullet}(u_{\zeta}(\uj),{\mathbb C}))
$$
and
$$
\mathcal{G}\big{(}\opH^{\bullet}(\Uz(\uj),{\mathbb C})\big{)} =
\Hom_{u_{\zeta}(\lj)}(M,\opH^{\bullet}(\Uz(\uj),{\mathbb C})) \cong
\Hom_{u_{\zeta}(\pj)}(M,\opH^{\bullet}(\Uz(\uj),{\mathbb C})).
$$
Therefore, $u_{\zeta}(\pj)$ acts trivially on this new spectral
sequence, and so this is a spectral sequence of $\BU(\pj)$-modules.

By Theorem \ref{multSteinberg}, the abutment is nonzero only when $a
+ b = \ell(w)$ in which case it is the trivial module $\mathbb C$.
And by Theorem \ref{thirdsectioncohocalthm}, as $\BU(\hh)$-modules,
we know that ${\mathcal G}\big{(}\opH^a(u_{\zeta}(\uj),{\mathbb
C})\big{)} \cong S^\frac{a - \ell(w)}{2}(\uj^*)$. Also, the algebra
$Z_J$ is a polynomial algebra and so its cohomology as an algebra is
an exterior algebra. Moreover, by \cite[Cor. 2.9.6]{ABG}, as
$\BU(\mathfrak{b})$-modules, $\opH^{\bullet}(Z_J,{\mathbb C}) \cong
\Lambda^{\bullet}(\uj^*)$ (i.e., the ordinary exterior algebra on
$\uj^*$), where the action of $\BU(\mathfrak{b})$ on
$\Lambda^{\bullet}(\uj^*)$ is given by the coadjoint action. Since
$\opH^{\bullet}(Z_J,{\mathbb C})$ is a $\BU(\pj)$-module and the
coadjoint action on $\Lambda^{\bullet}(\uj^*)$ can be extended to
$\pj$, by applying $\ind_{\BU(\mathfrak{b})}^{\BU(\pj)}(-)$,
$\opH^{\bullet}(Z_J,{\mathbb C}) \cong
\Lambda^{\bullet}(\uj^*)$ as $\pj$-modules.

As in Theorem \ref{thirdsectioncohocalthm}, $E_2^{a,b} = 0$ for $a
< \ell(w)$ and $E_2^{\ell(w) + a,b} = 0$ for odd $a$. Consider the
term
$$
E_2^{\ell(w) + 2,0} = {\mathcal G}\big{(}\opH^{\ell(w) +
2}(u_{\zeta}(\uj),{\mathbb C})\big{)} \cong S^1(\uj^*) = \uj^*,
$$
where the latter identifications are as $\BU(\hh)$-modules. We will
see this also holds as $\BU(\pj)$-modules. Since
$E_{\infty}^{\ell(w),1} = 0$ and $E_{\infty}^{\ell(w) + 2,0} = 0$,
the differential $d_2: E_2^{\ell(w),1} \to E_2^{\ell(w) + 2,0}$ is
an isomorphism of $\BU(\pj)$-modules.

Since $E_2^{\ell(w),1} = \opH^1(Z_J,{\mathbb C}) \cong \uj^*$ as a
$\BU(\pj)$-module and $E_2^{\ell(w) + 2,0} = {\mathcal
G}\big{(}\opH^{\ell(w) + 2}(u_{\zeta}(\uj),{\mathbb C})\big{)} \cong
\uj^*$ as a $\BU(\hh)$-module, this latter identification must also
hold as a $\BU(\pj)$-module.

As already noted, the $\BU(\pj)$-structure of the terms
$$E_2^{\ell(w),b} = \opH^b(Z_J,{\mathbb C}) \cong \Lambda^{b}(\uj^*)$$
is determined by the coadjoint action. Moreover, we can now say that
the $\BU(\pj)$-structure of the terms
$$E_2^{\ell(w) + 2,b} = {\mathcal G}\big{(}\opH^{\ell(w) + 2}(u_{\zeta}(\uj),{\mathbb C})\big{)}\otimes
\opH^b(Z_J,{\mathbb C}) \cong \uj^*\otimes\Lambda^b(\uj^*)$$ is
determined by the coadjoint action. Therefore, since
$E_{\infty}^{\ell(w) + 4,0} = 0$, the $\BU(\pj)$-structure of
$$
E_2^{\ell(w) + 4,0} = {\mathcal G}\big{(}\opH^{\ell(w) +
4}(u_{\zeta}(\uj),{\mathbb C})\big{)} \cong S^2(\uj^*)
$$
(with isomorphism as a $\BU(\hh)$-module), which is determined by
the structure on $E_2^{\ell(w),b}$ and $E_2^{\ell(w) + 2,b}$, must
also be determined by the coadjoint action. Inductively, we conclude
that indeed the $\BU(\pj)$-action on
$${\mathcal G}\big{(}\opH^{a}(u_{\zeta}(\uj),{\mathbb C})\big{)} \cong S^\frac{a - \ell(w)}{2}(\uj^*)$$
is given by the coadjoint action.

For parts (b) and (c), a similar argument may be used. The key here
(as in Theorem \ref{thirdsectioncohocalthm}) is that the additional
classes in ${\mathcal G}\big{(}\opH^i(\uj,{\mathbb C})\big{)}$ lie
in degrees which have the same parity as $\ell(w)$. Furthermore, the
additional classes have distinct (nonzero) weights whose differences 
are neither sums of positive root nor sums of negative roots as noted in Remark \ref{Steinbergrem}.
\end{proof}


\section{Proof of main result, Theorem 1.2.3, I}\label{5thsectioncohocal}\renewcommand{\thetheorem}{\thesection.\arabic{theorem}}
\renewcommand{\theequation}{\thesection.\arabic{equation}}\setcounter{equation}{0}
\setcounter{theorem}{0}

In this section, we  present a proof of the isomorphisms in Theorem~\ref{MainThm} as $G$-modules.
Let us recall the following vanishing result due to
Broer \cite[Thm. 2.2]{Br2} and extended to a more general setting by
Sommers \cite[Prop. 4]{So1}. Let $J$ be an arbitrary subset of
$\Pi$. Then
\begin{equation}\label{5thcohoeqn1}
R^{i}\text{ind}_{P_{J}}^{G} S^{\bullet}({\mathfrak
u}_{J}^{*})\otimes \lambda=0
\end{equation}
for $i > 0$ and $\lambda$ is a $P$-regular weights \cite[Section 1.1]{KLT} inside the character group $X(P_{J})$.
This vanishing result is proved using the Grauert-Riemenschneider
theorem. We can now prove the first of our main theorems which
provides a precise description of the cohomology of quantum groups
at roots of unity in the case that $\zeta$ is a primitive $l$th root
of unity.

\begin{proof} Consider the cases listed in (b)(i).  According to Proposition \ref{secondcohoprop},
Theorem \ref{thirdsectioncohocalthm}(a), and Lemma
\ref{4thsectionlemma}, we have, as $\BU(\pj)$-modules,
$$
\opH^j(u_{\zeta}({\mathfrak p}_{J}),\ind_{U_{\zeta}({\mathfrak
b})}^{U_{\zeta}({\mathfrak p}_{J})}w\cdot 0) \cong
S^{\frac{j-\ell(w)}{2}}({\mathfrak u}^{*}_{J}).$$ By substituting
this identification into the spectral sequence given in Theorem
\ref{firstsectioncohocalthm}, we have
$$E_{2}^{i,j}=R^{i}\ind_{{P}_{J}}^{G}
S^{\frac{j-\ell(w)}{2}}({\mathfrak u}^{*}_{J}) \Rightarrow
\operatorname{H}^{i+j-\ell(w)}(u_{\zeta}({\mathfrak g}),{\mathbb
C}).$$ We can now apply (\ref{5thcohoeqn1}) to conclude that
$E_{2}^{i,j}=0$ for $i>0$, thus the spectral sequence collapses to
yield
$$
\opH^{s}(u_{\zeta}({\mathfrak g}),{\mathbb
C})=\text{ind}_{P_{J}}^{G}S^{\frac{s}{2}}({\mathfrak u}_{J}^{*}).
$$
This gives part (a) for those types listed in (b)(i). According to
Theorem \ref{isomorphismofinducedthm},
$\text{ind}_{P_{J}}^{G}S^{\bullet}({\mathfrak u}_{J}^{*}) \cong
{\mathbb C}[G\cdot {\mathfrak u}_{J}]\cong {\mathbb C}[{\mathcal
N}(\Phi_{0})]$, which gives part (b)(i).

Consider now the cases listed in part (b)(ii). In this case we have the following spectral sequence (using
Proposition \ref{secondcohoprop}, Theorem
\ref{thirdsectioncohocalthm}(b), and Lemma \ref{4thsectionlemma}),
$$E_{2}^{i,j}=\bigoplus_{t=0}^{l-1} R^{i}\ind_{{P}_{J}}^{G}
S^{\frac{j-\ell(w)-(m+1)t(l-t)}{2}}({\mathfrak u}^{*}_{J})\otimes
\varpi_{t(m-1)}\Rightarrow
\operatorname{H}^{i+j-\ell(w)}(u_{\zeta}({\mathfrak g}),{\mathbb
C}).$$ One can again apply (\ref{5thcohoeqn1}) because we are
tensoring the symmetric algebra by weights in $X(P_{J})_{+}$, thus
the spectral sequence collapses and yields part (ii) of (a) and (b).
One can argue similarly for those cases listed in part (b)(iii) using Theorem
\ref{thirdsectioncohocalthm}(c). In that case $\ell(w) = 8$.
\end{proof}


\section{Spectral sequences, IV} \renewcommand{\thetheorem}{\thesection.\arabic{theorem}}
\renewcommand{\theequation}{\thesection.\arabic{equation}}\setcounter{equation}{0}
\setcounter{theorem}{0}
To show the isomorphisms in Theorem \ref{MainThm} are, in fact, isomorphisms of algebras, we will need some further
observations on one of the spectral sequences introduced earlier. These observations will also be used in
Section 6.3.

In Section \ref{spectralseqandEuler}, a filtration was
introduced on ${\mathcal U}_{\zeta}(\uj)$ which can be
restricted to $u_{\zeta}(\uj)$. Since $u_{\zeta}(\uj)$ is finite
dimensional, the induced filtration on the
cobar complex computing the cohomology
$\opH^{\bullet}(u_{\zeta}(\uj),{\mathbb C})$ is finite.
As in the proof of part (b) of Proposition \ref{euler}, there is a spectral
sequence as follows.

\begin{lem}\label{Mayspectralseq} There is a spectral sequence
$$
E_1^{i,j}= \opH^{i+j}(\gr u_{\zeta}(\uj),{\mathbb
C})_{(i)}\Rightarrow \opH^{i+j}(u_{\zeta}(\uj),{\mathbb C})
$$
of graded algebras and $U_{\zeta}^0$-modules.
\end{lem}

\noindent
Since the filtration on $u_{\zeta}(\uj)$ is finite,
this spectral sequence has only finitely many columns, and
hence eventually stops.

Globally, we have
$E_1^{\bullet,\bullet} \cong \opH^{\bullet}(\gr u_{\zeta}(\uj),{\mathbb C})$.
By \cite[Prop. 2.3.1]{GK}, there exists a graded algebra
isomorphism
\begin{equation}\label{gradediso}\opH^{\bullet}(\gr u_{\zeta}({\mathfrak u}_{J}),{\mathbb C})\cong
S^{\bullet}({\mathfrak u}_{J}^{*})^{[1]}\otimes
\Lambda^{\bullet}_{\zeta,J}.
\end{equation}
This is also an isomorphism of
$U_{\zeta}^0$-modules with $u_{\zeta}^0$ acting trivially on the
symmetric algebra. Moreover, under the isomorphism (\ref{gradediso}),
\begin{equation}\label{gradediso2}
\opH^n(\gr u_{\zeta}(\uj),{\mathbb C}) \cong \bigoplus_{2a + b = n}
S^{a}({\mathfrak u}_{J}^{*})^{[1]}\otimes\Lambda^{b}_{\zeta,J}.
\end{equation}

By the isomorphism (\ref{gradediso}), $E_1^{\bullet,\bullet}$ is
finitely generated over a subalgebra which is isomorphic (as
algebras and $U_{\zeta}^0$-modules) to $S^{\bullet}(\uj^*)^{[1]}$.
For notational convenience, we abusively consider
$S^{\bullet}(\uj^*)^{[1]}$ as a subalgebra of
$E_1^{\bullet,\bullet}$. Under mild conditions on $l$, this
subalgebra consists of universal cycles.

\begin{prop}\label{universalcycles} Let $l$ satisfy Assumption~\ref{assumption2} and $J\subseteq \Pi$. In the spectral sequence of
Lemma \ref{Mayspectralseq}, $d_{r}(S^{\bullet}(\uj^*)^{[1]})=0$ for $r\geq 1$.
\end{prop}

\begin{proof} We first consider the case when $r=1$. From (\ref{gradediso2}),
the submodule $S^1(\uj^*)^{[1]}$ is identified with a submodule of
$\opH^2(\gr u_{\zeta}({\mathfrak u}_{J}),{\mathbb C})$.
As such, the image of $S^1({\mathfrak u}_{J}^*)^{[1]}$ under $d_1$ must lie
in
$$\opH^3(\gr u_{\zeta}({\mathfrak u}_{J}),{\mathbb C}) \cong
(S^1({\mathfrak
u}_{J}^*)^{[1]}\otimes\Lambda^1_{\zeta,J})\oplus\Lambda^3_{\zeta,J}.$$
A $U^0_{\zeta}$-homogeneous element $x_{\si}$ of $S^1({\mathfrak
u}_{J}^*)^{[1]}$ has weight $l\si$ for some $\si \in \Phi^+ \backslash
\Phi_J^+$. Hence, by weight considerations, if the image of
$x_{\si}$ is not zero, it cannot lie in $S^1({\mathfrak
u}_{J}^*)^{[1]}\otimes\Lambda^1_{\zeta,J}$. On the other hand, for
$x_{\si}$ to have nonzero image in $\Lambda^3_{\zeta,k}$, we must
have $l\si = \ga_1 + \ga_2 + \ga_3$ for three distinct (positive)
roots $\ga_i \in \Phi^+ \backslash\Phi_J^+$. Under the given conditions on
$l$, this is not possible. To see this, we argue by the type of
$\Phi$.

For any weight $\eta$ which lies in the positive root lattice, we can write
$\eta = \sum_{\be \in \Pi}n_{\eta,\be}\be$ for unique $n_{\eta,\be} \in {\mathbb Z}_{\geq 0}$.
Set $\ga := \ga_1 + \ga_2 + \ga_3$ for $\ga_i$ as above. For type $A_2$ there is only one element in $\Lambda^3_{\zeta,k}$ and it is not of the form $l\si$.
If the root system is not of  of type $A_2$ the index of the root lattice in the weight lattice is not divisible by $3$. Therefore,
in order to have $\ga = l\si$, $l$ must divide $n_{\ga,\be}$ for each
$\be \in \Pi$. 

In type $A_n$, for each $\ga_i$ and $\be \in \Pi$,
we have $n_{\ga_i,\be} \leq 1$. Hence, $n_{\ga,\be} \leq 3$, and the claim immediately follows for $l > 3$.
For $l = 3$, we could only have $3\si = \ga$ if $\ga_1 = \ga_2 = \ga_3$ which contradicts our assumption.

For types $B_n$, $C_n$, and $D_n$, $n_{\ga,\be} \leq 6$. Hence, the claim follows if $l \geq 7$.
Suppose $l = 5$ and $n_{\ga,\be} = 5$ for some $\be \in \Pi$. Without a loss of generality
we may assume that $n_{\ga_1,\be} = 2$, $n_{\ga_2,\be} = 2$, and $n_{\ga_3,\be} = 1$. Observe that there is
necessarily some $\be' \in \Pi$ such that $n_{\ga_1,\be'} = 1$, $n_{\ga_2,\be'} = 1$, and $n_{\ga_3,\be'} \in \{0, 1\}$.
Thus $n_{\ga,\be'} \in \{2,3\}$ and so cannot be divisible by $5$. In types $B_n$ and $C_n$ when $l = 3$,
the condition can be satisfied (e.g., in type $B_2$, $\al_1 + (\al_1 + \al_2) + (\al_1 + 2\al_2) =
3(\al_1 + \al_2)$).

On the other hand, in type $D_n$, the condition is still not possible when $l = 3$.
To see this, suppose on the contrary that $3\si = \ga_1 + \ga_2 + \ga_3 = \ga$. If $n_{\ga_i,\be} \leq 1$
for each $\ga_i$ and all $\be \in \Pi$, then we are done as in type $A_n$. Suppose now that
$n_{\ga_1,\be} = 2$ for some $\be \in \Pi$. If $\eta$ is a positive root with $n_{\eta,\be} = 2$,
then (in the standard Bourbaki ordering)
$$
\eta = \al_i + \cdots +\al_j + 2\al_{j + 1} + \cdots + 2\al_{n-2} + \al_{n-1} + \al_n
$$
for some $j \geq i$.
Hence, we have $n_{\ga_1,\al_{n-2}} = 2$, $n_{\ga_1,\al_{n-1}} = 1$ and $n_{\ga_1,\al_{n}} = 1$. In order to
have $n_{\ga,\al_{n-1}} = 3 = n_{\ga,\al_{n}}$, we must also have $n_{\ga_2,\al_{n-1}} = n_{\ga_2,\al_{n}} =
n_{\ga_3,\al_{n-1}} = n_{\ga_3,\al_{n}} = 1$.
To have $n_{\ga,\al_{n-2}}$ divisible by 3, there are two cases to consider: either
$n_{\ga_2,\al_{n-2}} = 1$ and $n_{\ga_3,\al_{n-2}} = 0$ or
$n_{\ga_2,\al_{n-2}} = 2 = n_{\ga_3,\al_{n-2}}$. However, there are no such roots $\ga_3$
satisfying $n_{\ga_3,\al_{n-2}} = 0$, $n_{\ga_3,\al_{n-1}} = 1$, and $n_{\ga_3,\al_{n}} = 1$.
So the first case is not possible.

Suppose the latter case holds. Then for each $i$, $n_{\ga_i,\al_{n-3}} \in \{1,2\}$. If these
numbers are not all 2 or all 1, then we are done. If each $n_{\ga_i,\al_{n-3}} = 2$, inductively computing
$n_{\ga_i,\be}$ for $\be = \al_m$ with $m < n-3$ either we are done or we come to the case that
$n_{\ga_i,\be} = 1$ for each $i$. Continuing from that case, since the $\ga_i$ are distinct,
there is some $\be' = \al_{m'}$ with $m' < m$ such that $n_{\ga_1,\be'} = 1$, $n_{\ga_2,\be'} \in \{0,1\}$, and
$n_{\ga_3,\be'} = 0$. Hence, $n_{\ga,\be'}$ is not divisible by 3, and we are done.

For the exceptional types, one can check ``by hand'', using for example MAGMA, that the root
condition $l\si = \ga_1 + \ga_2 + \ga_3$ cannot hold for $l > 3$. Hence, under our assumptions on $l$, we must have 
$d_1(S^1({\mathfrak u}_{J}^*)^{[1]}) = 0$. Since the differentials in the spectral sequence are derivations with respect to the cup product, it
follows that $d_1(S^{\bullet}({\mathfrak u}_{J}^*)^{[1]}) = 0$. 

Now we can recursively apply this argument to show that $d_r(S^{\bullet}({\mathfrak u}_{J}^*)^{[1]}) = 0$ for $r\geq 1$. 
First observe that $d_{r}(S^{1}({\mathfrak u}_{J}^*)^{[1]})$ is always a subquotient of 
$$\opH^3(\gr u_{\zeta}({\mathfrak u}_{J}),{\mathbb C}) \cong (S^1({\mathfrak
u}_{J}^*)^{[1]}\otimes\Lambda^1_{\zeta,J})\oplus\Lambda^3_{\zeta,J}.$$
This means that the aforementioned weight arguments can be applied. Secondly, 
the result after taking kernels modulo images of $S^{\bullet}({\mathfrak u}_{J})^{[1]}$ in $E_{r}$ is always generated in degree one, 
so we can use the fact that the differentials are derivations to conclude that $d_{r}(S^{\bullet}({\mathfrak u}_{J}^*)^{[1]})=0$. 
\end{proof}

We now prove that the cohomology ring contains a subalgebra isomorphic to $S^{\bullet}({\mathfrak u}_{J}^*)^{[1]}$ by using 
our prior calculations in Section~\ref{combinSteinbergsec}. 

\begin{prop} \label{univsubalgebra} Let $l$ satisfy Assumption~\ref{assumption2} and $J\subseteq \Pi$ be as in Section~\ref{Steinbergweightssec}. 
There exists a subring isomorphic to $S:=S^{\bullet}({\mathfrak u}_{J}^*)^{[1]}$ contained in $R:=\opH^{\bullet}(u_{\zeta}(\uj),{\mathbb C})$
such that $R$ is finitely generated over $S$. 
\end{prop} 

\begin{proof} According to Proposition~\ref{universalcycles}, $d_r(S)=0$ in the spectral sequence of
Lemma \ref{Mayspectralseq}. 
After taking kernels modulo images of $S$ of the differentials $d_r$  one obtains a quotient
$S^{\prime}$ of $S$.
 Note that $S^{\prime}$ is a subring of $R$ such that 
$R$ is finitely generated over $S^{\prime}$.  Set $M:=(\text{ind}_{U_{\zeta}({\mathfrak b})}^{U_{\zeta}({\mathfrak p}_{J})} w\cdot 0)^{*}$ 
as in Section~\ref{Steinbergweightssec}. Consider the two spectral sequences: 
$$
E_1^{i,j}= \opH^{i+j}(\gr u_{\zeta}(\uj),{\mathbb
C})_{(i)}\Rightarrow \opH^{i+j}(u_{\zeta}(\uj),{\mathbb C}),
$$
$$
\widetilde{E}_1^{i,j}= \text{Hom}_{u_{\zeta}({\mathfrak l}_{J})}(M,\opH^{i+j}(\gr u_{\zeta}(\uj),{\mathbb
C})_{(i)})\Rightarrow \text{Hom}_{u_{\zeta}({\mathfrak l}_{J})}(M,\opH^{i+j}(u_{\zeta}(\uj),{\mathbb C})).
$$
Observe that $S=S^{\bullet}({\mathfrak u}_{J}^*)^{[1]} \subseteq \text{Hom}_{u_{\zeta}({\mathfrak u}_{J})}({\mathbb C},   \opH^{\bullet}(\gr u_{\zeta}(\uj),{\mathbb C})  )$.  
In the case when $J\subseteq \Pi$ is as in Section~\ref{Steinbergweightssec}, the second spectral sequence collapses 
and we have  
$$N:= \text{Hom}_{u_{\zeta}({\mathfrak l}_{J})}(M, \opH^{\bullet}(\gr u_{\zeta}(\uj),{\mathbb C}))\cong 
\text{Hom}_{u_{\zeta}({\mathfrak l}_{J})}(M,\opH^{\bullet}(u_{\zeta}(\uj),{\mathbb C})).$$ 
which is a finitely generated $S$-module. 

A description of $N$ is given in Theorem~\ref{thirdsectioncohocalthm}, and the action of $S$ on $N$ is given by left multiplication on the components 
involving the symmetric powers of ${\mathfrak u}_{J}^{*}$ in $N$. Now $R$ is finitely generated over $S^{\prime}$ and 
$S^{\prime}$ is contained in $\text{Hom}_{u_{\zeta}({\mathfrak l}_{J})}({\mathbb C},R)$ so $S^{\prime}$ will 
act on $\text{Hom}_{u_{\zeta}({\mathfrak l}_{J})}(M,R)$. Moreover $N$ is a $S^{\prime}$-summand of 
$R$ because of the projectivity of $M$, thus $N$ is a finitely generated $S^{\prime}$-module. Therefore, the 
(Krull) dimension of $S^{\prime}$ is greater than or equal to the dimension of $S$. Since $S$ is a polynomial ring and $S^{\prime}$ its quotient, we can conclude that 
none of the differentials in the spectral sequence could have non-zero image in $S$, thus $S\cong S^{\prime}$. 
\end{proof} 


\section{Proof of the main result, Theorem 1.2.3, II}\label{cohoring}\renewcommand{\thetheorem}{\thesection.\arabic{theorem}}
\renewcommand{\theequation}{\thesection.\arabic{equation}}\setcounter{equation}{0}
\setcounter{theorem}{0}

We will now prove that the isomorphisms in Theorem~\ref{MainThm} are algebra isomorphisms for those cases listed in part (b)(i).
Let $J\subseteq \Pi$ and consider the spectral sequence
$$E_{2}^{i,j}=R^{i}\text{ind}_{P_{J}}^{G} \text{H}^{j}(u_{\zeta}({\mathfrak p}_{J}),{\mathbb C})
\Rightarrow \text{H}^{i+j}(u_{\zeta}({\mathfrak g}),{\mathbb C}).$$

The (vertical) edge homomorphism is a map of algebras $$\phi:\text{H}^{\bullet}(u_{\zeta}({\mathfrak g}),{\mathbb C})
\rightarrow E_{2}^{0,\bullet},$$ where
$E_{2}^{0,\bullet}=\text{ind}_{P_{J}}^{G} \text{H}^{\bullet}(u_{\zeta}({\mathfrak p}_{J}),{\mathbb C})$.
The map of algebras is induced by the restriction map $\text{H}^{\bullet}(u_{\zeta}({\mathfrak g}),{\mathbb C})
\rightarrow \text{H}^{\bullet}(u_{\zeta}({\mathfrak p}_{J}),{\mathbb C})$.

Next consider the Lyndon-Hochschild-Serre spectral sequence:
$$E_{2}^{i,j}=\text{H}^{i}(u_{\zeta}({\mathfrak l}_{J}),\text{H}^{j}(u_{\zeta}({\mathfrak u}_{J}),{\mathbb C}))
\Rightarrow \text{H}^{i+j}(u_{\zeta}({\mathfrak p}_{J}),{\mathbb C}).$$
We also have the vertical edge homomorphism $\delta:\text{H}^{\bullet}(u_{\zeta}({\mathfrak p}_{J}),{\mathbb C})
\rightarrow \text{Hom}_{u_{\zeta}({\mathfrak l}_{J})}({\mathbb C},\text{H}^{\bullet}(u_{\zeta}({\mathfrak u}_{J}),{\mathbb C}))
$. Applying the induction functor to this map and composing with $\phi$ yields an algebra homomorphism

$$\Psi^{\prime}:\text{H}^{\bullet}(u_{\zeta}({\mathfrak g}),{\mathbb C})\rightarrow
\text{ind}_{P_{J}}^{G} \text{Hom}_{u_{\zeta}({\mathfrak l}_{J})}({\mathbb C},\text{H}^{\bullet}
(u_{\zeta}({\mathfrak u}_{J}),{\mathbb C})).$$

Set $R=\text{H}^{\bullet}(u_{\zeta}({\mathfrak u}_J),{\mathbb C})$.  From Proposition~\ref{univsubalgebra}, there
exists a subalgebra of universal cycles
isomorphic to $S^{\bullet/2}({\mathfrak u}_{J}^{*})^{[1]}$ in $R$ such that $R$ is finitely generated over the
subalgebra. This subalgebra can be identified as coming from $E_{1}$-term in the spectral
sequence in Lemma~\ref{Mayspectralseq}. There is an ideal $I'$ such that $E_{1}^{\bullet,\bullet}/I'\cong
S^{\bullet/2}({\mathfrak u}_{J}^{*})^{[1]}$. Therefore, there exists an ideal $I$ (equivariant under $P_{J}$) such that
$$R/I\cong S^{\bullet/2}({\mathfrak u}_{J}^{*})^{[1]}.$$
Consequently, there is  an algebra homomorphism:

$$\Psi:\text{H}^{\bullet}(u_{\zeta}({\mathfrak g}),{\mathbb C})\rightarrow
\text{ind}_{P_{J}}^{G} S^{\bullet/2}({\mathfrak u}_{J}^{*}).$$

Next we observe that our constructions are compatible with the other spectral sequences used in
Sections ~\ref{firstsectioncohocal} and ~\ref{secondsectioncohocal}. In the process of our work, we proved that there is a $G$-module isomorphism

$$\sigma:\text{H}^{\bullet-l(w)}(u_{\zeta}({\mathfrak g}),{\mathbb C})\rightarrow \text{ind}_{P_{J}}^{G} S^{(\bullet-l(w))/2}
({\mathfrak u}_{J}^{*}).$$

Set $M=\text{H}^{\bullet-l(w)}(u_{\zeta}({\mathfrak g}),{\mathbb C})$. Viewing $M$ as a module over 
$\text{H}^{\bullet}(u_{\zeta}({\mathfrak g}),{\mathbb C})$
and
$\text{ind}_{P_{J}}^{G} S^{(\bullet-l(w))/2}
({\mathfrak u}_{J}^{*})$ as a  $\text{ind}_{P_{J}}^{G} S^{\bullet/2}
({\mathfrak u}_{J}^{*})$-module, this isomorphism is compatible with $\Psi$ in 
the sense that $\sigma(x.y)=\Psi(x)\sigma(y)$ where $x\in \text{H}^{\bullet}(u_{\zeta}({\mathfrak g}),{\mathbb C})$ and $y\in M$. Now suppose that $\Psi(x)=0$ 
for some $x\in \text{H}^{\bullet}(u_{\zeta}({\mathfrak g}),{\mathbb C})$. Then $\sigma(x.y)=0$ for all $y\in M$.
Set $y=\text{id}\in \text{H}^{0}(u_{\zeta}({\mathfrak g}),{\mathbb C})$, so $x=0$ because $\sigma$
is an isomorphism. This shows that $\Psi$ is injective. By comparing dimensions, $\Psi$ must be surjective, and
hence an isomorphism of algebras.

%% file: chapt6BNPP.tex
\chapter{Finite Generation }\label{finitegen}\renewcommand{\thesection}{\thechapter.\arabic{section}}

This section provides a
proof of Theorem \ref{MainThm2} in Section \ref{mainresults}.  As pointed out in Chapter 1, it can also be found in \cite{MPSW}, in a more general context. Our result depends
heavily on the explicit calculations of the cohomology that we achieved in Chapter 5. 


\section{A finite generation result}\renewcommand{\thetheorem}{\thesection.\arabic{theorem}}
\renewcommand{\theequation}{\thesection.\arabic{equation}}\setcounter{equation}{0}\setcounter{theorem}{0}
Let $G$ be the complex semisimple, simply connected algebraic group with root system $\Phi$.  Maintain
the notation of \S2.3.

Let $J\subseteq \Pi$, and let $D$ be a $P_J$-$S^\bullet({\mathfrak u}^*_J)$-module. This
means that $D$ is a rational $P_J$-module and a module for $S^\bullet({\mathfrak u}^*_J)$ such that, if $g\in P_J$,
$a\in S^\bullet({\mathfrak u}^*_J)$, and $x\in D$, then $g\cdot(ax)=
(g\cdot a)(g\cdot x)$, where we use the natural conjugation action of $P_J$ on $S^\bullet({\mathfrak u}_J^*)$.

Let $A:=\ind_{P_J}^GS^\bullet({\mathfrak u}_J^*)={\mathbb C}[G\times^{P_{J}}{\mathfrak u}_{J}]$ be the
algebra of regular functions on the cotangent bundle $G\times^{P_J}{\mathfrak u}_J$ as discussed in \S3.6. Let ${\mathcal O}$ denote the
corresponding nilpotent orbit such that $G \cdot {\mathfrak u}_J = \overline{\mathcal O}$ (which was denoted
${\mathcal C}_J$ in \S\ref{richardsonorbits}). If the moment
map $\mu:G\times^{P_{J}}{\mathfrak u}_{J} \to G\cdot {\mathfrak u}_J$ is
a resolution of singularities,  then Lemma \ref{even} gives that
$${\mathbb
C}[G\times^{P_{J}}{\mathfrak u}_{J}] \cong {\mathbb C}[{\mathcal
O}].$$ 
Of course, ${\mathbb
C}[\overline{\mathcal O}] \subseteq {\mathbb C}[{\mathcal O}]$ and,
by \cite[Prop. 8.3]{Jan4}, the algebra ${\mathbb C}[{\mathcal O}]$ is the
integral closure of ${\mathbb C}[\overline{\mathcal O}]$ in its
field of fractions. Since ${\mathbb C}[\overline{\mathcal O}] =
{\mathbb C}[G \cdot {\mathfrak u}_J]$ is a finitely generated $\mathbb C$-algebra,
${\mathbb C}[{\mathcal O}] \cong {\mathbb
C}[G\times^{P_{J}}{\mathfrak u}_{J}] $ is also a
finitely generated $\mathbb C$-algebra by \cite[Chapter V, Thm. 9]{ZS} (which, in fact, plays
an essential role in the proof of \cite[Prop. 8.3]{Jan3}).  

We now state a basic result on how finite generation is related to the induction
functor which will play an important role in the discussion that follows.

\begin{prop}\label{indfg} Let $J\subseteq \Pi$ and $D$ be a $P_J$-$S^\bullet({\mathfrak u}^*_J)$-module which is a
finitely generated $S^\bullet({\mathfrak u}_J^*)$-module.
If $A := \ind_{P_J}^{G}S^{\bullet}(\ul_J^*)$ is a finitely generated ${\mathbb C}$-algebra, then $R^{n}\ind_{P_J}^G D$ is a finitely generated $A$-module for $n\geq 0$.
\end{prop}

\begin{proof}
By assumption, $A := \ind_{P_J}^GS^{\bullet}(\mathfrak{u}_J^*)$ is a
Noetherian ${\mathbb C}$-algebra.  Let $G_A$ and $(P_J)_A$ denote the group schemes obtained by
extending scalars from ${\mathbb C}$ to $A$.  Furthermore, let $D_A := D\otimes_{\mathbb C}A$ denote the
corresponding $P_J$-module extended to a $(P_J)_A$-module.  Since $G/P_J$ is projective
over ${\mathbb C}$, $G_A/(P_J)_A \cong (G/P_J)_A$ is projective over $A$.  Therefore, by
\cite[Prop. I.5.12(c)]{Jan1}, $R^n\ind_{(P_J)_A}^{G_A}D_A$ is finitely generated over $A$.
To complete the proof, we show that $R^n\ind_{P_J}^{G}D \cong R^n\ind_{(P_J)_A}^{G_A}D_A$,
and we do so with the aid of sheaf cohomology.

Let $X=G\times^{P_J}{\mathfrak u}_J$ and $Y=G/P_J$. Let $Y_A:=Y\times_{\mathbb C} \text{Spec}\,A$. Because
$A=\Gamma(X,{\mathcal O}_X)$ (the space of global sections of the sheaf
${\mathcal O}_X$), there is a natural morphism $\sigma_2:X\to \text{Spec}\,A\cong G\cdot {\mathfrak u}_J$.
There is also the projection morphism $\sigma_1:X\to Y$ which can be viewed as the cotangent bundle
of $X$. Let $f=\sigma_1\times\sigma_2:X\to Y_A$ be the pull-back
morphism.

The finitely generated  $S^{\bullet}({\mathfrak u}_{J}^*)$-module $D$ defines a coherent ${\mathcal O}_X$-module $F=F_D$ on $X$. To see this, just observe
that if $V$ is an open subvariety of $G$ of the form $U\times P_J$ (which exists by the Bruhat decomposition), then
$X$ contains an open affine subvariety $V':=V\times^{P_J}{\mathfrak u}_J\cong U\times {\mathfrak u}_J$. Then $\Gamma(V',F)\cong
{\mathbb C}[U]\otimes D$, which is certainly a finitely generated $\Gamma(V',{\mathcal O}_X)={\mathbb C}[U]\otimes S^\bullet({\mathfrak u}^*_J)$-module.

The morphism $f$ is an affine morphism, and it is easily verified that the direct image sheaf $f_*F$ is a coherent ${\mathcal O}_{Y_A}$-module.

Thus, $H^n(X,F)\cong H^n(Y_A,f_*F)\cong H^n(Y,\sigma_{1*}F)$ (see \cite[Ex. 4.1, p. 222]{Har}). We also use here the
fact that the projection $Y_A\to Y$ is affine.
By construction, $H^n(Y,\sigma_{1*}F) \cong R^n\ind_{P_J}^{G}D$ and
$H^n(Y_A,f_*F) \cong R^n\ind_{(P_J)_A}^{G_A}D_A$ (cf. \cite[Prop. I.5.12(a)]{Jan1}).
Hence $R^n\ind_{P_J}^{G}D \cong R^n\ind_{(P_J)_A}^{G_A}D_A$ as needed.

\end{proof}


\section{Proof of part (a) of Theorem 1.2.4}\renewcommand{\thetheorem}{\thesection.\arabic{theorem}}
\renewcommand{\theequation}{\thesection.\arabic{equation}}\setcounter{equation}{0}
\setcounter{theorem}{0} 

Throughout the remainder of Section 6, we will be working under Assumption~\ref{assumption2}. Thus, 
Assumption~\ref{assumption} is in force, and, if $\Phi$ has type $B_n$ or $C_n$, then $l>3$.
 First, we deal with the 
cases in which $l\nmid n+1$ when $\Phi$ is of type $A_{n}$, $l\neq 9$ when $\Phi$ is of type $E_{6}$, and $l \neq 7, 9$ when $\Phi$ is of
type $E_{8}$.  In these cases, Theorem \ref{MainThm} states that the cohomology ring 
$\operatorname{H}^{\bullet}(u_{\zeta}({\mathfrak g}),{\mathbb C})$ is the coordinate algebra
of the affine variety ${\mathcal N}(\Phi_0)$. It is therefore a finitely generated ${\mathbb C}$-algebra. 

Next, if $l = 7, 9$ when $\Phi$ is of type $E_8$, then Theorem \ref{MainThm} states that
$$\opH^{\bullet}(u_{\zeta}({\mathfrak g}), {\mathbb C})\cong
\ind_{P_J}^G S^{\bullet}({\mathfrak u}_J^*) \cong {\mathbb C}[G\times^{P_J}{\mathfrak u}_J].$$ 
In addition, Theorem \ref{isomorphismofinducedthm} says that $\mu:G\times^{P_J}{\mathfrak u}_J\to
G\cdot{\mathfrak u}_J$ is a desingularization  of $G\cdot{\mathfrak u}_J$. The discussion above
Proposition \ref{indfg} in the previous subsection thus implies that $\opH^\bullet(u_{\zeta}({\mathfrak g}),
{\mathbb C})$ is a finitely generated $\mathbb C$-algebra.

In order to handle the other cases (i.e., $l\mid n+1$ when $\Phi$ is of type $A_{n}$, or $l=9$ when $\Phi$ is of type
$E_{6}$),  we need to invoke a more general argument. Set $A=\text{ind}_{P_{J}}^{G}S^{\bullet}({\mathfrak u}^{*}_{J})$. 
In each case, Theorem~\ref{MainThm} implies that $\opH^\bullet(u_\zeta({\mathfrak g}),{\mathbb C})$ has the form
$\ind_{P_J}^G D$, where $D$ is a $P_J$-$S^\bullet({\mathfrak u}_J^*)$-module which is a finitely generated 
$S^\bullet({\mathfrak u}^*_J)$-module. 
Consequently, by Proposition~\ref{indfg}, $\ind_{P_J}^GD$ is a finitely generated $A$-module.
The action of $A$ on $\opH^\bullet(u_\zeta({\mathfrak g}),{\mathbb C})$ is by way of the spectral sequence: 
$$E_{2}^{i,j}=R^{i}\text{ind}_{P_{J}}^{G} \text{H}^{j}(u_{\zeta}({\mathfrak p}_{J}),{\mathbb C})
\Rightarrow \text{H}^{i+j}(u_{\zeta}({\mathfrak g}),{\mathbb C})$$
where $A$ is identified as a subring of universal cycles in the bottom of the filtration for the cohomology (cf. Section~\ref{cohoring}). 
Hence,  $\opH^\bullet(u_\zeta({\mathfrak g}),{\mathbb C})$ is finitely generated over $A$, and thus a finitely generated 
${\mathbb C}$-algebra by the Hilbert Basis Theorem.

\section{Proof of part (b) of Theorem 1.2.4}\label{proofofpartb}\renewcommand{\thetheorem}{\thesection.\arabic{theorem}}
\renewcommand{\theequation}{\thesection.\arabic{equation}} \setcounter{equation}{0}
\setcounter{theorem}{0}

We now prove part (b) of Theorem~\ref{MainThm2}. Let $M$ be a finite dimensional $u_{\zeta}({\mathfrak
g})$-module. Without a loss of generality, we may assume that $M$ is
an irreducible $u_{\zeta}({\mathfrak g})$-module because of the
following proposition which is easily proved by using induction on the
composition length of the module and the long exact sequence in
cohomology.

\begin{prop}\label{finitegenprop} Let $R:=\opH^{\bullet}
(u_{\zeta}({\mathfrak g}),{\mathbb C})$ and $M$ be a
finite dimensional $u_{\zeta}({\mathfrak g})$-module. Suppose that
$\opH^{\bullet}(u_{\zeta}({\mathfrak g}),S)$ is finitely generated
over $R$ for all irreducible $u_{\zeta}({\mathfrak g})$-modules $S$.
Then $\opH^{\bullet}(u_{\zeta}({\mathfrak g}),M)$ is finitely
generated over $R$.
\end{prop}

Let $S$ be an irreducible $u_{\zeta}({\mathfrak g})$-module. Since
$S$ lifts to a $U_{\zeta}({\mathfrak g})$-module, there exists a
spectral sequence of $R=\text{ind}_{P_{J}}^{G}
S^{\bullet}({\mathfrak u}_{J}^{*})$-modules (obtained in a manner
analogous to that of Theorem \ref{firstsectioncohocalthm}):
$$E_{2}^{i,j}=R^{i}\text{ind}_{P_{J}}^{G} \opH^{j}
(u_{\zeta}({\mathfrak p}_{J}),\text{ind}_{U_{\zeta}({\mathfrak
b})}^{U_{\zeta}({\mathfrak p}_{J})}w\cdot 0 \otimes S)
\Rightarrow\opH^{i+j-\ell(w)}(u_{\zeta}({\mathfrak g}),S).$$
Using this spectral sequence, it suffices to show that
$D:=\opH^{\bullet}(u_{\zeta}({\mathfrak p}_{J}),\text{ind}_{U_{\zeta}
({\mathfrak b})}^{U_{\zeta}({\mathfrak p}_{J})}w\cdot 0 \otimes S)$
is finitely generated over $S^{\bullet}({\mathfrak
u}_{J}^{*})^{[1]}$ because then $E_{2}^{i,\bullet}$ is finitely
generated over $R$ for each $i$ by Proposition~\ref{indfg}. Moreover,
this spectral sequence stops (i.e.,
$E_{r}=E_{\infty}$ for $r$ sufficiently large) because the higher
right derived functors $R^{i}\text{ind}_{P_{J}}^{G} -$ vanish when
$i> \dim G/P_{J}$. Thus $E_{\infty}$ is finitely generated over $R$,
and $\opH^{\bullet}(u_{\zeta}({\mathfrak g}),S)$ is finitely
generated over $R$.

To show that $D$ is finitely generated over $S^{\bullet}({\mathfrak
u}_{J}^{*})^{[1]}$, observe that one can use the LHS spectral
sequence (Lemma~\ref{ABG}) to show that
$$D \cong
\text{Hom}_{u_{\zeta}({\mathfrak
l}_{J})}((\text{ind}_{U_{\zeta}({\mathfrak
b})}^{U_{\zeta}({\mathfrak p}_{J})}w\cdot 0)^{*},
\opH^{\bullet}(u_{\zeta}({\mathfrak u}_{J}),S)).$$
By the same principles as used in the proposition above, it suffices to
show finite generation when $S$ is an irreducible $u_{\zeta}({\mathfrak
p}_{J})$-module. Note that irreducible $u_{\zeta}({\mathfrak
p}_{J})$-modules are obtained by inflating irreducible
$u_{\zeta}({\mathfrak l}_{J})$-modules, so $u_{\zeta}({\mathfrak
u}_{J})$ acts trivially on $S$. Now we have
$$D \cong \text{Hom}_{u_{\zeta}({\mathfrak l}_{J})}((\text{ind}_{U_{\zeta}
({\mathfrak b})}^{U_{\zeta}({\mathfrak p}_{J})}w\cdot 0)^{*}\otimes S^{*},
\opH^{\bullet}(u_{\zeta}({\mathfrak u}_{J}),{\mathbb C})).$$
But, $(\text{ind}_{U_{\zeta}({\mathfrak b})}^{U_{\zeta}({\mathfrak
p}_{J})}w\cdot 0)^{*}\otimes S^{*}$ is a projective
$u_{\zeta}({\mathfrak l}_{J})$-module. Thus we need to show that
$D_{P}:=\text{Hom}_{u_{\zeta}({\mathfrak l}_{J})}(P,
\opH^{\bullet}(u_{\zeta}({\mathfrak u}_{J}),{\mathbb C}))$ is
finitely generated over $S^{\bullet}({\mathfrak u}_{J}^{*})^{[1]}$
where $P$ is an arbitrary projective indecomposable
$u_{\zeta}({\mathfrak l}_{J})$-module. Observe that
$\opH^{\bullet}(u_{\zeta}({\mathfrak u}_{J}),{\mathbb C})\cong
\oplus_{P} D_{P}$ where the sum is taken over all projective
indecomposable $u_{\zeta}({\mathfrak l}_{J})$-modules. The claim now follows from 
Proposition \ref{univsubalgebra}
which says that $\opH^{\bullet}(u_{\zeta}({\mathfrak
u}_{J}),{\mathbb C})$ is a finitely generated
$S^{\bullet}({\mathfrak u}_{J}^{*})^{[1]}$-module.




%% file: chapt7BNPP.tex
 \chapter{Comparison with Positive Characteristic}
\renewcommand{\thesection}{\thechapter.\arabic{section}}
In this chapter, we consider the extent to which the methods used in computing $\opH^\bullet(u_\zeta,{\mathbb C})$
can be adapted to the calculation of the cohomology algebra of the restricted enveloping algebra of a reductive
algebraic group over a field of positive characteristic.

\section{The setting}\setcounter{equation}{0}
\setcounter{theorem}{0}
 Let $F$ be an algebraically closed field of positive characteristic $p$. In this chapter (and contrary to previous notation)
 $G$ denotes a connected, simple, simply connected algebraic group defined over the prime field
 ${\mathbb F}_p$.\footnote{The results in this chapter can be extended to general reductive groups. We leave this issue to the interested reader.}  Fix a maximal torus $T$, and let $\Phi$ be the root system of $T$ acting on the
 Lie algebra ${\mathfrak g}_F$ of $G$. As usual, fix a set $\Pi$ of simple roots. The standard notational conventions for the
 complex Lie algebra in earlier chapters apply equally in the present case.  However, the Lie algebra ${\mathfrak g}_F$ has an extra
 structure provided by a ``restriction
map" $x\mapsto x^{[p]}$; we let $u=u({\mathfrak g}_F)$ be the corresponding restricted enveloping algebra. Thus, $u$ is a finite
 dimensional (cocommutative) Hopf algebra.

 Let $\text{Fr}:G\to G$ be the Frobenius morphism and let $G_1$ be its scheme-theoretic kernel. It is well-known that the
 category of rational $G_1$-modules is equivalent to the category of restricted ${\mathfrak g}_F$-modules (i.e., to the category
 of $u$-modules). For this and historical reasons, we will state the results below in terms of the infinitesimal group scheme $G_1$,
 though the reader can replace each $G_1$ by $u$ if desired.

 For a rational $G$-module $M$, let $M^{(1)}$ denote the rational $G$-module obtained by
 making $g\in G$ act on $M$ through $\text{Fr}(g)$. In particular, $G_1$ (or ${\mathfrak g}_F$) acts trivially on
 $M^{(1)}$. Conversely, if $N$ is a rational $G$-module on which $G_1$ (or ${\mathfrak g}_F$) acts trivially, there exists a rational $G$-module $M$,
 uniquely defined up to isomorphism, such that  $M^{(1)}\cong N$. In this case, we can simply write $M=N^{(-1)}$.

 As usual, let ${\mathcal N}={\mathcal N}({\mathfrak g}_F)\subset{\mathfrak g}_F$ be the nullcone of ${\mathfrak g}_F$, the closed subvariety
 consisting of nilpotent elements (equivalently the closure of the Richardson class of regular nilpotent elements). The restricted nullcone ${\mathcal N}_1={\mathcal N}_1({\mathfrak g}_F)$ is the closed subvariety of
 ${\mathcal N}$ consisting of those $x\in{\mathcal N}$ satisfying $x^{[p]}=0$. It is an irreducible variety which can be explicitly
 described as the closure of a specific Richardson orbit; see \cite{CLNP} and \cite{UGA2}. In particular, ${\mathcal N}_1={\mathcal N}$ if and only
 if $p\geq h$, where $h$ is the Coxeter number of $G$ (i.~e., the maximum of the Coxeter numbers of the various simple components of
 the derived subgroup $G'$ of $G$).

  When $p>h$, it is known that $\opH^\bullet(G_1,F)=\opH^{2\bullet}(G_1,F)$ is isomorphic as a rational $G$-algebra to
 $F[{\mathcal N}]^{(1)}$. This result was first proved by Friedlander-Parshall \cite{FP2} for $p\geq 3(h-1)$, and then the
 bound was improved to $p>h$ by Andersen and Jantzen \cite{AJ} by different methods. Also, \cite{AJ} provided some ad hoc
 calculations of $G_1$-cohomology in the cases when $p\leq h$. We will demonstrate how these calculations fit into our general
 framework.
  As noted in the introduction (with references), it has been shown that $\text{Spec}\,\opH^{2\bullet}(G_1,F)$
 is homeomorphic (as a topological space) to ${\mathcal N}_1$.

\section{Assumptions}\setcounter{equation}{0}
\setcounter{theorem}{0}  For $J\subset \Pi$, we formulate two assumptions.
The first assumption involves Grauert-Riemenschneider vanishing which is known to hold in the cases when 
$\Phi_{J}=\varnothing$  or $\Phi_{J}$ is of type $A_{1}$ (in positive characteristic). 
In positive characteristic it has also been verified in a few other special cases which 
will be discussed in Section 7.4 below.

\vskip .15cm \noindent (A1) $R^{i}\text{ind}_{P_{J}}^{G}
S^{\bullet}({\mathfrak u_J}^{*})=0$ for $i>0$.

\vskip .15cm This condition can be reformulated in the following  fashion:

\noindent \vskip .15cm \noindent (A1)${}^{\prime}$
$R^{i}\text{ind}_{P_{J}}^{G} S^{\bullet}({\mathfrak u_J}^{*})\otimes
\lambda=0$ for $i>0$ and $\lambda\in X(P_{J})_{+}$. (Here $P_J$ is the parabolic subgroup of $G$ with Lie
algebra ${\mathfrak p}_J$ and $X(P_J)_+=X(P_J)\cap X_+$, the set of dominant weights on $T$ which extend to define
characters on $P_J$.)

 \vskip .15cm
\noindent The second assumption on $J$ is a condition on the
normality of the closure of the Richardson orbit defined by $J$:

\vskip .15cm \noindent (A2) The Richardson orbit closure $G\cdot {\mathfrak u}_{J}$
is normal. \vskip .15cm

In the calculation of $\opH^\bullet(u_\zeta,{\mathbb C})$, assumption (A1) held generally and assumption (A2) held for the relevant subsets $J\subset\Pi$.
However, in the positive characteristic case of this
chapter, much less is known about the validity
of (A1) and (A2). The situation for (A2) is as follows:

\vskip.15cm
(1) In type $A$, all orbits are Richardson orbits, and all orbit closures are normal (Donkin \cite{D}).

\vskip.15cm
(2) Assume that $J=\{\alpha\}$, $\alpha\in\Pi$, consists of a single root. Then $G\cdot{\mathfrak u}_J$ is the closure of the so-called subregular class ${\mathcal O}_{\text{subreg}}$ in $\mathcal N$. (It is independent of the choice of simple root $\alpha$ \cite[Thm. 5.7]{Hum1}.) Then $G\cdot{\mathfrak u}_J=
\overline{{\mathcal O}_{\text{subreg}}}$ is normal
(Kumar, Lauritzen, and Thomsen \cite{KLT}).

\vskip.15cm
(3) Generalizing (2) in some sense, assume that $J\subseteq\Pi$ consists of pairwise orthogonal short simple roots,
 the corresponding Richardson orbit closure $G\cdot {\mathfrak u}_J$ is normal (Thomsen \cite[Prop. 7]{Th}).

 \vskip.15cm
 (4)  Finally,  Christophersen \cite{C} has
recently determined the nilpotent orbits for type $E_6$ with $p\geq 5$ which have
normal orbit closure. See below for more specific information.


\section{Consequences}\setcounter{equation}{0}
\setcounter{theorem}{0}

 Using assumptions (A1) and (A2), we can determine when the cohomology algebra $\opH^\bullet(G_1,F)$  identifies with
the coordinate algebra $F[{\mathcal N}_1]$ of the restricted nullcone. In the present case, the
proof is simpler than that used in the calculation of $\opH^\bullet(u_\zeta,{\mathbb C})$ due, in part, to the
fact that the exterior algebra $\Lambda^{\bullet}(\uj^*)$ has a
natural structure as a $P_J$-module, whereas the quantized exterior
algebra $\Lambda^{\bullet}_{\zeta,J}$ does not (apparently) admit a natural
structure as a $U_{\zeta}(\pj)$-module.

\begin{theorem}\label{FrobthmA} Let $G$ be a connected, simple, simply connected algebraic
group over an algebraically closed field $F$ of positive characteristic $p$ as above. Assume that
$p\geq 3$, and that $p$ is a very good prime for $G$. Let $w\in W$ such that
$w(\Phi_{0}^+)=\Phi_{J}^+$. Then the following statements hold.
\begin{itemize}
\item[(a)] If $J\subseteq \Pi$ satisfies (A1), then
\begin{itemize}
\item[(i)] $\opH^{2\bullet}(G_{1},F)^{(-1)}\cong \operatorname{ind}_{P_{J}}^{G}
S^{\bullet}({\mathfrak u}_{J}^{*})$;
\item[(ii)] $\opH^{2\bullet+1}(G_{1},F)=0$.
\end{itemize}
\item[(b)] If $J\subseteq \Pi$ satisfies (A1) and (A2), then
\begin{itemize}
\item[(i)] $\opH^{2\bullet}(G_{1},F)^{(-1)}\cong F[{\mathcal N}(\Phi_{0})]$;
\item[(ii)] $\opH^{2\bullet+1}(G_{1},F)=0$.
\end{itemize}
\end{itemize}
Furthermore, these identifications are isomorphisms of rational $G$-algebras.
\end{theorem}

\begin{proof} Let $w\in W$ satisfy $w(\Phi_{0}^+)=\Phi_{w\cdot 0}^+=\Phi_{J}^+$, where
$J\subseteq \Pi$. An argument similar to that given in
Chapter \ref{cohomologysection} can be applied with the functors
$$
\Hom_{G_1}(F,\ind_{P_J}^{G}(-)) \text{ and }
 \ind_{P_J/(P_J)_1}^{G/G_1}(\Hom_{(P_J)_1}(F,-))
$$
(from $P_J$-mod to $G/G_1$-mod) to construct a first quadrant
spectral sequence
\begin{equation}\label{Frobeq1}
E_{2}^{i,j}=[R^{i}\ind_{{P}_{J}}^{G}
(\operatorname{H}^{j}((P_{J})_{1},
 \operatorname{ind}_{B}^{P_J} w\cdot 0)^{(-1)})]^{(1)}
\Rightarrow \operatorname{H}^{i+j-\ell(w)}(G_{1},F).
\end{equation}
Next, applying the Lyndon-Hochshild-Serre spectral sequence and the fact that
$\operatorname{ind}_{B}^{P_J} w\cdot 0$ is an injective
$(L_{J})_{1}$-module, we conclude that
\begin{equation}\label{Frobeq2}
\text{Hom}_{(L_{J})_{1}}(F,\operatorname{ind}_{B}^{P_J} w\cdot
0\otimes \opH^{j}((U_{J})_{1},F)) =\operatorname{H}^{j}((P_{J})_{1},
\operatorname{ind}_{B}^{P_J} w\cdot 0).
\end{equation}
In order to compute $\opH^{\bullet}((U_{J})_{1},F)$, there is a
first quadrant spectral sequence of $P_J$-modules \cite[Prop. 1.1]{FP1}
(cf. also \cite[I 9.20]{Jan1}, which can be re-indexed):
\begin{equation}\label{Frobeq3}
E_{2}^{2i,j}=S^{i}({\mathfrak u}_{J}^{*})^{(1)}\otimes
\opH^{j}({\mathfrak u}_{J},F) \Rightarrow
\opH^{2i+j}((U_{J})_{1},F).
\end{equation}
Here $\opH^j(\uj,F)$ is the ordinary cohomology of the Lie algebra
$\uj$ or equivalently the cohomology $\opH^j(\mathbb{U}(\uj),F)$ of
the universal enveloping algebra of $\uj$.
Since $\operatorname{ind}_{B}^{P_J} w\cdot 0$ is an injective
$(L_{J})_{1}$-module, we can compose the spectral sequence in
(\ref{Frobeq3}) with
$\text{Hom}_{(L_{J})_{1}}(F,\operatorname{ind}_{B}^{P_J} w\cdot
0\otimes -)$ and use (\ref{Frobeq2}) to construct a spectral
sequence:
\begin{equation}\label{Frobeq4}
E_{2}^{2i,j}=S^{i}({\mathfrak u}_{J}^{*})^{(1)}\otimes
\text{Hom}_{(L_{J})_{1}}(F,\operatorname{ind}_{B}^{P_J} w\cdot
0\otimes \opH^{j}(\uj,F)) \Rightarrow
\opH^{2i+j}((P_{J})_{1},\text{ind}_{B}^{P_J} w\cdot 0).
\end{equation}

Now using the same analysis as in Chapter
\ref{combinSteinbergsec} with the Steinberg module, we can conclude that
(analogous to Theorem \ref{multSteinberg})
$$
\text{Hom}_{(L_{J})_{1}}(F,\operatorname{ind}_{B}^{P_J} w\cdot
0\otimes \opH^{j}({\mathfrak u}_{J},F)) \cong
\begin{cases}
F &\text{ if } j = \ell(w)\\
0 &\text{ otherwise.}
\end{cases}
$$
Therefore, the spectral sequence (\ref{Frobeq4}) collapses to a single horizontal row and
yields:
$$\operatorname{H}^{i}((P_{J})_{1},\operatorname{ind}_{B}^{P} w\cdot 0)
\cong S^{\frac{i-\ell(w)}{2}}({\mathfrak u}_{J}^{*})^{(1)}.$$ By
using this isomorphism and assumption (A1), the spectral sequence
(\ref{Frobeq1}) collapses to a single horizontal row,
and we obtain part (a) of the theorem. For
part (b), we simply use (A2) and Theorem
\ref{isomorphismofinducedthm}.
\end{proof}

Consider the case not covered in the preceding
theorem which happens only when $\Phi$ is of type $A_n$. This
result encompasses the computation \cite[\S 6.A]{AJ} when $p=n+1$
(corresponding to $m = 0$).

\begin{theorem}\label{FrobthmB} Let $G=\operatorname{SL}_{n+1}(F)$, where $F$ is an algebraically
closed field of positive characteristic $p$. Assume that
$p\geq 3$ and $p\mid n+1$ with $n+1 = p(m+1)$. Let $w\in W$ be defined as in 
(\ref{wepi}) by taking $l = p$.  Then $w(\Phi_{0}^+)=\Phi_{J}^+$. If $J\subseteq \Pi$ satisfies
(A1)${}^{\prime}$, then
\begin{itemize}
\item[(a)] $\opH^{2\bullet}(G_{1},F)^{(-1)}\cong \bigoplus_{t=0}^{p-1}
\operatorname{ind}_{P_{J}}^{G}
S^{\frac{2\bullet-(m+1)t(p-t)}{2}}({\mathfrak u}_{J}^{*})\otimes
\varpi_{t(m+1)}$;
\item[(b)] $\opH^{2\bullet+1}(G_{1},F)=0$.
\end{itemize}
Furthermore, these are all isomorphisms of rational $G$-algebras.
\end{theorem}

\begin{proof} The proof proceeds as for Theorem \ref{FrobthmA}.  More precisely,
the discussion through (\ref{Frobeq4}) holds here as well.  In this case, however,
we apply a result analogous to part (b) of  Theorem \ref{multSteinberg} and conclude that
$$
\text{Hom}_{(L_{J})_{1}}(F,\operatorname{ind}_{B}^{P_J} w\cdot
0\otimes \opH^{j}({\mathfrak u}_{J},F)) \cong
\begin{cases}
F &\text{ if } j = \ell(w)\\
p\varpi_{t(m+1)}\oplus p\varpi_{(p-t)(m+1)} &\text{ if }  j = \ell(w) + (m+1)t(p-t)\\
& \qquad \text{ for } 1 \leq t \leq (p - 1)/2\\
0 &\text{ otherwise.}
\end{cases}
$$
Consider now the spectral sequence (\ref{Frobeq4}).  This consists of
$(p + 1)/2$ non-zero rows in degrees $\ell(w) + (m + 1)t(p-t)$ for
$0 \leq t \leq (p - 1)/2$.  Since $t(p - t)$ is even, these non-zero rows
all appear in degrees having the same parity.  Since the non-zero terms
appear only in even degree columns, and the differential $d_r$ has bidegree $(r,1-r)$,
all differentials are zero.  Hence we conclude that
$$\operatorname{H}^{i}((P_{J})_{1},\operatorname{ind}_{B}^{P} w\cdot 0) \cong
\bigoplus_{t=0}^{p}S^{\frac{i-\ell(w) - (m+1)t(p-t)}{2}}({\mathfrak u}_{J}^{*})^{(1)}\otimes p\varpi_{t(m+1)}.$$
By using this isomorphism and assumption (A1)$'$, the spectral sequence (\ref{Frobeq1})
collapses, and the theorem follows.
\end{proof}


\section{Special cases}\renewcommand{\thetheorem}{\thesection.\arabic{theorem}}
\renewcommand{\theequation}{\thesection.\arabic{equation}} \setcounter{equation}{0}
\setcounter{theorem}{0}
 When $J$ consists of a single simple root the
assumptions (A1) and (A2) were verified in \cite[Prop. 7,
proof of Thm. 2, Lemma 14]{Th}, and (A2) has also been verified
in \cite[Thm. 6]{KLT}. In this case
${\mathcal N}(\Phi_{0})$ is the closure of the subregular orbit
${\mathcal O}_{\operatorname{subreg}}$ which occurs precisely when
$p=h-1$. Consequently, we have the following corollary.

\begin{cor}\label{FrobcorA} Let $G$ be a simple algebraic group over an algebraically closed field
$F$ of positive characteristic $p$ as above. Assume that
$p=h-1$. Then the following statements hold.
\begin{itemize}
\item[(a)] $\opH^{2\bullet}(G_{1},F)^{(-1)}\cong F[\overline{\mathcal O}_{\operatorname{subreg}}]$;
\item[(b)] $\opH^{2\bullet+1}(G_{1},F)=0$.
\end{itemize}
These identifications are isomorphisms of rational $G$-algebras.
\end{cor}

Using \cite{Th}, it is possible to compute the cohomology
algebra $\opH^\bullet(G_1,F)$ for more general examples when $J$ consists of pairwise
orthogonal simple short roots. For such $J$, conditions (A1) and (A2) hold.  Combining
this with the previous corollary gives the following.

\begin{cor}\label{FrobcorB} Let $G$ be a simple algebraic group over an algebraically closed field
$F$ of positive characteristic $p$ as above. Assume the following constraints on $p$:
\begin{itemize}
\item[(i)] $p > h/2$ if $\Phi$ is of type $A_n$, $C_n$, or $D_n$,
\item[(ii)] $p \geq h/2$ if $\Phi$ is of type $B_n$,
\item[(iii)] $p \geq 7$ if $\Phi$ is of type $F_4$ or $E_6$,
\item[(iv)] $p \geq 11$ if $\Phi$ is of type $E_7$,
\item[(v)] $p \geq 17$ if $\Phi$ is of type $E_8$.
\end{itemize}
Then the following statements hold.
\begin{itemize}
\item[(a)] $\opH^{2\bullet}(G_{1},F)^{(-1)}\cong F[\mathcal{N}(\Phi_0)]$;
\item[(b)] $\opH^{2\bullet+1}(G_{1},F)=0$.
\end{itemize}
These identifications are isomorphisms of rational $G$-algebras.
\end{cor}

Recent results of Christophersen \cite[Thm. 1, Example 3.15]{C} verify (A1) and (A2) for
the group $E_{6}$
when $p \geq 5$ for those subsets $J \subseteq \Pi$ listed in Appendix
\ref{tables1}, which gives the following result.

\begin{cor}\label{FrobcorC} Let $G$ be a simple algebraic group over an algebraically closed field
$F$ having positive characteristic $p\geq 5$ as above. Assume that $G$ has root system of type $E_6$. Then the
following statements hold.
\begin{itemize}
\item[(a)] $\opH^{2\bullet}(G_{1},F)^{(-1)}\cong F[{\mathcal N}(\Phi_{0})]$;
\item[(b)] $\opH^{2\bullet+1}(G_{1},F)=0$.
\end{itemize}
Here the subset $\Phi_0=\Phi_{0,p}$ is explicitly described in Appendix A.1. These identifications 
are isomorphisms of rational $G$-algebras.
\end{cor}

With the assistance of undergraduate students at the University of Wisconsin-Stout,
particularly J. Mankovecky, the original computer program written by
Christophersen \cite[Appendix A]{C} for root systems of type $E_6$ has been extended
to arbitrary types; see \cite{BMMR}.  With the aid of this program and the
algorithm in \cite[Example 3.15]{C},  condition (A1)
can be verified in Type $E_7$ when $p \geq 7$ for those subsets
$J \subseteq \Pi$ listed in Appendix \ref{tables1}.

%% file: chapt8BNPP.tex
\chapter{Support Varieties over $u_{\zeta}$ for  the Modules $\nabla_{\zeta}(\lambda)$ and $\Delta_{\zeta}(\lambda)$}

In this chapter we return to the small quantum group $u_{\zeta}$ in characteristic
$0$. We apply our results on the cohomology of the small quantum group
to obtain information on the support varieties of  the modules $\nabla_{\zeta}(\lambda)$
and $\Delta_{\zeta}(\lambda)$. We will first compute the support varieties of the 
(quantum) induced modules $\nabla_{\zeta}(\lambda)$. Later, using  this calculation, 
we will show that ${\mathcal V}_{\mathfrak g}(\Delta_{\zeta}(\lambda))=
{\mathcal V}_{\mathfrak g}(\nabla_{\zeta}(\lambda))$ for all $\lambda\in X_{+}$. 


\section{Quantum support varieties} We will assume that $l$ satisfies
Assumption~\ref{assumption2}. In what follows, we let
$$R:=\opH^{2\bullet}(u_{\zeta}({\mathfrak g}),{\mathbb C})_{\text{red}}
=\opH^{2\bullet}(u_\zeta({\mathfrak g}),
{\mathbb C})/{\text{\rm rad}}(\opH^{2\bullet}(u_{\zeta}({\mathfrak g}),{\mathbb C})),$$
the quotient of the (even) cohomology algebra by its (Jacobson) radical. 
We have proven that $R$ is a commutative finitely generated ${\mathbb C}$-algebra.
Moreover, if $M$ and $N$ are finite dimensional $u_{\zeta}({\mathfrak
g})$-modules, then $\text{Ext}^{\bullet}_{u_{\zeta}({\mathfrak g})}(M,N)$ is a
finitely-generated $R$-module. Let $J_{M,N}$ be the annihilator of the
action of $R$ on $\text{Ext}^{\bullet}_{u_{\zeta}({\mathfrak g})}(M,N)$, and
set ${\mathcal V}_{\mathfrak g}(M,N)$ equal to the maximum
ideal spectrum of $R/J_{M,N}$. The variety ${\mathcal V}_{\mathfrak
g}(M,N)$ is called a {\it relative support variety} of $M$. This variety is a closed,
conical subvariety of the variety ${\mathcal V}_{\mathfrak g}:={\mathcal V}_{\mathfrak g}({\mathbb C},{\mathbb C})$. 
The {\it support variety} of $M$ is ${\mathcal V}_{\mathfrak g}(M):={\mathcal V}_{\mathfrak g}(M,M)$. 
We note that if $M$ is a $U_{\zeta}({\mathfrak g})$-module then ${\mathcal V}_{\mathfrak
g}(M)$ is stable under the adjoint action of $G$.

In the generic case (cf.  Theorem~\ref{MainThm}(b)(i)) our computation of the cohomology algebra shows that 
${\mathcal V}_{\mathfrak g}$ identifies with ${\mathcal N}(\Phi_{0})$
where ${\mathcal N}(\Phi_0)$ is the subvariety of $\mathcal N$ defined in (\ref{definePhi0}).
We will prove that this also holds in the non-generic case, but we will need to 
use more sophisticated techniques to verify this. 


\section{Lower bounds on the dimensions of support varieties}

Let $\lambda\in X$ and let $\Phi_{\lambda}=\{\alpha\in \Phi:\
\langle\lambda+\rho,\alpha^{\vee}\rangle \in l{\mathbb Z}\}$.
Using the notation of Section 2.10, let $\nabla_{\zeta}(\la)=\opH^0_{\zeta}(\la)$ denote the
$U_{\zeta}({\mathfrak g})$-module induced from the one dimensional
$U_{\zeta}({\mathfrak b})$-module ${\mathbb C}_{\la}$ determined by
the character $\la$.  In what follows, we will consider $\nabla_{\zeta}(\la)$ to be a
$u_{\zeta}({\mathfrak g})$-module by restriction.

The first observation to make is that one can find a lower bound on
$\dim {\mathcal V}_{\mathfrak g}(\nabla_{\zeta}(\lambda))$ which is not
dependent on whether $l$ is good or bad.

\begin{prop}\label{supportlowerbound} Let $\lambda\in X_+$. Then
$$\dim {\mathcal V}_{\mathfrak g}(\nabla_{\zeta}(\lambda))\geq |\Phi|-|\Phi_{\lambda}|.$$
\end{prop}

\begin{proof} One can use the proof given in \cite[\S2, Corollary 2.5]{UGA3} by
replacing $G_{1}$ with $u_{\zeta}({\mathfrak g})$ (resp. $B_{1}$ by
$u_{\zeta}({\mathfrak b})$). However, we should remark that
${\mathcal V}_{\mathfrak b}(\nabla_{\zeta}(\lambda)) \subseteq
{\mathcal V}_{{\mathfrak g}}(\nabla_{\zeta}(\lambda))\cap {\mathfrak u}$.
Since ${\mathcal V}_{\mathfrak g}(\nabla_{\zeta}(\lambda))$ is a $G$-variety
(i.e., a union of $G$-orbit closures), one can use a result of Spaltenstein (cf.
\cite[Proposition 6.7]{Hum1}) to conclude that $\dim {\mathcal V}_{\mathfrak b}(\nabla_{\zeta}(\lambda))
\leq \frac{1}{2} \dim {\mathcal V}_{\mathfrak g}(\nabla_{\zeta}(\lambda))$. For the quantum
case, one should now replace the last line in \cite[Corollary 2.5]{UGA3} by
$$\dim {\mathcal V}_{\mathfrak g}(\nabla_{\zeta}(\lambda))\geq
2\dim {\mathcal V}_{\mathfrak b}(\nabla_{\zeta}(\lambda))\geq
2(|\Phi^{+}|-|\Phi_{\lambda}^{+}|)=|\Phi|-|\Phi_{\lambda}|.$$
\end{proof}


\section{Support varieties of $\nabla_{\zeta}(\lambda)$: general results}

We can now present a result which allows one to compute the supports of
induced/Weyl modules $\nabla_\zeta(\lambda),\Delta_\zeta(\lambda)$ provided that it is possible to $W$-conjugate the
stabilizer set $\Phi_\lambda$ into a subroot system $\Phi_J$ which is generated by a set $J$ of simple roots. The methods used will be
those provided in \cite[Sections 5 and 6]{NPV}. We outline some of the details of
these arguments. For $\lambda \in X$, let $Z_{\zeta}(\lambda)=
\text{ind}_{u_{\zeta}({\mathfrak b})}^{u_{\zeta}({\mathfrak g})}
\lambda$. These are finite dimensional modules having dimension $l^{|\Phi^{+}|}$ and are quantum
analogs of the ``baby Verma modules''. Moreover, for $\lambda\in X_{+}$, let $L_{\zeta}(\lambda)=
\text{soc}_{U_{\zeta}({\mathfrak g})}
\nabla_{\zeta}(\lambda)$ be the finite-dimensional irreducible $U_{\zeta}({\mathfrak g})$-module of
highest weight $\lambda$.

\begin{prop}\label{supportupperbound} Let $\lambda\in X_{+}$, $w\in W$. Then
\begin{itemize}
\item[(i)] ${\mathcal V}_{\mathfrak g}(L_{\zeta}(\lambda))\subseteq
G\cdot {\mathcal V}_{\mathfrak g}(Z_{\zeta}(w\cdot \lambda))$;
\item[(ii)] ${\mathcal V}_{\mathfrak g}(\nabla_{\zeta}(\lambda))\subseteq
G\cdot {\mathcal V}_{\mathfrak g}(Z_{\zeta}(w\cdot \lambda))$.
\end{itemize}
\end{prop}

\begin{proof}  For (i), we apply the arguments given in \cite[Section 5]{NPV}.
Using the proofs one can show that
$${\mathcal V}_{\mathfrak g}(H_{\zeta}^{l(w)}(w\cdot \lambda))
\subseteq G\cdot {\mathcal V}_{\mathfrak g}(Z_{\zeta}(w\cdot \lambda))$$
for all $w\in W$. Note that these proofs make use of the relative support varieties
${\mathcal V}_{\mathfrak g}(M,N)$. With the aforementioned inclusion of support
varieties, (ii) follows by using induction on the ordering of dominant weight
(cf. \cite[(5.6.1) Theorem]{NPV}).
\end{proof}

\begin{prop}\label{supportverma} Let $\mu\in X$ such that
$\langle \mu+\rho,\alpha^{\vee}\rangle\in l{\mathbb Z}$ for all $\alpha\in J$. Then
$${\mathcal V}_{\mathfrak g}(Z_{\zeta}(\mu))\subseteq {\mathcal V}_{{\mathfrak u}_{J}}.$$
\end{prop}

\begin{proof} Let ${\mathfrak p}_{J}$ be the parabolic subalgebra associated to
$J$ where ${\mathfrak p}_{J}\cong {\mathfrak l}_{J}\oplus {\mathfrak u}_{J}$. Here
${\mathfrak l}_{J}$ is the Levi subalgebra and ${\mathfrak u}_{J}$ is the unipotent
radical of ${\mathfrak p}_{J}$. Set $Z_{\zeta}^{J}(\mu)=\text{ind}_{u_{\zeta}({\mathfrak b}_{J})}
^{u_{\zeta}({\mathfrak l}_{J})} \mu$ where ${\mathfrak b}_{J}$ is a Borel subalgebra of ${\mathfrak l}_{J}$.
We can make $Z_{\zeta}^{J}(\mu)$ into a $u_{\zeta}({\mathfrak p}_{J})$-module by letting
$u_{\zeta}({\mathfrak u}_{J})$ act trivially. Then
$$Z_{\zeta}(\mu)\cong \text{ind}_{u_{\zeta}({\mathfrak p}_{J})}^{u_{\zeta}({\mathfrak g})}
Z_{\zeta}^{J}(\mu).$$
By the standard Frobenius reciprocity argument (cf. \cite[Prop. (2.3.1)]{NPV}),
${\mathcal V}_{\mathfrak g}(Z_{\zeta}(\mu))\subseteq {\mathcal V}_{{\mathfrak p}_{J}}(Z_{\zeta}^J(\mu))$.
For any $u_\zeta({\mathfrak g})$-module $M$, we have the following commutative diagram
\begin{equation}
\CD
\text{H}^{2\bullet}(u_{\zeta}({\mathfrak p}_{J}),{\mathbb C}) @>\text{res}>>
\text{H}^{2\bullet}(u_{\zeta}({\mathfrak u}_{J}),{\mathbb C})\\
@V\gamma VV @VV\delta V\\
\Ext^{\bullet}_{u_{\zeta}({\mathfrak p}_{J})}(M,M) @>>\beta>  \Ext^{\bullet}_{u_{\zeta}({\mathfrak u}_{J})}(M,M)
\endCD
\end{equation}
Here $\gamma=-\otimes M$, $\delta=-\otimes M$, and the bottom horizonal restriction map is labeled $\beta$.  
Now set $M:=Z_{\zeta}^{J}(\mu)$. The action of $u_{\zeta}({\mathfrak u}_{J})$ on $Z_{\zeta}^{J}(\mu)$ is trivial, thus
$\Ext^\bullet_{u_{\zeta}({\mathfrak u}_{J})}(M,M)\cong
\Ext^\bullet_{u_{\zeta}({\mathfrak u}_{J})}({\mathbb C},{\mathbb C})\otimes M^*\otimes M$.
This shows that $\delta$ is an injection.

Under the hypothesis of the proposition (i.e., $\langle \mu+\rho,\alpha^{\vee} \rangle
\in l{\mathbb Z}$ for all $\alpha\in J$),
we can conclude that $Z_{\zeta}^{J}(\mu)$ is projective as a $u_{\zeta}({\mathfrak l}_{J})$-module. By
applying the Lyndon-Hochschild-Serre spectral sequence for $u_{\zeta}({\mathfrak u}_{J})$ normal in
$u_{\zeta}({\mathfrak p}_{J})$, we have
$$\Ext^{\bullet}_{u_{\zeta}({\mathfrak p}_{J})}(M,M)\cong
\Ext^{\bullet}_{u_{\zeta}({\mathfrak u}_{J})}(M,M)^{u_{\zeta}({\mathfrak l}_{J})}.$$
This also shows that $\beta$ is injective.  Therefore, $\text{ker}(\text{res})=
\text{ker}(\gamma)$.  By definition ${\mathcal V}_{{\mathfrak p}_{J}}(Z^J_{\zeta}(\mu))$
is the variety associated to the ideal $\text{ker}(\gamma)$, and the image of the map
${\mathcal V}_{{\mathfrak u}_{J}}\to {\mathcal V}_{{\mathfrak p}_{J}}$ is given by the ideal
$\text{ker}(\text{res})$. It follows that
${\mathcal V}_{{\mathfrak p}_{J}}(Z^J_{\zeta}(\mu))= {\mathcal V}_{{\mathfrak u}_{J}}$.
\end{proof}

The following two results (generalizing \cite[(6.2.1) Thm.]{NPV}) provide a description of the
support varieties of the modules
$\nabla_{\zeta}(\lambda)$ for $\lambda\in X_{+}$ in terms of
closures of Richardson orbits. This is Theorem \ref{MainThm3}.
Observe that there are restrictions on $l$ even in the case when $l>h$.

\begin{theorem}\label{supportthm} Let $\lambda\in X_+$. Suppose there exists $w\in W$
such that $w(\Phi_{\lambda})=\Phi_{J}$, then
$${\mathcal V}_{\mathfrak g}(\nabla_{\zeta}(\lambda))=G\cdot{\mathfrak u}_{J}.$$
\end{theorem}

\begin{proof} Let $w\in W$ such that $\Phi_{w\cdot \lambda}=w(\Phi_{\lambda})=\Phi_{J}$, and
${\mathfrak u}_{J}\subseteq {\mathcal N}(\Phi_{0})$. Moreover,
$$\text{H}^{2\bullet}(u_{\zeta}({\mathfrak u}_{J}),{\mathbb C})_{\text{red}}\cong
S^{\bullet}({\mathfrak u}_{J}^{*})^{[1]}$$
by using the proofs in Section 5.6 and 5.7. Therefore, ${\mathcal V}_{{\mathfrak u}_{J}}\cong {\mathfrak u}_{J}$.

According to Proposition ~\ref{supportupperbound}(ii) and Proposition ~\ref{supportverma}:
\begin{equation}
{\mathcal V}_{\mathfrak g}(\nabla_{\zeta}(\lambda))\subseteq
G\cdot {\mathcal V}_{\mathfrak g}(Z_{\zeta}(w\cdot \lambda))\subseteq G\cdot{\mathfrak u}_{J}.
\end{equation}
Now $G\cdot {\mathfrak u}_{J}$ is an irreducible variety of dimension equal to $|\Phi|-|\Phi_{J}|=|\Phi|-|\Phi_{\lambda}|$.
The statement of the theorem now follows by Proposition ~\ref{supportlowerbound}.
\end{proof}

\begin{cor}\label{supportcor} Let ${\mathfrak g}$ be a complex simple Lie algebra and $l$ be an
odd positive integer which is good for $\Phi$. If $\lambda\in
X_{+}$, then there exists $w\in W$ such that
$w(\Phi_{\lambda})=\Phi_{J}$ for some $J\subseteq \Pi$, and
${\mathcal V}_{\mathfrak g}(\nabla_{\zeta}(\lambda))=G\cdot
{\mathfrak u}_{J}$.
\end{cor}

\begin{proof} This follows
immediately from Lemma \ref{subsystemlemma} and Theorem
\ref{supportthm}.
\end{proof}


\section{Support varieties of $\Delta_{\zeta}(\lambda)$ when $l$ is good} 

We will now show that ${\mathcal V}_{\mathfrak g}(\Delta_{\zeta}(\lambda))=
{\mathcal V}_{\mathfrak g}(\nabla_{\zeta}(\lambda))$ for all $\lambda\in X_{+}$.  
Recall that $L_{\zeta}(\lambda)$ is the irreducible module which appears as the 
socle of $\nabla_{\zeta}(\lambda)$. 

\begin{theorem} \label{maintheoremdelta} Let ${\mathfrak g}$ be a complex simple Lie algebra and $l$ be an
odd positive integer which is good for $\Phi$.
Let $\lambda \in X_+$, and choose $J\subseteq \Pi$ such that $w(\Phi_{\lambda})= \Phi_{J}$ for some $w \in W$. Then
\begin{itemize}
\item[(a)]  ${\mathcal V}_{\mathfrak g}(L_\zeta(\lambda))\subseteq G\cdot{\mathfrak u}_J$; 
\item[(b)]  ${\mathcal V}_{\mathfrak g}(\Delta_{\zeta}(\lambda))= G\cdot {\mathfrak u}_{J}$. 
\end{itemize} 
\end{theorem}

\begin{proof} We can prove part (a) by using induction on the ordering of weights. If $\mu$ is linked under the dot action of the affine Weyl group to $\lambda$ and is minimal among all dominant weights 
$\leq \lambda$, then $L_\zeta(\mu)=\nabla_\zeta(\mu)$. Moreover, $\Phi_\mu$ is $W$-conjugate to $\Phi_\lambda$. According to Corollary~\ref{supportcor} and the fact that 
$L_{\zeta}(\mu)=\nabla_{\zeta}(\mu)$ we have 
$${\mathcal V}_{\mathfrak g}(L_\zeta(\mu)) = {\mathcal V}_{\mathfrak g}(\nabla_\zeta(\mu)) \subseteq G \cdot {\mathfrak u}_J.$$ 
Now assume that for all $\mu$ linked to $\lambda$ with $\mu$ dominant and $\mu<\lambda$ we have ${\mathcal V}_{\mathfrak g}(L_\zeta(\mu))\subseteq G \cdot {\mathfrak u}_J$. 
There exists a short exact sequence: 
$$0\rightarrow L_{\zeta}(\lambda)\rightarrow \nabla_{\zeta}(\lambda) \rightarrow N \rightarrow 0$$ 
where $N$ has composition factors $<\lambda$ and linked to $\lambda$. Therefore, ${\mathcal V}_{\mathfrak g}(N)\subseteq G\cdot {\mathfrak u}_{J}$ (by the induction hypothesis). 
If $0\to M_1\to M_2\to M_3\to 0$ is a short exact sequence of finite-dimensional $u_\zeta({\mathfrak g})$-modules, it is an elementary fact that ${\mathcal V}_{\mathfrak g}(M_{\sigma(1)})\subseteq 
{\mathcal V}_{\mathfrak g}(M_{\sigma(2)})\cup {\mathcal V}_{\mathfrak g}(M_{\sigma(3)})$ for any $\sigma\in{\mathfrak S}_3$ \cite[Lemma 5.2]{PW}. Therefore, 
$${\mathcal V}_{\mathfrak g}(L_{\zeta}(\lambda))\subseteq 
{\mathcal V}_{\mathfrak g}(\nabla_{\zeta}(\lambda))\cup {\mathcal V}_{\mathfrak g}(N)\subseteq G\cdot {\mathfrak u}_{J}\cup G\cdot {\mathfrak u}_{J}\subseteq G\cdot {\mathfrak u}_{J}.$$ 

For part (b) first observe that all composition factors of $\Delta_{\zeta}(\lambda)$ have highest weights 
which are strongly linked to $\lambda$ under the affine Weyl group. By part (a) these composition factors have their support varieties contained in 
$G\cdot {\mathfrak u}_{J}$, thus ${\mathcal V}_{\mathfrak g}(\Delta_{\zeta}(\lambda))\subseteq G\cdot {\mathfrak u}_{J}$. 
Since the characters of $\nabla_{\zeta}(\lambda)$ and $\Delta_{\zeta}(\lambda)$ are equal, it follows that their graded dimensions 
are equal and one can apply the argument in Proposition~\ref{supportlowerbound} to show that 
$$\dim {\mathcal V}_{\mathfrak g}(\Delta_{\zeta}(\lambda))\geq |\Phi|-|\Phi_{\lambda}|. $$
Now by the same reasoning provided in Theorem~\ref{supportthm}, we have ${\mathcal V}_{\mathfrak g}(\Delta_{\zeta}(\lambda))= G\cdot {\mathfrak u}_{J}$. 

\end{proof}


\section{A question of naturality of support varieties}\label{}

One fact that has not been established in the quantum setting is the
``naturality'' of support varieties. In particular for $M$ a $u_{\zeta}({\mathfrak g})$-module
we would like to have the following statement:
\begin{equation} \label{nat}
{\mathcal V}_{\mathfrak b}(M)={\mathcal V}_{\mathfrak g}(M)\cap {\mathcal V}_{\mathfrak b}.
\end{equation}
In fact, one inclusion is true:
\begin{equation}
{\mathcal V}_{\mathfrak b}(M)\subseteq {\mathcal V}_{\mathfrak g}(M)\cap {\mathcal V}_{\mathfrak b}.
\end{equation}
The equality of support varieties (\ref{nat}) has been established in the restricted
Lie algebra setting because of the description of support varieities via rank varieties.
We anticipate that this will eventually be established for quantum groups.

Throughout the remainder of the paper, we will make the following assumption on naturality
of support varieties for ${\mathcal V}_{\mathfrak g}(\nabla(\lambda))$.

\begin{assump}\label{regassumption} Given $\lambda\in X_{+}$,
${\mathcal V}_{\mathfrak b}(\nabla(\lambda))={\mathcal V}_{\mathfrak g}(\nabla(\lambda))\cap
{\mathcal V}_{\mathfrak b}$.
\end{assump}

Clearly, the validity of (\ref{nat}) implies the validity of  Assumption \ref{regassumption}. One key result which is needed for
analyzing the bad $l$ situation is a quantum analogue of an important result of Jantzen \cite[Prop. (2.4)]{Jan5}; see also \cite[Cor. (4.5.1)]{NPV}.  In its statement, given $J\subseteq\Pi$, we set $x_J:=\sum_{\alpha\in J}x_\alpha$, where $0\not= x_\alpha\in
{\mathfrak g}_\alpha$. We can take $x_\varnothing=0$, by definition. 

\begin{prop} \label{T:upperbound} Assume that Assumption~\ref{regassumption} is valid.
Let $J\subseteq \Pi$ and $\la\in X_+$. If $w(\Phi_{\lambda})\cap \Phi_{J}\ne\varnothing$ for all
$w\in W$, then $x_{J}\notin {\mathcal V}_{\mathfrak g}(\nabla_{\zeta}(\lambda))$.
\end{prop}

In order to prove this proposition, one can apply the machinery set
up in \cite[Section 4]{NPV} involving relative support varieties and ``baby'' Verma module.
However, in the proof of \cite[Cor. (4.5.1)]{NPV},
the statement of Assumption \ref{regassumption} is a key ingredient in the proof.


\section{The Constrictor Method I} In order to compute support varieties
${\mathcal V}_{\mathfrak g}(\nabla_{\zeta}(\lambda))$ when $l$ is a bad and
there does not exist $w \in W$ such that $w(\Phi_{\lambda}) = \Phi_J$ for some
$J \subseteq \Pi$, we need to
utilize the ``constrictor method'' as presented in \cite{UGA3} to obtain upper bounds on the
support varieties. Let ${\mathcal O}$ be an orbit in ${\mathcal N}(\Phi_{0})$.
The {\it constrictors} of ${\mathcal O}$ are the orbits contained in ${\mathcal N}(\Phi_{0})
-\overline{\mathcal O}$ which are minimal with respect to the closure ordering of orbits
in ${\mathcal N}$. The following theorem holds in the quantum case as long as
Assumption~\ref{regassumption} is valid.

\begin{theorem} \label{T:constrictor} Let $\lambda\in X_{+}$. Assume that Assumption~\ref{regassumption} holds.
Let ${\mathcal O}$ be an orbit in ${\mathcal N}(\Phi_{0})$ and
$\{{\mathcal O}_{1},{\mathcal O}_{2},\dots,{\mathcal O}_{s}\}$ be the set of
constrictors of ${\mathcal O}$. Assume that the following conditions are
satisfied:
\begin{itemize}
\item[(i)] $|\Phi|-|\Phi_\lambda|\geq \dim {\mathcal O}$;
\item[(ii)] for $i=1,2,\dots,s$, ${\mathcal O}_{i}=G\cdot x_{J_{i}}$ for some
$J_{i}\subseteq \Pi$;
\item[(iii)] for $i=1,2,\dots,s$,
$w(\Phi_{\lambda})\cap \Phi_{J_{i}}\ne\varnothing$ for all $w\in W$.
\end{itemize}
Then ${\mathcal V}_{\mathfrak g}(\nabla_{\zeta}(\lambda))=\overline{\mathcal O}$.
\end{theorem}


\section{The Constrictor Method II} In our computations, a more
powerful method is necessary to force the containment of ${\mathcal V}_{\mathfrak g}(\nabla_{\zeta}(\lambda))$
inside the closure of an orbit. In the Lie algebra case it was enough to use Theorem~\ref{T:constrictor}
because the constrictors had orbit representatives given by sums of simple root vectors.
As we will see later this is not true in the quantum case.

\begin{prop} \label{P:preconstrictor} Assume that
${\mathcal O}$ is a $G$-orbit in ${\mathcal N}$. Let $\lambda\in
X$ and assume that
\begin{itemize}
\item[(i)] there exists $w\in W$ with $w(\Phi_{\lambda})=
\Phi_{J}$ with $J\subseteq \Pi$,
\item[(ii)] $|\Phi_{\lambda}|> \operatorname{codim}_{\mathcal N}\ {\mathcal O}$.
\end{itemize}
Then
${\mathcal O}\cap {\mathcal V}_{\mathfrak g}(Z_{\zeta}(\lambda))=\varnothing$.
\end{prop}

\begin{proof} Suppose that ${\mathcal O}\cap {\mathcal V}_{\mathfrak g}(Z_{\zeta}(\lambda))
\neq \varnothing$. Now
$$G\cdot {\mathcal V}_{\mathfrak g}(Z_{\zeta}(\lambda))\subseteq
\bigcup {\mathcal V}_{\mathfrak g}(\nabla_{\zeta}(w^{\prime}\cdot\lambda))$$
where the union is taken over $w^{\prime}\in W_{l}$ with $w^{\prime}\cdot\lambda\in X_{+}$
(cf. \cite[ Thm. (4.6.1)]{NPV}). Therefore,
${\mathcal O}\subseteq {\mathcal V}_{\mathfrak g}(\nabla_{\zeta}(y\cdot\lambda))$
for some $y\in W_{l}$ because ${\mathcal V}_{\mathfrak g}(\nabla_{\zeta}(y\cdot\lambda))$
is $G$-stable.

By assumption there exists $w\in W$ with $w(\Phi_{\lambda})=
\Phi_{J}$ with $J\subseteq \Pi$. Moreover,
$${\mathcal V}_{\mathfrak g}(Z_{\zeta}(w\cdot\lambda))\subseteq {\mathfrak u}_{J}.$$
Applying the methods in \cite[Thm. (5.6.1)]{NPV}, we have
\begin{equation*}
{\mathcal V}_{\mathfrak g}(\nabla_{\zeta}(y\cdot\lambda))\subseteq
G\cdot {\mathcal V}_{\mathfrak g}(Z_{\zeta}(w\cdot\lambda))\subseteq
G\cdot {\mathfrak u}_{J}.
\end{equation*}
Therefore,
$$\text{codim}_{\mathcal N}\ G\cdot {\mathfrak u}_{J}\leq \text{codim}_{\mathcal N}\ {\mathcal O} <|\Phi_{\lambda}|.$$
On the other hand,
$$\text{codim}_{\mathcal N}\ G\cdot {\mathfrak u}_{J}=|\Phi|-2\dim {\mathfrak u}_{J}=|\Phi_{J}|=|\Phi_{\lambda}|.$$
\end{proof}

We can now prove a generalization of Theorem~\ref{T:constrictor}.

\begin{theorem} \label{T:constrictor2} Let ${\mathcal O}$ be an orbit in ${\mathcal N}(\Phi_{0})$
with ${\mathcal O}=G\cdot {\mathcal O}_{I}$ where ${\mathcal O}_{I}$ is
an $L_{I}$-orbit for some $I\subset \Pi$. Let $\lambda\in X_{+}$ and
assume that Assumption~\ref{regassumption} holds. Suppose that
\begin{itemize}
\item[(i)] there exists $y\in W$ with $y(\Phi_{\lambda})\cap \Phi_{I}=
\Phi_{J}$ with $J\subseteq I$,
\item[(ii)] $|w(\Phi_{\lambda})\cap \Phi_{I}|> \operatorname{codim}_{{\mathcal N}({\mathfrak l}_{I})}
\ {\mathcal O}_{I}$ for all $w\in W$.
\end{itemize}
Then
${\mathcal O}\cap {\mathcal V}_{\mathfrak g}(\nabla_{\zeta}(\lambda))=\varnothing$.
\end{theorem}

\begin{proof} Suppose that $|w(\Phi_{\lambda})\cap \Phi_{I}|>
\operatorname{codim}_{{\mathcal N}({\mathfrak l}_{I})}\ {\mathcal O}_{I}$ for all $w\in W$.
It follows from Proposition~\ref{P:preconstrictor} that
${\mathcal O}_{I}\cap {\mathcal V}_{{\mathfrak l}_{I}}(Z^{I}_{\zeta}(w\cdot\lambda))=\varnothing$
for all $w\in W$. Therefore,
${\mathcal O}_{I}\cap {\mathcal V}_{{\mathfrak l}_{I}}(\oplus_{w\in W} Z^{I}_{\zeta}(w\cdot\lambda))=\varnothing$,
and
${\mathcal O}_{I}\cap {\mathcal V}_{{\mathfrak g}}(\oplus_{w\in W} Z_{\zeta}(w\cdot\lambda))=\varnothing$
(cf. \cite[Prop. (4.2.1)]{NPV}). Consequently,
${\mathcal O}\cap {\mathcal V}_{{\mathfrak g}}(\oplus_{w\in W} Z_{\zeta}(w\cdot\lambda))=\varnothing$,
and by Proposition~\ref{supportupperbound}
${\mathcal O}\cap {\mathcal V}_{{\mathfrak g}}(\nabla_{\zeta}(\lambda))=\varnothing$.
\end{proof}


\section{Support varieties of $\nabla_{\zeta}(\lambda)$ when $l$ is bad}
Assume for the remainder of the paper that Assumption~\ref{regassumption} holds for
all $\lambda\in X_{+}$. Under the assumptions stated at the beginning of the
paper, the remaining cases that we are forced to consider are the following (arranged in the order of difficulty):
\vskip .15cm
$\bullet$ $E_{6}$ when $3\mid l$;
\vskip .15cm
$\bullet$ $F_4$ when $3\mid l$;
\vskip .15cm
$\bullet$ $E_{7}$ when $3\mid l$;
\vskip .15cm
$\bullet$ $E_{8}$ when $3\mid l$;
\vskip .15cm
$\bullet$ $E_{8}$ when $5\mid l$.
\vskip .25cm
Let $\Phi_{\lambda}^{\vee}=\{\alpha^{\vee}: \ \alpha\in \Phi_{\lambda}\}$. Then
$\Phi_{\lambda}^{\vee}$ is a subroot system (cf. \cite[Defn. 12.1]{Ka}) in the dual root system $\Phi^{\vee}$.
If $\Gamma$ is a subroot system of $\Phi^{\vee}$ then let $W(\Gamma)$ be a the
subgroup of the Weyl group (for $\Phi^{\vee}$) generated by the reflections in $\Gamma$.

Our goal is to classify all $\Phi_{\lambda}^{\vee}$ (or equivalently $\Phi_{\lambda}$) such that
$\Phi_{\lambda}$ is not $W$ conjugate to $\Phi_{J}$ for any $J\subseteq \Pi$. In this process we will use the Borel-de Siebenthal Theorem \cite[Thm. 12.1]{Ka}:

\begin{theorem} Let $\Phi$ be an irreducible root system, $\Pi=\{\alpha_{1},\alpha_{2},\dots,
\alpha_{n}\}$ be a set of simple roots and $\alpha_{h}=\sum_{i=1}^{n}h_{i}\alpha_{i}$ be the highest root of $\Phi$.
Then the maximal closed subroot systems of $\Phi$ (up to $W$ conjugation) are those with the
fundamental systems:
\begin{itemize}
\item[(i)] $\Pi-I$ where $I$ consists of a subset of simple roots $\alpha_{i}$ where $h_{i}=1$,
\item[(ii)] $\{-\alpha_{h},\alpha_{1},\alpha_{2},\dots,\alpha_{n}\}-I$ where $I$ consists of a subset of
simple roots $\alpha_{i}$ where $h_{i}=p$ where $p$ is a prime.
\end{itemize}
\end{theorem}

\begin{rem} In what follows, for a given $\la$, if there does exist $w \in W$ and 
$J \subseteq \Pi$ such that $w(\Phi_{\la}) = \Phi_J$, then 
${\mathcal V}_{\mathfrak g}(\nabla_{\zeta}(\lambda))$ is determined by 
Theorem \ref{supportthm}.
\end{rem}


\section{$E_{6}$ when $3\mid l$}

We will first proceed to classify all $\Phi_{\lambda}$ (or equivalently $\Phi_{\lambda}^{\vee}$)
such that $\Phi_{\lambda}$ is not of the form $w(\Phi_{J})$ for any $w\in W$ and $J\subseteq \Delta$.
One can apply the Borel-de Siebenthal Theorem which describes maximal subroot systems
via the extended Dynkin diagrams. For $E_{6}$ there are three possibilities:
$\Phi_{\lambda}^{\vee}$ is a subroot system of either $D_{5}$, $A_{1}\times A_{5}$ or $A_{2}\times A_{2} \times A_{2}$.

One can immediately rule out $D_{5}$ because this is a subroot system of the ordinary Dynkin diagram.
In the other cases, where $\Gamma=A_{1}\times A_{5}$ or $A_{2}\times A_{2}\times A_{2}$ we have
${\mathbb Z}\Gamma/{\mathbb Z}\Phi_{\lambda}^{\vee}$ has torsion prime to $l$ so
there exists $w\in W$ such that $w(\Phi_{\lambda})=\Phi_{I}$ where $I\subseteq
\Pi(\Gamma)$. Here $\Pi(\Gamma)$ is a ``set of simple roots'' for the root system 
$\Gamma$ given by the Borel-de Siebenthal Theorem.

We can analyze both cases when $\Gamma=A_{1}\times A_{5}$ or $A_{2}\times A_{2}\times A_{2}$.
First note that if the ``additional simple root from the extended diagram'' (cf. \cite[p. 137]{Ka})
is not in $\Phi_{I}$ then $\Phi_{I}\subseteq \Phi$ and $\Phi_{\lambda}$ is $W$-conjugate to $\Phi_{I}$.
Therefore, we can assume that $I$ contains the additional simple root from the extended Dynkin diagram.
We can conclude that the possibilities of non-conjugate $\Phi_{\lambda}$ to $\Phi_{J}$ where
$J\subset \Pi$ reduces to the cases when $\Phi_{\lambda}$ is possibly:
\begin{itemize}
\item[(i)] $A_{5}\times A_{1}$,
\item[(ii)] $A_{3}\times A_{1} \times A_{1}$,
\item[(iii)] $A_{1}\times A_{1} \times A_{1} \times A_{1}$,
\item[(iv)] $A_{2}\times A_{2} \times A_{2}$.
\end{itemize}
Further reductions can be made by arguing in the following way. For example,
if $\Phi_{\lambda}$ is of type $A_{5}\times A_{1}$ then $\langle \lambda+\rho,
\alpha_{0}^{\vee} \rangle \in 3{\mathbb Z}$. Since $\langle \lambda+\rho,
\alpha_{j}^{\vee} \rangle \in 3{\mathbb Z}$ for $j\neq 2$ this forces
$\langle \lambda+\rho,2\alpha_{2}^{\vee} \rangle \in 3{\mathbb Z}$. Therefore,
$\langle \lambda+\rho,\alpha_{2}^{\vee} \rangle \in 3{\mathbb Z}$ which is a
contradiction.

If $\Phi_{\lambda}$ is of type $A_{3}\times A_{1}\times A_{1}$ then
$\langle \lambda+\rho, 2\alpha_{2}^{\vee}+2\alpha_{5}^{\vee} \rangle \in 3{\mathbb Z}$, thus
$\langle \lambda+\rho, \alpha_{2}^{\vee}+\alpha_{5}^{\vee} \rangle \in 3{\mathbb Z}$. Furthermore,
$\langle \lambda+\rho, \alpha_{4}^{\vee}\rangle \in 3{\mathbb Z}$, so
$\langle \lambda+\rho, \alpha_{2}^{\vee}+\alpha_{4}^{\vee}+\alpha_{5}^{\vee} \rangle \in 3{\mathbb Z}$
which is a contradition. One can argue in a similar way to rule out $\Phi_{\lambda}$
having type $A_{1}\times A_{1} \times A_{1} \times A_{1}$.

This reduces us to the case when $\Phi_{\lambda}$ is of type $A_{2}\times A_{2}
\times A_{2}$ which can be realized by the weight $(l-1)(1,1,1,q-1,1,1)$ where $l=3q$.
Moreover, we can use the calculations and Constrictor Method I as outlined in \cite[4.2]{UGA3} to deduce the
following result.

\begin{theorem} Let $\Phi$ be of type $E_{6}$ and let $3\mid l$. If $\lambda\in X_+$ and
$\Phi_{\lambda}$ is not $W$ conjugate to $\Phi_{J}$ for any $J\subseteq \Pi$, then
\begin{itemize}
\item[(a)] $\Phi_{\lambda}$ is of type $A_{2}\times A_{2} \times A_{2}$;
\item[(b)] ${\mathcal V}_{\mathfrak g}(\nabla_{\zeta}(\lambda))=
\overline{{\mathcal O}(A_{2}\times A_{1})}$.
\end{itemize}
\end{theorem}


\section{$F_{4}$ when $3\mid l$}

We first apply the Borel-de Siebenthal Theorem which describes maximal subroot systems
containing $\Phi_{\lambda}^{\vee}$ via the extended Dynkin diagrams. For $F_{4}$ there are three possibilities:
$A_{1}\times C_{3}$, $B_{4}$ and $A_{2}\times A_{2}$. We have to consider all three cases since
the ``additional simple root'' in the extended Dynkin diagram is involved. In these cases,
${\mathbb Z}\Gamma/{\mathbb Z}\Phi_{\lambda}^{\vee}$ has torsion prime to $l$ so
there exists $w\in W$ such that $w(\Phi_{\lambda})=\Phi_{I}$ where $I\subseteq
\Pi(\Gamma)$.

We now analyze each of these cases and we only need to consider when the
``additional simple root from the extended diagram'' is in $\Phi_{I}$. We can conclude
that the possibilities of non-conjugate $\Phi_{\lambda}$ to $\Phi_{J}$ where $J\subset \Pi$ reduces to the cases when
$\Phi_{\lambda}$ is possibly:
\begin{itemize}
\item[(i)] $C_{3}\times A_{1}$,
\item[(ii)] $C_{2}\times A_{1}$,
\item[(iii)] $A_{1}\times A_{1} \times A_{1}$,
\item[(iv)] $A_{1}\times A_{1}$,
\item[(v)] $B_{4}$,
\item[(vi)] $A_{3}$,
\item[(vii)] $A_{2}\times A_{2}$.
\end{itemize}

Further reductions can be made by arguing in the following way. For example,
if $\Phi_{\lambda}$ is of type $C_{3}\times A_{1}$ then $\langle \lambda+\rho,
\alpha_{h}^{\vee} \rangle \in 3{\mathbb Z}$. Since $\langle \lambda+\rho,
\alpha_{j}^{\vee} \rangle \in 3{\mathbb Z}$ for $j=2,3,4$ this forces
$\langle \lambda+\rho,2\alpha_{1}^{\vee} \rangle \in 3{\mathbb Z}$. Therefore,
$\langle \lambda+\rho,\alpha_{1}^{\vee} \rangle \in 3{\mathbb Z}$ which is a
contradiction. One can argue in a similar way to rule out $\Phi_{\lambda}$
in cases (ii)-(vi).

This reduces us to the case when $\Phi_{\lambda}$ is of type $A_{2}\times A_{2}$.
One can invoke Constrictor Method I as outlined in \cite[4.2]{UGA3} to deduce the
following result.

\begin{theorem} Let $\Phi$ be of type $F_{4}$ and let $3\mid l$. If $\lambda\in X_+$ and
$\Phi_{\lambda}$ is not $W$ conjugate to $\Phi_{J}$ for any $J\subseteq \Pi$, then
\begin{itemize}
\item[(a)] $\Phi_{\lambda}$ is of type $A_{2}\times A_{2}$;
\item[(b)] ${\mathcal V}_{\mathfrak g}(\nabla_{\zeta}(\lambda))=
\overline{{\mathcal O}(A_{1}\times \widetilde{A}_{2})}$.
\end{itemize}
\end{theorem}


\section{$E_{7}$ when $3\mid l$}

The Borel-de Siebenthal Theorem can be used to describes maximal subroot systems of $E_{7}$
containing $\Phi_{\lambda}^{\vee}$ via the extended Dynkin diagrams. There are four possibilities:
$E_{6}$, $A_{1}\times D_{6}$, $A_{7}$, and $A_{2}\times A_{5}$. The $E_{6}$-case reduces us
to looking at a subroot system of type $A_{2}\times A_{2}\times A_{2}$. For the other three cases
the ``additional simple root'' in the extended Dynkin diagram is involved. In the these cases,
${\mathbb Z}\Gamma/{\mathbb Z}\Phi_{\lambda}^{\vee}$ has torsion prime to $l$ so
there exists $w\in W$ such that $w(\Phi_{\lambda})=\Phi_{I}$ where $I\subseteq
\Pi(\Gamma)$.

We now analyze each of these cases and we only need to consider when the
``additional simple root from the extended diagram'' is in $\Phi_{I}$. We can conclude
that the possibilities of non-conjugate $\Phi_{\lambda}$ to $\Phi_{J}$ where $J\subset \Pi$ reduces to the cases when
$\Phi_{\lambda}$ is possibly:
\begin{itemize}
\item[(i)] $D_{6}\times A_{1}$,
\item[(ii)] $A_{1}\times A_{1}\times A_{3}\times A_{1}$,
\item[(iii)] $A_{1}\times A_{1}\times A_{1}\times A_{1}$,
\item[(iv)] $D_{4}\times A_{1}\times A_{1}$,
\item[(v)] $A_{7}$,
\item[(vi)] $A_{6}$,
\item[(vii)] $A_{3}\times A_{3}$,
\item[(viii)] $A_{2}\times A_{5}$,
\item[(ix)] $A_{2}\times A_{2}\times A_{2}$.
\end{itemize}

The same type of arguments given in the $E_{6}$ and $F_{4}$ cases reduces us
to looking at $\Phi_{\lambda}$ of type $A_{2}\times A_{5}$ and $A_{2}\times A_{2} \times A_{2}$.
We need to now invoke the Constrictor Method II to deduce the following result.

\begin{theorem} Let $\Phi$ be of type $E_{7}$ and let $3\mid l$. If $\lambda\in X_+$ and
$\Phi_{\lambda}$ is not $W$ conjugate to $\Phi_{J}$ for any $J\subseteq \Pi$, then
either $\Phi_{\lambda}$ is of type $A_{2}\times A_{5}$ or $A_{2}\times A_{2} \times A_{2}$.
\begin{itemize}
\item[(a)] If $\Phi_{\lambda}$ is of type $A_{2}\times A_{5}$, then
${\mathcal V}_{\mathfrak g}(\nabla_{\zeta}(\lambda))=\overline{{\mathcal O}(2A_{2}+A_{1})}$.
\item[(b)] If $\Phi_{\lambda}$ is of type $A_{2}\times A_{2}\times A_{2}$, then
${\mathcal V}_{\mathfrak g}(\nabla_{\zeta}(\lambda))=\overline{{\mathcal O}(A_{5}+A_{1})}$.
\end{itemize}
\end{theorem}

\begin{proof} For part (a), we use Constrictor Method I (with constrictor ${\mathcal O}(A_{3})$),
and show that $|w(\Phi_{\lambda})\cap \Phi_{J}|>0$ for all $w\in W$ and $\Phi_{J}\cong A_{3}$.
This was accomplished by using MAGMA \cite{BC, BCP}. In order to verify part (b), we use
Constrictor Method I with the orbit ${\mathcal O}(A_{5}^{\prime})$ and
Constrictor Method II with the orbit ${\mathcal O}(D_{5}(a_{1})+A_{1})$. In the latter case
we verified that $|w(\Phi_{\lambda})\cap \Phi_{J}|>2$ for all $w\in W$ and $\Phi_{J}\cong D_{5}\times A_{1}$.
\end{proof}


\section{$E_{8}$ when $3\mid l$, $5\mid l$}

First observe that we need to handle both cases simultaneously because there are cases when
both $3\mid l$ and $5\mid l$. One can apply the Borel-de Siebenthal Theorem and for $E_{8}$ there are five
possibilities for maximal subroot systems:
$D_{8}$, $A_{1}\times E_{7}$, $A_{8}$, $A_{2}\times E_{6}$, and $A_{4}\times A_{4}$. In all five cases, 
the ``additional simple root'' in the extended Dynkin diagram is involved. Furthermore, in the case
$D_{8}$, $A_{8}$, $A_{4}\times A_{4}$, ${\mathbb Z}\Gamma/{\mathbb Z}\Phi_{\lambda}^{\vee}$ has torsion prime to $l$ so
there exists $w\in W$ such that $w(\Phi_{\lambda})=\Phi_{I}$ where $I\subseteq
\Pi(\Gamma)$. For the other cases, $A_{1}\times E_{7}$, $A_{2}\times E_{6}$, we rely on our prior analysis
for the $E_{6}$ and $E_{7}$ cases.

Once again we analyze each of these cases and only need to consider when the
``additional simple root from the extended diagram'' is in $\Phi_{I}$. The possibilities of non-conjugate
$\Phi_{\lambda}$ to $\Phi_{J}$ where $J\subset \Pi$ reduces to the cases:
\begin{itemize}
\item[(i)] $D_{8}$,
\item[(ii)] $D_{6}\times A_{1}$,
\item[(iii)] $D_{4}\times A_{3}$,
\item[(iv)] $D_{4}\times A_{1}\times A_{1}$,
\item[(v)] $A_{1}\times A_{1}\times A_{5}$,
\item[(vi)] $A_{1}\times A_{1}\times A_{1}\times A_{3}$,
\item[(vii)] $A_{1}\times A_{1}\times A_{1}\times A_{1}\times A_{1}$,
\item[(viii)] $A_{1}\times E_{7}$,
\item[(ix)] $A_{1}\times A_{2}\times A_{5}$,
\item[(x)] $A_{1}\times A_{2}\times A_{2}\times A_{2}$,
\item[(xi)] $A_{8}$,
\item[(xii)] $A_{5}\times A_{2}$,
\item[(xiii)] $A_{2}\times A_{2}\times A_{2}$,
\item[(xiv)] $A_{2}\times E_{6}$,
\item[(xv)] $A_{2}\times A_{2}\times A_{2}\times A_{2}$,
\item[(xvi)] $A_{4}\times A_{4}$.
\end{itemize}

We can use the prior arguments involving ``finding additional roots'' in $\Phi_{\lambda}$ to
rule out cases (i)--(ix). This reduces us to the case when $\Phi_{\lambda}$ is of type
(x)-(xvi). The following theorems summarize our findings.

\begin{theorem}[A] Let $\Phi$ be of type $E_{8}$ and let $3\mid l$ or $5\mid l$. If $\lambda\in X_+$ and
$\Phi_{\lambda}$ is not $W$ conjugate to $\Phi_{J}$ for any $J\subseteq \Pi$, then
$\Phi_{\lambda}$ is of type
\begin{itemize}
\item[(i)] $A_{1}\times A_{2}\times A_{2}\times A_{2}$,
\item[(ii)] $A_{8}$,
\item[(iii)] $A_{5}\times A_{2}$,
\item[(iv)] $A_{2}\times A_{2}\times A_{2}$,
\item[(v)] $A_{2}\times E_{6}$,
\item[(vi)] $A_{2}\times A_{2}\times A_{2}\times A_{2}$,
\item[(vii)] $A_{4}\times A_{4}$.
\end{itemize}
\end{theorem}

\begin{theorem}[B]  Let $\Phi$ be of type $E_{8}$ and let $3\mid l$ or $5\mid l$. Suppose that
$\lambda\in X_+$ and $\Phi_{\lambda}$ is not $W$ conjugate to $\Phi_{J}$ for any $J\subseteq \Pi$.
\begin{itemize}
\item[(i)]  If $\Phi_{\lambda}$ is of type $A_{1}\times A_{2}\times A_{2}\times A_{2}$,then
${\mathcal V}_{\mathfrak g}(\nabla_{\zeta}(\lambda))=\overline{{\mathcal O}(E_{8}(b_{6}))}$.
\item[(ii)]  If $\Phi_{\lambda}$ is of type $A_{8}$, then
${\mathcal V}_{\mathfrak g}(\nabla_{\zeta}(\lambda))=\overline{{\mathcal O}(2A_{2}+2A_{1})}$.
\item[(iii)]  If $\Phi_{\lambda}$ is of type $A_{5}\times A_{2}$, then
${\mathcal V}_{\mathfrak g}(\nabla_{\zeta}(\lambda))=\overline{{\mathcal O}(E_{6}(a_{3})+A_{1})}$.
\item[(iv)]  If $\Phi_{\lambda}$ is of type $A_{2}\times A_{2}\times A_{2}$, then
${\mathcal V}_{\mathfrak g}(\nabla_{\zeta}(\lambda))=\overline{{\mathcal O}(E_{6}+A_{1})}$.
\item[(v)]  If $\Phi_{\lambda}$ is of type $A_{2}\times E_{6}$, then
${\mathcal V}_{\mathfrak g}(\nabla_{\zeta}(\lambda))=\overline{{\mathcal O}(2A_{2}+A_{1})}$.
\item[(vi)]  If $\Phi_{\lambda}$ is of type $A_{2}\times A_{2}\times A_{2}\times A_{2}$, then
${\mathcal V}_{\mathfrak g}(\nabla_{\zeta}(\lambda))=\overline{{\mathcal O}(D_{6})}$.
\item[(vi)]  If $\Phi_{\lambda}$ is of type $A_{4}\times A_{4}$, then
${\mathcal V}_{\mathfrak g}(\nabla_{\zeta}(\lambda))=\overline{{\mathcal O}(A_{4}+A_{3})}$.
\end{itemize}
\end{theorem}

\begin{proof} The cases were verified by using Constrictor Methods I and II. For example,
in case (iii), in order to get the inclusion of ${\mathcal V}_{\mathfrak g}(\nabla_{\zeta}(\lambda))$ in
$\overline{{\mathcal O}(E_{6}(a_{3})+A_{1})}$, it suffices to show that
${\mathcal V}_{\mathfrak g}(\nabla_{\zeta}(\lambda))\cap {\mathcal O}(E_6(a_{3}))=\varnothing$ and
${\mathcal V}_{\mathfrak g}(\nabla_{\zeta}(\lambda))\cap {\mathcal O}(D_6(a_{2}))=\varnothing$. From
Theorem~\ref{T:constrictor2}, it is enough to prove that
$|w(\Phi_{\lambda})\cap \Phi_{J_1}|>6$ for all $w\in W$ and $\Phi_{J_1}\cong E_6$,
and $|w(\Phi_{\lambda})\cap \Phi_{J_2}|>4$ for all $w\in W$ and $\Phi_{J_2}\cong D_{6}$ which
was verified by MAGMA.
\end{proof}


\section{Support varieties of $\Delta_{\zeta}(\lambda)$ when $l$ is bad} We remark that given the computation of ${\mathcal V}_{\zeta}(\nabla_{\zeta}(\lambda))$ for $\lambda\in X_{+}$ 
when $l$ is bad, one can use the same ideas as in the proof of Theorem~\ref{maintheoremdelta} to show that 
when $l$ is bad then 
$${\mathcal V}_{\mathfrak g}(\nabla_{\zeta}(\lambda))={\mathcal V}_{\mathfrak g}(\Delta_{\zeta}(\lambda))$$ 
for all $\lambda\in X_{+}$.

%% file: chapt9BNPP.tex
 \chapter{}

\renewcommand{\thesection}{\thechapter.\arabic{section}}

\section{Tables I}\label{tables1}\renewcommand{\thetheorem}{\thesection.\arabic{theorem}}
\renewcommand{\theequation}{\thesection.\arabic{equation}}\setcounter{equation}{0}
\setcounter{theorem}{0}

For each odd integer $l > 1$ which is not equal to a bad prime for $\Phi$,
the following tables give an element
 $w \in W$ and subset $J \subset \Pi$ such that
$w(\Phi_0) = \Phi_J$. In fact, the element $w$ is chosen so that $w(\Phi_0^+) = \Phi_J^+$, 
and identified by
means of a reduced expression in terms of the simple reflections
$s_i=s_{\al_i}$. Similar short-hand notation is used to denote the simple
roots in $J$. Also given is the type of $\Phi_0$ (or equivalently
$\Phi_J$), the Bala-Carter label for the nilpotent orbit
$\mathcal{N}(\Phi_0) = G\cdot \ul_J$, and the dimension of that
orbit.

\vskip .5cm \noindent Type $F_{4}$: \vskip .25cm
\begin{center}
\begin{tabular}{|c|c|c|c|c|c|}\hline
$l$ & $\dim {{\mathcal N}(\Phi_{0})}$ & $\Phi_{0}$ & $w$ & $J$ &
orbit \\ \hline $5$ & $40$ & $A_{2}\times A_{1}$ &
$s_2s_3s_4s_2s_3s_2s_1s_2s_3s_1s_2s_3$ & $\{1,3,4\}$ &
$F_{4}(a_{3})$ \\ \hline $7$ & $44$ & $A_{1}\times A_{1}$ &
$s_2s_1s_4s_3s_2s_3s_4s_1s_3s_2s_4s_3s_2$ & $\{1,3\}$ &
$F_{4}(a_{2})$ \\ \hline $9$ & $46$ & $A_1$ & $s_4s_2s_3s_1s_2$ &
$\{3\}$ & $F_4(a_1)$ \\ \hline $11$ & $46$ & $A_{1}$ &
$s_4s_2s_3s_1s_2s_3s_4$ & $\{3\}$ & $F_{4}(a_{1})$ \\ \hline $\geq
12$ & $48$ & $\varnothing$ & -- & $\varnothing$ & $F_{4}$ \\ \hline
\end{tabular}
\end{center}

\vskip 1cm \noindent Type $G_{2}$: \vskip .25cm
\begin{center}
\begin{tabular}{|c|c|c|c|c|c|}\hline
$l$ & $\dim {{\mathcal N}_1(\Phi_{0})}$ & $\Phi_{0}$ & $w$ & $J$ &
orbit\\ \hline $5$ & $10$ & $A_{1}$ & $s_2s_1$ & $\{1\}$ &
$G_{2}(a_{1})$ \\ \hline $\geq 6$ & $12$ & $\varnothing$ & -- &
$\varnothing$ & $G_{2}$ \\ \hline
\end{tabular}
\end{center}
\vskip .5cm

\newpage\noindent Type $E_{6}$: \vskip .25cm
\begin{center}
\begin{tabular}{|c|c|c|c|c|}\hline
$l$ & $\dim {{\mathcal N}(\Phi_{0})}$ & $\Phi_{0}$ & $J$ & orbit\\
\hline $5$ & $62$ & $A_{2}\times A_{1}\times A_{1}$ & $\{1,2,3,5\}$
& $A_{4}+A_{1}$ \\ \hline $7$ & $66$ & $A_{1}\times A_{1}\times
A_{1}$ & $\{2,3,5\}$ & $E_{6}(a_{3})$ \\ \hline $9$ & $70$ & $A_1$ &
$\{4\}$ & $E_6(a_1)$ \\ \hline $11$ & $70$ & $A_{1}$ & $\{4\}$ &
$E_{6}(a_{1})$ \\ \hline $\geq 12$ & $72$ & $\varnothing$ &
$\varnothing$ & $E_{6}$ \\ \hline
\end{tabular}
\end{center}
\vskip .5cm
\begin{center}
\begin{tabular}{|c|c|}\hline
$l$ & $w$ \\ \hline $5$ &
$s_4s_3s_2s_1s_5s_4s_2s_5s_6s_3s_5s_6s_2s_4s_1s_3s_4s_3s_2s_1s_4s_3s_1s_4s_3s_2s_1$
\\ \hline $7$ &
$s_4s_2s_1s_3s_4s_6s_5s_4s_3s_6s_1s_5s_3s_2s_4s_6s_5$ \\ \hline $9$
& $s_3s_1s_2s_5s_4s_6s_5s_3$ \\ \hline $11$ &
$s_5s_6s_2s_3s_4s_5s_1s_3s_4s_2$\\ \hline
\end{tabular}
\end{center}

\vskip 1cm
\noindent Type $E_{7}$: \vskip .25cm
\begin{center}
\begin{tabular}{|c|c|c|c|c|}\hline
$l$ & $\dim {{\mathcal N}_1(\Phi_0)}$ & $\Phi_{0}$ & $J$ & orbit\\
\hline $5$ & $106$ & $A_{3}\times A_{2}\times A_{1}$ &
$\{1,2,3,5,6,7\}$ & $A_{4}+A_{2}$ \\ \hline $7$ & $114$ &
$A_{2}\times A_{1}\times A_{1}\times A_{1} $ & $\{1,2,3,5,7\}$ &
$A_{6}$ \\ \hline $9$ & $118$ & $A_1 \times A_1 \times A_1 \times
A_1$ & $\{2,3,5,7\}$ & $E_6(a_1)$ \\ \hline $11$ & $120$ &
$A_{1}\times A_{1}\times A_{1}$ &$\{2,3,5\}$ & $E_{7}(a_{3})$ \\
\hline $13$ & $122$ & $A_{1}\times A_{1} $ & $\{4,6\}$ &
$E_{7}(a_{2})$ \\ \hline $15$ & $124$ & $A_1$ & $\{1\}$ & $E_7(a_1)$
\\ \hline $17$ & $124$ & $A_{1}$ & $\{1\}$ & $E_{7}(a_{1})$ \\
\hline $\geq 18$ & $126$ & $\varnothing$ & $\varnothing$ & $E_{7}$ \\
\hline
\end{tabular}
\end{center}
\vskip .5cm
\begin{center}
\begin{tabular}{|c|c|}\hline
$l$ & $w$ \\ \hline $5$ &
$s_4s_5s_3s_2s_4s_2s_3s_1s_3s_6s_5s_4s_3s_5s_6s_5s_3s_2s_4s_2s_3s_2s_4s_6s_5s_4s_7s_6s_4s_3s_5s_6$
\\ \hline $7$ &
$s_4s_2s_3s_6s_5s_4s_6s_3s_1s_2s_5s_4s_3s_7s_6s_5s_4s_2$ \\ \hline
$9$ &
$s_4s_2s_5s_4s_1s_6s_5s_3s_4s_5s_2s_3s_4s_7s_6s_1s_2s_3s_4s_5s_6s_7$
\\ \hline $11$ &
$s_4s_3s_2s_5s_6s_5s_7s_6s_1s_3s_4s_2s_3s_5s_4s_3s_1s_2s_6s_5s_4s_3$
\\ \hline $13$ &
$s_7s_5s_6s_2s_3s_4s_5s_1s_2s_3s_4s_5s_6s_7s_1s_2s_3s_4s_5s_6$ \\
\hline $15$ & $s_3s_4s_2s_5s_4s_3s_6s_5s_4s_2s_7s_6s_5s_4$ \\ \hline
$17$ & $s_3s_4s_2s_5s_4s_3s_6s_5s_4s_7s_6s_5s_2s_4s_3s_1$ \\ \hline
\end{tabular}
\end{center}
\vskip .5cm

\newpage
\noindent Type $E_{8}$: \vskip .25cm
\begin{center}
\begin{tabular}{|c|c|c|c|c|}\hline
$l$ & $\dim {{\mathcal N}(\Phi_{0})}$ & $\Phi_{0}$ & $J$ & orbit \\
\hline $7$ & $212$ & $A_{4}\times A_{2}\times A_{1}$ &
$\{1,2,3,5,6,7,8\}$ & $A_{6}+A_{1}$ \\ \hline $9$ & $220$ & $A_3
\times A_2 \times A_1$ & $\{1,2,4,6,7,8\}$ & $E_8(b_6)$ \\ \hline
$11$ & $224$ & $A_{2}\times A_{2}\times A_{1}\times A_{1} $ &
$\{1,2,3,5,7,8\}$ & $E_{8}(a_{6})$ \\ \hline $13$ & $228$ &
$A_{2}\times A_{1}\times A_{1}\times A_{1}$ & $\{2,3,5,6,8\}$ &
$E_{8}(a_{5})$ \\ \hline $15$ & $232$ & $A_1\times A_1 \times A_1
\times A_1$ & $\{1,4,6,8\}$ & $E_{8}(a_4)$ \\ \hline $17$ & $232$ &
$A_{1}\times A_{1} \times A_{1} \times A_{1}$ & $\{2,3,5,7\}$ &
$E_{8}(a_{4})$ \\ \hline $19$ & $234$ & $A_{1}\times A_{1}\times
A_{1}$ & $\{2,3,5\}$ & $E_{8}(a_{3})$ \\ \hline $21$ & $236$ & $A_1
\times A_1$ & $\{6,8\}$ & $E_8(a_2)$ \\ \hline $23$ & $236$ &
$A_{1}\times A_{1}$ & $\{6,8\}$ & $E_{8}(a_{2})$ \\ \hline $25$ &
$238$ & $A_1$ & $\{1\}$ & $E_8(a_1)$ \\ \hline $27$ & $238$ & $A_1$
& $\{1\}$ & $E_8(a_1)$ \\ \hline $29$ & $238$ & $A_{1}$ & $\{1\}$ &
$E_{8}(a_{1})$ \\ \hline $\geq 30$ & $240$ & $\varnothing$ &
$\varnothing$ & $E_{8}$ \\ \hline
\end{tabular}
\end{center}
\vskip.5cm
\begin{center}
\begin{tabular}{|c|c|}\hline
$l$ & $w$ \\ \hline $7$ &
$s_4s_3s_2s_4s_5s_4s_2s_1s_3s_1s_4s_2s_3s_5s_4s_6s_5s_4s_3s_1s_2s_4s_3s_7s_6s_8s_7s_5$
\\ \hline -- &
$s_4s_2s_6s_5s_4s_3s_4s_5s_6s_7s_6s_8s_7s_8s_1s_3s_2s_4s_5s_6s_4s_1s_3s_2s_4s_1s_2s_3$
\\ \hline $9$ &
$s_5s_6s_4s_5s_2s_4s_3s_4s_5s_6s_1s_2s_3s_4s_2s_5s_4s_3s_4s_2s_6s_5s_4s_7s_6s_5s_8s_7s_6s_1s_3$
\\ \hline -- &
$s_4s_2s_5s_4s_3s_4s_5s_6s_5s_4s_1s_2s_3s_4s_5s_1s_3s_4s_7s_6s_5s_2s_4s_5s_6s_7s_1s_2s_3s_4s_5$
\\ \hline $11$ &
$s_4s_2s_5s_3s_4s_3s_2s_6s_5s_4s_2s_7s_6s_3s_5s_4s_5s_6s_7s_1s_3s_4s_5s_2s_4s_6s_5s_8$
\\ \hline -- &
$s_7s_6s_7s_1s_3s_4s_5s_2s_4s_2s_6s_7s_3s_1s_3s_4s_2s_5s_4s_3s_4s_3s_1s_3s_8s_7s_6s_5$
\\ \hline $13$ &
$s_4s_2s_3s_4s_5s_7s_6s_8s_7s_5s_4s_2s_3s_4s_1s_3s_4s_2s_5s_4s_3s_1s_2s_6s_5$
\\ \hline -- &
$s_4s_3s_2s_5s_4s_3s_1s_7s_8s_6s_5s_7s_6s_4s_2s_5s_3s_4s_5s_2s_6s_3s_1s_4s_3$\\
\hline $15$ &
$s_3s_4s_2s_5s_4s_3s_6s_5s_4s_7s_8s_6s_5s_7s_6s_2s_4s_5s_2s_4s_6s_7s_5s_6s_1s_3s_8s_7$
\\ \hline -- &
$s_4s_5s_3s_4s_1s_3s_6s_5s_4s_2s_7s_6s_5s_8s_3s_4s_2s_5s_4s_3s_6s_5s_4s_2s_7s_6s_5s_4$
\\ \hline $17$ &
$s_4s_2s_5s_4s_6s_5s_7s_6s_1s_3s_4s_2s_5s_4s_8s_7s_6s_5s_3s_4s_2s_1s_3s_4s_5s_6s_7s_8s_4s_5s_6s_7s_2s_4s_5s_6s_1s_3s_4$
\\ \hline -- &
$s_2s_5s_4s_3s_1s_8s_7s_4s_5s_3s_4s_1s_3s_6s_5s_4s_2s_7s_6s_5s_4s_8s_7s_3s_4s_2s_5s_4s_3s_6s_5s_4s_2s_7s_6s_5s_4s_3s_1$\\
\hline $19$ &
$s_4s_2s_3s_4s_1s_3s_6s_5s_4s_2s_7s_6s_5s_4s_3s_1s_8s_7s_6s_5s_4s_3s_2s_4s_5s_6s_7s_8s_4s_5s_3s_4$
\\ \hline -- &
$s_1s_3s_6s_5s_4s_2s_7s_6s_5s_4s_8s_7s_6s_5s_3s_4s_2s_5s_4s_3s_6s_5s_4s_2s_7s_6s_5s_4s_3s_1s_8s_7$
\\ \hline $21$ &
$s_7s_6s_5s_4s_2s_3s_4s_5s_6s_7s_1s_3s_4s_5s_6s_2s_4s_5s_3s_4s_1s_3s_2s_4s_5s_6$
\\ \hline -- &
$s_3s_4s_2s_5s_4s_3s_6s_5s_4s_2s_7s_6s_5s_4s_3s_1s_8s_7s_6s_5$ \\
\hline $23$ &
$s_7s_6s_5s_4s_2s_3s_4s_5s_6s_7s_1s_3s_4s_5s_6s_2s_4s_5s_3s_4s_1s_3s_2s_4s_5s_6s_7s_8$
\\ \hline -- &
$s_3s_4s_2s_5s_4s_3s_6s_5s_4s_2s_7s_6s_5s_4s_3s_1s_8s_7s_6s_5s_4s_2$
\\ \hline $25$ &
$s_3s_4s_2s_5s_4s_3s_6s_5s_4s_2s_7s_6s_5s_4s_3s_1s_8s_7s_6s_5s_4s_2s_3s_4$
\\ \hline $27$ &
$s_3s_4s_2s_5s_4s_3s_6s_5s_4s_2s_7s_6s_5s_4s_3s_1s_8s_7s_6s_5s_4s_2s_3s_4s_5s_6$
\\ \hline $29$ &
$s_3s_4s_2s_5s_4s_3s_6s_5s_4s_2s_7s_6s_5s_4s_3s_1s_8s_7s_6s_5s_4s_2s_3s_4s_5s_6s_7s_8$
\\ \hline
\end{tabular}
\end{center}

\newpage


\section{Tables II}\renewcommand{\thetheorem}{\thesection.\arabic{theorem}}
\renewcommand{\theequation}{\thesection.\arabic{equation}}\label{tables2}
 The information listed below for the exceptional Lie algebras is used for
justifying the arguments given in Section \ref{exceptionalI}. For each $l$, the
element $w \in W$ referred to here is the $w$ identified in the
tables of Appendix \ref{tables1}. All weights (e.g., $w\cdot 0$) are
listed in the weight (or $\varpi$) basis.

For a given $J$ and $w$, the values of $w\cdot 0$ and $-w_{0,J}(w\cdot 0)$ are
listed.  In the arguments of Section \ref{exceptionalI} a weight $\la$ and a
sum of coroots $\delta^{\vee}$ are introduced.  The values of
$\la$, $\langle -w_{0,J}(w\cdot 0),\delta^{\vee}\rangle$, and
$\langle\la,\delta^{\vee}\rangle$ are listed.  Finally, for the root systems of
type $E_n$, a third table lists the value of $\langle -w_{0,J}(w\cdot 0),\al^{\vee}\rangle$
along with bounds on possible values for an inner product $\langle x,\al^{\vee}\rangle$
for $\al \in \Pi\backslash J$.  The object $x$ is introduced in Section \ref{exceptionalI}
and represents various possible combinations of sums of positive roots.  This information
is only needed in the arguments of Section \ref{exceptionalI} for certain values of $l$,
and only the relevant information is listed.

\vskip .5cm \noindent Type $F_{4}$: \vskip .25cm
\begin{center}
\begin{tabular}{|c|c|c|c|c|c|c|c|c|}
\hline
$l$ & $J$ & $w\cdot 0$ & $-w_{0,J}(w\cdot 0)$ & $\langle-w_{0,J}(w\cdot 0),\delta^{\vee}\rangle = (l - 1)|J|$ & $\la$ & $\langle\la,\delta^{\vee}\rangle$\\
\hline
$11$ & $\{3\}$ & $(0,-6,10,-6)$ & $(0,-4,10,-4)$ & $10$ & $(0,-4,10,-4)$ & $10$ \\
\hline
$9$ & $\{3\}$ & $(1,-5,8,-5)$ & $(-1,-3,8,-3)$ & $8$ & $(0,-4,10,-4)$ & $10$ \\
\hline
$7$ & $\{1,3\}$ & $(6,-7,6,-6)$ & $(6,-5,6,0)$ & $12$ & $(4,-5,9,-3)$ & $13$ \\
\hline
$5$ & $\{1,3,4\}$ & $(4,-8,4,4)$ & $(4,-4,4,4)$ & $12$ & $(4,-4,4,4)$ & $12$\\
\hline
\end{tabular}
\end{center}

\vskip 1cm \noindent Type $G_{2}$: \vskip .5cm
\begin{center}
\begin{tabular}{|c|c|c|c|c|c|c|c|}
\hline
$l$ & $J$ & $w\cdot 0$ & $-w_{0,J}(w\cdot 0)$ & $\langle-w_{0,J}(w\cdot 0),\delta^{\vee}\rangle = (l - 1)|J|$ & $\la$ & $\langle\la,\delta^{\vee}\rangle$\\
\hline
$5$ & $\{1\}$ & $(4,-3)$ & $(4,-1)$ & $4$ & $(4,-1)$ & $4$\\
\hline
\end{tabular}
\end{center}

\newpage

\vskip .5cm \noindent Type $E_{6}$: \vskip .25cm
\begin{center}
\begin{tabular}{|c|c|c|c|}\hline
$l$ & $J$ & $w\cdot 0$ & $-w_{0,J}(w\cdot 0)$ \\
\hline
$11$ & $\{4\}$ & $(0,-6,-6,10,-6,0)$ & $(0,-4,-4,10,-4,0)$ \\
\hline
$9$ & $\{4\}$ & $(1,-5,-5,8,-5,1)$ & $(-1,-3,-3,8,-3,-1)$ \\
\hline
$7$ & $\{2,3,5\}$ & $(-3,6,6,-11,6,-6)$ & $(-3,6,6,-7,6,0)$ \\
\hline
$5$ & $\{1,2,3,5\}$ & $(4,4,4,-10,4,-5)$ & $(4,4,4,-6,4,1)$ \\
\hline
\end{tabular}
\end{center}
\vskip.25cm
\begin{center}
\begin{tabular}{|c|c|c|c|c|c|}\hline
$l$ & $\langle-w_{0,J}(w\cdot 0),\delta^{\vee}\rangle = (l - 1)|J|$ & $\la$ & $\langle\la,\delta^{\vee}\rangle$ \\
\hline
$11$ & $10$ & $(0,-4,-4,10,-4,0)$ & $10$\\
\hline
$9$ & $8$ & $(0,-4,-4,10,-4,0)$ & $10$\\
\hline
$7$ & $18$ & $(-2,6,6,-7,6,-2)$ & $18$ \\
\hline
$5$ & $16$ & $(3,6,2,-6,6,-2)$ & $17$\\
\hline
\end{tabular}
\end{center}
\vskip.25cm
\begin{center}
\begin{tabular}{|c|c|c|c|}
\hline
$l$ & $\al$ & $\langle x,\al^{\vee}\rangle$ & $\langle -w_{0,J}(w\cdot 0),\al^{\vee}\rangle$\\
\hline
$7$ & $\al_1$ & $-4 \leq * \leq 1$ & $-3$\\
\hline
$7$ & $\al_4$ & $-7 \leq * \leq -6$ & $-7$\\
\hline
$7$ & $\al_6$ & $-4 \leq * \leq 1$ & $0$\\
\hline
\end{tabular}
\end{center}

\vskip 1cm \noindent Type $E_{7}$: \vskip .25cm
\begin{center}
\begin{tabular}{|c|c|c|c|}
\hline
$l$ & $J$ & $w\cdot 0$ & $-w_{0,J}(w\cdot 0)$ \\
\hline
$17$ & $\{1\}$ & $(16, 0, -11, 0, 0, 0, 0)$ & $(16,0,-5,0,0,0,0)$ \\
\hline
$15$ & $\{1\}$ & $(14,0,-10,0,0,1,0)$ & $(14,0,-4,0,0,0,0)$ \\
\hline
$13$ & $\{4,6\}$ & $(0,-7,-7,12,-13,12,-8)$ & $(0,-5,-5,12,-11,12,-4)$ \\
\hline
$11$ & $\{2,3,5\}$ & $(-6,10,10,-17,10,-6,0)$ & $(-4,10,10,-13,10,-4,0)$ \\
\hline
$9$ & $\{2,3,5,7\}$ & $(-5,8,8,-14,8,-8,8)$ & $(-3,8,8,-10,8,-8,8)$ \\
\hline
$7$ & $\{1,2,3,5,7\}$ & $(6,6,6,-14,6,-7,6)$ & $(6,6,6,-10,6,-5,6)$ \\
\hline
$5$ & $\{1,2,3,5,6,7\}$ & $(4,4,4,-15,4,4,4)$ & $(4,4,4,-9,4,4,4)$ \\
\hline
\end{tabular}
\end{center}
\vskip.25cm
\begin{center}
\begin{tabular}{|c|c|c|c|c|c|}\hline
$l$ & $\langle-w_{0,J}(w\cdot 0),\delta^{\vee}\rangle = (l - 1)|J|$ & $\la$ & $\langle\la,\delta^{\vee}\rangle$ \\
\hline
$17$ & $16$ & $(16,0,-5,0,0,0,0)$ & $16$\\
\hline
$15$ & $14$ & $(16,0,-5,0,0,0,0)$ & $16$\\
\hline
$13$ & $24$ & $(0,-5,-5,12,-11,12,-4)$ & $24$\\
\hline
$11$ & $30$ & $(-4,10,10,-13,10,-4)$ & $30$\\
\hline
$9$ & $32$ & $(-4,8,10,-11,8,-7,8)$ & $34$\\
\hline
$7$ & $30$ & $(5,8,4,-10,8,-7,8)$ & $33$\\
\hline
$5$ & $24$ & $(5,10,4,-10,3,0,5)$ & $27$\\
\hline
\end{tabular}
\end{center}
\vskip.25cm
\begin{center}
\begin{tabular}{|c|c|c|c|}
\hline
$l$ & $\al$ & $\langle x,\al^{\vee}\rangle$ & $\langle -w_{0,J}(w\cdot 0),\al^{\vee}\rangle$\\
\hline
$15$ & $\al_2, \al_4, \al_5, \al_6, \al_7$ & $-10 \leq * \leq 12$ & $0$\\
\hline
$15$ & $\al_3$ & $-15 \leq * \leq 4$ & $-4$\\
\hline \hline
$9$ & $\al_1$ & $-8 \leq * \leq 4$ & $-3$\\
\hline
$9$ & $\al_4$ & $-13 \leq * \leq -6$ & $-10$\\
\hline
$9$ & $\al_6$ & $-12 \leq * \leq 0$ & $-8$\\
\hline
\end{tabular}
\end{center}

\vskip .5cm \noindent Type $E_{8}$: \vskip .25cm
\begin{center}
\begin{tabular}{|c|c|c|c|}
\hline
$l$ & $J$ & $w\cdot 0$ & $-w_{0,J}(w\cdot 0)$ \\
\hline
$29$ & $\{1\}$ & $(28,0,-17,0,0,0,0,0)$ & $(28,0,-11,0,0,0,0,0)$ \\
\hline
$27$ & $\{1\}$ & $(26,0,-17,1,0,0,0,0)$ & $(26,0,-9,-1,0,0,0,0)$ \\
\hline
$25$ & $\{1\}$ & $(24,0,-15,0,0,1,0,0)$ & $(24,0,-9,0,0,-1,0,0)$ \\
\hline
$23$ & $\{6,8\}$ & $(0,0,0,0,-13,22,-25,22)$ & $(0,0,0,0,-9,22,-19,22)$ \\
\hline
$21$ & $\{6,8\}$ & $(2,0,0,0,-13,20,-21,20)$ & $(-2,0,0,0,-7,20,-19,20)$ \\
\hline
$19$ & $\{2,3,5\}$ & $(-10,18,18,-29,18,-13,4,0)$ & $(-8,18,18,-25,18,-5,-4,0)$ \\
\hline
$17$ & $\{2,3,5,7\}$ & $(-14,16,16,-25,16,-16,16,-11)$ & $(-2,16,16,-23,16,-16,16,-5)$ \\
\hline
$15$ & $\{1,4,6,8\}$ & $(14,-10,-17,14,-13,14,-13,14)$ & $(14,-4,-11,14,-15,14,-15,14)$ \\
\hline
$13$ & $\{2,3,5,6,8\}$ & $(-7,12,12,-26,12,12,-21,12)$ & $(-5,12,12,-22,12,12,-15,12)$ \\
\hline
$11$ & $\{1,2,3,5,7,8\}$ & $(10,10,10,-22,10,-18,10,10)$ & $(10,10,10,-18,10,-12,10,10)$ \\
\hline
$9$ & $\{1,2,4,6,7,8\}$ & $(8,8,-13,8,-24,8,8,8)$ & $(8,8,-11,8,-16,8,8,8)$ \\
\hline
$7$ & $\{1,2,3,5,6,7,8\}$ & $(6,6,6,-25,6,6,6,6)$ & $(6,6,6,-17,6,6,6,6)$ \\
\hline
\end{tabular}
\end{center}
\begin{center}
\begin{tabular}{|c|c|c|c|c|c|}\hline
$l$ & $\langle-w_{0,J}(w\cdot 0),\delta^{\vee}\rangle = (l - 1)|J|$ & $\la$ & $\langle\la,\delta^{\vee}\rangle$ \\
\hline
$29$ & $28$ & $(28,0,-11,0,0,0,0,0)$ & $28$\\
\hline
$27$ & $26$ & $(28,0,-11,0,0,0,0,0)$ & $28$\\
\hline
$25$ & $24$ & $(28,0,-11,0,0,0,0,0)$ & $28$\\
\hline
$23$ & $44$ & $(0,0,0,0,-8,22,-21,22)$ & $44$\\
\hline
$21$ & $40$ & $(0,0,0,0,-8,22,-21,22)$ & $44$\\
\hline
$19$ & $54$ & $(-8,18,18,-25,18,-8,0,0)$ & $54$\\
\hline
$17$ & $64$ & $(-7,16,16,-22,16,-15,16,-6)$ & $64$\\
\hline
$15$ & $56$ & $(16,-5,-15,16,-15,16,-15,16)$ & $64$\\
\hline
$13$ & $60$ & $(-7,16,16,-21,7,8,-14,16)$ & $63$\\
\hline
$11$ & $60$ & $(8,16,7,-21,16,-14,8,7)$ & $62$\\
\hline
$9$ & $48$ & $(18,7,-16,10,-14,6,0,10)$ & $51$\\
\hline
$7$ & $42$ & $(9,18,8,-21,6,0,0,9)$ & $50$\\
\hline
\end{tabular}
\end{center}

\vskip.25cm
\begin{center}
\begin{tabular}{|c|c|c|c|}
\hline
$l$ & $\al$ & $\langle x,\al^{\vee}\rangle$ & $\langle -w_{0,J}(w\cdot 0),\al^{\vee}\rangle$\\
\hline
$27$ & $\al_2, \al_5, \al_6, \al_7, \al_8$ & $-18\leq * \leq 20$ & $0$\\
\hline
$27$ & $\al_3$ & $-27 \leq * \leq 4$ & $-9$\\
\hline
$27$ & $\al_4$ & $-18\leq * \leq 20$ & $-1$\\
\hline \hline
$25$ & $\al_2, \al_4, \al_5, \al_7, \al_8$ & $-20 \leq * \leq 22$ & $0$\\
\hline
$25$ & $\al_3$ & $-27 \leq * \leq 6$ & $-9$\\
\hline
$25$ & $\al_6$ & $-20 \leq * \leq 22$ & $-1$\\
\hline \hline
$23$ & $\al_1, \al_2, \al_3, \al_4$ & $-10 \leq * \leq 12$ & $0$\\
\hline
$23$ & $\al_5$ & $-20 \leq * \leq 1$ & $-9$\\
\hline
$23$ & $\al_7$ & $-26\leq * \leq -14$ & $-19$\\
\hline \hline
$21$ & $\al_1$ & $-14 \leq * \leq 16$ & $-2$\\
\hline
$21$ & $\al_2, \al_3, \al_4$ & $-14 \leq * \leq 16$ & $0$\\
\hline
$21$ & $\al_5$ & $-24 \leq * \leq 6$ & $-7$\\
\hline
$21$ & $\al_7$ & $-26 \leq * \leq -9$ & $-19$\\
\hline \hline
$19$ & $\al_1$ & $-16 \leq * \leq 1$ & $-8$\\
\hline
$19$ & $\al_4$ & $-25 \leq * \leq -24$ & $-25$\\
\hline
$19$ & $\al_6$ & $-16\leq * \leq 1$ & $-5$\\
\hline
$19$ & $\al_7$ & $-10 \leq * \leq 12$ & $-4$\\
\hline
$19$ & $\al_8$ & $-10 \leq * \leq 12$ & $0$\\
\hline \hline
$17$ & $\al_1$ & $-14 \leq * \leq 1$ & $-2$\\
\hline
$17$ & $\al_4$ & $-25 \leq * \leq -18$ & $-23$\\
\hline
$17$ & $\al_6$ & $-20 \leq * \leq -8$ & $-16$\\
\hline
$17$ & $\al_8$ & $-14 \leq * \leq 1$ & $-5$\\
\hline \hline
$13$ & $\al_1$ & $-17 \leq * \leq 5$ & $-5$\\
\hline
$13$ & $\al_4$ & $-24 \leq * \leq -13$ & $-22$\\
\hline
$13$ & $\al_7$ & $-22 \leq * \leq -3$ & $-15$\\
\hline \hline
$11$ & $\al_4$ & $-24 \leq * \leq -15$ & $-18$\\
\hline
$11$ & $\al_6$ & $-21 \leq * \leq -5$ & $-12$\\
\hline

\end{tabular}
\end{center}

%% file: MemBNPP.bbl
\begin{thebibliography}{A}





\bibitem[\sf A]{A} H.~H.~Andersen,
The strong linkage principle for quantum groups at roots of $1$,
{\em J. Algebra}, {\sf 260}, (2003), 2--15.





\bibitem[\sf AJ]{AJ} H.~H.~Andersen, J.~C.~Jantzen,
Cohomology of induced representations for algebraic groups, {\em
Math. Ann.}, {\sf 269}, (1984), 487--525.





\bibitem[\sf AJS]{AJS} H.~H.~Andersen, J.~C.~Jantzen, W. Soergel,
Representations of quantum groups at a $p$th root of unity and of
semisimple groups in characteristic $p$, {\it Ast\'erique}, {\sf
220}, (1994).





\bibitem[\sf APW]{APW} H.~H.~Andersen, P. Polo, K. Wen,
Representations of quantum algebras, {\em Invent. Math.}, {\sf 104},
(1991), 1--59.



\bibitem[\sf AG]{AG} S. Arkhipov, D. Gaitsgory, Another realization
of the category of modules over the small quantum group, {\em Adv. Math.},
{\sf 173}, (2003), 114--143.



\bibitem[\sf ABBGM]{ABBGM} S. Arkhipov, A. Braverman, R. Bezrukavnikov,
D. Gaitsgory, I. Mirkovic, Modules over the small quantum group and semi-infinite flag manifold,
{\em Transformation Groups}, {\sf 10}, (2005), 279--362.




\bibitem[\sf ABG]{ABG} S. Arkhipov, R. Bezrukavnikov,
V. Ginzburg, Quantum groups, the loop Grassmannian, and the Springer
resolution, {\em
 J. Amer. Math. Soc.}, {\sf 17}, (2004), 595--678.







\bibitem[\sf BMMR]{BMMR} C.~P. Bendel, B. Mandler, J. Mankovecky, H. Rosenthal,
Program for verifying vanishing of line bundle cohomology, 
http://www3.uwstout.edu/faculty/bendelc/studentresearch.cfm.





\bibitem[\sf Be]{Be} R. Bezrukavnikov, Cohomology of tilting modules over
quantum groups and $t$-structures on derived categories of coherent sheaves,
{\em Invent. Math.}, {\sf 166}, (2006), 327--357.





\bibitem[\sf BC]{BC} W. Bosma, J. Cannon, {\em Handbook on Magma Functions}, Sydney University, 1996.






\bibitem[\sf BCP]{BCP} W. Bosma, J. Cannon, C. Playhoust, The Magma Algebra System I: The User Language,
{\em J. Symbolic Computation}, {\sf 3/4} no. 24, (1997), 235-265.






\bibitem [\sf Bo]{Bo} N. Bourbaki, {\it Elements of Mathematics:
 Lie Groups and Lie Algebras,} Chapters 4---6, Springer (1972).




\bibitem[\sf Br1]{Br1} A. Broer, Normal nilpotent varieties in $F_{4}$,
{\em J. Algebra}, {\sf 207}, (1998), 427--448.





\bibitem[\sf Br2]{Br2} A. Broer, Normality of some nilpotent varieties and
cohomology of line bundles on the cotangent bundle of the flag
variety (in Lie theory and Geometry, Boston), {\em Progress in
Mathematics}, {\sf 123}, Birkh{\" a}user, 1994, 1--19.





\bibitem[\sf CLNP]{CLNP} J.~F.~Carlson, Z.~Lin, D.~K.~Nakano,
B.~J.~Parshall, The restricted nullcone, {\em Cont. Math.}, {\sf
325}, (2003), 51--75.






\bibitem[\sf Car]{Car} R.~W. Carter, {\em Finite Groups of Lie type: Conjugacy
classes and complex characters}, Wiley Classic Library Edition,
1993.

\bibitem[\sf Ch]{Ch} S. Chemla, Rigid dualizing complex for quantum enveloping algebras
and algebras of generalized differential operators, {\em J. Algebra} {\bf 276} (2004), 80--102.

\bibitem[\sf C]{C} A.~L. Christophersen, {\em A Classification of the Normal
Nilpotent Varieties for Groups of Type $E_6$}, Ph.D. Thesis,
University of Aarhus, 2006.



\bibitem[\sf CPS]{CPS} E.~Cline, B.~Parshall, L.~Scott, A Mackey Imprimitivity
Theorem for algebraic groups, {\em Math. Zeit}, {\sf 182}, (1983), 447-471.



\bibitem
[\sf CM]{CM} D.~H. Collingwood, W.~M. McGovern, {\em Nilpotent
Orbits in Semisimple Lie Algebras}, Van Nostrand Reinhold, 1993.






\bibitem[\sf DK]{DK} C. de Concini, V. Kac, Representations of
quantum groups at roots of 1, {\it Prog. in Math.,} {\sf 92},
Birkh\"auser, Boston (1990), 471--506.






\bibitem[\sf DKP]{DKP} C. de Concini, V. Kac, C. Procesi,
Some quantum analogues of solvable Lie groups, {\it Geometry and
Analysis, Bombay 1992}, Tata Inst. Fund. Res., Bombay (1995), 41-65.




\bibitem[\sf DP]{DP} C. de Concini, C. Procesi, Quantum Groups,
{\it Lecture Notes in Math.,} {\sf 1565}, Springer, Berlin (1993),
31-140.





\bibitem[\sf D]{D} S. Donkin, The normality of closures of conjugacy
clases of matrices, {\em Invent. Math.}, {\sf 101}, (1990),
717--736.

\bibitem[\sf Dr1]{Dr1} C. Drupieski, {Representations and cohomology of Frobenius-Lusztig kernels},  
{\em  J. Pure and Applied Algebra}, to appear.

\bibitem[\sf Dr2]{Dr2} C. Drupieski, {On injective modules and support varieties for the small quantum group},  
{\em International Math. Research Notices}, to appear.

\bibitem[\sf Ev]{Ev} L. Evens, {\em The Cohomology of Groups},
Oxford Mathematical Monographs, Oxford University Press, 1991.

\bibitem[\sf FeW]{FeW} J. Feldvoss, S. Witherspoon, 
{Support varieties and representation type of small quantum groups}, 
{\em International Math. Research Notices}, (2010), 1346-1362. 

\bibitem[\sf Fie]{Fie} P. Fiebig, An upper bound on the exceptional 
characteristics for Lusztig's character formula, arXiv.0811.167v2.


\bibitem[\sf FP1]{FP1} E.~M.~Friedlander, B.~J.~Parshall,
Cohomology of infinitesimal and discrete groups, {\em Math. Ann.},
{\sf 273}, (1986), 353-374.



\bibitem[\sf FP2]{FP2} E.~M. Friedlander, B.~J. Parshall, Cohomology of
Lie algebras and algebraic groups, {\em Amer. J. Math.}, {\sf
108}, (1986), 235-253.




















\bibitem[\sf GK]{GK} V. Ginzburg, S. Kumar,
Cohomology of quantum groups at roots of unity, {\em Duke Math.
Journal}, {\sf 69}, (1993), 179--198.






\bibitem[\sf GW]{GW} R. Goodman, N.~R. Wallach,
{\em Representations and Invariants of the Classical Groups},
Encyclopedia of Mathematics and its Applications 68, Cambridge
University Press, 1998.



\bibitem[\sf Har]{Har} R. Hartshorne, {\em Algebraic Geometry}, 
Springer-Verlag, 1977. 

\bibitem[\sf H]{H} W.~H. Hesselink, Polarizations in the classical groups,
{\em Math. Zeit.}, {\sf 160}, (1978), 217-234.




\bibitem[\sf HS]{HS} P.~J.~Hilton, U. Stammbach, {\it A Course in
Homological Algebra}, Springer-Verlag, 1971.




\bibitem[\sf HK]{HK} J. Hong, S.-J. Kang, {\it Introduction to Quantum Groups
and Crystal Bases}, Grad. Studies in Math. {\sf 42}, Amer. Math.
Soc., 2002.



\bibitem[\sf Hum]{Hum} J.~E. Humphreys,
{\em Introduction to Lie Algebras and Representation Theory},
Springer-Verlag, 1972.


\bibitem[\sf Hum1]{Hum1} J.~E. Humphreys, {\em Conjugacy Classes in Semisimple Algebraic Groups}, Amer. Math.
Soc., 1995.



\bibitem[\sf Jan1]{Jan1} J.~C.~Jantzen, {\em Representations of
Algebraic Groups,} 2nd ed., Math. Surveys and Monographs {\sf 107}, American Mathematical Society, Providence (2003).
 






\bibitem[\sf Jan2]{Jan3} J.~.C Jantzen, {\em Lectures on Quantum
Groups}, Grad. Studies in Math., {\sf 6}, Amer. Math. Soc., 1996





\bibitem[\sf Jan3]{Jan4} J.~C. Jantzen, Nilpotent orbits
in representation theory, {\em Lie theory, Progr. Math.,} {\sf 228},
Birkh\"auser (2004), 1--211.

\bibitem[\sf Jan4]{Jan5} J.~C. Jantzen, Support varieties of Weyl modules, {\em Bull. London Math. Soc.,}
{\sf 19} (1987), 91--108.

\bibitem[\sf JR]{JR} D.~S. Johnston, R.~W. Richardson,
Conjugacy classes in parabolic subgroups of semisimple algebraic
groups. II., {\em Bull. London Math. Soc.}, {\sf 9}, (1977),
245-250.


\bibitem[\sf Ka]{Ka} R. Kane, {\em Reflection Groups and Invariant Theory}, 
Springer, 2001. 









\bibitem[\sf KP1]{KP1} H. Kraft, C. Procesi, Closures of conjugacy
classes of matrices are normal, {\em Invent. Math.}, {\sf 53},
(1979), 227--247.






\bibitem[\sf KP2]{KP2} H. Kraft, C. Procesi, On the geometry of
conjugacy classes in classical groups, {\em Comm. Math. Helv.}, {\sf
57}, (1982), 539--602.






\bibitem[\sf KLT]{KLT} S. Kumar, N. Lauritzen, J. Thomsen, Frobenius
splitting of cotangent bundles of flag varieties, {\em Invent.
Math.}, {\sf 136}, (1999), 603-621.




\bibitem[\sf Kun]{Kun} E. Kunz, {\em Introduction to Commutative
Algebra and Algebraic Geometry}, Birkhauser, 1985.



\bibitem[\sf Lac1]{Lac1} A. Lachowska, On the center of the small
quantum group, {\em J. Algebra}, {\sf 262}, (2003), 313-331.



\bibitem[\sf Lac2]{Lac2} A. Lachowska, A counterpart of the Verlinde
algebra for the small quantum group, {\em Duke Math. Journal}, {\sf 118},
(2003), 37--60.



\bibitem[\sf LS]{LS} S. Levendorski\u\i, Y. Soibelman, Algebras of
functions on compact quantum groups, Schubert cells and quantum
tori, {\em Comm. Math. Physics}, {\sf 139} (1991), 141--170.








\bibitem[\sf L1]{L1} G. Lusztig, Modular representations of quantum groups, 
{\em Cont. Math.}, {\sf 82}, (1989), 59-77. 



\bibitem[\sf L2]{L2} G. Lusztig, Finite dimensional Hopf algebras
arising from quantized universal enveloping algebras, {\em J. Amer.
Math. Soc.,} {\sf 3} no. 1, (1990), 257--296.




\bibitem[\sf Mac]{Mac} S. Mac Lane, {\em Homology}, Springer-Verlag
(Classics in Mathematics), 1995.

\bibitem[\sf MPSW]{MPSW} M. Mastnak, J. Pevtsova, P. Schauenberg, S. Witherspoon,
Cohomology of finite-dimensional pointed Hopf algebras, {\em Proc. Lond. Math. Soc. (3),}
{\sf 100} no. 2, (2010), 377--404.




\bibitem[\sf Mat]{Mat} H. Matsumura, {\em Commutative Ring Theory},
Cambridge studies in advanced mathematics, {\sf 8}, Cambridge
University Press, 1990.





\bibitem[\sf Mon]{Mon} S. Montgomery, {\em Hopf Algebras and Their Actions on Rings}, CBMS
Regional Conference Series in Mathematics, {\sf 82}, American
Mathematical Society, 1993.





\bibitem [\sf NPV]{NPV} D.~K. Nakano, B.~J. Parshall, D.~C. Vella, Support
varieties for algebraic groups, {\em J. Reine Angew. Math.}, {\sf
547}, (2002), 15--49.






\bibitem[\sf Ost]{Ost} V. Ostrik, Support varieties for quantum groups,
{\em Funct.\ Anal.\ and its Appl.}, {\sf 32}, (1998), 237--246.











\bibitem[\sf PW]{PW} B.~J. Parshall, J.~P. Wang, Cohomology of quantum
groups: the quantum dimension, {\em Canadian J. Math.}, {\sf 45},
(1993), 1276--1298.






\bibitem[\sf RH]{RH} S. Ryom-Hansen,
A $q$-analogue of Kempf's vanishing theorem, {\em Mosc. Math. J.},
{\sf 3}, (2003), 173--187.






\bibitem[\sf So1]{So1} E. Sommers, Normality of nilpotent varieties
in $E_6$, {\em J. Algebra}, {\sf 270}, (2003), 288--306.






\bibitem[\sf So2]{So2} E. Sommers, Normality of very even nilpotent varieties
in $D_{2l}$, {\em Bull. London Math. Soc.}, {\sf 37}, (2005),
351--360.





\bibitem[\sf So3]{So3} E. Sommers, private communication.

\bibitem[\sf SS]{SS} T. A. Springer, R. Steinberg, Conjugacy classes, {\em 1970 seminar on algebraic groups and related finite groups}, 
Lecture Notes in Math. {\sf 131}, Springer, Berlin, 167--266.

\bibitem[\sf St]{St} R. Steinberg, Conjugacy classes in algebraic groups, Springer Lecture
Notes in Math., {\sf 366}, Springer, (1974).


\bibitem[\sf SFB1]{SFB1} A. Suslin, E.~M. Friedlander, C.~P. Bendel,
Infinitesimal 1-parameter subgroups and cohomology, {\em Jour. Amer.
Math. Soc.}, {\sf 10}, (1997), 693-728.




\bibitem[\sf SFB2]{SFB2} A. Suslin, E.~M. Friedlander, C.~P. Bendel,
Support varieties for infinitesimal group schemes, {\em Jour. Amer.
Math. Soc.}, {\sf 10}, (1997), 729-759.





\bibitem[\sf Th]{Th} J.~F. Thomsen, Normality of certain nilpotent varieties
in positive characteristic, {\em J. Algebra}, {\sf 227}, (2000),
595--613.






\bibitem
[\sf UGA1]{UGA1} University of Georgia VIGRE Algebra Group,
Varieties of nilpotent elements for simple Lie algebras I: good
primes, {\em J.\ Algebra}, {\sf 280}, (2004), 719--737.





\bibitem
[\sf UGA2]{UGA2} University of Georgia VIGRE Algebra Group,
Varieties of nilpotent elements for simple Lie algebras II: bad
primes, {\em J.\ Algebra}, {\sf 292}, (2005), 65--99.





\bibitem
[\sf UGA3]{UGA3} University of Georgia VIGRE Algebra Group, Support
varieties for Weyl modules over bad primes, {\em J.Algebra}, {\sf
312}, (2007), 602--633.




\bibitem
[\sf W]{W} G. Warner, {\em Harmonic Analysis on Semi-Simple Lie
Groups I}, Springer-Verlag, 1972.


\bibitem
[\sf ZS]{ZS} O. Zariski and P. Samuel, {\it Commutative Algebra, I}, Springer, 1960.















\end{thebibliography}
